\documentclass[11pt,reqno]{amsart}
 
\usepackage[T1]{fontenc}
\usepackage{amssymb}
\usepackage{actuarialsymbol}
\usepackage{amsmath,amsthm}
\usepackage{amsfonts}
\usepackage{leftidx}
\usepackage{mathrsfs} 
\usepackage{enumerate}	
\usepackage{abstract}
\usepackage{stmaryrd}
\usepackage{ulem}
\usepackage{hyperref}
\usepackage{graphicx}
 \usepackage[top=0.9in, bottom=0.9in, left=1in, right=1in]{geometry}
\def\v{v}

\def\R{\mathbb{R}}
\def\Z{\mathbb{Z}}

\def\d{|\nabla|}

\def\C{\mathbb{C}^{ }}
\def\p{\partial}

\def\h{\frac{1}{2}}

\def\be{\begin{equation}}
\def\ee{\end{equation}}

\newtheorem{theorem}{Theorem}[section]

\newtheorem{lemma}{Lemma}[section]
\newtheorem{proposition}{Proposition}[section]
\theoremstyle{definition}
\newtheorem{definition}{Definition}[section]

\theoremstyle{remark}
\newtheorem{remark}{Remark}[section]

\numberwithin{equation}{section}
\begin{document}
\title[Massive Einstein-Vlasov system]{ Global stability of the Minkowski spacetime for the Einstein-Vlasov system}
 
\author{Xuecheng Wang}
\address{YMSC, Tsinghua University \& BIMSA\\ Beijing, China 100084}
\email{xuecheng@tsinghua.edu.cn, xuechengthu@outlook.com}
\thanks{Acknowledgement:\qquad  On various stages of this work, the author was benefited a lot from many discussions with Alexandru Ionescu and Benoit Pausader. The author is supported partially by NSFC-12141102, and MOST-2020YFA0713003. }
\maketitle 
\begin{abstract}
 We prove   global stability of  the  Minkowski spacetime  in the wave  coordinates  system for the massive Einstein-Vlasov system. In particular, compared with previous results by Lindblad-Taylor \cite{taylor}, in which the Vlasov part is assumed to have compact support assumption, and Fajman-Joudioux-Smulevci \cite{FJS}, in which the spacetime is assumed to be exact Schwarzschild in the exterior region, 
 we do not impose  any compact support condition for the Vlasov part and   allow a   large class of non-isotropic perturbations for the metric part, which decay at rate $\langle r\rangle^{-1+}$ towards space infinity. 

\end{abstract}
\tableofcontents
\section{Introduction}

We are interested in  global stability problem for the Einstein field equations. One of   the   astrophysically important matter fields is dust like collisionless matter, which is described by the Vlasov equation. There is a vast literature on this topic, we refer readers to   Andr\'easson \cite{Andreasson},  Rendall \cite{Rendall}, and Stewart \cite{Stewart} for more on the backgrounds.

Mathematically speaking,  we consider the stability of the Minkowski spacetime $(\R^{1+3}, \eta)$ in the presence of the massive collisionless  particles, where $\eta=\textup{diag}(-1,1,1,1)$. The   unknown Lorentzian manifold   $(M,g)$   has signature $(-,+,+,+)$.  We assume that the particle system satisfies classical equations of motion (called on shell), which means that   the energy-momentum relation as   defined in (\ref{massshell}) holds for the four-momentum $(v^0,v^1, v^2, v^3).$ The \textit{mass shell} $\mathcal{M}$, which is a subbundle of the tangent space $TM$,  is defined as follows,
\be\label{massshell}
\mathcal{M}:=\{(t,x^i, v^0, v^i)\in TM: (v^0, v^i) \textup{\, is future directed},  v^{\nu} v_{\nu}=-1\},
\ee
 where, as in the standard notation,  Greek indices run over $0,1,2,3$ and Latin indices run over $1,2,3$. In the definition of mass shell, we have already  normalized   the speed of light $c$    and the mass of    particles   $m$ to be $1$.

The  Einstein-Vlasov system,  which describes the dynamics of spacetime and the distribution function of particles, reads as follows, 
\be\label{EV}
\left\{\begin{array}{l}
\textup{Ric}(g) - \frac{1}{2}R(g) g = T(f)\\ 
\\
T_{g}f =0,
\end{array}\right.
\ee 
where $\textup{Ric}(g)$ is the Ricci curvature tensor of  $g$, $R(g)$ is the scalar curvature of $g$,  $T(f)$ is the energy-momentum tensor associated with the distribution of particles $f: \mathcal{M} \longrightarrow \R_{+}$, $T_g$ is    the generator of the geodesic flow   of the Lorentzian manifold  $(M, g)$.

 The main goal of this paper   is to prove the global stability of $(\eta, 0)$ in wave  coordinates  for the {massive} Einstein-Vlasov system for {unrestricted} initial data. Compared with previous results by  Lindblad-Taylor \cite{taylor}, in which the metric is assumed to be asymptotically flat  and the Vlasov part has compact support assumption, and Fajman-Joudioux-Smulevci \cite{FJS}, in which the  metric is assumed to be    exact Schwarzchild spacetime in the exterior region and there is no compact support assumption on the Vlasov part, we do not require any compact support assumption for the Vlasov part and   allow a   large class of non-isotropic perturbations. Indeed, our assumptions on the metric on the initial slice are  essentially of the  following type
\be
g_{\alpha\beta}= \eta_{\alpha\beta} + \epsilon \mathcal{O}(\langle r\rangle^{-1+}), \quad \p g_{\alpha\beta}= \epsilon \mathcal{O}(\langle r\rangle^{-2+}), 
\ee
which is slightly better than the asymptotically flat assumption of the metric.

We use the framework  developed by Ionescu-Pausader  \cite{IP} in   the study of the Einstein-Klein-Gordon system,  which allows us to combine the strength of both the vector fields method and the Fourier method. The Fourier analysis is mainly used in the  analysis of resonant sets  and a special “designer” norm ($Z$-norm), to prove the decay estimate over time for both the metric part and the density-type function for particles  in lower space regularity norms.

\subsection{The PDE formulation in wave coordinates}

Due to the differeomorphism  covariance  of the Einstein-Vlasov system,  we  choose the so-called wave coordinates.  By working in this wave coordinate system, the Einstein equations are reduced into a system of quasilinear wave equations. We refer readers to \cite{IP} and   Wald \cite{Wald} for related discussions.

Recall that  the  wave  coordinate system $\{x^\mu\}_{\mu},$  satisfies the following equation
\be\label{wavecoordinate}
\square_g x^{\mu}=0.
\ee 
 
From (\ref{wavecoordinate}), we have  
\be\label{wavecoordinatecondition}
\Gamma_{\mu}= g^{\alpha \beta} g_{\nu \mu} \Gamma_{\alpha\beta}^{\nu}=g^{\alpha\beta} \Gamma_{\mu\alpha\beta}=\h g^{\alpha\beta}\big(\p_{\alpha} g_{\beta\mu} + \p_{\beta} g_{\alpha \mu} - \p_{\mu} g_{\alpha \beta} \big)=0. 
\ee
Recall that the connection coefficients are defined as follow,
\[
\Gamma_{\mu\alpha \beta}:=g(\p_{\mu}, \nabla_{\p_\beta} \p_{\alpha} )= \h\big(\p_{\alpha} g_{\beta\mu} + \p_{\beta} g_{\alpha \mu} - \p_{\mu} g_{\alpha \beta} \big),
\]
where $\nabla$ is the torsion free  covariant derivative. In the wave coordinates system, we have
\be\label{riccicurvature}
2R_{\alpha\beta}= -g^{\gamma\mu}\p_{\gamma}\p_{\mu}g_{\alpha\beta} + Q_{\alpha\beta} + P_{\alpha\beta}, 
\ee
where
\[
Q_{\alpha\beta}=    \frac{1}{2}g^{\gamma\mu}g^{\lambda\nu}\big[ 
  \big(\p_{\lambda} g_{\alpha\mu}\p_{\nu} g_{\gamma\beta}+ \p_{\mu}g_{\lambda\alpha} \p_{\gamma} g_{\nu\beta} \big) 
+ 2 \big(  \p_{\alpha} g_{\mu\nu} \p_{\lambda} g_{\gamma\beta} -  \p_{\lambda} g_{\mu\nu} \p_{\alpha} g_{\gamma\beta} 
\]
\[
 +   \p_{\beta} g_{\mu\nu} \p_{\lambda} g_{\gamma\alpha} -  \p_{\lambda} g_{\mu\nu} \p_{\beta} g_{\gamma\alpha}\big)+ \big( \p_{\mu} g_{\lambda\nu} \p_{\alpha} g_{\gamma\beta}  +    \p_{\mu} g_{\lambda\nu} \p_{\beta} g_{\gamma\alpha} -\p_{\alpha} g_{\lambda\nu} \p_{\mu} g_{\gamma\beta}  -    \p_{\beta} g_{\lambda\nu} \p_{\mu} g_{\gamma\alpha} \big) \]
\be\label{nullstructure} 
 + \big( \p_{\nu} g_{\alpha\lambda} \p_{\mu} g_{\gamma\beta}- \p_{\mu} g_{\lambda\alpha} \p_{\nu}g_{\gamma\beta}  +     \p_{\nu } g_{\beta\lambda} \p_{\mu} g_{\gamma\alpha}
   -\p_{\nu} g_{\alpha\gamma} \p_{\mu}g_{\lambda\beta}  \big)\big],
\ee
\be\label{weaknull} 
P_{\alpha\beta}=   \frac{1}{2}g^{\gamma\mu}g^{\lambda\nu}\big[ \big( -\p_{\alpha}g_{\mu\lambda} \p_{\beta} g_{\gamma\nu} \big)+\h \p_{\alpha}g_{\mu\gamma} \p_{\beta}g_{\lambda\nu}\big]. 
\ee

We consider the Vlasov part in (\ref{EV}).  Let $\pi:\mathcal{M}\longrightarrow M$ denote  the canonical projection. The energy-momentum tensor in local coordinate system is given as follows, 
 \be\label{march19eqn6}
T_{\mu\nu}(f)(x^\alpha) = \int_{\pi^{-1}(x^\alpha)} f(x^\alpha, v^\alpha) v_{\mu} v_{\nu} \frac{1}{\sqrt{| \det g|} v^0} d v_i,
 \ee
 where the volume form  $1/\big({\sqrt{| \det g|} v^0}\big) d v_i$ is the induced volume form of the hypersurface $\pi^{-1}(x^\alpha)\subset T_{ x^\alpha }M $,   the tangent space $ T_{ x^\alpha }M $ is endowed with the metric $g_{\mu\nu}(x^\alpha) d v^{\mu} d v^\nu$. 

We derive a precise formula for the geodesic spray $T_g$.  Let  $(U, x^\alpha)  $ be a local  coordinate system on $M$,  we define
 the  conjugate coordinate system $(x^\alpha, v^\alpha)$ on $TM$, where $(x^\alpha, v^\alpha)$ denotes the point $ v^\alpha \p_{x^\alpha}|_{(x^\alpha)}\in TM.$   Note that massive particles travel along timelike  geodesics. Suppose that  the geodesic is parametrized by $x^\alpha: (-\varepsilon, \varepsilon) \longrightarrow M$ for some $\varepsilon \in\R_{+}$,    we have 
\[
f(x^\alpha(s), v^1(s), v^2(s), v^3(s)) = f(x^\alpha(0), v^1(0), v^2(0), v^3(0)).
\]
Therefore, we have
\be\label{vlasovequation}
 \frac{d x^\alpha}{d t} \p_{x^\alpha} f + \frac{d v^i(t)}{d t} \p_{v^i} f=0, \qquad \Longrightarrow \big( v^\alpha \p_{x^\alpha} - \Gamma^i_{\alpha\beta} v^{\alpha} v^{\beta}\p_{v^i} \big) f =0, 
\ee
and 
 the geodesic spray $T_g$  can be written as follows, 
 \be\label{vlasovup}
T_g= v^{\alpha}\p_{x^\alpha}- v^\alpha v^\beta \Gamma_{\alpha \beta}^i \p_{v^i}, \quad \Gamma^i_{\alpha \beta}= g^{\mu i } \Gamma_{\mu\alpha\beta}= \h g^{\mu i }\big(\p_{\alpha} g_{\mu \beta}  + \p_\beta g_{\mu\alpha}- \p_{\mu} g_{\alpha\beta} \big). 
 \ee

 To find a  favorable decomposition  inside the acceleration term in \eqref{vlasovequation}, which is  stated in Lemma \ref{innerproduct},  we choose to work in the cotangent bundle   and change coordinates system as follows, 
\[
(x^\alpha, v^i) \longrightarrow (x^\alpha, v_i), \qquad \tilde{f}(x^\alpha, v_i):= f(x^\alpha, v^i(x^\alpha, v_j)),
\]
where $v^i(x^\alpha, v_j):=  g^{i\alpha} v_\alpha.$  In the new coordinates system $(x^\alpha, v_i)$, the Vlasov equation (\ref{vlasovequation}) reads as follows, 
\be\label{vlasov}
v^\alpha \p_{x^\alpha} \tilde{f}- \h v_\alpha \v_\beta \p_{x^j} g^{\alpha \beta}\p_{v_j} \tilde{f}=0. 
\ee
Without causing any confusion,  we still use $f$ instead of $\tilde{f}$ to denote the distribution function of particles in terms of the coordinate system $(x^\alpha, v_i)$.

Moreover, from (\ref{EV}) and (\ref{march19eqn6}),   we derive the following equation for the scalar curvature, 
\be\label{april3eqn3}
-R= g^{\mu\nu} T_{\mu\nu}(f) = - \int_{\pi^{-1}(x^\alpha)} f(x^\alpha, v_i)  \frac{1}{\sqrt{| \det g|} v^0} d v_i:=T(f). 
\ee

 \subsubsection{The initial data set and the constraint equations}

 Assume that initially we have a spacelike hypersurface $(\Sigma, \bar{g}, k)$, where $\bar{g}$ is the  Riemannian metric on $\Sigma$, $k$ is the second fundamental form of the surface $\Sigma$, and initial particle distribution function $f_0: \pi^{-1}(\Sigma)\longrightarrow \R_{+}, $  where  $\pi: \mathcal{M}\longrightarrow M$ is the natural projection. The following constraint equations must hold on $\Sigma$,  
\be\label{constraintofini}
\overline{div} k_j -(d \overline{tr} k)_j = T_{0j}(f_0), \qquad \overline{R}+ (\overline{tr} k)^2 - |k|_{\bar{g}}^2 =2 T_{00}(f_0) , 
\ee
where $\overline{div}$, $\overline{tr}$, $\overline{R}$ denote the divergence operator, trace operator, and the scalar curvature of $\bar{g}$ respectively.

 

\subsection{Previous results}

Global stability of physical solutions is an important topic in General Relativity. The ground breaking work of  Christodoulou-Klainerman \cite{Christodoulou1}  shows  the global nonlinear stability of the Minkowski spacetime for the Einstein-vacuum equation,   see also  Klainerman-Nicol\`o \cite{Klainerman15}, Lindblad-Rodnianski \cite{Lindblad2,Lindblad3}, Bieri\cite{Bieri1}, Speck  \cite{Speck} for various extensions on this line of research for the Einstein-vacuum equation. Moreover, from the work of    Lindblad-Rodnianski \cite{Lindblad2,Lindblad3}, in which they gave an alternative and  conceptually   simpler proof of the stability of Minkowski spacetime in the  wave coordinates system, we know that    Einstein equations in harmonic coordinates satisfy  weak null condition. This observation plays a key role in later study of   Einstein-scalar field systems.

More recently,   global stability  of the Minkowski spacetime have also been proved for other coupled Einstein field equations, e.g.,  Zisper \cite{Zipser} considered   the Einstein-Maxwell system,  Speck \cite{Speck}  considered  the 
 Einstein-nonlinear electromagnetic system, Lefloch-Ma \cite{LeFloch1,LeFloch2} and Wang \cite{Qian} considered the   Einstein-Klein-Gordon system for the restricted data, which coincides with the Schwarzschild in the exterior region, Ionescu-Pausader \cite{IP} considered the   Einstein-Klein-Gordon system for the \emph{unrestricted} data.

 For the Einstein-Vlasov system, Bigorgne-Fajman-Joudioux-Smulevici-Thaller    \cite{Bigorgne} considered the massless case for general initial data, in which the particles travel at the same speed of light. For the massive case, Lindblad-Taylor \cite{taylor} considered the initial data with   the compact support in Vlasov part and asymptotically flat metric,  Fajman-Joudioux-Smulevci \cite{FJS} considered the initial data with the assumption  of the exact Schwarzschild in the exterior region and without compact support assumption on the Vlasov part. 

 Due to the finite speed of propagation, the use of  restricted initial data  coupled with the hyperbolic foliation method, leads to significant simplifications of the global analysis, particularly at the level of proving decay estimates.

 Our goal in this paper are to   extend the scope of the framework developed by  Ionescu-Pausader in \cite{IP} and to   work with \emph{unrestricted} initial data for the massive case. 

\subsection{The Main result   }

To state main results of this paper, we first introduce several function spaces. Let $N_0=200, \gamma=  10^{-22} $, we   define the energy spaces $E_{ l},  l\in[0, N_0]\cap \Z$, by the following norms,  
\be\label{energyespace}
\| f\|_{E_{ l}}:= \big(\sum_{k\in \Z} 2^{-k_{-} +2 \gamma k_{-}+2 l k_{+}} \|   P_k (f) \|_{L^2}^2\big)^{1/2}.
\ee
 
We   state the main theorem as follows.
\begin{theorem}\label{maintheorem}
Let $\delta:=10^{-20}$, $N_0=200, N_1=10^3N_0$, $\Sigma=\{(t,x)\in \R^4: t=0\}$,  $ ( \bar{g}, k,f_0)$ be an initial data of the Einstein-Vlasov system  \eqref{EV}  on $\Sigma$ such that the constrain equations in \eqref{constraintofini}  are satisfied. There exists a sufficiently small constant $\epsilon_0$ such that if the following smallness condition holds, 
\[
\sum_{n, l\in \Z_+, n+l\leq N_0}\sum_{\Gamma\in \{S, L_i, \Omega_{ij}\},|\alpha|\leq n  }\| (\p_t + i \d )\Gamma^\alpha (g_{\alpha\beta}-\eta_{\alpha\beta})\big|_{t=0} \|_{E_{ l}}   
\]
\[
+ \sum_{|\alpha_1|+|\alpha_2|\leq N_0} \int_{\R_x^3} \int_{\R_v^3} (1+|x|+|v|^{1/\delta})^{2N_1} |\nabla_x^{\alpha_1} \nabla_v^{\alpha_2} f_0(x,v)|^2  d x d v 
\leq \epsilon_0,
\]
then there exist  globally defined wave coordinates $\{x^{\mu} \}$ and  a unique future geodesically complete solution $(g_{\alpha\beta}, f)$ of the Einstein-Vlasov system  \eqref{EV}  in $M:=\{(t,x)\in \R^4, t\geq 0\}$, with the given initial data $(  \bar{g}, k,f_0)$ on $\Sigma$.    
\end{theorem}

\subsection{Main ideas of the proof}

The mechanism of the small data global regularity for the Einstein-Vlasov system   consists of two parts:
\begin{enumerate}
\item[(i)] Control of the $L^2$-type high order energy, which involves the high order Sobolev norm and the vector fields. For the Vlasov part, it also involves well-chosen weight functions to be propagated adapted to the hierarchy of space derivatives and vector fields.

\item[(ii)] Prove   almost sharp $L^\infty_x$-type decay estimates for the perturbed metric and the density type function of the Vlasov part. 

\end{enumerate}

We  introduce  some main ideas used in the above two main steps.

\subsubsection{Energy estimate for the wave part}

As in \cite{IP}, to exploit the benefit of the harmonic gauge condition, which allows us to identify clearly the  null structure inside the Einstein's equation,   we also use the double Hodge-decomposition variables $\{F, \underline{F}, \omega_j, \Omega_j, \rho, \vartheta_{mn}\}$.

Comparing with the Einstein-Klein-Gordon system studied in \cite{IP}, an extra benefit is that the scaling vector field $S$ commutes with the Einstein-Vlasov system.  We exploit this benefit and  set the hierarchy for the energy functionals based on the total number of space derivatives and the total number of the vector fields applied on the profiles of the metric part.

To keep the discussion concise and highlight the main difficulties in the Einstein-Vlasov system,  we refer readers to \cite{IP} for more details on the discussion of the   wave-wave type interaction  that comes from purely the Einstein-vacuum equations.

We focus on the  discussion of the main ideas used in the wave-Vlasov type interaction in the Einstein's equation, which mainly comes from the energy-momentum tensor determined by the Vlasov part. 

Firstly, due to the linear density type function in the energy-momentum tensor, we modify the profile of the metric, which roughly defined as follows $U:=(\p_t -i\d) h_{\alpha\beta}$. Roughly  speaking, we have
\[
(\p_t+i \d) U = \rho(f) + \textup{nonlinear}, \quad \rho(f):= \int_{\R^3} f(t,x,v) d v,
\]
where $\rho(f)$ is a  density type function that is linear in $f$.

To capture  the nonlinear effect of $\rho(f)$  over time, we modified the profile $U$ as follows 
\[
\widetilde{U}:= U + T(f), \quad T(f):=- \int_{\R^3} \int_{\R^3} e^{i x\cdot \xi} \frac{\widehat{f}(t, \xi, v)}{i(|\xi|- \hat{v}\cdot \xi)}d \xi d v, \quad \hat{v}:= \frac{v}{\sqrt{1+|v|^2}}. 
\]
As a result, we remove the linear density and have 
\[
(\p_t + i \d) \widetilde{U}= \textup{nonlinear}. 
\]

In certain sense, we view   $\widetilde{U}$ as the metric part  less effected by the Vlasov part and view $ U $ to be the sum of two parts: (i) the modified profile part $\widetilde{U}$; (ii) the linear density type part, which depends purely on the distribution function of the Vlasov part.

Motivated by this idea, we allow the Vlasov part to grow faster than the the modified profile in our bootstrap assumption. Roughly speaking,  the modified profile grows at rate $\langle t \rangle^{\delta_1}$ and  the Vlasov part grows  at rate $\langle t\rangle^{ \delta_0}$, where $\delta_1\ll \delta_0. $

There are two main ideas used in the energy estimate of the wave-Vlasov type interaction, which are also the main reasons why we can let the Vlasov part grows faster.

  Firstly, thanks to the polynomial decay rate of the distribution function, we know that either the distribution function is small, i.e., $|x|+|v|$ large,  or the distance with respect to the light cone is large, in which case the   metric part decays faster. More precisely, we have
\[
\langle |t|-|x+t\hat{v}|\rangle  \langle |x |\rangle  \langle|v|\rangle^{ 2}\gtrsim \langle t \rangle. 
\]
Therefore, if the frequency of the wave part is not too small,   the decay rates are improved such that they are not near the critical ``$1/\langle t \rangle$'' rate. 

Secondly, for the very low frequency of the wave part case, in which  there is no gain  from the distance to the light cone, we observe that this is not the time resonance case. Intuitively speaking,  we have the following type of phase associated with the wave-Vlasov type interaction, 
\[
|\xi|-\mu|\eta|-\hat{v}\cdot(\xi-\eta), \quad \mu\in\{+,-\}.
\]
Note that, let $t\in [2^{m-1}, 2^m], $ $\langle v \rangle \leq 2^{m/10}, |\eta|\approx 2^{-m}, |\xi|\approx 1,$ we have 
\[
 \big| |\xi|-\mu|\eta|-\hat{v}\cdot(\xi-\eta)\big|\gtrsim |\xi|\langle v \rangle^{-2}.
\]
Therefore, we can do form transformation to  exploit the non-time-resonance property.  As a result, we gain at least $\langle t\rangle^{-1/2}$, which is sufficient.

\subsubsection{Energy estimate for the Vlasov  part}

For the Vlasov part, we use the framework used by the   author in the study of the Vlasov-Nordstr\"om system \cite{wang} and the Vlasov-Maxwell system \cite{wang2}. We propagate   polynomially weighted $L^2_{x,v}$-type norm of  two sets of vector fields for the Vlasov part.

 The first set of vector fields    commutes with the Einstein-Vlasov system. The second set of vector fields, which is only applied to the Vlasov part, helps to control the $\nabla_v f $, which is the main difficulty of energy estimate for the Vlasov equation because $\nabla_v$ does not commute with the linear transport operator $\p_t +\hat{v}\cdot \nabla_x$. We should point out that, when the second set of   vector fields unavoidably acts on the wave part, it can be  reduced to the classical vector fields we propagate for the wave part, see Lemma \ref{decompositionofderivatives}. 

 Moreover,  based on the total number of space derivatives and the total number of the vector fields applied on the profiles of the Vlasov part, we also set   the hierarchy for the energy of the Vlasov part. 

To be less technical,   one key   observation is  that   there exists null structure for the nonlinearity of the Vlasov part. Roughly speaking, we encounter the following type of term in the energy estimate of the Vlasov part, 
\[
\int_{t_1}^{t_2}\int_{\R^3} \int_{\R^3}  {u}_1(t,x,v) \omega(x,v) U^{\mu}(t, x+t\hat{v}) u_2(t,x,v) d x d  v, \quad \forall \mu\in\{+, -\},
\]
where $\omega(x,v)$ is a weight function, $u_1(t,x,v)$ is the profile of the Vlasov part with some vector field,   $u_2(t,x,v)$ is the profile of the Vlasov part with some other vector field,  $U(t,\cdot)$ is the profile of the metric part, and $U^{+}:=U, U^{-}:=\overline{U}.$  

Let 
\[
\mathcal{N}(t,x,v):= u_1(t,x,v) \omega(x,v) u_2(t,x,v), \quad V:=e^{i t \d} U. 
\]
On the Fourier side, we have
\[
\int_{t_1}^{t_2}\int_{\R^3} \int_{\R^3}  {u}_1(t,x,v) \omega(x,v) U^{\mu}(t, x+t\hat{v}) u_2(t,x,v) d x d  v 
\]
\[
=\int_{t_1}^{t_2}\int_{\R^3} \int_{\R^3} e^{-i t\mu |\xi| + i t\hat{v}\cdot \xi} \widehat{V^{\mu}}(t, \xi) \widehat{\mathcal{N}}(t,-\xi,v)  d \xi dv d t. 
\]
 Therefore, the time resonance set is given by $\{(\xi, v): |\xi|- \mu \hat{v}\cdot \xi =0\}$.

 Roughly speaking, for the acceleration force in the Vlasov equation, see (\ref{vlasov}),    we have 
 \be 
   \hat{v}_{\mu}\hat{v}_{\nu} P_k(   h_{\mu\nu}) = \sum_{U\in\{ \textup{profiles of }F, \underline{F}, \omega_l,  \vartheta_{mn}\}}
  \sum_{\mu\in\{+,-\}} \big(T_{k;\mu}^{\tilde{h}}(v) U    \big)^{\mu}+  \textup{error terms}
  \ee
where   the symbol $T_{k;\mu}^{\tilde{h}}(v)(\xi), \tilde{h}\in \{F, \underline{F}, \omega_l, \vartheta_{mn}\},   $ of the Fourier multiplier operator $T_{k;\mu}^{\tilde{h}}(v)$   vanishes if $|\xi|- \mu \hat{v}\cdot \xi =0$, see \eqref{april26eqn2} for   details. 

Therefore, we can do normal form transformation to exploit the benefit of null structure.

\subsubsection{  $Z$-norm estimates}

The Z-norm method, with different choices of the norm itself, depending on the problem, was used  in several small data global regularity problems, see \cite{Deng,Deng2,Guo1,IP,IP1,IP2,IP3}.  The $Z$-norm method essentially depends on identifying the ``correct'' designer $Z$-norm   to prove sharp decay estimates for the metric part and the density-type function.

 For the good parts, i.e., $\{F, \omega_j, \vartheta_{ij}\}$, we show that their $Z$-norms do not grow over time, which implies that the good parts decay sharply at $1/\langle t\rangle $ rate over time. 
 For the bad part $\underline{F}$, it grows at most  at rate $\langle t\rangle^{\delta_1}$. 

 We show that the wave-Vlasov type interaction plays minor role in the $Z$-norm analysis. As a result, the modified scattering property of the metric part obtained in \cite{IP} is also valid for the Einstein-Vlasov system.

 For the Vlasov part, as suggested from the linear decay estimate of the density type function, see Lemma \ref{decayestimateofdensity},   the $Z$-norm we used for the Vlasov part  is same as in  \cite{wang,wang2}. 

 Roughly speaking, we show that the zero-frequency of the profile as well as its $v$-derivatives do not grow over time. As a result, the density function and its derivatives decay sharply.

Because  of the zero-output-frequency,  which ensures that the size of the input-frequency of the metric part and  the size of the input-frequency of the Vlasov part are same, the loss from the normal form transformation can be recovered by the symbol of  the Vlasov part.   As a result, the decay rate of nonlinearity are not at $1/\langle t \rangle$ critical rate. Hence     the  $Z$-norm of  the Vlasov part is uniformly bounded over time.

\subsection{Notation  }

For any $x\in \cup_{n}\R^n$, we use the Japanese bracket $\langle x \rangle$ to denote $(1+|x|^2)^{1/2}$.  We  fix an even smooth function $\tilde{\psi}:\R \rightarrow [0,1]$, which is supported in $[-3/2,3/2]$ and equals to ``$1$'' in $[-5/4, 5/4]$. For any $k\in \mathbb{Z}$, we define the cutoff functions $\psi_k, \psi_{\leq k}, \psi_{\geq k}:\cup_{n=1,3}\R^n\longrightarrow \R$ as follows, 
\[
\psi_{k}(x) := \tilde{\psi}(|x|/2^k) -\tilde{\psi}(|x|/2^{k-1}), \quad \psi_{\leq k}(x):= \tilde{\psi}(|x|/2^k)=\sum_{l\leq k}\psi_{l}(x), \quad \psi_{\geq k}(x):= 1-\psi_{\leq k-1}(x).
\]

Let $P_k$, $P_{\leq k}$, and $P_{\geq k}$ be the Fourier multiplier operators with symbols $\psi_{k}(\xi),\psi_{\leq k}(\xi), $ and $\psi_{\geq k}(\xi)$ respectively.  We use $f^{+}$ to denote $f$ and use $f^{-}$ to denote $\bar{f}$. For $k\in\Z$, we use $f_k$ to denote $P_k f$.

  For an integer $k\in\mathbb{Z}$, we use $k_{+}$ to denote $\max\{k,0\}$ and  use $k_{-}$ to denote $\min\{k,0\}$.    For any unit vectors $u, v\in \mathbb{S}^2$, we use $\angle(u, v)$ to denote the angle between $u$ and $v$ and use the convention that   $\angle(u, v)\in[0,\pi] $.   For any $k\in \Z$, we define the following set of integers, 
\be\label{2020april10eqn1}
\begin{split} 
\chi_k^1&:=\{(k_1,k_2): k_1, k_2\in \Z, |k_1-k_2|\leq 10, k\leq k_1+10\}, \\ 
\chi_k^2&:=\{(k_1,k_2): k_1, k_2\in \Z, k_1\leq k_2-10,  |k-k_2|\leq 5\}, \\ 
\chi_k^3&:=\{(k_1,k_2): k_1, k_2\in \Z, k_2\leq k_1-10,  |k-k_1|\leq 5\}.
\end{split}
\ee 

Given a symbol $a=a(x, \xi):\R^3\times \R^3\longrightarrow \C $, we define the Weyl quantization operator $T_a$ as follows, 
\be\label{weylquantization}
T_a f(x):= \int_{\R^3} \int_{\R^3} e^{i x\cdot \xi} \mathcal{F}_{x}[a](\xi-\eta, \frac{\xi+\eta}{2}) \hat{f}(\eta) \psi_{\leq-100}(\frac{|\xi-\eta|}{|\xi+\eta|}) d \eta d \xi.
\ee

For any $a\in \R$, we define  the  $\mathcal{S}^a$-class of symbols as follows, 
\be\label{symbolclas}
\mathcal{S}^a:=\{m(\xi): \sup_{\alpha\in \Z_+^3} \sup_{k\in \Z} 2^{-a k + |\alpha| k } \| \mathcal{F}^{-1}_{\xi\rightarrow x}[\nabla_\xi^\alpha m(\xi)\psi_k(\xi)](x)\|_{L^1_x}< \infty \}.
\ee

For any $k\in\mathbb{Z}_{+}$, we use $\Lambda_{k}[\mathcal{N}]$ to denote the $k$-th order term of the nonlinearity $\mathcal{N}$, where the Taylor expansion are in terms of the perturbed metric $h_{\alpha\beta}$ and the distribution function $f$ of particles. Similarly, we use $\Lambda_{\geq k}[\mathcal{N}]:=\sum_{l\geq k } \Lambda_{l}[\mathcal{N}]$ to denote the $k$-th and higher order terms of the nonlinearity $\mathcal{N}$.

\subsection{Organization}

The plan of this paper is organized as follows, 
\begin{enumerate}
  \item[$\bullet$] In section \ref{prel}, we introduce the set-up of the problem, classify the nonlinearities of both the metric part and the Vlasov part, and state our main bootstrap assumption. 
\item[$\bullet$]In section \ref{fixedtimeestimate}, we  estimate $L^2$-type and $L^\infty$-type norms for some basic quantities and reduce the weighted norm estimate to the energy estimates of the wave part and the Vlasov part. 
\item[$\bullet$]   
In section \ref{energyestimatewave}, we estimate the increment of energy for the perturbed metric components. 
 \item[$\bullet$]  In section \ref{Znormestimate}, we estimate the increment of $Z$-norm of the metric component. 
 \item[$\bullet$]   In section \ref{energyestimateVlasov}, we estimate the increment of energy for the Vlasov part  

\item[$\bullet$] In section \ref{ZnormVlas}, we estimate the increment of  $Z$-norm for the Vlasov part.
\end{enumerate}

\section{Preliminary}\label{prel}

 \subsection{Vector fields for the Einstein-Vlasov system}

 For the wave-Vlasov  type coupled system,  we have the following joint vector fields, 
\be\label{april9eqn111}
P :=\{S:=t\p_t + x\cdot \nabla_x, \Omega_{i,j}:=x^i\p_{x^j}-x^j\p_{x^i} , L_i:= t\p_{x^i} + x^i \p_t  \},
\ee
\be\label{april9eqn112}
\mathcal{P}  :=\{t\p_t + x\cdot \nabla_x,\tilde{\Omega}_{i,j}:= x^i\p_{x^j}-x^j\p_{x^i} +v_i\p_{v_j}-v_j\p_{v_i}, \tilde{L}_i:= t\p_{x^i} + x^i \p_t+ \sqrt{1+|v|^2}\p_{v_i} \},
\ee
where vector fields in $P$ are classic vector fields commute with the wave operator and vector fields in $\mathcal{P}$ commute with the linear operator $\p_t +\hat{v}\cdot\nabla$ of the Vlasov part.

As stated in the following Lemma, by using the above defined vector fields, we observe a favorable structure for the nonlinearity of the Vlasov equation (\ref{vlasov}).

\begin{lemma}\label{innerproduct}
The following decomposition holds for the following special form, 
\be\label{march31eqn1}
 \nabla_x \phi(t,x)\cdot \nabla_v f(t,x,v)= \big(1+|v|^2 \big)^{-1/2}\big(\p_{x^i}\phi\tilde{L}_i f -   L_i \phi \p_{x^i}f+    \p_t \phi S f   -   S\phi \p_{t}f\big).
\ee 
\end{lemma}
\begin{proof}
Note that
\[
\nabla_x \phi(t,x)\cdot \nabla_v f(t,x,v)= \p_{x^i}\phi \p_{v_i} f =  \big(1+|v|^2 \big)^{-1/2}\big( \p_{x^i}\phi\tilde{L}_i f  -  \p_{x^i}\phi(t\p_{x^i}+x^i\p_t)f \big)\]
\[
=  \big(1+|v|^2 \big)^{-1/2}\big( \p_{x^i}\phi\tilde{L}_i f -  t \p_{x^i}\phi \p_{x^i}f -  x^i\p_{x^i}\phi \p_{t}f \big)=\big(1+|v|^2 \big)^{-1/2}\big( \p_{x^i}\phi\tilde{L}_i f -  L_i \phi \p_{x^i}f 
\]
\[
+    \p_t \phi x^i \p_{x^i}f -  S\phi \p_{t}f   
 +  t\p_t\phi \p_{t}f \big)
=  \big(1+|v|^2 \big)^{-1/2}\big(  \p_{x_i}\phi\tilde{L}_i f -   L_i \phi \p_{x_i}f-  S\phi \p_{t}f +   \p_t \phi S f     \big). 
\]
Hence finishing the proof of the desired equality (\ref{march31eqn1}). 
\end{proof}

The key point of the decomposition (\ref{march31eqn1}) is that, as $\nabla_v $ does not commute with the transport operator $\p_t + \hat{v}\cdot \nabla_x $, generally speaking, we can not control $\nabla_v f$ directly. However, thanks to the decomposition \eqref{march31eqn1}, all terms on the right hand side of \eqref{march31eqn1} depends only on the commutable  vector fields.

\begin{definition}\label{vectorfieldscla}
For any non-negative integer $n\in \mathbb{Z}_{+}$, we define the set of vector fields $P^{}_{n}$ as follows, 
\be\label{july13eqn1}
\begin{split}
P^{ }_{n}&:=\{ \Gamma^{\alpha} , \quad \Gamma\in \{\Omega_{ij},S, L_i, i,j\in\{1,2,3\}\} ,   |\alpha| \leq  n \},\\
  \mathcal{P}^{ }_{n}&:=\{ \tilde{\Gamma}^{\alpha} , \quad \tilde{\Gamma}\in \{\tilde{\Omega}_{ij},S, \tilde{L}_i, i,j\in\{1,2,3\}\} ,   |\alpha| \leq  n \}.
\end{split}
\ee

For $L_1=\nabla_x^{\alpha_1}\Gamma_1^{\beta_1}  , L_2=\nabla_x^{\alpha_2} \Gamma_2^{\beta_2} \in P_{n}^{   }$, we define $L_1\prec L_2$ ($L_1\preceq   L_2$) if    $|\alpha_1| \leq |\alpha_2| $,  $|\beta_1| \leq|\beta_2| $, and $|\alpha_1|+|\beta_1|<  |\alpha_2|+|\beta_2|$($|\alpha_1|+|\beta_1|\leq  |\alpha_2|+|\beta_2|$). For any $L=\nabla_x^\alpha \Gamma^\beta$, we let $|L|:=|\beta|$ and $\|L\|:=|\alpha|+|\beta|$.  

We can define $\mathcal{L}_1\prec\mathcal{L}_2(\mathcal{L}_1\preceq\mathcal{L}_2),$ $\mathcal{L}_i\in \mathcal{P}_{n}^{  },i=1,2 $, and $|\mathcal{L}|$ in a similar fashion. For any $L\in P^{ }_{n}$, there exists a uniquely determined vector field $\mathcal{L}\in \mathcal{P}^{ }_{n}$ such that  for any smooth function $h(t,x)$, we have $L h = \mathcal{L} h $.  We use the convention that $L$ and its uniquely determined  vector field $\mathcal{L}$ are treated same in the $\prec$ and $\preceq$ relations. For   convenience  in notation, we define $P_0:=\{Id\}$ and $ \mathcal{P}_0:=\{Id\}$.
\end{definition}

Now, we compute the equations satisfied by the distribution function with vector fields act on it.  Recall (\ref{vlasov}). We have
\be\label{july5eqn41}
(\p_t + \hat{v}\cdot \nabla_x ) f  =  - \Lambda_{\geq 1}[(v^0)^{-1}] v\cdot\nabla_ x f  -  \h  (v^0)^{-1}  v_{\alpha} v_{\beta} \p_{j} g^{\alpha\beta} \p_{v_j} f. 
\ee  
Note that,  for any $\mathcal{L}\in \mathcal{P}_n$, we have
\[
\mathcal{L}\big(       \p_{j} g^{\alpha\beta  } \p_{v_j} f\big) = \sum_{\mathcal{L}_1, \mathcal{L}_2 \in \mathcal{P}_a^b, \mathcal{L}_1 \circ \mathcal{L}_2 \preceq   {\mathcal{L}}}  c_{\mathcal{L}_1\mathcal{L}_2 }^{   {\mathcal{L}};j,l}(v)  \p_{j} \mathcal{L}_1 g^{\alpha\beta  } \p_{v_l}\mathcal{L}_2 f. 
\]

Alternatively, 
from the equality (\ref{march31eqn1}) in Lemma \ref{innerproduct}, the following equality holds for some uniquely determined coefficients,
\[
\mathcal{L}\big(       \p_{j} g^{\alpha\beta  } \p_{v_j} f\big) =  \mathcal{L}\big(      \big(1+|v|^2 \big)^{-1/2}\big(\p_{x^i}  g^{\alpha\beta } \tilde{L}_i f -   L_i g^{\alpha\beta  }  \p_{x^i}f+    \p_t  g^{\alpha\beta }  S f   -   S g^{\alpha\beta }  \p_{t}f\big)
\]
\be\label{2020april7eqn1}
=   \sum_{
\begin{subarray}{c}
 \mathcal{L}_1 \circ \mathcal{L}_2 \preceq  {\mathcal{L}},  
  ( \Gamma_1, \Gamma_2)\in P\times\{\p_{\alpha}\} \cup \{\p_{\alpha}\}\times \mathcal{P}
 \end{subarray}} c_{\mathcal{L}_1\mathcal{L}_2 }^{ \mathcal{L};\Gamma_1,\Gamma_2}(v) \Gamma_1 \mathcal{L}_1 f   \Gamma_2 \mathcal{L}_2 g^{\alpha\beta}, \quad | c_{\mathcal{L}_1\mathcal{L}_2 }^{ \mathcal{L};\Gamma_1,\Gamma_2}(v)  |\lesssim (1+|v|)^{-1}. 
\ee 
 Moreover,  to highlight the role of extra derivatives in quasilinear terms, we point out that the following equalities hold, 
\be\label{2020april7eqn2}
   c_{id\mathcal{L}  }^{   {\mathcal{L}};j,l}(v)  \p_{j}   g^{\alpha\beta } \p_{v_l}\mathcal{L}  f= \sum_{ \Gamma_1\in P\cup\{\p_{\alpha}\}, \Gamma_2\in \mathcal{P}\cup\{\p_{\alpha} \} }c_{  id   {\mathcal{L}}}^{  {\mathcal{L}};\Gamma_1,\Gamma_2}(v) \Gamma_1   \mathcal{L}   f  \Gamma_2   g^{\alpha\beta} =  \nabla_x g^{\alpha\beta}\cdot\nabla_v  {\mathcal{L}} f , 
\ee
\be\label{2020april7eqn3}
c_{\mathcal{L} id }^{   {\mathcal{L}};j,l}(v)  \p_{j} \mathcal{L}  g^{\alpha\beta } \p_{v_l}  f=\sum_{ \Gamma_1\in P\cup\{\p_{\alpha}\}, \Gamma_2\in \mathcal{P}\cup\{\p_{\alpha}\}}c_{      {\mathcal{L}}id}^{   {\mathcal{L}};\Gamma_1,\Gamma_2}(v) \Gamma_1 f  \Gamma_2 \mathcal{L}  g^{\alpha\beta}=  \nabla_x  {\mathcal{L}} g^{\alpha\beta}\cdot\nabla_v  f .
\ee
Therefore, from the above equalities   (\ref{july5eqn41}--\ref{2020april7eqn3}),     we can write the Vlasov equation schematically as follows, 
\be\label{2020april2eqn1}
   (\p_t + \hat{v}\cdot \nabla_v) \mathcal{L} f =   \nabla_x\cdot \mathcal{N}_1^{\mathcal{L}}(t,x,v) +\nabla_v \cdot \mathcal{N}_2^{\mathcal{L}}(t,x,v)+ \mathcal{N}_3^{\mathcal{L}}(t,x,v)+ \mathcal{N}_4^{\mathcal{L}}(t,x,v) , 
\ee
where
\be\label{2020april7eqn4}
\mathcal{N}_1^{\mathcal{L}} = - \Lambda_{\geq 1}[(v^0)^{-1}] v  \mathcal{L}f, \quad \mathcal{N}_2^{\mathcal{L}}(t,x,v)= -  \h  (v^0)^{-1}v_{\alpha} v_{\beta} \nabla_x g^{\alpha\beta}  \mathcal{L} f,  
\ee
\[
 \mathcal{N}_3^{\mathcal{L}}= -  \h  (v^0)^{-1} v_{\alpha} v_{\beta}   \nabla_x \mathcal{L} g^{\alpha\beta}\cdot \nabla_v  f =\nabla_v \cdot \widetilde{ \mathcal{N}_3^{\mathcal{L};1} }+ \widetilde{ \mathcal{N}_3^{\mathcal{L};2} },\quad  \widetilde{ \mathcal{N}_3^{\mathcal{L};1} }:= -  \h  (v^0)^{-1} v_{\alpha} v_{\beta}   \nabla_x \mathcal{L} g^{\alpha\beta}  f 
\]
\[
 \widetilde{ \mathcal{N}_3^{\mathcal{L};2} }=  \h  \nabla_v\big((v^0)^{-1} v_{\alpha} v_{\beta}\big)  \cdot \nabla_x \mathcal{L} g^{\alpha\beta}  f,
\]
\[
   \mathcal{N}_4^{\mathcal{L}}  =\sum_{ \mathcal{L}_1 \circ \mathcal{L}_2\preceq  {\mathcal{L}}, \mathcal{L}_2\prec \mathcal{L}}   \tilde{c}_{\mathcal{L}_1\mathcal{L}_2}^{\mathcal{L}}(v) \Lambda_{\geq 1}[\mathcal{L}_1(v^0)^{-1}]  \cdot \nabla_x \mathcal{L}_2 f +v\cdot \nabla_x \Lambda_{\geq 1}[(v^0)^{-1}](t, x )  \mathcal{L}f (t,x,v) \]
 
\be\label{oct1eqn21}
 +\sum_{
\begin{subarray}{c}
 \mathcal{L}_1 \circ \mathcal{L}_2\circ \mathcal{L}_3 \preceq  {\mathcal{L}}, 
  \mathcal{L}_2\prec {\mathcal{L}} ,   ( \Gamma_1, \Gamma_2)\in P\times\{\p_{\alpha}\} \cup \{\p_{\alpha}\}\times \mathcal{P}
 \end{subarray}} \tilde{c}_{\mathcal{L}_1\mathcal{L}_2\mathcal{L}_3 }^{  {\mathcal{L}};\Gamma_1,\Gamma_2}(v)     \Gamma_1 \mathcal{L}_1 g^{\alpha\beta}\Gamma_2 \mathcal{L}_2 f  \mathcal{L}_3\big(  ({2v^0})^{-1}  v_{\alpha} v_{\beta}  \big),
\ee
 where the coefficients $ \tilde{c}_{\mathcal{L}_1\mathcal{L}_2}^{\mathcal{L}}(v)$ and $ \tilde{c}_{\mathcal{L}_1\mathcal{L}_2\mathcal{L}_3 }^{  {\mathcal{L}};\Gamma_1,\Gamma_2}(v)  $ satisfy the following estimate, 
\be\label{2020april22eqn21}
(1+|v|)| \tilde{c}_{\mathcal{L}_1\mathcal{L}_2\mathcal{L}_3 }^{  {\mathcal{L}};\Gamma_1,\Gamma_2}(v)  | +|\tilde{c}_{\mathcal{L}_1\mathcal{L}_2}^{\mathcal{L}}(v)|\lesssim 1. 
\ee

  For   any vector $\mathcal{L}\in \mathcal{P}^{ }_n$, we define
$u^{\mathcal{L}}(t, x, v):= { \mathcal{L}}  f (t, x+\hat{v}t, v )$ to be  the profile of the Vlasov part, which is the pull-back of the nonlinear solution along the linear flow.

Recall   (\ref{2020april2eqn1}). As a result of direct computations,   we rewrite the nonlinearity of $ \p_t u^{\mathcal{L}}(t, x,v)$  in terms of   profiles as follows,
\be\label{2020april7eqn6}
 \p_t u^{\mathcal{L}}(t, x,v)=  \nabla_x \cdot \mathfrak{N}_1^{\mathcal{L
 }}(t,x,v) +  D_v\cdot  \mathfrak{N}_2^{\mathcal{L
 }}(t,x,v) +   \mathfrak{N}_3^{\mathcal{L
 }}(t,x,v) +   \mathfrak{N}_4^{\mathcal{L
 }}(t,x,v),  
\ee
where
 \be\label{2020july9eqn5}
 D_v:=\nabla_v -t \nabla_v \hat{v}\cdot\nabla_x,\quad \mathfrak{N}_i^{\mathcal{L
 }}(t,x,v)=  \mathcal{N}_i^{\mathcal{L}}(t,x+t\hat{v},v), \quad i\in\{1,\cdots,4\}.
 \ee

To control the energy of  $\nabla_v^\alpha u^{\mathcal{L}}(t,x,v)$,  which provides the control of the $L^\infty_x$-norm of the density type functions, 
we  also utilize a  second set of vector fields for the Vlasov equation, see \cite{wang} for more details about the   the construction of the   second set  of vector fields. More precisely, we define
\be
K_v:= \nabla_v - \sqrt{1+|v|^2} \omega(x,v)\nabla_v  \hat{v}\cdot \nabla_x , 
\ee
\be\label{eqq10}
 S^v: = \tilde{v} \cdot \nabla_v,\quad  S^x:= \tilde{v}  \cdot \nabla_x,  \quad \Omega^v_i= \tilde{V}_i \cdot \nabla_v, \quad   \Omega^x_i= \tilde{V}_i \cdot \nabla_x, 
\ee
\be\label{sepeqq2}
   \widehat{S}^v:= \tilde{v}\cdot K_v= S^v - \frac{\omega(x,v)}{ {1+|v|^2}} S^x, \quad \widehat{\Omega}^{v}_i:=  \tilde{V}_i \cdot K_v = \Omega_i^v - \omega(x,v) \Omega_i^x,\quad  K_{v_i}:=K_v \cdot e_i, 
\ee
where
\be\label{eqn16}
\omega(x, v)=  \psi_{\geq 0}(|x|^2 +(x\cdot v)^2)\big(x\cdot v +\sqrt{(x\cdot {v})^2  + |x|^2 }\big), 
\ee
\be 
  \tilde{V}_i=e_i\times \tilde{v}, \quad \tilde{v}:=\frac{v}{|v|}, \quad \tilde{v}_i:=\tilde{v}\cdot e_i, \quad \hat{v}_i:=\hat{v}\cdot e_i,\quad e_1:=(1,0,0),  e_2:=(0,1,0),   e_3:=(0,0,1). 
\ee
With the above defined vector fields, the following favorable decomposition holds, 
\be\label{2020april21eqn11}
D_v=\sum_{\iota\in \mathcal{S}, |\iota|=1} c_{\iota}(t,x,v) \Lambda^{\iota}, \quad | c_{\iota}(t,x,v)|\lesssim (1+|v|)| |t|-|x+t\hat{v}||. 
\ee

We use the following notation to   represent the above vector fields uniformly, 
\be\label{dec26eqn5}
\begin{split}
 &  \Gamma_1= \psi_{\geq 1}(|v|) \widehat{S}^v, \quad \Gamma_2:=\psi_{\geq 1}(|v|)S^x, \quad   \Gamma_{i+2}:=\psi_{\geq 1}(|v|)\widehat{\Omega}_i^v, \\ 
 & \Gamma_{i+5}:=\psi_{\geq 1}(|v|) \Omega^x_i,\quad   \Gamma_{i+8}:=\psi_{\leq 0}(|v|) K_{v_i},\quad \Gamma_{i+11}:=\psi_{\leq 0}(|v|) \p_{x_i}, \quad i=1,2,3, \\ 
  & \mathcal{K}:=\{\vec{e}:\,\,\vec{e}\in \{0,1\}^{14}, |\vec{e}|=0,1\},\quad \vec{0}:=(0,\cdots, 0), 
\vec{e}_i:=(0,\cdots,\overbrace{1}^{\text{$i$-th} },\cdots,0),\\
 & \textit{if\,\,} \vec{0}, \vec{e}_i\in \mathcal{K}, \quad  
\Lambda^{\vec{ {0}}}:=Id,              \quad \Lambda^{\vec{e}_i}:=\Gamma_i,\quad \mathcal{S}:=\cup_{k\in\mathbb{N}_{+}}\mathcal{K}^{k},\quad  \forall e, f\in \mathcal{S}, \Lambda^{e\circ f}:= \Lambda^{e}\Lambda^{f}.
\end{split} 
\ee

For    any $\beta\in \mathcal{S}$, we have $\beta\thicksim \iota_1\circ \cdots \iota_{|\beta|},  \iota_i\in \mathcal{K}/\{\vec{0}\}$. We let $|\beta|$ to denote total number of derivatives. Moreover, we define  the index ``$c_{ }(\beta)$''  to distinguishes the velocity derivative in radial direction from the  the velocity derivative in rotational  directions and define the index ``$\tilde{c}(\beta)$'' to count the total number of derivatives in $v$ in $\Lambda^\beta$  as follow, 
\be\label{countingnumber}
 \begin{split}
 &c_{ }(\beta)= \sum_{i=1,\cdots, |\beta|}c_{ }(\iota_i), \quad   c_{ }(\iota) =\left\{\begin{array}{ll}
1 & \textup{if\,} \Lambda^\iota \thicksim  \psi_{\geq 1}(|v|) \widehat{S}^v  \\
-1 & \textup{if\,} \Lambda^\iota \thicksim  \psi_{\geq 1}(|v|) \widehat{\Omega}^v_i, i\in\{1,2,3\} \\
0 & \textup{otherwise} \\
\end{array}\right. ,  \iota\in  \mathcal{K}/\{\vec{0}\}, \\ 
&\tilde{c}(\beta):=|\{ \iota_i: \,\, \beta\sim \iota_1\circ\iota_2\circ\cdots \iota_{|\beta|},   \iota_i\sim   \psi_{\geq 1}(|v|) \widehat{S}^v, \textup{or},\, \psi_{\geq 1}(|v|) \widehat{\Omega}_i^v, \textup{or},\,\psi_{\leq 0}(|v|) K_{v_j}, j\in\{1,2,3\} \}|. 
\end{split}
\ee

To compute the   equation satisfied by $\Lambda^\alpha u^{\mathcal{L}}$, by using  the following Lemma, we know that, if the second set of vector fields act on the perturbed metric component, they can be represented by  the classic vector fields in $P$, see \eqref{april9eqn111}.

\begin{lemma}\label{decompositionofderivatives}
The following identity holds for any $\rho \in \mathcal{S},  $  
\be\label{sepeqn610}
\Lambda^{\rho}\big( f  (t,x+\hat{v}t  )\big)=\sum_{L\in P_{|\rho|}}  {c}_{ \rho}^{L} (t,x,v) L f (t,x+\hat{v} t) ,
  \ee
where  the coefficients $ {c}_{ \rho}^{L} (x,v) $, $ L\in  P_{|\rho|}$,  satisfy the following estimate,
\be\label{sepeqn88}
|t\p_t  {c}_{ \rho}^{L} (t,x,v)|  + |  {c}_{ \rho}^{L} (t,x,v) |  \lesssim  (1+|x|)^{|\rho|-|L|}  (1+|v|)^{  |\rho|-| L|-c_{ }(\rho)}.  
\ee
Moreover, the following rough estimate holds for any $\kappa\in \mathcal{S}, $
\be\label{noveq781}
| \Lambda^{\kappa}\big( {c}_{ \rho}^{L} (t,x,v)\big)| \lesssim (1+|x|)^{|\kappa|+|\rho|-|L| } (1+|v|)^{|\kappa|+ |\rho|-|L| -c_{ }(\rho) }   .
\ee
 \end{lemma}
\begin{proof}
See \cite{wang}[Lemma 4.1]
\end{proof}

From the equation satisfied by $u^{\mathcal{L}}$ in (\ref{2020april7eqn6}) and the equality (\ref{sepeqn610}) in Lemma \ref{decompositionofderivatives}, we  classify the equation satisfied by $\Lambda^{\rho}u^{\mathcal{L}}$  as follows, 
\be\label{2020april8eqn31}
\p_t \Lambda^{\rho}u^{\mathcal{L}}(t,x,v)= \nabla_x \cdot \mathfrak{N}_{\rho,1}^{\mathcal{L}}(t,x,v) + D_v\cdot  \mathfrak{N}_{\rho,2}^{\mathcal{L}}(t,x,v) +  \mathfrak{N}_{\rho,3}^{\mathcal{L}}(t,x,v) +  \mathfrak{N}_{\rho,4}^{\mathcal{L}}(t,x,v),
\ee
where $\mathfrak{N}_{\rho,1}^{\mathcal{L}}(t,x,v)  $ and $  \mathfrak{N}_{\rho,2}^{\mathcal{L}}(t,x,v)$ reveal the symmetric structure of the equation $\p_t \Lambda^{\rho}u^{\mathcal{L}}(t,x,v)$, in which all vector fields act on the profile of the Vlasov part, $\mathfrak{N}_{\rho,3}^{\mathcal{L}}(t,x,v) $ denotes terms in which all  vector fields act on the metric part, lastly, $\mathfrak{N}_{\rho,4}^{\mathcal{L}}(t,x,v)$ denotes  the rest terms. 

More precisely, we have 
\be\label{2020april8eqn32}
\mathfrak{N}_{\rho,1}^{\mathcal{L}}(t,x,v) = - \Lambda_{\geq 1}[(v^0)^{-1}](t, x+t\hat{v}) v  \Lambda^{\rho} u^{\mathcal{L}}(t,x,v),
\ee
\be\label{2020april8eqn33}
 \mathfrak{N}_{\rho,2}^{\mathcal{L
 }}(t,x,v) =  -  \h \big( (v^0)^{-1}   v_{\alpha} v_{\beta} \nabla_x g^{\alpha\beta}\big)(t, x+t\hat{v})  \Lambda^{\rho}u^{\mathcal{L}}(t,x,v) ,
\ee
\be\label{2022may16eqn1}
 \mathfrak{N}_{\rho,3}^{\mathcal{L
 }}(t,x,v) =  \sum_{\begin{subarray}{c}
 \tilde{L}\in \cup_{\alpha\in \Z_+^3, l\in \Z_+}\nabla_x^\alpha P_l\\  |\tilde{L}|\leq  \tilde{c}(\rho),  |
 \alpha|\leq |\rho|- \tilde{c}(\rho) 
 \end{subarray}} 
  \h  c^{\tilde{L} }_\rho(x,v) \big( (v^0)^{-1}   v_{\alpha} v_{\beta} \nabla_x \tilde{L} \mathcal{L} g^{\alpha\beta}\big)(t, x+t\hat{v}) \cdot D_v u^{ }(t,x,v), 
 \ee
\[
 \mathfrak{N}_{\rho,4}^{\mathcal{L
 }}(t,x,v) =  \sum_{\begin{subarray}{c}
   \rho_1 \circ \rho_2\preceq \rho,   \tilde{L}\in \cup_{ \alpha\in \Z_+^3, l\in \Z_+}\nabla_x^\alpha P_l,    
  \mathcal{L}_1 \circ \mathcal{L}_2\preceq  {\mathcal{L}}  \\ 
\mathcal{L}_1\prec \mathcal{L},  |\tilde{L}|\leq  \tilde{c}(\rho_2), |\alpha|\leq |\rho_2|-\tilde{c}(\rho_2)
 \end{subarray}}   \tilde{c}_{\mathcal{L}_1\mathcal{L}_2;\tilde{L}}^{\mathcal{L};\rho,\rho_1,\rho_2}(x,v)  \cdot \nabla_x \Lambda^{\rho_1}  u^{\mathcal{L}_1}(t,x,v) 
 \]
 \[
 \times  \Lambda_{\geq 1}[\tilde{L}\mathcal{L}_2(v^0)^{-1}](t, x+t\hat{v},v) \]
 \[
   +\sum_{
\begin{subarray}{c}
 \mathcal{L}_1 \circ \mathcal{L}_2\circ \mathcal{L}_3 \preceq  {\mathcal{L}}, 
  \rho_1\circ\rho_2\circ \rho_3\preceq \rho,  \tilde{L}_i  \in \cup_{ \alpha_i\in \Z_+^3, l\in \Z_+}\nabla_x^{\alpha_i} P_l,  \\
   |\tilde{L}_i|\leq \tilde{c}(\rho_{i+1}), |\alpha_i|\leq  |\rho_{i+1}|-\tilde{c}(\rho_{i+1}), i\in\{1,2\},\\ 
   |\tilde{c}(\rho_i)|+|\mathcal{L}_i|<|\rho|+|\mathcal{L}|, ( \Gamma_1, \Gamma_2)\in P\times\{\p_{\alpha}\} \cup \{\p_{\alpha}\}\times \mathcal{P}\\ 
 \end{subarray}} c_{\mathcal{L}_1\mathcal{L}_2\mathcal{L}_3;\tilde{L}_1, \tilde{L}_2}^{  {\mathcal{L}};\Gamma_1,\Gamma_2;\rho, \rho_1, \rho_2}(x,v)    \tilde{L}_1 \Gamma_2 \mathcal{L}_2 g^{\alpha\beta}(t,x+t\hat{v}) 
 \]
 \[
   \times \Lambda^{\rho_1} u^{\Gamma_1 \mathcal{L}_1}(t,x,v)   \tilde{L}_2 \mathcal{L}_3\big(  ({2v^0})^{-1} { v_{\alpha} v_{\beta}} \big)(t,x+t\hat{v}, v)  + v\cdot \nabla_x \Lambda_{\geq 1}[(v^0)^{-1}](t, x+t\hat{v})  
 \]
\be\label{2020april8eqn35}
\times \Lambda^{\rho} u^{\mathcal{L}}(t,x,v)  +\h D_v\cdot\big(\big( (v^0)^{-1}   v_{\alpha} v_{\beta} \nabla_x g^{\alpha\beta}\big)(t, x+t\hat{v})\big)  \Lambda^{\rho}u^{\mathcal{L}}(t,x,v) , 
\ee 
where, from the estimates (\ref{sepeqn88}) and (\ref{noveq781})  in Lemma \ref{decompositionofderivatives}, the coefficients satisfy the following point-wise estimates, 
\be\label{2020july30eqn1}
\begin{split}
 |c^{\tilde{L} }_\rho(x,v) |&\lesssim (1+|x|)^{|\rho |-|\tilde{L}|}(1+|v|)^{|\rho |-|\tilde{L}|- c(\rho )}, \\
 | \tilde{c}_{\mathcal{L}_1\mathcal{L}_2;\tilde{L}}^{\mathcal{L};\rho,\rho_1,\rho_2}(x,v)  | &\lesssim (1+|x|)^{|\rho_2|-|\tilde{L}|}(1+|v|)^{1+|\rho_2|-|\tilde{L}|- c(\rho_2)},\\
 | c_{\mathcal{L}_1\mathcal{L}_2\mathcal{L}_3;\tilde{L}_1, \tilde{L}_2}^{  {\mathcal{L}};\Gamma_1,\Gamma_2;\rho, \rho_1, \rho_2,\rho_3}(x,v)    | & \lesssim (1+|x|)^{|\rho_2|+|\rho_3|-|\tilde{L}_1|-|\tilde{L}_2|}(1+|v|)^{ |\rho_2|+|\rho_3|-|\tilde{L}_1|-|\tilde{L}_2|-c(\rho_2)-c(\rho_3)}. 
\end{split}
\ee
\subsection{Double Hodge decomposition of the metric tensor}

Firstly, we rewrite the Einstein equations in terms of the perturbation with respect to the Minkowski spacetime. Define
\be\label{march24eqn11}
h_{\alpha\beta}= g_{\alpha\beta} - \eta_{\alpha\beta},\qquad \Longrightarrow   g^{\alpha\beta}= \eta^{\alpha\beta} - \eta^{\alpha\gamma} h_{\gamma\lambda} \eta^{\lambda\beta} +\,\, \textup{quadratic and higher order terms}.  
\ee
Hence, from (\ref{riccicurvature}), we have
\be\label{april3eqn2}
\square h_{\alpha\beta}=H_{i\gamma} \p_i\p_{\gamma} h_{\alpha\beta} + \mathcal{N}_{\alpha\beta}^{wa} + \mathcal{N}_{\alpha\beta}^{vl},\quad \square:= \p_t^2 - \Delta, 
\ee 
where $\mathcal{N}_{\alpha\beta}^{wa}$ denotes the semilinear nonlinearity of wave-wave type interaction and  $\mathcal{N}_{\alpha\beta}^{vl}$ denotes the semilinear nonlinearity  of wave-Vlasov type interaction. More precisely,   the detailed formulas of $H_{i\gamma}$, $\mathcal{N}_{\alpha\beta}^{wa} $, and $\mathcal{N}_{\alpha\beta}^{vl} $ are given as follows, 
\be\label{nov23eqn2}
 H_{i\gamma} =\left\{\begin{array}{cc}
 2 (g^{00})^{-1}  \Lambda_{\geq 1}[g^{\gamma i}] & \textup{if}\,\, \gamma=0\\ 
 &\\ 
 \Lambda_{\geq 1}[(g^{00})^{-1}]\delta_i^{\gamma}+ (g^{00})^{-1}\Lambda_{\geq 1}[g^{ i\gamma}] &\textup{if} \,\,\gamma\in\{1,2,3\},\\
 \end{array}\right.
\ee
 \be\label{2020april3eqn1}
\mathcal{N}_{\alpha\beta}^{wa} =- \big(g^{00} \big)^{-1}(P_{\alpha\beta}+Q_{\alpha\beta}), \quad \mathcal{N}_{\alpha\beta}^{vl} =  \big(g^{00} \big)^{-1} ( T_{\alpha\beta}(f) - \h g_{\alpha\beta} T(f)).  
 \ee

  For any $ {L}\in\cup_{\alpha\in \Z^3,n \in \Z_{+}}\nabla_x^\alpha  P_n$, recall (\ref{april3eqn2}), we have 
\be\label{sep28eqn26}
\square L h_{\alpha\beta} =  \sum_{\tilde{L}\in \nabla_x^\alpha P_c^d, \tilde{L}\preceq L} c_{L\tilde{L}}\tilde{L}\big(   H_{i\gamma} \p_i\p_{\gamma} h_{\alpha\beta} + \mathcal{N}_{\alpha\beta}^{wa} + \mathcal{N}_{\alpha\beta}^{vl}\big).
\ee

To exploit the harmonic gauge condition,   following  the study of  the  Einstein-Klein-Gordon system in \cite{IP}, for any $L\in P_n, n\in \Z_+$,   we define  
\begin{multline}\label{doublehodge1}
F^{ L }:=\h\big( L h_{00} + R_i R_j  L  h_{ij}\big),\quad \underline{F}^{  L}:= \h\big(  L h_{00}- R_i R_j L   h_{ij}\big),\quad \rho^{  }:= R_i L  h_{0i}, \\
\omega_{j}^{  }:= \varepsilon_{jkl} R_k  {L } h_{0l}, \quad \Omega_j^{ {L}}:=  \varepsilon_{jkl} R_k R_m   {L} h_{lm}, \quad \vartheta_{jk}^{  }:= \varepsilon_{jmp}\varepsilon_{knq} R_m R_n L  h_{pq}.
\end{multline}
 We can recover the perturbation of metric $L h_{\alpha\beta}$ from the above defined double Hodge decomposition, 
 \begin{multline}\label{april1eqn1}
L  h_{00}^{ }= F^{L }+\underline{F}^{ L}, \quad L h_{0j} = -R_{j} \rho^{ L} + \varepsilon_{jkl} R_k \omega_l^{ L} \\ 
 L h_{jk}^{ }= R_jR_k\big(F^{L }-\underline{F}^{L } \big)- \big( \varepsilon_{klm} R_j + \varepsilon_{jlm}R_k \big) R_l \Omega_m^{L } + \varepsilon_{jpm}\varepsilon_{kqn} R_p R_q \vartheta_{mn}^{ L}. 
 \end{multline}

To represent the connection between the perturbed metric and the double Hodge decomposition variables schematically, we know that there exist  uniquely determined zero order operators $ R_{\alpha\beta}^{\tilde{h}},  R^{\alpha\beta}_{\tilde{h}}\in \{R_i, R_iR_j\}\times \C$ such that the following equalities holds, 
 \be\label{connection}
 L h_{\alpha\beta} =\sum_{\tilde{h}\in\{F, \underline{F}, \omega_j, \Omega_j, \rho, \vartheta_{mn} \}} R_{\alpha\beta}^{\tilde{h}} \tilde{h}^L, \quad \tilde{h}^L= R^{\alpha\beta}_{\tilde{h}} L_{\alpha\beta}, \quad \tilde{h}\in\{F, \underline{F}, \omega_j, \Omega_j, \rho, \vartheta_{mn} \}. 
 \ee

 Due to the harmonic coordinate gauge condition,  the variables $\rho$, $\Omega_j$ can be recovered from the other defined variables. Hence,  it would be sufficient to control $\{F, \underline{F}, \omega_{j}, \vartheta_{mn}\}$ over time  to control the perturbed metric. More precisely, we have

\begin{lemma}\label{linearrelationhodge}
Let $ R_0:=\p_t \d^{-1}$. For any  ${L}\in   P_n, n\in \Z_{+}$, we define 
\begin{multline}\label{sep28eqn79}
E^{comm}_{ {L}, \mu}:= \eta^{\alpha\beta} [\p_\alpha,  {L}] h_{\beta \mu} - \h \eta^{\alpha\beta} [\p_{\mu},  {L}] h_{\alpha\beta}=  \sum_{\tilde{L}\prec L} c_{L;\tilde{L}}^{\mu, \nu;\alpha, \beta} \p_{\nu} \tilde{L} h_{\alpha\beta}, \\
 E_{\mu}:=\eta^{\alpha\beta} \p_{\alpha} h_{\beta\mu} - \h \eta^{\alpha\beta}\p_{\mu} h_{\alpha\beta}=-\Lambda_{\geq 2}[\Gamma_\mu].
\end{multline}
Then the following equalities hold for the variables $\rho^{ {L}}$ and $\omega_j^{ {L}}$,
\begin{multline}\label{april1eqn11}
\rho^{ {L}} = R_0\big(\underline{F}^{ {L}} + \tau^{ {L}}\big) +\d^{-1}  \big( {L} E_{0} + E^{comm}_{ {L},0}\big),\quad \tau^{ {L}} = -\h \delta_{jk} \vartheta^{ {L}}_{jk}, \\ 
 \Omega_j^{ {L}}:= R_0 \omega_j^{ {L}} + \varepsilon_{jlk}\d^{-1} R_l \big(  {L} E_k +   E_{ {L},k}^{comm} \big),
\end{multline}
where  $\tau^{ {L}}$ satisfies the following equalities 
\be\label{aug25eqn1}
2\d^2 \tau^{ {L}} = \p_{\alpha}  {L} E_{\alpha} +\p_{\alpha} E^{ comm}_{ {L}, \alpha}-\h\big(\square  {L} h_{00} + R_i R_j \square  {L} h_{ij}\big) + \h \delta_{jk}  \varepsilon_{jmp}\varepsilon_{knq} R_m R_n \square  {L} h_{pq}, \ee
\be\label{aug25eqn2}
2 \d \p_t \tau^{ {L}} =- \d  {L} E_0 + R_k\p_t {L} E_k -\d E_{ {L},0}^{com} + \p_t R_k E_{ {L},k}^{comm} - R_i \square  {L} h_{0i}. 
\ee
\end{lemma}
\begin{proof}
See \cite{IP}[Lemma 6.7]. 
\end{proof}
\begin{remark}
 Recall (\ref{aug25eqn1}) and  (\ref{aug25eqn2}).  Unlike   the Einstein-Klein-Gordon system  studied in \cite{IP}, a key difference in the  Einstein-Vlasov system  is that $\tau^{\mathcal{L}}$ and $\p_t \tau^{\mathcal{L}}$ in the Einstein-Vlasov system contain  linear density type term which depends on the Vlasov part.  More precisely, 
 \[
  \Lambda_{1}[\tau^{\mathcal{L}}] =  \frac{1}{4}\d^{-2}\Lambda_{1}[\big(- \big(\square  \mathcal{L} h_{00}  + R_i R_j \square \mathcal{L} h_{ij} \big) +   \delta_{jk}  \varepsilon_{jmp}\varepsilon_{knq} R_m R_n \square \mathcal{L} h_{pq} \big) ]
 \]
 \be\label{2022march19eqn1}
  =    \sum_{   \mathcal{L} \preceq  L   }\d^{-2} M_{\alpha\beta}\big(  \int_{\R^3} {b}^{L\mathcal{L} }_{\alpha\beta}(v)    \mathcal{L} f   (t,x, v) d v\big),
 \ee
\be\label{sep26eqn21}
 \Lambda_{1}[\p_t \tau^{\mathcal{L}}] =  -\h\d^{-1} R_i  \Lambda_{1}[\square \mathcal{L} h_{0i}]  =  \sum_{     \mathcal{L} \preceq  L     }\d^{-1} R_i\big(  \int_{\R^3} {b}^{L\mathcal{L} }_{0i }(v) \mathcal{L} f   (t,x, v) d v\big),
\ee
where $M_{\alpha\beta}\in \{c_1 R_i, c_2R_iR_j, i,j\in\{1,2,3,\}, c_1,c_2\in \C\} .$
\end{remark}
Recall the double Hodge decomposition in (\ref{doublehodge1}). From the equation (\ref{sep28eqn26}), we classify schematically the nonlinearities of $\square \tilde{h}^{L}$ as follows, 
\[
\square \tilde{h}^{L} = R_{\alpha\beta}^{\tilde{h}^L}(\square Lh_{\alpha\beta})= \sum_{\tilde{L}\in \nabla_x^\alpha P_c^d, \tilde{L}\preceq L} c_{L\tilde{L}} R_{\alpha\beta}^{\tilde{h}^L}\big(\tilde{L}\big(   H_{i\gamma} \p_i\p_{\gamma} h_{\alpha\beta} + \mathcal{N}_{\alpha\beta}^{wa} + \mathcal{N}_{\alpha\beta}^{vl}\big)\big) 
\]
\be\label{aug11eqn21}
=  H_{i\gamma}\p_{i}\p_{\gamma}\tilde{h}^{L} +   N_{vl}^{  {\tilde{h} ;L }}   + R_{mt }^{\tilde{h};L }  + Q_{mt }^{ \tilde{h};L} ,\quad \tilde{h}\in \{F, \underline{F}, \omega_j, \vartheta_{mn}\},
\ee
where the uniquely determined linear operator $ R_{\alpha\beta}^{\tilde{h}^L}\in \{R_i, R_iR_j, i,j\in\{1,2,3\}\}\times\mathbb{C}$,   $ N_{vl}^{  {\tilde{h} ;L }}$ denotes the nonlinearity that depends on the Vlasov part, $Q_{mt }^{ \tilde{h} ;L }$ denotes the quadratic terms of the  wave-wave type, and $R_{mt }^{ \tilde{h};L  } $ denotes the cubic and higher order nonlinearity which only depends on the metric component.

 \subsection{Modified profile for the perturbed metric}

Since the non-linearity $N_{vl}^{  {\tilde{h} ;L }}$  in (\ref{aug11eqn21}) contains  linear level terms which depend  only on the Vlasov part, to capture the linear effect, in this subsection, we define corresponding modified profile of the perturbed metric. 

Roughly speaking, the \emph{modified} profile of the perturbed metric has similar property of the original (un-modified) profile of the perturbed metric in the Einstein-Klein-Gordon system \cite{IP} as it takes out the leading effect of the Vlasov part on the perturbed metric.

Firstly, for  $L\in P_n$, $   \tilde{h}\in \{F, \underline{F},\rho, \Omega_j, \omega_j, \vartheta_{mn}\}$, we define  the profiles of $\tilde{h}^L$ and $Lh_{\alpha\beta}$ as follows,
\be\label{aug11eqn151}
V^{ \tilde{h}^L }=e^{  it \d} U^{   \tilde{h}^L }, \quad   U^{   \tilde{h}^L }:= (\p_t - i \d) \tilde{h}^{ L}, \quad \Longrightarrow \p_t V^{ \tilde{h}^L } =  e^{ it \d}\big(\square \tilde{h}^{ L} \big),
\ee
\be\label{2020april15eqn35}
V^{ L h_{\alpha\beta} }=e^{  it \d} U^{   L h_{\alpha\beta} }, \quad   U^{    L h_{\alpha\beta} }:= (\p_t - i \d)  L h_{\alpha\beta}, \quad \Longrightarrow \p_t V^{ L h_{\alpha\beta} } =  e^{ it \d}\big(\square  L h_{\alpha\beta} \big),
\ee 
\be\label{aug12eqn1}
    \tilde{h}^L= \sum_{\mu\in\{+,-\}}   c_{\mu}\d^{-1}\big(    U^{L \tilde{h}}(t)\big)^{\mu}, \quad \p_t  \tilde{h}^L= \sum_{\mu\in\{+,-\}}     \big(    U^{L \tilde{h}}(t)\big)^{\mu}, \quad c_{\mu}:=  i \mu/2.
\ee

Next,  we identify the linear  density type functions in $ N_{vl}^{  {\tilde{h};L}}$. Recall (\ref{aug11eqn21}).   
   As a typical example, we compute the linear term of $N_{vl}^{F;L}$  as follows, 
  \be\label{aug12eqn1100d}
  \begin{split}
  \Lambda_{1}[N_{vl}^{F;L} ]=&  \sum_{\begin{subarray}{c}
  \hat{L}\in \cup_{  n\in \Z_+}  P_n,   \hat{L}\preceq L 
  \end{subarray}} c_{L\hat{L}}\big[\hat{L}\big( \Lambda_{1}[ -  T_{00}(f) + \h g_{00} T(f)\big)] \\ & + R_iR_j\big(\hat{L}\big( \Lambda_{1}[ -  T_{ij}(f) + \h g_{ij} T(f)\big)] \big)\big]. 
  \end{split}
  \ee 
 Note  that, for any $L\in P_{n}^{  }$ and the corresponding vector field $\mathcal{L}\in \mathcal{P}_n^{ }$, the following equality of general form holds for some determined coefficients, 
\be\label{jan29eqn51}
L-\mathcal{L}= \sum_{\tilde{\mathcal{L}} \prec \mathcal{L}, |\alpha|+|\tilde{\mathcal{L}}|\leq |\mathcal{L}| } a_{\alpha\tilde{\mathcal{L}}}^{L}(v) \nabla_v^{\alpha}  \tilde{\mathcal{L}},\quad | a_{\alpha\tilde{\mathcal{L}}}^{L}(v) |\lesssim \langle |v|\rangle^{|\alpha|}.
\ee
 After using the above  equality  (\ref{jan29eqn51}) and doing integration by parts in ``$v$'',  the following equality holds for any fixed $L\in P_{n}$,
\begin{multline}\label{april8eqn71}
L\big(  - \Lambda_{1}[ T_{\alpha\beta}(f)]  \big) 
= L\big(-\int_{\R^3} f(t,x, v ) \frac{\Lambda_{0}[v_\alpha] \Lambda_{0}[v_\beta]}{\sqrt{1+|v|^2}} d v  d v  \big)\\ 
=  \sum_{\tilde{\mathcal{L}} \preceq \mathcal{L}, |\mu|+|\tilde{\mathcal{L}}|\leq |\mathcal{L}| }    -\int_{\R^3}  a_{\mu \tilde{\mathcal{L}} }^{L}(v)  \nabla_v^{\mu} \tilde{\mathcal{L}} f(x^\alpha, v ) \frac{\Lambda_{0}[v_\alpha] \Lambda_{0}[v_\beta]}{\sqrt{1+|v|^2}} d v =   \sum_{\tilde{\mathcal{L}} \preceq \mathcal{L}} \int_{\R^3}  \tilde{a}^{L \tilde{\mathcal{L}}}_{\alpha\beta}(v) \tilde{\mathcal{L}} f   (t,x, v) d v    ,
\end{multline}
where  ``$  \tilde{a}^{L \tilde{\mathcal{L}}}_{\alpha\beta}(v) $''   are some determined coefficients that satisfy the following estimate, 
\be\label{july11eqn78}
\sum_{|\alpha|\in \mathbb{Z}_{+}}    \|\langle |v|\rangle^{|\alpha|-10}\nabla_v^\alpha  \tilde{a}^{L \tilde{\mathcal{L}}}_{\alpha\beta}(v) \|_{L^\infty_v}\lesssim 1. 
\ee
The reduction we did in (\ref{april8eqn71}) is standard and it will be used directly later when we apply the macroscopic vector fields to the density type variables associated with the Vlasov part.  With minor modifications, we know that the following equality holds for some uniquely  determined coefficients $   {b}^{L\mathcal{L} }_{\alpha\beta}(v)$, 
\be\label{sep20eqn61}
\Lambda_{1}[L \mathcal{N}_{vl}^{ h_{\alpha\beta}} ]=  \sum_{ \mathcal{L} \preceq L  }  \int_{\R^3} {b}^{L\mathcal{L} }_{\alpha\beta}(v) \mathcal{L} f   (t,x, v) d v,\quad  \sum_{|\alpha|\in \mathbb{Z}_{+}} \| \langle |v|\rangle^{|\alpha|-10}\nabla_v^\alpha    {b}^{L\mathcal{L} }_{\alpha\beta}(v)\|_{L^\infty_v}  \lesssim 1. 
\ee
 
    Similar to the process we did for the linear term in (\ref{april8eqn71}), recall (\ref{2020april3eqn1}) and (\ref{march19eqn6}), we can formulate the quadratic term of $L\big( \mathcal{N}_{\alpha\beta}^{vl}\big)$ as follows for  some uniquely determined coefficients   $ {c}^{L {L}_1\mathcal{L}_2 }_{\alpha\beta;\gamma\kappa}(v)$,
    \be\label{2020april14eqn51}
\Lambda_{2}[ L\big( \mathcal{N}_{\alpha\beta}^{vl}\big) ]=  \sum_{  | {L}_1| + |\mathcal{L}_2|\leq |L| }  L_1  {h}_{\gamma\kappa} \big( \int_{\R^3} {c}^{L {L}_1\mathcal{L}_2 }_{\alpha\beta;\gamma\kappa}(v)  \mathcal{L}_2 f   (t,x, v) d v \big),
\ee
where the following estimate holds for the coefficients ${c}^{L {L}_1\mathcal{L}_2 }_{\alpha\beta;\gamma\kappa}(v)  $, 
\be
  \sum_{|\alpha|\in \mathbb{Z}_{+}}    \|  \langle |v|\rangle^{|\alpha|-10}\nabla_v^\alpha   {c}^{L {L}_1\mathcal{L}_2 }_{\alpha\beta;\gamma\kappa}(v)\|_{L^\infty_v} \lesssim 1.
\ee

Lastly, we are ready to define the     modified profiles of the perturbed metric. To exploit  the fact that the speed of particle is strictly less than the speed of light and balance out  the linear effect caused by the Vlasov part in the equations satisfied by the perturbed metric, $ \forall \tilde{h}\in \{F, \underline{F}, \rho, \Omega_j, \omega_j, \vartheta_{mn}\}$ we define   
\be\label{2020april1eqn4}
\begin{split}
 \widehat{ \widetilde{   V^{  \tilde{h}^L } }}(t, \xi)&:= \widehat{    V^{  \tilde{h}^L }   }(t, \xi) -   \sum_{    \mathcal{L} \preceq  L    }   \int_{\R^3} e^{ i t |\xi|- i t\hat{v}\cdot \xi} {b}^{L\mathcal{L}  }_{ \tilde{h}}(v) \big(i (|\xi|-\hat{v}\cdot\xi )\big)^{-1}   {k_{ \tilde{h}}(\xi) \widehat{u^{\mathcal{L}}}(t, \xi, v)}  d v,\\
\widehat{ \widetilde{   V^{ L h_{\alpha\beta} }  } }(t, \xi)&:= \widehat{    V^{   L h_{\alpha\beta} }   }(t, \xi) -   \sum_{   \mathcal{L} \preceq  L    }   \int_{\R^3} e^{ i t |\xi|- i t\hat{v}\cdot \xi} {b}^{L\mathcal{L} }_{\alpha\beta}(v) \big(i (|\xi|-\hat{v}\cdot\xi )\big)^{-1}{  \widehat{u^{\mathcal{L}}}(t, \xi, v)}  d v.
\end{split}
\ee

  Recall the equation satisfied by $\tilde{h}^L$ in (\ref{aug11eqn21}). From the equalities (\ref{aug11eqn151}) and  (\ref{2020april1eqn4}),  we highlight  the quasilinear  terms, which are responsible for losing derivatives,  as follows, 
\be\label{2020june30eqn1}
\begin{split}
 e^{-it\d} \p_t  \widetilde{   V^{  \tilde{h}^L} }& =   H_{i\gamma}\p_{i} {\p_{\gamma}\tilde{h}^{L}}  
 +   \Lambda_{\geq 2}[N_{vl}^{  {\tilde{h} ;L }}]   + R_{mt }^{\tilde{h};L }  + Q_{mt }^{ \tilde{h};L} + \mathcal{N}_{vl;m}^{  {\tilde{h} ;L }}, \\
 \mathcal{N}_{vl;m}^{  {\tilde{h} ;L }}& = \sum_{\mathcal{L}\in \mathcal{P}_{a}^{b },|\mathcal{L}|\leq |L|   }  - \mathcal{F}_{\xi}^{-1}\big[ \int_{\R^3} e^{ - i t\hat{v}\cdot \xi} {b}^{L\mathcal{L} }_{\alpha\beta}(v) \frac{ k^{\tilde{h}}_{\alpha\beta}(\xi)  \p_t \widehat{u^{\mathcal{L}}}(t, \xi, v)}{i2\big(|\xi|-\hat{v}\cdot\xi\big)} d  v \big]  .
\end{split}
\ee

Correspondingly, we define a modified half wave variables and the modified perturbed metric  as follows,  
\be\label{2022march19eqn11}
\widetilde{   U^{L h_{\alpha\beta} } }(t, x):= \mathcal{F}_{\xi}^{-1}[ e^{-it |\xi|} \widehat{ \widetilde{   V^{L h_{\alpha\beta}} }}(t, \xi) ](x),
\ee
\be\label{2022march19eqn12}
\forall \tilde{h}\in \{F, \underline{F}, \rho, \Omega_j,\omega_j, \vartheta_{ij}\}, \quad\widetilde{   U^{ \tilde{h}^L } }(t, x):= \mathcal{F}_{\xi}^{-1}[ e^{-it |\xi|} \widehat{ \widetilde{   V^{ \tilde{h}^L  } }}(t, \xi) ](x), 
\ee
\be\label{2022march19eqn13}
 \widetilde{ Lh_{\alpha\beta}}(t):= \sum_{\mu\in\{+,-\}} c_{\mu}\d^{-1} \big(   \widetilde{   U^{L h_{\alpha\beta} } }(t)  \big)^{\mu},\quad \widetilde{\tilde{h}^{L}}(t):= \sum_{\mu\in\{+,-\}} c_{\mu}\d^{-1}  \big( \widetilde{   U^{ \tilde{h}^L } }(t)  \big)^{\mu},
\ee 
\be\label{2022march19eqn14}
  \widetilde{ \p_t Lh_{\alpha\beta}}(t):= \sum_{\mu\in\{+,-\}} \frac{ \big(   \widetilde{   U^{L h_{\alpha\beta} } }(t)  \big)^{\mu} }{2}   \qquad
 \widetilde{\p_t \tilde{h}^{L}}(t):= \sum_{\mu\in\{+,-\}}  \frac{\big( \widetilde{   U^{ \tilde{h}^L } }(t)  \big)^{\mu}}{2} ,
\ee
\be\label{2022march19eqn16}
\tilde{\rho}^{h_{\alpha\beta}}_{L}(f)(t):= U^{Lh_{\alpha\beta}}(t)- \widetilde{U^{Lh_{\alpha\beta}}}(t),\quad \tilde{\rho}^{\tilde{h}}_{L}(f)(t):= U^{ \tilde{h}^L}(t)- \widetilde{U^{ \tilde{h}^L}}(t),
\ee 
\be\label{modifiedperturbedmetric}
  \widetilde{\rho}^{h_{\alpha\beta}}_{L}(f)(t):=     \sum_{  \mathcal{L} \preceq  L     }  \mathcal{F}^{-1}\big[ \int_{\R^3} e^{  - i t\hat{v}\cdot \xi} {b}^{L\mathcal{L} }_{\alpha\beta}(v)    
 \big(i|\xi|-i\hat{v}\cdot\xi\big)^{-1} \big(  \widehat{u^{\mathcal{L}}}(t, \xi, v) \big)  d v\big](t,x),  
\ee
\be\label{2022march19eqn3}
 \widetilde{\rho}^{ \tilde{h}}_{L}(f)(t):=     \sum_{  \mathcal{L} \preceq  L     }  \mathcal{F}^{-1}\big[ \int_{\R^3} e^{  - i t\hat{v}\cdot \xi}  {b}^{L\mathcal{L}  }_{ \tilde{h}}(v)    k_{ \tilde{h}}(\xi)\big(i|\xi|-i\hat{v}\cdot\xi\big)^{-1} \big(  \widehat{u^{\mathcal{L}}}(t, \xi, v) \big)  d v\big](t,x).
\ee

We remark that,  from the above definition, the modified perturbed metric components are all real valued functions. In particular,  although $\widetilde{\p_t \tilde{h}^{L}}(t)$ and $\p_t \widetilde{  \tilde{h}^{L}}(t)$ are different,  they are same at the linear level.  From (\ref{modifiedperturbedmetric}), as a result of direct computation, we have
\be\label{2022may17eqn33}
\p_t \widetilde{  \tilde{h}^{L}}(t)-\widetilde{\p_t \tilde{h}^{L}}(t)= \sum_{\mu\in\{+,-\}} c_{\mu}\d^{-1}  \big( (\p_t +  i \d) \widetilde{   U^{ \tilde{h}^L } }(t)  \big)^{\mu} = \textup{Im}\big[ \d^{-1} e^{-i t\d}  \p_t \widetilde{   V^{ \tilde{h}^L } }(t) \big],
\ee
\be\label{2020june30eqn32}
\begin{split}
\p_t \widetilde{\p_t \tilde{h}^{L}}(t) &=  \sum_{\mu\in\{+,-\}} \frac{1}{2} \big( - i \d \widetilde{   U^{ \tilde{h}^L } }(t)  \big)^{\mu}   + \frac{1}{2}  \big( e^{-i t\d}  \p_t \widetilde{   V^{ \tilde{h}^L } }(t)  \big)^{\mu} \\
& =   \Delta \widetilde{   \tilde{h}^{L}}(t)+ \sum_{\mu\in\{+,-\}} \frac{1}{2}   \big( e^{-i t\d}  \p_t \widetilde{   V^{ \tilde{h}^L } }(t)  \big)^{\mu} . 
\end{split}
 \ee

 From the equality in  (\ref{connection}) and results in   Lemma \ref{linearrelationhodge},  we can reformulate the   profile $   {   U^{L h_{\alpha\beta} } }(t)$  in terms of the modified profile of the essential parts $\widetilde{   U^{ \tilde{h}^L } }(t, x)$, $\tilde{h}\in \{F, \underline{F}, \omega_j, \vartheta_{mn}\}$ as follows,
\be\label{2020june1eqn1}
   {   U^{L h_{\alpha\beta} } }(t)=\sum_{\tilde{h}\in\{F, \underline{F}, \omega_j,  \vartheta_{mn} \}} \tilde{R}_{\alpha\beta}^{\tilde{h}}      \widetilde{   U^{ \tilde{h}^L } }(t) + \sum_{\tilde{L}\prec L } R_{\alpha\beta}^{\gamma\kappa} \widetilde{U^{\tilde{L}h_{\gamma\kappa} }}  + Err_{\alpha\beta}^{L;wa}(t) + Err_{\alpha\beta}^{L;vl}(t),
\ee
where the low order error type terms comes from the commutators $E_{ {L},\alpha}^{comm}$ in Lemma \ref{linearrelationhodge}. More precisely, 
 \be\label{2020july6eqn61}
 \tilde{R}_{\alpha\beta}^{\tilde{h}}    = \left\{\begin{array}{ll} 
  {R}_{\alpha\beta}^{\tilde{h}}    & \textup{if\,\,} \tilde{h}\in \{F, \vartheta_{mn}\}\\ 
 {R}_{\alpha\beta}^{\underline{F}} R_{\gamma} + i R_{\alpha\beta}^{\rho}   & \textup{if\,\,} \tilde{h}=\underline{F} \\
   {R}_{\alpha\beta}^{\omega_j}R_{\gamma} + i  R_{\alpha\beta}^{\Omega_j}   &  \textup{if\,\,} \tilde{h}=\omega_j,   \\ 
  \end{array}\right.  
  \ee
  \[
     Err_{\alpha\beta}^{L;wa}(t)=  R_{\alpha\beta}^{\rho}\big(  \d^{-1}\Lambda_{\geq 2}[\square \underline{F}^{L}]  + \d^{-1}(\p_t+ i \d) ( {L} E_{0}) + \d^{-1}\big(\Lambda_{\geq 2}[\square \tau^L]+i\Lambda_{\geq 2}[(\p_t + i\d)\tau^L ] \big) \big)  
  \]
  \be\label{2022may18eqn1}
  + R_{\alpha\beta}^{\Omega_j}\big(  \d^{-1}\Lambda_{\geq 2}[\square \omega_j^{L}] + \d^{-1} (\p_t+ i \d) \varepsilon_{jlk} R_l (  {L} E_k) \big), 
  \ee
  \[
 Err_{\alpha\beta}^{L;vl}(t)=      \sum_{\tilde{h}\in\{F, \underline{F},\rho, \Omega_j, \omega_j , \vartheta_{mn} \}} \sum_{\tilde{L}\preceq L} R_{\alpha\beta}^{\tilde{h}} \tilde{\rho}^{\tilde{h}}_{\tilde{L}}(f)(t)+R_{\alpha\beta}^{\rho}\big(\d^{-1}(\Lambda_{1}[ \square( \underline{F}^{L}+  \underline{\tau}^{L})] 
  \]
  \be
    + i\Lambda_{1}[(\p_t + i\d)\tau^L ] \big) 
  +\d^{-1}R_{\alpha\beta}^{\Omega_j} \Lambda_{1}[\square \omega_j^{L} ]  .
\ee
\subsection{Classification of quadratic terms of the perturbed metric}
 For   convenience in later argument, in  this subsection, we elaborate and  classify the quadratic terms of the equations satisfied by perturbed metric component in (\ref{aug11eqn21}). Depending on the roughness  allowed or the delicacy required to close the bootstrap argument, in different scenarios, we use different strategies to decompose the quadratic terms.  
Recall (\ref{april3eqn2}), (\ref{2020april3eqn1}), the detailed formulas of $  P_{\alpha\beta} $ and $ Q_{\alpha\beta} $ in (\ref{nullstructure}) and  (\ref{weaknull}). For any $ {L}\in\cup_{\alpha\in \Z^3, n\in \Z_+}\nabla_x^\alpha  P_n,$ we have
\begin{multline}\label{2020april8eqn1}
 {L}(\Lambda_{2}[ \mathcal{N}_{\alpha\beta}^{wa}])+ L(\Lambda_{2}[  H_{i\gamma} \p_i\p_{\gamma} h_{\alpha\beta}]  ):=- L\big( \Lambda_{2}[ P_{\alpha\beta}] +\Lambda_{2}[ Q_{\alpha\beta}] \big) + L \Lambda_{2}[S_{\alpha\beta}], 
 \\
L(\Lambda_2[K_{\alpha\beta}])=  \sum_{  L_1\circ L_2\preceq L,  k\in \mathbb{Z}, (k_1,k_2)\in \chi_k^1\cup \chi_k^2\cup \chi_k^3} c_{L_1L_2}^L QK_{\alpha\beta;k,k_1,k_2}^{L;L_1,L_2} ,\quad K\in\{S, P, Q\}. 
\end{multline}

In particular,  we   single out   the term that is responsible for losing derivatives, which is the case all vector fields hit on the quasilinear part, i.e.,  $L_2=L$,  
\be
  c_{IdL}^{L} QS_{\alpha\beta;k,k_1,k_2}^{L;Id,L}=  \Lambda_{1}[H_{i\gamma}] \p_i\p_{\gamma} Lh_{\alpha\beta}.
\ee
    The detailed formulas of $QK_{\alpha\beta;k,k_1,k_2}^{L;L_1,L_2}, K\in\{S, Q, P\}$, are given as follow, 
\be\label{nov5eqn1}
QS_{\alpha\beta;k,k_1,k_2}^{L;L_1,L_2} = \left\{\begin{array}{cc}
P_k\big[P_{k_1}( L_1  h_{00}  ) P_{k_2}( \Delta L_2 h_{\alpha\beta}) -2 P_{k_1}( L_1  h_{0i}  ) P_{k_2}( \p_t \p_i L_2 h_{\alpha\beta}) &\\
 +  P_{k_1}( L_1  h_{ij}  ) P_{k_2}( \p_i \p_j L_2 h_{\alpha\beta}) \big] & \textup{if}\,\, |L_1|+|L_2|=|L| \\ &\\
c^{\mu\nu}_{ i \gamma} P_k\big(P_{k_1}(  L_1  h_{\mu\nu} ) P_{k_2}(\p_{i} \p_{\gamma}  L_2 h_{\alpha\beta})\big) & \textup{if}\,\,  |L_1|+|L_2| \leq |L|, \\ 
\end{array}\right. 
\ee 
\be\label{nov5eqn4}
QQ_{\alpha\beta;k,k_1,k_2}^{L;L_1,L_2}  = \left\{\begin{array}{ll}
 c^{\alpha\beta;\mu\nu}_{\gamma,\kappa;\iota, \tau,1} P_k\big( Q_{\mu,\nu}( P_{k_1}(L_1 h_{\gamma\kappa}), P_{k_2}(L_2 h_{\iota\tau})) \big) &  \textup{if}\,\,  |L_1|+|L_2|=|L| \\   
 &\\ 
  c^{\alpha\beta;\mu\nu}_{\gamma,\kappa;\iota, \tau,2}P_k\big( \p_{\mu} P_{k_1} (L_1 h_{\gamma\kappa}) \p_{\nu} P_{k_2} ( L_2 h_{\iota\tau}) \big) &  \textup{if}\,\,  |L_1|+|L_2| \leq |L|, \\ 
 \end{array}\right. 
\ee
\be\label{nov5eqn5}
QP_{\alpha\beta;k,k_1,k_2}^{L;L_1,L_2}  = \left\{\begin{array}{cc}
   \frac{1}{4}\eta^{\gamma\mu}\eta^{\lambda\nu}P_k\big[2 \big( -\p_{\alpha}P_{k_1}(L_1h_{\mu\lambda}) P_{k_2} (\p_{\beta} L_2 h_{\gamma\nu}) \big) &\\ 
   +  \p_{\alpha}P_{k_1}(L_1 h_{\mu\gamma}) \p_{\beta}P_{k_2}(L_2 h_{\lambda\nu})\big] &  \textup{if}\,\,  |L_1|+|L_2|=|L| \\   
 &\\ 
  c^{\alpha\beta;\mu\nu}_{\gamma,\kappa;\iota, \tau,3}P_k\big( \p_{\mu} P_{k_1} (L_1 h_{\gamma\kappa}) \p_{\nu} P_{k_2} ( L_2 h_{\iota\tau}) \big) &  \textup{if}\,\,  |L_1|+|L_2| \leq |L|, \\ 
 \end{array}\right. 
\ee
and  the null bilinear operators $Q_{\alpha, \beta}(\cdot, \cdot), \alpha, \beta\in\{0,1,2,3\}, $ are defined as follows, 
\be\label{nov5eqn6}
Q_{0,0}(f,g):=\p_t f \p_t g - \nabla_x f \cdot \nabla_x g, \quad Q_{\alpha,i}(f,g) = \p_{\alpha} f \p_i g -\p_{i} f \p_{\alpha} g. 
\ee

Recall the relation in (\ref{connection}) and the equation satisfied by $\tilde{h}^L$ in  (\ref{aug11eqn21}). We have
\be\label{2020june1eqn11}
Q_{mt }^{ \tilde{h};L}=\sum_{k\in \Z, (k_1,k_2)\in \chi_k^1\cup\chi_k^2 \cup \chi_k^3} \sum_{\begin{subarray}{c}
L_1\circ L_2\preceq L,
\tilde{L}_1\circ \tilde{L}_2 \preceq L , L_2\prec L, H\in\{  P, Q\}\\ 
\end{subarray}} R^{\alpha\beta}_{\tilde{h} }  QS_{\alpha\beta;k,k_1,k_2}^{L;L_1,L_2} +     R^{\alpha\beta}_{\tilde{h} }   QH_{\alpha\beta;k,k_1,k_2}^{L;\tilde{L}_1,\tilde{L}_2}.
\ee

In our bootstrap argument, we allow the Vlasov part to grow faster than   the modified metric part. Due to different growth rates,  to better measure the growth rate of quadratic terms, we split the perturbed metric into the modified perturbed metric and the density type function. As a result, generally speaking, we  have the following structural decomposition, 
\be\label{nov7eqn21}
QK_{\alpha\beta;k,k_1,k_2}^{L;L_1,L_2}=  QK_{\alpha\beta;k,k_1,k_2}^{ess;L,L_1,L_2} +   QK_{\alpha\beta;k,k_1,k_2}^{vl;L,L_1,L_2}     +  QK_{\alpha\beta;k,k_1,k_2}^{err;L,L_1,L_2},  
\ee
where $ QK_{\alpha\beta;k,k_1,k_2}^{ess;L,L_1,L_2}  $ denotes the essential part of the wave-wave type interaction, in which the inputs are modified perturbed metric instead of the original perturbed metric,  $ QK_{\alpha\beta;k,k_1,k_2}^{vl;L,L_1,L_2}  $ denotes    the wave-Vlasov type interaction, and  $ QK_{\alpha\beta;k,k_1,k_2}^{err;L,L_1,L_2}  $ denotes the error type term.

 The detailed formulas of the above decomposition change depend on the size of $|L_1|+|L_2|$ and the roughness allowed.  For the rough  cases, we use the modified perturbed metric variables   in  (\ref{modifiedperturbedmetric}). As a result, for any $K\in \{S, P, Q\}$, we have
\be\label{2020june18eqn1}
QK_{\alpha\beta;k,k_1,k_2}^{ess;L,L_1,L_2}=  \sum_{ \mu, \nu\in \{+, -\}}P_kQK_{\alpha\beta; \mu\nu;ess  }^{L;L_1,L_2 }( (\widetilde{U_{k_1}^{L_1 h_{\iota\tau}}})^{\mu} ,(\widetilde{U_{k_2}^{L_2 h_{\gamma\kappa}}})^{\nu} ), \quad QK_{\alpha\beta;k,k_1,k_2}^{err;L,L_1,L_2}=0,
\ee
\[
QK_{\alpha\beta;k,k_1,k_2}^{vl;L,L_1,L_2}=  \sum_{ \mu, \nu\in \{+, -\}}P_kQK_{\alpha\beta; \mu\nu;vl }^{L;L_1,L_2 } ( (\widetilde{U_{k_1}^{L_1 h_{\iota\tau}}})^{\mu}, P_{k_2}( \tilde{\rho}_{L_2}^{h_{\gamma\kappa}})^{\nu} )
\]
\be\label{2020june18eqn2}
 + P_kQK_{\alpha\beta; \mu\nu ;vl}^{L;L_1,L_2 } ( P_{k_1}( \tilde{\rho}_{L_1}^{h_{\iota\tau}}))^{\mu},( {U_{k_2}^{L_2 h_{\gamma\kappa}}} )^{\nu}).
\ee
 
Alternatively, for some delicate cases, we use the modified double Hodge decomposition variables.  Recall    the detailed formulas in  (\ref{nov5eqn1}), (\ref{nov5eqn4}), and (\ref{nov5eqn5}).  We substitute the perturbed metric by  its corresponding half wave defined in (\ref{2020april15eqn35}) and using    the equality (\ref{2020june1eqn1}),  the following decomposition holds for any $K\in\{P, Q, S\},$
\be\label{2020june17eqn21}
QK_{\alpha\beta;k,k_1,k_2}^{ess;L,L_1,L_2}  = \sum_{\tilde{h}_1, \tilde{h}_2\in \{F,\underline{F}, \omega_j, \vartheta_{mn}  \}}  P_k QK_{\alpha\beta;\mu\nu}^{\tilde{h}_1, \tilde{h}_2} (    (\widetilde{   U^{ \tilde{h}_1^{L_1} }_{k_1} }(t))^{\mu} , ( \widetilde{   U_{k_2}^{ \tilde{h}_2^{L_2} } }(t))^{\nu} ) ),
\ee
\[
QK_{\alpha\beta;k,k_1,k_2}^{err;L,L_1,L_2}:=\sum_{\tilde{h}   \in \{F,\underline{F}, \omega_j, \vartheta_{mn}\}}\sum_{ \mu, \nu\in \{+, -\}}  P_k  QK_{\alpha\beta;\mu\nu}^{\tilde{h}_1, \tilde{h}_2} \big( (   ( {   U_{k_1}^{ \tilde{h}^{L_1} } }(t))^{\mu} , P_{k_2}( Err_{\alpha'\beta'}^{L_2;wa}(t))^{\nu} )\big)
\]
\[
+\sum_{\tilde{L}_i\preceq L_i,|\tilde{L}_1|+ |\tilde{L}_2|< |L_1|+|L_2|  }  P_k QK_{\alpha\beta;\mu\nu}^{\tilde{L}_1,\tilde{L}_2;\gamma_1\kappa_1,\gamma_2\kappa_2  } (    (\widetilde{   U^{ \tilde{L}_1 h_{\gamma_1\kappa_1} }_{k_1} }(t))^{\mu} , ( \widetilde{   U_{k_2}^{ \tilde{L}_2 h_{\gamma_2\kappa_2} } }(t))^{\nu} ) )
\]
\be\label{2020june17eqn22}
+ P_k QK_{\alpha\beta;\mu\nu}^{\tilde{h}_1, \tilde{h}_2}   \big(   P_{k_1}( Err_{\alpha'\beta'}^{L_1;wa}(t))^{\mu} ),   (\widetilde{   U_{k_2}^{ \tilde{h}^{L_2} } }(t))^{\nu}  \big), 
\ee
\[
QK_{\alpha\beta;k,k_1,k_2}^{vl;L,L_1,L_2}:=\sum_{\begin{subarray}{c}
\tilde{h}   \in \{F,\underline{F}, \omega_j, \vartheta_{mn}\}\\ 
\mu, \nu\in \{+, -\}
\end{subarray}
}  \sum_{\tilde{L}_i\preceq L_i,|\tilde{L}_1|+ |\tilde{L}_2|< |L_1|+|L_2|  } P_k QK_{\alpha\beta;\mu\nu}^{\tilde{h}_1, \tilde{h}_2} \big( (   ( {   U_{k_1}^{ \tilde{h}^{L_1} } }(t))^{\mu} , P_{k_2}( Err_{\alpha'\beta'}^{L_2;vl}(t))^{\nu} )\big)
\]
\[
+ P_k QK_{\alpha\beta;\mu\nu}^{\tilde{h}_1, \tilde{h}_2}   \big(   P_{k_1}( Err_{\alpha'\beta'}^{L_1;vl}(t))^{\nu} ),   (\widetilde{   U_{k_2}^{ \tilde{h}^{L_2} } }(t))^{\nu}  \big)
\]
\[
+ P_k QK_{\alpha\beta;\mu\nu}^{\tilde{L}_1,\tilde{L}_2;\gamma_1\kappa_1,\gamma_2\kappa_2  } (    (P_{k_1}( \widetilde{\rho}^{h_{\gamma_1\kappa_1}}_{\tilde{L}_1}(f)(t)))^{\mu} , ( \widetilde{   U_{k_2}^{ \tilde{L}_2 h_{\gamma_2\kappa_2} } }(t))^{\nu} ) )
\]
\be\label{2020june17eqn23}
+ P_k QK_{\alpha\beta;\mu\nu}^{\tilde{L}_1,\tilde{L}_2;\gamma_1\kappa_1,\gamma_2\kappa_2  } (     ( {   U^{ \tilde{L}_1 h_{\gamma_1\kappa_1} }_{k_1} }(t))^{\mu} , ( P_{k_2}( \widetilde{\rho}^{h_{\gamma_2\kappa_2}}_{\tilde{L}_2}(f)(t)))^{\nu} ) ).
\ee

\subsection{Null structure of the Einstein's equations}
The  null structure inside the Einstein equations, in particular,  its form on the Fourier side has been studied in details by  Ionescu-Pausader in \cite{IP}.  We record   main conclusions here and refer readers to \cite{IP} for more details.   Let 
 \begin{multline}\label{zeroordersymbol}
\mathcal{M}^{ }:=\{m: (\R^3/\{0\})^2  \rightarrow \mathbb{C}:\quad m(x,y)=m_1(x,y)m'(x+y), \quad m':\R^3\longrightarrow \C, \\ 
|x|^\alpha |\nabla_x^\alpha m'(x)|\lesssim_{\alpha} 1, |x|^\alpha |y|^\beta |\nabla_x^\alpha \nabla_y^\beta m_1(x,y)|\lesssim_{\alpha, \beta} 1 , x, y\in \R^3, \alpha, \beta\in \Z_+^3  \}.
\end{multline}
Moreover, we define the following classes of null multipliers,
\be\label{nullmutipliers} 
\begin{split}
\mathcal{M}_{\pm}^{null}&:=\{n:(\R^3/\{0\})^2: n(x,y):= (\frac{x_i}{|x|} \mp \frac{y_i}{|y|} ) m(x,y), m\in\mathcal{M}, \textup{for some \,} i \in \{1,2,3\}\},\\
\tilde{\mathcal{M}}_{\pm}^{null}&:=\{n:(\R^3/\{0\})^2: n(x,y):=(\frac{x_i+y_i}{|x+y|} \mp \frac{y_i}{|y|}) m(x,y),   m\in\mathcal{M}, \textup{for some \,} i \in \{1,2,3\}\}.
\end{split}
\ee
For $a,b,c,d\in \R, k,k_1,k_2\in \Z$, we define
\be\label{2022may22eqn31}
\begin{split}
\mathcal{S}^\infty_{k,k_1,k_2}&:=\{m:\R^3_\xi \times \R^3_\eta \longrightarrow \C, \|m(\xi-\eta, \eta)\|_{\mathcal{S}^\infty_{k,k_1,k_2}}:=\sup_{\alpha, \beta\in \Z_{+}^3  }  2^{|\alpha|k_1+|\beta|k_2}\\
& \times \|\mathcal{F}^{-1}_{\xi, \eta}[\nabla_{\xi}^\alpha \nabla_\eta^\beta m(\xi-\eta, \eta ) \psi_k(\xi) \psi_{k_1}(\xi-\eta)\psi_{k_2}(\eta)](x,y)\|_{L^1_{x, y}}< +\infty \}
 \\
 P_v^d \mathcal{M}^{a;b,c}& :=\{m:\R^3_\xi \times \R^3_\eta\times  \R_v^3 \longrightarrow \C, \| m(\xi-\eta, \eta,v)\|_{P_v^d \mathcal{M}^{a;b,c} }:=\sup_{k,k_1,k_2\in \Z_+^3} \sum_{ \alpha, \beta\in \Z_{+}^3} 2^{-ak-bk_1-c k_2} \\ 
 &\times  \| \langle v \rangle^{-d} \mathcal{F}^{-1}_{\xi, \eta}[\nabla_{\xi}^\alpha \nabla_\eta^\beta m(\xi-\eta, \eta,v)\psi_k(\xi) \psi_{k_1}(\xi-\eta)\psi_{k_2}(\eta)](x,y)\|_{ L^\infty_v  L^1_{x,y}} < \infty \}. 
\end{split}
 \ee
  As summarized in the  following Lemma \ref{nullstructurefact}, all quadratic terms have null structure except the $ Q(\vartheta_{mn}, \vartheta_{pq})$ type quadratic term in the equation satisfied by the component $\underline{F}.$ 
    \begin{lemma}\label{nullstructurefact}
Let $qk_{\alpha\beta;\mu\nu}^{\tilde{h}_1, \tilde{h}_2}(\xi-\eta, \eta), k\in \{s,p,q\}$, be the symbol of the bilinear operators $QK_{\alpha\beta;\mu\nu}^{\tilde{h}_1, \tilde{h}_2} $ in \textup{(\ref{2020june17eqn22})}. For any $\tilde{h}_1, \tilde{h}_2\in \{F,\underline{F}, \omega_j, \vartheta_{mn}\}, \mu, \nu\in \{+, -\}$, we have  $ qs_{\alpha\beta;\mu\nu}^{\tilde{h}_1, \tilde{h}_2}(\xi-\eta, \eta)\in |\xi-\eta||\eta|^{-1} \mathcal{M}_{\mu\nu}^{null} \cup |\xi-\eta|^{-1}|\eta| \mathcal{M}_{\mu\nu}^{null},$ and $  qq_{\alpha\beta;\mu\nu}^{\tilde{h}_1, \tilde{h}_2} \in \mathcal{M}_{\mu\nu}^{null} .$ Moreover, if $\tilde{h}_1, \tilde{h}_2\notin \{\vartheta_{mn}\}$, we have $  qp_{\alpha\beta;\mu\nu}^{\tilde{h}_1, \tilde{h}_2}(\xi-\eta, \eta) \in \mathcal{M}_{\mu\nu}^{null}.$ If $\tilde{h}\neq \underline{F}$ and $\tilde{h}_1, \tilde{h}_2\in \{\vartheta_{mn}\}$, we have $R_{\tilde{h}}^{\alpha\beta}(\xi)qp_{\alpha\beta;\mu\nu}^{\tilde{h}_1, \tilde{h}_2}(\xi-\eta, \eta) \in  \tilde{\mathcal{M}}_{-\nu}^{null}$ where $R_{\tilde{h}}^{\alpha\beta}(\xi)$ is the symbol of the Fourier multiplier operator $R_{\tilde{h}}^{\alpha\beta}$ defined in \textup{(\ref{connection})}. 
\end{lemma}
\begin{proof}
The above property depends solely on the Einstein equation, which has been verified and used in the study of the Einstein-Klein-Gordon system. We only provide an intuitive explanation here, 
see \cite{IP} for more details.

As a result of direct computations,   $\Lambda_{2}[  H_{i\gamma} \p_i\p_{\gamma} h_{\alpha\beta}]$,   $\Lambda_{2}[P_{\alpha\beta}] $ $+1/2\p_{\alpha} R_pR_q \vartheta_{mn} \p_\beta R_pR_q \vartheta_{mn}  $, and $\Lambda_{2}[Q_{\alpha\beta}]$, are linear combinations of null forms with symbols in $\mathcal{M}_{\mu\nu}^{null}\cup |\xi-\eta||\eta|^{-1} \mathcal{M}_{\mu\nu}^{null} \cup \xi-\eta|^{-1}|\eta| \mathcal{M}_{\mu\nu}^{null}$. For quasilinear term, the corresponding symbol belongs to $|\xi-\eta||\eta|^{-1} \mathcal{M}_{\mu\nu}^{null} \cup \xi-\eta|^{-1}|\eta| \mathcal{M}_{\mu\nu}^{null}$. For semilinear term, the corresponding symbol belongs to $ \mathcal{M}_{\mu\nu}^{null}  $.

Moreover, let $f=R_pR_q\vartheta_{mn}$,   we note that $\p_{0} f\p_0 f + R_iR_j(\p_i f\p_j f )$, $\varepsilon_{jkl}R_k\big(\p_0 f \p_l f \big)$, and $\varepsilon_{jmp}\varepsilon_{knq}$ $ R_m R_n( \p_p f \p_q f) $ are all null forms with symbol in $ \tilde{\mathcal{M}}_{-\nu}^{null}$, i.e., $R_{\tilde{h}}^{\alpha\beta}(\xi)qp_{\alpha\beta;\mu\nu}^{\tilde{h}_1, \tilde{h}_2}(\xi-\eta, \eta) \in    \tilde{\mathcal{M}}_{-\nu}^{null}$, where $\tilde{h}\in \{F, \omega_j, \vartheta_{mn}\}$.

That  is to say,  all quadratic terms have null structure except the $ Q(\vartheta_{mn}, \vartheta_{pq})$ type quadratic term in the equation satisfied by the component $\underline{F}.$ 

\end{proof}

\subsection{Bootstrap assumption}

Let $\delta_1:=10^{-20}$, $\gamma:=\delta_1/100$,   ${\delta}_0:=100 \delta_1,$ and $N_0=200 $. Our bootstrap assumption for the perturbed metric component is  stated as follows, 
\[
 \sup_{t\in[0,T] }   \sum_{ n, l\in \Z_+, n+l\leq N_0}\sum_{   {L}\in {P}^{  }_n}   
   \langle t \rangle^{-   d( n, l) \delta_0 }   \|     U^{{L}h_{\alpha\beta}}  (t)) \|_{E_{   l}} +  \sum_{\tilde{h}\in \{F, \omega_j, \vartheta_{mn}\}}   \langle t \rangle^{-   e( |L|;\tilde{h})\delta_1 } \|      \widetilde{V^{ \tilde{h}^L }}(t)  \|_{E_{   l}} 
    \]
    \[
+\langle t \rangle^{-  e(   |L|;\underline{F})\delta_1}  \|    \widetilde{V^{  \underline{F}^L }}(t)    \|_{E_{ l}}  
        + \sum_{\begin{subarray}{c}
\tilde{L}   \in {P}^{  }_l, l\leq N_0-10\\
i\in\{1,2,3\}\\
\end{subarray}} \sup_{k\in \Z}
   \langle t \rangle^{-   {H}(l+1)\delta_1}  2^{k-k_{-}/2+\gamma k_{-}+ (N_0-l-1)k_{+} }
    \]
\be\label{BAmetricold}
\times  \| P_k(x_i  V^{ \tilde{L}h_{\alpha\beta}})(t)\|_{L^2}   + \|  {V^{\tilde{h}}}(t)\|_{Z}
    + \langle t\rangle^{-\delta_1 }\|  {V^{\underline{F}}}(t)\|_{Z} \lesssim \epsilon_1:=\epsilon_0^{5/6}   ,
 \ee
where the energy space  $E_{ l}$is defined in (\ref{energyespace}),     the $Z$-normed space is defined as follows,
\[
\| f\|_{Z}:=\sup_{k\in \mathbb{Z}} 2^{1.001k_{-} + 40k_{+}} \|\widehat{f}(\xi)\psi_k(\xi)\|_{L^\infty_\xi},
\]
and $d(n,l)$, $e(|L|; \tilde{h}), \tilde{h}\in \{F, \underline{F}, \omega_j, \vartheta_{mn}\}$, and $H(n), n\in \Z_{+}$ are defined as follows
\[
 e(|L|;\tilde{h}):=\left\{\begin{array}{ll}
 H(|L|) & \tilde{h}= \underline{F}\\
 H(|L|)-\beta(|L|) & \tilde{h}\in \{F, \omega_j, \vartheta_{mn} \}\\
\end{array}\right.,\quad   d(n,l):=H(n+l)+ 10n, 
 \]
\be\label{2022may3eqn11}
  H(n)= \left\{\begin{array}{ll}
30n+3  & n< 2\\ 
200n^2  & n\geq 2\\ 
\end{array}\right., \quad \beta(n) =\left\{\begin{array}{ll}
1+1/50 & \textup{if\,} n< 2\\
H(1) +1/50 & \textup{if\,} n\geq 2 \\
\end{array}\right..
\ee

The choice of function $H(n)$ guarantees the following estimates hold, 
\be\label{july15eqn21}
\begin{split}
 \textup{if\,\,} a\leq b, & \quad  H(a+1)+H(b)\leq H(a+b) + H(1 ), \\  
 \textup{if\,\,} n \geq 1, & \quad   H(n)- H(n-1) \geq 20\beta(n). 
  \end{split}
\ee 
 
  The bootstrap assumption we stated in (\ref{BAmetricold}) and the energy spaces we use are stronger than the one used in \cite{IP}, i.e., the  bootstrap assumption stated in (\ref{BAmetricold}) is automatically satisfied. The    $Z$-normed space is   same as the one used in  \cite{IP}.  This setting  allows us to take the analysis of the increment of  $Z$-norm  of the perturbed metric  contributed from the Einstein part as a black box  and be more focus on the overall effect of the Vlasov part, which is the main theme of this paper.

 Recall the equation (\ref{2020april8eqn31}) and the index $\tilde{c}(\beta),\beta\in \mathcal{S},$ defined in \eqref{countingnumber}.  Our bootstrap assumption for the Vlasov part is stated as follows, 
\be\label{BAvlasovori}
\sup_{t\in[0,T]}   \sum_{
\begin{subarray}{c}
 \mathcal{L}\in \cup_{l\leq N_0  }\mathcal{P}^{  }_{l},\rho\in \mathcal{S}, 
    | \mathcal{L}| +  |\rho |\leq N_0 
\end{subarray}}     (1+t)^{-H(\mathcal{L}, \rho) \delta_0   }   E_{\mathcal{L}, \rho}^{vl}(t) + Z^{vl}(t) \lesssim \epsilon_1,
\ee
where $ E_{\mathcal{L}, \rho}^{vl}(t)$ aims to control the wighted energy of the Vlasov part and    the $Z$-norm  of the Vlasov part $ Z^{vl}(t)$  aims to show that the density of particles and its derivatives decay sharply over time in a low regularity space. More precisely, 
\begin{multline}
 H(\mathcal{L}, \rho):=  d(|\mathcal{L}|+|\tilde{c}(\rho)|, |\rho|-|\tilde{c}(\rho)|),\qquad    E_{\mathcal{L}, \rho}^{vl}(t):=  \| \omega^{\mathcal{L}}_{\rho} (t,x, v) \Lambda^{\rho}{u}^{\mathcal{L}}(t,x,v)\|_{L^2_{x,v}},\\
 Z^{vl}(t):=   \sum_{
\begin{subarray}{c}
 \mathcal{L}\in \cup_{l\leq N_0   }\mathcal{P}^{  }_{l}, \alpha\in\Z_{+}^3,|\alpha|+|\mathcal{L}|\leq N_0 
\end{subarray}}   \|\widetilde{  \omega^{\mathcal{L}}_{\alpha} }(v)\big(\nabla_v^\alpha \widehat{u^{\mathcal{L} }}(t, 0, v) - \nabla_v\cdot {u^{\mathcal{L};\alpha}_{corr}}(t,v) \|_{L^2_{x,v}}, 
\end{multline}
 where the above weight functions  $\omega^{\mathcal{L}}_{\rho} (t, x, v)$, and $\widetilde{  \omega^{\mathcal{L}}_{\alpha} }(v)$  are defined as follows, 
  \be\label{weightfunctions2}
  \omega^{\mathcal{L}}_{\rho} (t, x, v) = (1+|v|)^{c(\rho)} \big( 1+|v|^{1 /\delta_1}+ |x|\phi(|x|/(1+|t|)^{4H( 4)\delta_1 })\big)^{200N_0-10(| \mathcal{L}| + |\rho| )},
 \ee
  \be\label{weightfunctions3}
  \widetilde{  \omega^{\mathcal{L}}_{\alpha} }(v) = (1 +|v|^{1 /\delta_1}  )^{100N_0-10  (| \mathcal{L}| + |\alpha| )},
 \ee
  where $c(\rho)$ and $\tilde{c}(\rho)$ are defined in (\ref{countingnumber}) and    the cutoff function  $\phi:\R_{+}\longrightarrow \R_{+}$ is defined as follows, 
 \be\label{2020april27eqn61}
 \phi(x):=\left\{\begin{array}{ll}
 0 & \textup{if}\,\, x\leq 1\\ 
 (x-1)^2 & \textup{if}\,\, x\in(1,2]\\
3-2e^{-(x-2)} & \textup{if}\,\, x\in(2,\infty)\\
 \end{array}\right. , \quad \Longrightarrow \phi'(x)\geq 0, \quad \p_t \omega^{\mathcal{L}}_{\rho} (t, x, v) \leq 0   .
 \ee

The   correction term  ``${u^{\mathcal{L};\alpha}_{corr}}(t,v)$'', which is introduced to get around a losing derivative issue,  is defined as follows, 
\be\label{correctionterm}
{u^{\mathcal{L};\alpha}_{corr}}(t,v):=\left\{\begin{array}{ll}
\displaystyle{\int_0^t \int_{\R^3} - \h \big(  \frac{v_{\alpha} v_{\beta} }{v^0}  \nabla_x g^{\alpha\beta}\big)(s, x+s\hat{v})    \nabla_v^{\alpha} u^{\mathcal{L}}(s,x,v)  d x d s }& \textup{if\,\,} |\alpha|+|\mathcal{L}|= N_0\\ 
&\\ 
 0  & \textup{if\,\,}  |\alpha|+|\mathcal{L}| < N_0. \\
\end{array}\right. 
\ee

\subsection{Bootstrap argument and the proof of the theorem \ref{maintheorem}}

Assume that $(g, f)$ is a solution of the Einstein-Vlasov system \eqref{EV} on time $[0, T]$ such that the  bootstrap assumptions 
 (\ref{BAmetricold}) and \eqref{BAvlasovori} are valid, the existence of $T>0$ is ensured by the local existence result of  Choquet-Bruhat \cite{Bruhat}.

Firstly, from the estimate \eqref{may7eqn21} in Proposition \ref{energyvlasovest} and the estimate \eqref{oct11eqn6} in Proposition \ref{estimateofzerofrequency},   the bootstrap assumption for the Vlasov part is improved. Secondly,  from the estimate \eqref{aug19eqn70} in Proposition \ref{energyestwave},  the estimate of $\widetilde{V^{ \tilde{h}^L }}(t) , \tilde{h}\in\{F, \underline{F}, \omega_i, \vartheta_{mn}\}$ is improved. Moreover, from      the estimate \eqref{2020july17eqn1} in Proposition \ref{fixedtimenonlinarityestimate},  the weighted estimate of  $V^{ \tilde{L}h_{\alpha\beta}} (t)$ is improved.  Lastly, recall    (\ref{connection}),  (\ref{april1eqn11}), and \eqref{2020april1eqn4},  from the estimate (\ref{oct7eqn73}) in Lemma  \ref{harmonicgaugehigher} and the improved estimate of $\widetilde{V^{ \tilde{h}^L }}(t), \tilde{h}\in\{F, \underline{F}, \omega_i, \vartheta_{mn}\}$ and the improved estimates of the Vlasov part,  the estimate of $   U^{{L}h_{\alpha\beta}}  (t))$  is also improved. 

Hence closing the bootstrap argument, we can extend the time interval $[0, T]$ to $[0, \infty)$, i.e., the global regularity holds.  

\section{ Estimates of nonlinearities for any fixed time}\label{fixedtimeestimate}

In this section,  one of  main goals  is to prove the following proposition \ref{fixedtimenonlinarityestimate}, in which  we prove rough estimates for the nonlinearity of the metric part, which are convenient at places where it's not necessary to be sharp.    We also reduce the goal of improving the bootstrap assumption on $P_k(x_i  V^{ \tilde{L}h_{\alpha\beta}})(t)$ to the energy estimates of the metric part and the Vlasov part and the the $Z$-norm estimate of the Vlasov part, see \eqref{2020july17eqn1}. 

 Moreover, as byproducts, we   obtain  some estimates  of the error type terms in the double Hodge decomposition   in subsection \ref{auxiliary} and an improved estimate for the modified profile at low frequency in subsection \ref{lowfreq}. 

\begin{proposition}\label{fixedtimenonlinarityestimate}
Under the bootstrap assumption \textup{(\ref{BAmetricold})} and \textup{(\ref{BAvlasovori})}, the following 
estimate holds for any fixed $t\in[2^{m-1},2^m]\subset[0,T]$, $L\in \cup_{n\leq N_0}P_n$, $\tilde{L}\in \cup_{l\leq N_0-10} P_{l},  $ 
\[
 \forall l\in [0, N_0-|L|]\cap \Z, \quad 2^{-k_{-}/2+\gamma k_{-}+ (l-1) k_{+}}\big(\|\p_t \widehat{ \widetilde{   V^{L\tilde{h} } }}(t, \xi) \psi_k(\xi)\|_{L^2}  + \| \Lambda_{\geq 2}[ \p_t \widehat{   {   V^{L h_{\alpha\beta} } }}(t, \xi) ]\psi_k(\xi)\|_{L^2} \big)
\]
\be\label{oct4eqn41}
 \lesssim   2^{-m+ H(|L|)\delta_1 m + \beta_{}(|L|) \delta_1 m} \epsilon_1^2 + 2^{-m+k_{-}+d(|L|+1, l)\delta_0m } \epsilon_1^2, 
\ee
\[
  \sup_{k\in \mathbb{Z}}    2^{ k_{-}/2+\gamma k_{-}+ (N_0-|\tilde{L}| -1)k_{+} }  \|\nabla_\xi \widehat{ {   V^{ \tilde{L}  {h}_{\alpha\beta} } }}(t, \xi) \psi_k(\xi)\|_{L^2}\lesssim    2^{     {H}( |\tilde{L}|+1)\delta_1 m} \epsilon_1^2+  Z^{vl}(t) 
    \]
 \be\label{2020july17eqn1}
 +   \sup_{k\in \mathbb{Z}} \sum_{\Gamma\in P_1}   2^{ -k_{-}/2+\gamma k_{-}+(N_0-|\tilde{L}| -1)k_{+} }  \| \widehat{ {   U^{\Gamma \tilde{L}  {h}_{\alpha\beta} } }}(t, \xi) \psi_k(\xi)\|_{L^2}+ \sum_{|\mathcal{L}|+|\rho|\leq N_0} 2^{-m/2} E_{\mathcal{L}, \rho}^{vl}(t).
\ee

\end{proposition}
\begin{proof}
For better presentation, we postpone  the proof   to subsection \ref{proofpropfixed}.  
\end{proof}

\subsection{Linear decay estimates for the density type function and the wave part }

We first record the following $L^\infty$-decay estimate for the density type function of the Vlasov part. 
 \begin{lemma}\label{decayestimateofdensity}
For   $x\in \R^3$, $  t\in \R$, $k\in\Z, n\in \Z_{+},$, s.t., $|t|\geq 1$,  and any given zero order symbol $m(\xi,v)\in L^\infty_v \mathcal{S}^0$, the following decay estimate holds, 
\[
 \sup_{\xi\in \R^n}\big|\int_{\R^3} e^{  i t \hat{v} \cdot \xi} m(\xi, v)  \widehat{g}(t, \xi, v) \psi_{k}(\xi) d v    \big|
 \]
 \[
 \lesssim \sum_{|\alpha|+|\beta|\leq  n} 2^{-nk} |t|^{-n} \|\nabla_v^\beta m(\xi,v)\|_{L^\infty_{v, \xi}} \min\big\{ \|(1+|v|)^{ 4n} \nabla_v^\alpha \widehat{g}(t,x,v) \|_{L^1_{x,v} } ,
 \]
\be\label{densitydecay}
   \|(1+|v|)^{ 4n} \nabla_v^\alpha \widehat{g}(t,0,v) \|_{L^1_v }  +\sum_{|\gamma|\leq 1}|t|^{-1 } \| (1+|v|)^{4+4n} (1+|x|)\nabla_v^\gamma \nabla_v^\alpha g(t,x,v)\|_{L^1_x L^1_v}\big\}.
\ee
Similarly, the following estimate holds for any fixed $x\in \R^3, a\in \R\cap(-3, \infty)$,  and $h(x,v)\in H^{10+a}_{x,v}$,
\[
 \big|\int_{\R^3}\int_{\R^3} e^{i x\cdot \xi + i t \hat{v} \cdot \xi} m(\xi, v)|\xi|^{a}\widehat{g}(t, \xi, v) h(x,v) d v d\xi   \big|
 \]
 \[
 \lesssim \sum_{|\alpha|\leq 5+\lfloor a\rfloor} \big(\sum_{|\beta|\leq 5+\lfloor a\rfloor}   \|\nabla_v^\beta m(\xi,v)\|_{L^\infty_v\mathcal{S}^\infty}\big)   \big[|t|^{-3-a} \|(1+|v|)^{5+|a|} \nabla_v^\alpha\big( h(x,v) \widehat{g}(t,0,v)\big) \|_{L^1_v } 
\]
\be\label{julyeqn21}
+|t|^{-4-a} \| (1+|v|)^{5+|a|} (1+|y|)\nabla_v^\alpha\big( h(x,v)g(t,y,v)\big)\|_{L^1_y L^1_v}\big].
\ee
For any $n\in\mathbb{Z}_{+},p\in[1,\infty) $, the following estimate holds, 
\[
\big\|\int_{\R^3}\int_{\R^3} e^{i x\cdot \xi + i t \hat{v} \cdot \xi} m(\xi, v)|\xi|^{n}\widehat{g}(t, \xi, v) h(x,v) \psi_k(\xi) d v d\xi    \big\|_{L^p_x}
\]
\be\label{oct6eqn61}
\lesssim \sum_{|\alpha|,|\beta|,|\gamma|\leq n}|t|^{-n}\|\nabla_v^\alpha m(\xi,v)\|_{L^\infty_v\mathcal{S}^\infty} \|\nabla_v^\beta h(x,v)\|_{L^\infty_{x,v}} \| (1+ |v|)^{10+3n}\nabla_v^\gamma g(t,x,v)\|_{L^p_{x,v}}.
\ee 
\end{lemma}
\begin{proof}
The main ideas are  singling out the zero frequency (or not) for different desired estimates and doing integration by parts in $v$. See  \cite{wang3}[The proof of Lemma 3.1] for more details.
\end{proof}

In the following Lemma, as in the same spirit of \cite{IP}[Lemma 2.11], we prove $L^\infty_x$ decay estimate for the frequency super localized linear half-wave. 
 
\begin{lemma}\label{superlocalizedaug}
Let $t\in \R, x\in \R^3, \mu \in\{+,-\}$, $k, n,\tilde{k}\in \mathbb{Z}$, s.t., $|t|\gg 1$, $2^{k}\geq  |t|^{-1}$,  $\tilde{k}\leq k$. 
If  $|x|\leq 2^{-100}|t|$, we have 
\[
 \big|\int_{\R^3}  e^{ i x\cdot \xi-i \mu t |\xi|  }     m(\xi) \widehat{f}(\xi) \psi_k(\xi) \psi(2^{-\tilde{k}} |\xi|-n)  d \xi \big| 
\]
\be\label{decayestimatesmall}
  \lesssim  |t|^{-1} \|m(\xi)\|_{\mathcal{S}^\infty_k}  2^{ \tilde{k}/2}\big(2^k\|\nabla_\xi \widehat{f}(\xi)\psi_k(\xi) \psi(2^{-\tilde{k}} |\xi|-n)\|_{L^2_\xi} +\|  \widehat{f}(\xi)\psi_k(\xi) \psi(2^{-\tilde{k}} |\xi|-n)\|_{L^2_\xi}
 \big).
\ee
If  $|x|\geq 2^{-100}|t|$, we have 
\[
 \big|\int_{\R^3}  e^{ i x\cdot \xi-i \mu t |\xi|  }     m(\xi) \widehat{f}(\xi) \psi_k(\xi) \psi(2^{-\tilde{k}} |\xi|-n)  d \xi \big|   \lesssim     \sum_{\Gamma \in \{\Omega_{ij}
\}, |\alpha|\leq 1, |\beta|\leq 2 } |x|^{-1}  2^{\tilde{k}/2}  \|m(\xi)\|_{\mathcal{S}^\infty_k} 
 \]
 \be\label{july27eqn4}
\times \|    \widehat{\Gamma^\alpha f}(\xi )\psi(2^{-\tilde{k}} |\xi|-n)  \psi_k(|\xi|)\|_{L^2(\xi) }^{1-\delta_1^2} \|    \widehat{\Gamma^\beta f}(\xi )\psi(2^{-\tilde{k}} |\xi|-n)\psi_k(|\xi|)\|_{L^2(\xi) }^{\delta_1^2},
\ee
\[
 \big|\int_{\R^3}  e^{ i x\cdot \xi-i \mu t |\xi|  }     m(\xi) \widehat{f}(\xi) \psi_k(\xi) \psi(2^{-\tilde{k}} |\xi|-n)  d \xi \big|   \lesssim    \sum_{\Gamma \in \{\Omega_{ij}
\},   |\beta|\leq 2 }  |x|^{-1}  2^{\tilde{k}/2}  \|m(\xi)\|_{\mathcal{S}^\infty_k} 
\]
\be\label{linearwavedecay2}
 \times \big[2^{k+\tilde{k}/2}\|\widehat{f}(\xi)\psi_k(\xi) \psi(2^{-\tilde{k}} |\xi|-n) \|_{L^\infty_\xi} + |x|^{-1/3} 2^{-k/3} \|    \widehat{\Gamma^\beta f}(\xi )\psi(2^{-\tilde{k}} |\xi|-n)\psi_k(|\xi|)\|_{L^2(\xi) }\big].
\ee
\end{lemma}
\begin{proof}
Note that $\nabla_\xi(x\cdot\xi-\mu t|\xi|)=0$ if and only if \mbox{${\xi}/{|\xi|} = \mu {x}/{t} = \mu{x}/{|x|} := \xi_{0} $}.   Let $\bar{l} $ be the least integer such that  $2^{ \bar{l} } \geq   2^{ -k/2} |x|^{-1/2}$. From the volume of support of $\xi$ and the Sobolev embedding in angular variables, we have 
\[
 \Big| 
\int_{\R^3}    e^{ i x\cdot \xi-i \mu t |\xi|  }  
 \widehat{f}(\xi) m(\xi) \psi_{k}(|\xi|) \psi_{\leq  \bar{l} }(\angle(\xi,\xi_0))\psi(2^{-\tilde{k}} |\xi|-n) 
  d \xi \Big| 
   \]
 \be\label{july27eqn2}
   \lesssim \|m(\xi)\|_{\mathcal{S}^\infty_k}  2^{(2k+\tilde{k})/2+2\bar{l}}  \|\widehat{f}(r\theta)\psi_k(r)\psi(2^{-\tilde{k}} r-n) \|_{L^2(r^2 dr ) L^\infty_\theta}  .
\ee

For the case when the angle is localized around $2^l$ where $l> \bar{l}$, we first do integration by parts in $\xi$ in the rotational directions   once. As a result, we have
\begin{multline}
 \int_{\R^3}  e^{ i x\cdot \xi-i \mu t |\xi|  }  
   \widehat{f}(\xi) m(\xi) \psi_{k}(\xi) \psi(2^{-\tilde{k}} |\xi|-n)  \psi_{l }(\angle(\xi,\xi_0))
  d \xi  =I_l^1+I_l^2,\\
  I_l^1=  \int_{\R^3}  e^{ ix\cdot \xi - i \mu t|\xi|  }   \big(  \frac{i x \times\xi }{|x\times \xi|^2}    \cdot    (\xi\times 
 \nabla_\xi)\widehat{f}(\xi)   m(\xi)\psi(2^{-\tilde{k}} |\xi|-n) \psi_{k}(\xi) \psi_{l }(\angle(\xi,\xi_0))\big) 
  d \xi,  \\
    I_l^2=\int_{\R^3}  e^{ ix\cdot \xi - i \mu t|\xi|  }      \big[ (\xi\times 
 \nabla_\xi)\cdot  \big(  \frac{i x \times\xi }{|x\times \xi|^2}      m(\xi)\psi(2^{-\tilde{k}} |\xi|-n) \psi_{k}(\xi) \psi_{l }(\angle(\xi,\xi_0))\big)\big]  \widehat{f}(\xi) 
  d \xi. 
  \end{multline}
  For  $ I_l^2$,  we   do integration by parts in $\xi$ in the rotational directions   one more time. As a result, we have 
  \begin{multline}\label{aug2eqn3}
   I_l^2=\int_{\R^3}  e^{ ix\cdot \xi - i \mu t|\xi|  }  (\xi\times \nabla_\xi)\cdot\Big[ \frac{i x \times\xi }{|x\times \xi|^2}   \big[ (\xi\times 
 \nabla_\xi)\cdot  \big(  \frac{i x \times\xi }{|x\times \xi|^2}     m(\xi)\psi(2^{-\tilde{k}} |\xi|-n)\\
 \times  \psi_{k}(\xi) \psi_{l }(\angle(\xi,\xi_0))\big)  \widehat{f}(\xi) \big]\Big]
  d \xi. 
\end{multline}
 From the volume of support of ``$\xi$'', we have 
 \[
   |I_l^1|+  |I_l^2|\lesssim \sum_{\Gamma \in \{\Omega_{ij}
\}}  \big(  |x|^{-1} 2^{- l} + |x|^{-2} 2^{-2l} 2^{-k-l}\big)
 \]
 \[
  \times 2^{-   k } 2^{(2k+\tilde{k})/2+ (1+\delta_1^3){l} } \|    \widehat{\Gamma f}(r\theta )\psi(2^{-\tilde{k}} r-n)\psi_k(r)\|_{L^2(r^2 d r)L^{2/(1-\delta_1^3)}_\theta }\]
  \be\label{july27eqn1}
+ |x|^{-2} 2^{-2l} 2^{-2k-2l}   2^{(2k+\tilde{k})/2+2 {l}}   \|\widehat{f}(r\theta)\psi_k(r)\psi(2^{-\tilde{k}} r-n) \|_{L^2(r^2 dr ) L^\infty_\theta} 
 \ee
From the above estimate (\ref{july27eqn1}), we have
\[
\sum_{\bar{l}< l \leq 2} |I_{l}^1| +|I_{l}^2|  \lesssim    \sum_{\Gamma \in \{\Omega_{ij}
\}}    |x|^{-1}  2^{\tilde{k}/2}  \|m(\xi)\|_{\mathcal{S}^\infty_k} \big[ \|\widehat{f}(r\theta)\psi_k(r)\psi(2^{-\tilde{k}} r-n) \|_{L^2(r^2 dr ) L^\infty_\theta} 
\]
 \be\label{july27eqn3}
+   \|    \widehat{\Gamma f}(r\theta )\psi(2^{-\tilde{k}} r-n)\psi_k(r)\|_{L^2(r^2 d r)L^{2/(1-\delta_1^3)}_\theta }\big].
\ee 
 The desired estimate  (\ref{july27eqn4}) holds from the estimates (\ref{july27eqn2}) and  (\ref{july27eqn3}) and the Sobolev embedding on the sphere. 

 Similarly,  if we rerun the above argument, do    integration by parts in $\xi$ in the rotational directions   one more time for $  I_l^1 $. As a result, the following estimate holds, 
 \[
  \big|\int_{\R^3}  e^{ i x\cdot \xi-i \mu t |\xi|  }     m(\xi) \widehat{f}(\xi) \psi_k(\xi) \psi(2^{-\tilde{k}} |\xi|-n)  d \xi \big| \lesssim \sum_{l\geq \bar{l}}
   \sum_{\Gamma \in \{\Omega_{ij}
\}}  |x|^{-2} 2^{-2l}\big[ 2^{-k-l}   2^{-   k } 2^{(2k+\tilde{k})/2+ 5{l}/3 } \]
  \be 
\times \|    \widehat{\Gamma f}(r\theta )\psi(2^{-\tilde{k}} r-n)\psi_k(r)\|_{L^2(r^2 d r)L^{6}_\theta }+  2^{-2k-2l}   2^{(2k+\tilde{k})/2+2 {l}}   \|\widehat{f}(r\theta)\psi_k(r)\psi(2^{-\tilde{k}} r-n) \|_{L^2(r^2 dr ) L^\infty_\theta} \big].
 \ee
 Therefore, our   desired estimate (\ref{linearwavedecay2}) holds from the above estimate and  the Sobolev embedding on the sphere. 
\end{proof}

\subsection{Estimates of basic quantities  }

\subsubsection{Difference between the perturbed metric and the modified perturbed metric}

In the following Lemma, we    show that the growth rates of the modified perturbed metric and the perturbed metric are almost comparable in both $L^2$-type space and the $L^\infty$-type space. They are only different at very top order. 
\begin{lemma}\label{basicestimates}
Under the bootstrap assumptions \textup{(\ref{BAmetricold})} and \textup{(\ref{BAvlasovori})}, for any $t\in[2^{m-1}, 2^m],   \tilde{L}\in \cup_{n\leq N_0}P_n,   {L}\in \cup_{n\leq N_0-5}P_n$, $\mu\in\{0,1,2,3\},$   we have 
\[ 
  \| \Lambda_{1}[ \p_t \widehat{   {   V^{\tilde{L} h_{\alpha\beta} } }}(t, \xi) ]\psi_k(\xi)\|_{L^2} + 2^k\|  P_{k}\big(U^{\tilde{L} h_{\alpha\beta}}(t)-\widetilde{U^{\tilde{L}h_{\alpha\beta}}}(t) \big)\|_{L^2_{x}} 
\]
\be\label{2022march19eqn30}
 \lesssim  
  2^{3k_{-}/2} \min\{2^{d( |\tilde{L}|, 0)\delta_0 m}, 2^{k+d( |\tilde{L}|, 0)\delta_0 m }+1  
  + 2^{m+k}\mathbf{1}_{|\tilde{L}|=N_0 
  }\}\epsilon_1, 
\ee
\be\label{2022march19eqn34}
\sum_{ |L|\leq N_0 -2} \sum_{l\in [0, N_0-|L|]\cap\Z}  \sup_{k\in \Z   } 2^{-k_{-}/2+\gamma k_{-} +  l k_{+}}\|  P_{k}\big(U^{ {L} h_{\alpha\beta}}(t)-\widetilde{U^{ {L}h_{\alpha\beta}}}(t) \big)\|_{L^2_{x}} \lesssim \epsilon_1,
\ee
\[
  2^{-k_{-}+ (N_0-3-|L|) k_{+}}\big( \| P_{k}(  \widetilde{U^{ {L} h_{\alpha\beta}  }})(t )   \|_{L^\infty_x}  + \| P_{k}(  {U^{ {L} h_{\alpha\beta}  }})(t )  \|_{L^\infty_x}\big) 
  +\| L h_{\alpha\beta}\|_{L^\infty_x}
\]
\be\label{2020julybasicestiamte}
 \lesssim 2^{-m +H(|L|+1)\delta_1 m+1.1\gamma m  }\epsilon_1  . 
 \ee
\end{lemma}
\begin{proof}

Note that the following decomposition  holds for any differentiable symbol $m(\xi, v)$, 
 \begin{multline}\label{oct11eqn11}
 \int_{\R^3} e^{-it\hat{v}\cdot \xi}\widehat{u^{\mathcal{L}}}(t, \xi, v)  m(\xi,v)  \psi_k(\xi)    d v  =  \int_0^1 \int_{\R^3} e^{-it\hat{v}\cdot \xi}  m(\xi,v)    \xi\cdot \nabla_\xi \widehat{u^{\mathcal{L}}}(t, s\xi, v)   \psi_k(\xi) d v d s  \\
 + \int_{\R^3} e^{-it\hat{v}\cdot \xi}  m(\xi,v)     \big(\widehat{u^{\mathcal{L}}}(t, 0, v) -\nabla_v\cdot \tilde{u}^{\mathcal{L}}_{corr}(t,v)\big)\psi_k(\xi) d v  - \int_{\R^3} e^{-it\hat{v}\cdot \xi}  \nabla_v  m(\xi,v)    \cdot  \tilde{u}^{\mathcal{L}}_{corr}(t,v) \psi_k(\xi)d v\\ 
 + \int_{\R^3} e^{-it\hat{v}\cdot \xi}   m(\xi,v)     t \nabla_v( i \hat{v}\cdot \xi) \cdot  \tilde{u}^{\mathcal{L}}_{corr}(t,v) \psi_k(\xi)d v. 
\end{multline}
Recall  (\ref{sep20eqn61}) and the definition of the modified perturbed metric in (\ref{modifiedperturbedmetric}). From the above decomposition and the volume of support of $\xi$ if $k\leq 0$, we  have   
\[
 \|P_{k}\big(\Lambda_{1}[L \mathcal{N}_{\alpha\beta}^{vl}]\big)(t,x) \|_{L^2_{x}}    \lesssim  
  2^{3k_{-}/2-l k_{+}} \min\{2^{d( |\tilde{L}|, l)\delta_0 m},    2^{k+ d( |\tilde{L}|, l)\delta_0 m }  +1 + 2^{m+k}\mathbf{1}_{| {L}|=N_0 
  }\}\epsilon_1
\]
\[
 \|  P_{k}\big(U^{Lh_{\alpha\beta}}(t)-\widetilde{U^{Lh_{\alpha\beta}}}(t) \big)\|_{L^2_{x}}
\]
\be\label{2020june24eqn40}
 \lesssim  2^{3k_{-}/2-k- l k_{+}} \min\{2^{ d( |\tilde{L}|, l) \delta_0 m},    2^{k+ d( |\tilde{L}|, l) \delta_0 m }  +1 + 2^{m+k}\mathbf{1}_{| {L}|=N_0 
  } \}\epsilon_1.
\ee
Hence finishing the  proof of the desired estimate   (\ref{2022march19eqn30}).

From the decomposition (\ref{oct11eqn11}), the estimate (\ref{densitydecay})  in Lemma \ref{decayestimateofdensity}, the following estimate holds for any $k\in \Z$ if $|L|\leq N_0-2$, and $l\in [0, N_0-|L|]\cap \Z, $
\[
  2^{-k_{-}/2+\gamma k_{-} +  l k_{+}} \|  P_{k}\big(U^{Lh_{\alpha\beta}}(t)-\widetilde{U^{Lh_{\alpha\beta}}}(t) \big)\|_{L^2_{x}} 
  \]
  \[
  \lesssim  \min\{ 2^{- m- k_{-} + d(| {L}|+1, (l-1)_+ )\delta_0 m } , 2^{l k_+}( 1+2^{k+d(|L|,0)\delta_0 m } )\}\epsilon_1\lesssim \epsilon_1. 
 \]
Hence finishing the proof of the desired estimate   (\ref{2022march19eqn34}).

From the linear decay estimates (\ref{decayestimatesmall}) and  (\ref{july27eqn4})  in Lemma \ref{superlocalizedaug} and the obtained   estimate   (\ref{2022march19eqn34}),   the following estimate holds for any $t\in[2^{m-1},2^m]\subset [0,T], L\in \cup_{n\leq N_0-10} P_n ,$
\be\label{aug16eqn1}
 \| P_{k}(  {U^{ {L} h_{\alpha\beta}  }}) \|_{L^\infty_x}  \lesssim 2^{k_{-}/2-\gamma k_{-}}\min\{2^{3k/2-(N_0-|L|) k_{+}},2^{-(N_0-|L|-2)k_{+}} 2^{-m +k/2+  H(| {L}|+1)\delta_1 m +\gamma m/100   } \}\epsilon_1. 
\ee
   From the decay estimate (\ref{densitydecay}) of the density type function in Lemma  \ref{decayestimateofdensity}, the following estimate holds if $|L|\leq N_0-5$, 
\be\label{2020april16eqn51}
\| P_k( \tilde{\rho}^{h_{\alpha\beta}}_{L}(f))(t)\|_{L^\infty_x }  + \sum_{\tilde{h}\in\{F, \underline{F}, \omega_j, \vartheta_{mn}\}} \|  P_k(\tilde{\rho}^{\tilde{h}}_{L}(f))(t)\|_{L^\infty_x } \lesssim \min\{2^{2k}, 2^{-3m- k }\} 2^{-(N_0-|L|-3)k_{+}}\epsilon_1.
\ee
Recall the definition of modified half-wave in (\ref{modifiedperturbedmetric}).  We know that the last  estimate   in (\ref{2020julybasicestiamte}) hold   from the estimates (\ref{aug16eqn1}) and (\ref{2020april16eqn51}).  

\end{proof}

\subsubsection{General bilinear estimates for the wave-wave type interaction}
  \begin{lemma}\label{aug9bilinearLemma}
 For any  $ L_1,L_2\in \cup_{  0\leq n\leq N_0, |\alpha|\leq N_0-n }  \nabla_x^\alpha P_n, $ s.t., $ L_1\circ L_2 = L \in \cup_{  0\leq n\leq N_0, |\alpha|\leq N_0-n }  \nabla_x^\alpha P_n$, $m\in \mathbb{Z}_{+},t\in [2^{m-1}, 2^m]\subset[0, T]$, $f,g\in \{F, \underline{F},\omega_j, \vartheta_{mn}, h_{\alpha\beta}\}$, and  Fourier multiplier operators $T_1,T_2\in \mathcal{S}^0$, the following estimates hold  under  the bootstrap assumptions \textup{(\ref{BAmetricold})} and \textup{(\ref{BAvlasovori})}, 
\be\label{aug9eqn84}
\| T_1(L_1 f(t)) T_2( L_2 g(t))\|_{L^2_x} \lesssim 2^{-m/2 +   d( |L|+1, \|L\|-|L| )\delta_0 m   }\epsilon_1^2,  
\ee
\be\label{oct3reqn52}
\| \p_{\alpha} T_1( L_1 f(t)) T_2( L_2 g(t))\|_{L^2_x} \lesssim 2^{-2m/3 + d( |L|+1, \|L\|-|L| )\delta_0 m   }\epsilon_1^2, 
\ee
where  $\|L\|,$ and $|L|$ are defined in Definition \ref{vectorfieldscla}. 
 \end{lemma}
 \begin{proof}

 After doing dyadic decompositions for two inputs and using the orthogonality in $L^2$, we have
 \[
 \| T_1(L_1 f(t)) T_2( L_2 g(t))\|_{L^2_x}^2 \lesssim \sum_{k \in \Z }   \|P_{k }(T_1(L_1 f(t))) P_{\leq k-10}(T_2( L_2 g(t)))\|_{L^2_x}^2
 \]
 \be\label{2020aug24eqn10}
   + \|P_{\leq k-10 }(T_1(L_1 f(t))) P_{ k }(T_2( L_2 g(t)))\|_{L^2_x}^2+ \big(\sum_{|k_1-k_2|\leq 10} \|P_{k_1 }(T_1(L_1 f(t))) P_{k_2}(T_2( L_2 g(t)))\|_{L^2_x} \big)^2.
 \ee
  Based on the relative size of $|L_1|$ and $|L_2|$, we  separate into two cases  as follows.

\noindent $\bullet$ If $|L_1|\leq |L_2|$. \qquad Note that we have $|L_1|\leq |L|/2$ for this case. If $k_1\leq k_2 +10$, then  from the $L^2-L^\infty$ type estimate and the $L^\infty$-estimate (\ref{2020julybasicestiamte}) in Lemma \ref{basicestimates}, we  have  
\be\label{2020aug24eqn11}
 \|P_{k_1}(T_1(L_1 f(t))) P_{k_2}(T_2( L_2 g(t)))\|_{L^2_x} \lesssim  2^{-k_1-k_2} \min\{ \| U^{L_1 f}_{k_1}\|_{L^\infty} , 2^{3k_1/2} \| U^{L_1 f}_{k_1}\|_{L^2}\}\| U^{L_2 g}_{k_2}\|_{L^2}.
\ee
If $k_2\leq k_1-10$ and $|L|\leq N_0-5$, then   the following estimate holds after putting $  U^{L_2 g}_{k_2}  $ in $L^\infty$,
\be\label{2020aug24eqn12}
 \|P_{k_1}(T_1(L_1 f(t))) P_{k_2}(T_2( L_2 g(t)))\|_{L^2_x} \lesssim  2^{-k_1-k_2 } \min\{ \| U^{L_2 f}_{k_2}\|_{L^\infty} , 2^{3k_2/2} \| U^{L_2 f}_{k_2}\|_{L^2}\} \| U^{L_1 g}_{k_1}\|_{L^2}.
\ee
If $k_2\leq k_1-10$ and $|L|\geq N_0-5$, the following estimate holds after putting $  U^{L_1 f}_{k_1}  $ in $L^\infty$ or putting $  U^{L_2 g}_{k_2}  $ in $L^\infty$  and using $L^\infty\longrightarrow L^2$ type Sobolev embedding, 
\be\label{2020aug24eqn13}
 \|P_{k_1}(T_1(L_1 f(t))) P_{k_2}(T_2( L_2 g(t)))\|_{L^2_x} \lesssim  2^{-k_1-k_2 } \min\{ 2^{3k_2/2}\| U^{L_1 f}_{k_1}\|_{L^2},\| U^{L_1 f}_{k_1}\|_{L^\infty}   \} \| U^{L_2 g}_{k_2}\|_{L^2}.
\ee
After combining the estimates (\ref{2020aug24eqn10}) and (\ref{2020aug24eqn11}--\ref{2020aug24eqn13}), we have 
\be\label{2020aug24eqn14}
\begin{split}
\| T_1(L_1 f(t)) T_2( L_2 g(t))\|_{L^2_x} & \lesssim 2^{  d( |L|+1, \|L\|-|L| )\delta_0 m   -\delta_0 m } \big(\sum_{k\leq -m} 2^{k/2-2\gamma k_{-}} + \sum_{k\geq -m} 2^{- m -k-2\gamma k_{-}} \big)  \epsilon_1^2\\
& \lesssim 2^{-m/2+d( |L|+1, \|L\|-|L| )\delta_0 m  } \epsilon_1^2.
 \end{split}
\ee
Following the same strategy, with minor modifications, we have
 
\be\label{2020aug24eqn15}
\begin{split}
\| \p_{\alpha}T_1(L_1 f(t)) T_2( L_2 g(t))\|_{L^2_x}  & \lesssim 2^{   d( |L|+1, \|L\|-|L| )\delta_0 m   -\delta_0 m  } \big(\sum_{k\leq -2m/3} 2^{k-2\gamma k_{-}} \\
 + \sum_{k\geq -2m/3} 2^{- m -k/2-2\gamma k_{-}} \big)  \epsilon_1^2 & \lesssim 2^{-2m/3+ d( |L|+1, \|L\|-|L| )\delta_0 m    } \epsilon_1^2.
  \end{split}
\ee
 
\noindent $\bullet$ If $|L_1|\geq |L_2|$. \qquad Since the estimate of $T_1(L_1 f(t)) T_2( L_2 g(t))$ is symmetric with respect to $L_1$ and $L_2$, we only need to  estimate  $\p_{\alpha} T_1(L_1 f(t)) T_2( L_2 g(t))$.

 Note that the following estimate holds from the $L^2-L^\infty$ type bilinear estimate, 
\be\label{oct3reqn55}
  \|\p_{\alpha} T_1(L_1 f(t)) T_2( L_2 g(t))\|_{L^2} \lesssim  
\| U^{L_1 f}_{ }(t)\|_{L^2} \| U^{L_2 g}_{}(t)\|_{L^\infty} \lesssim 2^{-m+d( |L|+1, \|L\|-|L| )\delta_0 m }. 
  \ee
  To sum up,  our desired estimates (\ref{aug9eqn84})  and (\ref{oct3reqn52}) hold from   the estimates (\ref{2020aug24eqn14}--\ref{oct3reqn55}). 
 \end{proof}

\subsubsection{Estimates of basic quantities}

Now, we apply basic linear and bilinear estimates obtained in previous subsubsections to control some basic quantities, which are mostly higher order terms. We estimate them in details here so that they can be treated like error terms in later energy estimates. 

 \begin{lemma}\label{highorderterm1}
 Under the bootstrap assumptions  \eqref{BAmetricold} and \eqref{BAvlasovori}, for any $\gamma\in\{0,1,2,3\}$, $t\in [2^{m-1}, 2^m]\subset[0, T], m\in \Z_+,$
    $ L_1 \in \cup_{  0\leq n\leq N_0, |\alpha|\leq N_0-n }  \nabla_x^\alpha P_n$ and $L_2\in  \cup_{  0\leq n\leq N_0-5, |\alpha|\leq N_0-n } \nabla_x^\alpha P_n$, we have
 \[
  2^{m/2-d( |L_1|+1, \|L_1\|-|L_1| ) \delta_0m }\| {L}_1(\Lambda_{\geq 2}[\textup{det} g ])\|_{L^2}   +  2^{2m/3-d( |L_1|+1, \|L_1\|-|L_1| ) \delta_0m }\| \p_\alpha {L}_1(\Lambda_{\geq 2}[\textup{det} g ])\|_{L^2}  
 \] 
 \be\label{june29eqn1}
+     2^{2m  -d( |L_2|+2, \|L_2\|-|L_2| ) \delta_0m }\| {L}_2(\Lambda_{\geq 2}[\textup{det} g ])\|_{L^\infty} \lesssim 2^{2\delta_0m} \epsilon_1^2.
    \ee

 \end{lemma}
 \begin{proof}
Note that
\[
\textup{det}(g) = \sum_{\alpha_i, \beta_i\in\{0,1,2,3\}}c^{\beta_1\cdots\beta_4}_{\alpha_1\cdots\alpha_4}\prod_{i=1,2,3,4}(1+h_{\alpha_i\beta_i}),
\]
where $c^{\beta_1\cdots\beta_4}_{\alpha_1\cdots\alpha_4}\in\{0,-1,1\}$ is some uniquely determined constants. Therefore, we have
 \be\label{june29eqn2}
 {L}(\Lambda_{\geq 2}[\textup{det} g ])=   \sum_{\begin{subarray}{c}
 {L}= {L}_1\circ \cdots  {L}_4,  {L}_i \in P^{b_i }_{a_i},
\{j_1,\cdots,j_4\}=\{1,\cdots,4\}, \alpha_i, \beta_i\in\{0,1,2,3\} \\
\end{subarray} }  c^{\beta_1\cdots\beta_4}_{\alpha_1\cdots\alpha_4}\prod_{i=1,2,3,4} \Lambda_{\geq 2}[  {L}_{j_i}(1+h_{\alpha_i\beta_i})].
 \ee
 From    the estimates (\ref{aug9eqn84}) and (\ref{oct3reqn52}) in Lemma \ref{aug9bilinearLemma}, we have
 \[
 2^{m/2-d( |L_1|+1, \|L_1\|-|L_1| )\delta_0m }\| {L}_1(\Lambda_{\geq 2}[\textup{det} g ])\|_{L^2}   +  2^{2m/3-d( |L_1|+1, \|L_1\|-|L_1| )\delta_0m }\| \p_\alpha {L}_1(\Lambda_{\geq 2}[\textup{det} g ])\|_{L^2}
 \]
 \[
     \lesssim 2^{2\delta_0 m } \epsilon_1^2. 
 \]
 From  the $L^\infty-L^\infty$ type mutlilinear estimate, the estimate (\ref{2020julybasicestiamte}) in Lemma \ref{basicestimates}, and the estimate (\ref{july15eqn21}), we have
 \[
  \| {L}_2(\Lambda_{\geq 2}[\textup{det} g ])\|_{L^{\infty}} \lesssim  \sum_{    \tilde{L}_1\circ \tilde{L}_2 \preceq L_2 } \|  {L}_1 h_{\alpha\beta}\|_{L^\infty}  \|  {L}_2 h_{\alpha\beta}\|_{L^\infty} \lesssim 2^{-2m+ d( |L_2|+2, \|L_2\|-|L_2| )\delta_0 m  } \epsilon_1^2.
 \]
 Hence finishing the proof of the desired estimate  (\ref{june29eqn1}). 
\end{proof}

\begin{lemma}\label{highorderterm2}
 Under the bootstrap assumptions  \eqref{BAmetricold} and \eqref{BAvlasovori},   for any $\alpha, \beta, \gamma \in\{0,1,2,3\}$,  $t\in [2^{m-1}, 2^m]\subset[0, T], m\in \Z_+,$ $ L_1 \in \cup_{  0\leq n\leq N_0, |\alpha|\leq N_0-n }  \nabla_x^\alpha P_n$ and $L_2\in  \cup_{  0\leq n\leq N_0-5, |\alpha|\leq N_0-n } \nabla_x^\alpha P_n$, we have
\[
  2^{m/2- d( |L_1|+1, \|L_1\|-|L_1| ) \delta_0m    }\big[\big(\| {L}_1(\Lambda_{\geq 2}[( {-\textup{det} g })^{-1/2}])\|_{L^2_x} +   \| {L}_1(\Lambda_{\geq 2}[ g^{\alpha\beta}])\|_{L^2_x}\big)
  \]
  \[
  +  2^{ m/6 }\big( \| \p_{\gamma} {L}_1(\Lambda_{\geq 2}[( {-\textup{det} g })^{-1/2}])\|_{L^2_x}   + \|  \p_{\gamma} {L}_1(\Lambda_{\geq 2}[ g^{\alpha\beta}])\|_{L^2_x}  \big)\big]
 \]
 \be\label{june30eqn1}
  +   2^{ 2m-  d( |L_2|+2, \|L_2\|-|L_2| ) \delta_0 m   } \big(\| {L}_2( \Lambda_{\geq 2}[ ( {-\textup{det} g })^{-1/2}])\|_{L^\infty_x}  +  \| {L}_2(\Lambda_{\geq 2}[ g^{\alpha\beta} ])\|_{L^\infty_x}\big)\lesssim  2^{2\delta_0 m }\epsilon_1^2.
 \ee
  
\end{lemma}
\begin{proof}
Due to the small data regime, we will use a fixed point type argument to show that the size of ``$ {L} (\Lambda_{\geq 2}[ ( {-\textup{det} g })^{-1/2} ])$''  is comparable with the size of ``$ {L} (\Lambda_{\geq 2}[  \textup{det} g  ])$''. More precisely, let
\[
y:=\Lambda_{\geq 2}[-\textup{det}(g)],\quad  x:= \Lambda_{1}[-\textup{det}(g)]=h_{00}-h_{ii},\]
\[
 z:=\Lambda_{\geq 2}[  ( {-\textup{det} (g) })^{-1/2}] ,  \quad w:= \Lambda_{1}[  ( {-\textup{det} (g) })^{-1/2}]=-\h  x.
\]
Note that
\begin{multline}\label{june29eqn31}
(1+z+w)^2= ({1+x+y})^{-1} = 1-x +\Lambda_{\geq 2}[({1+x+y})^{-1}], \\ 
 \Longrightarrow 2z   + 2 zw +z^2 +w^2 =  \Lambda_{\geq 2}[({1+x+y})^{-1}],
\end{multline}
  \[
 \Lambda_{\geq 2}[\frac{1}{1+ (x+y) }]= \Lambda_{\geq 2}[\frac{- (x+y)}{1+ (x+y)}]
 \]
 \[
 = -\frac{\Lambda_{\geq 2}[(x+y)]}{1+ (x+y)}-x\Lambda_{\geq 1}[\frac{1}{1+ (x+y)}] = -y\Lambda_{\geq 2}[\frac{1}{1+(x+y)}] - y (1-x)
 \]
\be\label{june29eqn44}
-x\big(-x +\Lambda_{\geq 2}[\frac{1}{1+ (x+y)}] \big)= - (x+y) \Lambda_{\geq 2}[\frac{1}{1+x+y}] - y (1-x) + x^2. 
\ee
Let
\[
X_{\alpha,n }:=  \sum_{  {L} \in  \nabla_x^\alpha P^{   }_n}     2^{m/2-d( |L |+1, \|L \|-|L | ) \delta_0m -2\delta_0 m  } \big[\| {L}  (\Lambda_{\geq 2}[({1+x+y})^{-1}])\|_{L^2}
\]
\[
+ 2^{m/6 } \|\p_{\gamma} {L}  (\Lambda_{\geq 2}[ ({1+x+y})^{-1}])\|_{L^2}\big], \]
\[
 Y_{\alpha,n  }:=   \sum_{  {L} \in   \nabla_x^\alpha P^{   }_n}    2^{m/2-d( |L |+1, \|L \|-|L | ) \delta_0m -2\delta_0 m  } \big[  \| {L}   z   \|_{L^2} +  2^{ m/6} \|\p_{\gamma} {L}   z   \|_{L^2}\big], 
\]
\[
 \tilde{X}_{\alpha,n  }:=  \sum_{  {L} \in   \nabla_x^\alpha P^{   }_n}  2^{ 2m- d( |L |+ 2, \|L \|-|L | ) \delta_0 m -2\delta_0 m} \| {L}  (\Lambda_{\geq 2}[ ({1+x+y})^{-1}])\|_{L^\infty}, 
 \]
 \[
  \tilde{Y}_{\alpha,n  }:=  \sum_{  {L} \in  \nabla_x^\alpha P^{   }_n}  2^{ 2m- d( |L |+2, \|L \|-|L | )\delta_0 m -2\delta_0 m}  \| {L}   z \|_{L^\infty}.
\]
From the equality  (\ref{june29eqn44}), $L^\infty-L^\infty$ type estimate, $L^2-L^\infty$ type estimate, and the estimate (\ref{june29eqn1}) in Lemma \ref{highorderterm1},  we know that the following estimates hold, 
\be\label{june29eqn51}
\tilde{X}_{\alpha,n }\lesssim    \sum_{0\leq l\leq n, |\beta|\leq |\alpha|  } \epsilon_1 \tilde{X}_{\beta, l } +\epsilon_1^2\Longrightarrow \sum_{0\leq |\alpha|+ N(n) \leq N_0-5}\tilde{X}_{n } \lesssim \epsilon_1^2,   
\ee
\be\label{june29eqn53}
 X_{\alpha,n}\lesssim  \sum_{0\leq l \leq n/2+2, |\beta|\leq |\alpha|} \epsilon_1   \tilde{X}_{\beta, l}   + \sum_{0\leq n'  \leq  n, |\beta|\leq |\alpha|   } \epsilon_1  {X}_{\beta, n' }  +\epsilon_1^2,\quad \Longrightarrow \quad  \sum_{0\leq |\alpha|+N(n) \leq N_0 } X_{\alpha, n }\lesssim \epsilon_1^2.
\ee
Similarly, from the equality (\ref{june29eqn31}), the following estimates hold from the estimate (\ref{june29eqn51}),  
\[
\tilde{Y}_{\alpha,n }\lesssim    \sum_{0\leq l\leq n, |\beta|\leq |\alpha|  } \epsilon_1 \tilde{Y}_{\beta, l }+ \tilde{X}_{\alpha, n }+ \epsilon_1^2 + \big( \sum_{ 0\leq l\leq n, |\beta|\leq |\alpha| }   \tilde{Y}_{\beta, l} \big)^2,  
\]
\be\label{june29eqn61}
 \Longrightarrow     \sum_{n\in[0, N_0-5]\cap \Z, |\alpha|\in [0, N_0-n]\cap \Z} \tilde{Y}_{\alpha, n}\lesssim \epsilon_1^2,
  \ee
\begin{multline}\label{june29eqn62}
  Y_{\alpha, n}\lesssim  \sum_{0\leq l' \leq n/2+2, |\beta|\leq |\alpha|} \epsilon_1   \tilde{Y}_{\beta, l' }   + \sum_{0\leq l  \leq  n,|\beta|\leq |\alpha|   } \epsilon_1  {Y}_{\beta, l } +\epsilon_1^2 + \big( \sum_{0\leq l'  \leq n/2+2, |\beta|\leq |\alpha|} \tilde{Y}_{\beta, l'}\big) \big( \sum_{0\leq l  \leq n, |\beta|\leq |\alpha|  } Y_{\beta, l }\big)\\ 
   \Longrightarrow \quad \sum_{n\in[0, N_0 ]\cap \Z, |\alpha|\in [0, N_0-n]\cap \Z} Y_{\alpha, n  }  \lesssim \epsilon_1^2. 
\end{multline}
 
Note that, for any fixed $\alpha, \beta\in\{0,1,2,3\}$, the following equality holds for some uniquely determined constants $ c_{\alpha_1\beta_1\cdots \alpha_3\beta_3}^{\alpha\beta}\in\{0,1,-1\}$,
\be\label{june30eqn11}
\Lambda_{\geq 2}[g^{\alpha\beta}]=\sum_{\alpha_i, \beta_i\in\{0,1,2,3\} } \prod_{i=1,2,3} c_{\alpha_1\beta_1\cdots \alpha_3\beta_3}^{\alpha\beta} \Lambda_{\geq 2}[(\textup{det} g)^{-1} (1+h_{\alpha_i\beta_i})].
\ee
Note that we have already derived the basic $L^2$-estimate and the basic $L^\infty$-estimate for $\Lambda_{\geq 2}[(\det g)^{-1}]$, see estimates (\ref{june29eqn51}) and (\ref{june29eqn53}). After applying the $L^2-L^\infty$ type bilinear estimate and the $L^\infty-L^\infty$ type bilinear estimate for the equality (\ref{june30eqn11}), it is easy to see that our desired estimate (\ref{june30eqn1}) holds. 
\end{proof}
Due to the on shell relation, see \eqref{massshell}, $v_0$ as well as $v^0$ depend on  the perturbed metric. For convenience, we identify the zero order and the   first order of  $v_0$ ($v^\alpha$) as follows, 
\begin{multline}\label{june16eqn51}
g^{\mu\nu} v_{\mu} v_{\nu}=-1, \quad \Longrightarrow \Lambda_0[v^{0}] =   \sqrt{1+|v|^2} , \quad \Lambda_{0}[v_0] =- \sqrt{1+|v|^2},\\ 
 \Lambda_{0}[v^i] = v_i,\quad \Lambda_{1}[v^i] =   -h_{0i} \sqrt{1+|v|^2}  - h_{ij}v_j,\quad   \Lambda_1[v^0]= \h\sqrt{1+|v|^2} \big( h_{00} -  \hat{v}_i \hat{v}_j   h_{ij} \big),\\ 
\Lambda_1[v_0]= \h\big( \sqrt{1+|v|^2}h_{00} + 2  v_ih_{0i} + \frac{v_i v_j}{\sqrt{1+|v|^2}}h_{ij} \big).
\end{multline}

For the quadratic and higher order terms of $v_0$, their explicit formulas are not necessary. It would be sufficient to give their rough estimates, which are summarized in the following two lemmas. 
\begin{lemma}\label{highorderterm3}
 Under the bootstrap  assumptions  \eqref{BAmetricold} and \eqref{BAvlasovori},  for any   $t\in [2^{m-1}, 2^m]\subset[0, T], m\in \Z_+,$
  for  $ L_1 \in \cup_{  0\leq n\leq N_0, |\alpha|\leq N_0-n }  \nabla_x^\alpha P_n$ and $L_2\in  \cup_{  0\leq n\leq N_0-5, |\alpha|\leq N_0-n } \nabla_x^\alpha P_n$,  we have 
  \[
  2^{m/2- d( |L_1|+1, \|L_1\|-|L_1| ) \delta_0m    } \big[\big( \|\langle v \rangle^{-1 } {L}_1(\Lambda_{\geq 2}[v^0 ])\|_{L^2_x L^\infty_v} + \|\langle v \rangle^{-1 } {L}_1(\Lambda_{\geq 2}[v_0 ])\|_{L^2_xL^\infty_v}  
  \]
  \[
   +  \|\langle v \rangle^{-1 } {L}_1(\Lambda_{\geq 2}[v^i ])\|_{L^2_x L^\infty_v}\big) +2^{m/6  }\big( \|\langle v \rangle^{-1 } \p_{\gamma }{L}_1(\Lambda_{\geq 2}[v^0 ])\|_{L^2_x L^\infty_v}   
 \]
 \[
+ \|\langle v \rangle^{-1 } \p_{\gamma }{L}_1(\Lambda_{\geq 2}[v_0 ])\|_{L^2_xL^\infty_v}+  \|\langle v \rangle^{-1 } \p_{\gamma } {L}_1(\Lambda_{\geq 2}[v^i ])\|_{L^2_x L^\infty_v}\big)\big]
 \]
 \[
 +    2^{ 2m-  d( |L_2|+2, \|L_2\|-|L_2| ) \delta_0 m   } \big(  \|\langle v \rangle^{-1 } {L}_2(\Lambda_{\geq 2}[v^i ])\|_{L^\infty_{x,v}} 
 \]
\be\label{june29eqn71}
 + \|\langle v \rangle^{-1 }  {L}_2(\Lambda_{\geq 2}[v^0 ])\|_{L^\infty_{x,v}}+ \|\langle v \rangle^{-1 }  {L}_2(\Lambda_{\geq 2}[v_0 ])\|_{L^\infty_{x,v}}\big) \lesssim 2^{2\delta_0 m }\epsilon_1^2.
 \ee
 \end{lemma}
 \begin{proof}

Note that
\[
 g^{00} \big(  -\sqrt{1+|v|^2}+\Lambda_1[v_0] +\Lambda_{\geq 2}[v_0] \big)^2=  g^{00} \big[\big(  -\sqrt{1+|v|^2}+\Lambda_1[v_0]  \big)^2 \]
 \be\label{june29eqn97}
 +2  \Lambda_{\geq 2}[v_0]\big(  -\sqrt{1+|v|^2}+\Lambda_1[v_0]  \big) + \big(\Lambda_{\geq 2}[v_0] \big)^2 \big]= 2  \Lambda_{\geq 2}[v_0]    \sqrt{1+|v|^2}+ R,
\ee
where
\be\label{june29eqn91} 
R= g^{00}\big(  -\sqrt{1+|v|^2}+\Lambda_1[v_0]  \big)^2
-2\Lambda_{\geq 2}[v_0] \big(  \Lambda_{\geq 1}[g^{00}]   \sqrt{1+|v|^2} + g^{00}   \big(   \Lambda_1[v_0]  \big) \big)+ g^{00} \big(\Lambda_{\geq 2}[v_0] \big)^2.
\ee
From the  equality   (\ref{june29eqn97}), we have
\[
2\sqrt{1+|v|^2}  \Lambda_{\geq 2}[v_0] = -R -1 -g^{ij} v_i v_j +2 g^{0i} \sqrt{1+|v|^2}- 2g^{0i} v_i \Lambda_{1}[v_0]- 2g^{0i} v_i \Lambda_{\geq 2}[v_0]  
\]
\be\label{july15e1qn98}
=- 2 g^{0i} v_i \Lambda_{\geq 2}[v_0] -\Lambda_{\geq 2}[R] - \Lambda_{\geq 2}[g^{ij}]v_iv_j+2 \Lambda_{\geq 2}[  g^{0  i}  ]\sqrt{1+|v|^2}- 2\Lambda_{\geq 1}[g^{0i}] v_i \Lambda_{1}[v_0].
\ee
Moreover, 
\be\label{june30eqn86}
v_0= g_{\alpha 0} v^{\alpha},\quad v^i= g^{i\alpha} v_{\alpha}, \quad \Longrightarrow \Lambda_{\geq 2}[v_0] = - \Lambda_{\geq 2}[v^0] + h_{0i} \Lambda_{\geq 2}[ v^i] + h_{0i} \Lambda_{1}[v^i],
\ee
\be\label{june30eqn87}
\Lambda_{\geq 2}[v^i] = \Lambda_{\geq 2}[g^{i 0} v_0] + \Lambda_{\geq 2}[g^{i j}] v_j= g^{i0}\Lambda_{\geq 2}[v_0] + \Lambda_{\geq 2}[g^{i 0} \Lambda_{\leq 1}[v_0]] +\Lambda_{\geq 2}[g^{ij}] v_j. 
\ee
For any fixed $v$, we define the following two quantities 
\[
Z_{\alpha, n }:= \sum_{  {L} \in \nabla_x^\alpha P^{ }_n }    2^{m/2-d( |L |+1, \|L \|-|L | ) \delta_0m -2\delta_0 m  }\big[ \|\langle v \rangle^{-1 } {L}  (\Lambda_{\geq 2}[ v^{0}])\|_{L^2_x L^\infty_v}
\]
\be\label{june30eqn34}
 +  \langle t \rangle^{ 1/6  } \|\langle v \rangle^{-1 } \p_{\gamma }{L}  (\Lambda_{\geq 2}[ v^{0}])\|_{L^2_x L^\infty_v}\big], 
\ee
\be\label{june30eqn35}
 \tilde{Z}_{\alpha, n }:=  \sum_{  {L} \in \nabla_x^\alpha P^{ }_n } 2^{ 2m-  d( |L |+2, \|L \|-|L_2| ) \delta_0 m -2\delta_0 m }\|\langle v \rangle^{-1 } {L}  (\Lambda_{\geq 2}[ v^{0}])\|_{L^\infty_{x,v}}.
\ee
\[
W_{\alpha, n}:=   \sum_{  {L} \in \nabla_x^\alpha P^{ }_n }         2^{m/2-d( |L |+1, \|L \|-|L | ) \delta_0m -2\delta_0 m  } \big[\big(\|\langle v \rangle^{-1 } {L}  (\Lambda_{\geq 2}[ v_{0}])\|_{L^2_xL^\infty_v} +\|\langle v \rangle^{-1 } {L}  (\Lambda_{\geq 2}[ v^{i}])\|_{L^2_x L^\infty_v}\big)
\]
\be\label{june30eqn64}
+   2^{ m/6  }\big(\|\langle v \rangle^{-1 } \p_{\gamma} {L}  (\Lambda_{\geq 2}[ v_{0}])\|_{L^2_xL^\infty_v} +\|\langle v \rangle^{-1 } \p_{\gamma} {L}  (\Lambda_{\geq 2}[ v^{i}])\|_{L^2_x L^\infty_v}\big)\big], 
\ee
\be\label{june30eqn65}
\tilde{W}_{\alpha, n }:=  \sum_{  {L} \in \nabla_x^\alpha P^{ }_n }  2^{ 2m-  d( |L |+2, \|L \|-|L | ) \delta_0 m -2\delta_0 m } \big(\|\langle v \rangle^{-1 } {L}  (\Lambda_{\geq 2}[ v_{0}])\|_{L^\infty_{x,v}} +\sum_{i=1,2,3}\|\langle v \rangle^{-1 } {L}  (\Lambda_{\geq 2}[ v^{i}])\|_{L^\infty_{x,v}} \big).
\ee
By using a similar argument as in the proof of Lemma \ref{highorderterm2} and the estimate (\ref{june30eqn1}) in Lemma \ref{highorderterm2},  the following estimate holds for any fixed $v\in \R^3$, 
\be
\tilde{Z}_{\alpha, n }  + \tilde{W}_{\alpha, n}  \lesssim      \sum_{0\leq l\leq n,0\leq |\beta|\leq |\alpha| } \epsilon_1 (\tilde{Z}_{\beta,l }  + \tilde{W}_{\beta, l } ) +\epsilon_1^2,   \Longrightarrow    \sum_{n\in[0, N_0-5]\cap \Z, |\alpha|\in [0, N_0-n]\cap \Z} \tilde{Z}_{\alpha,n }  + \tilde{W}_{\alpha,n }  \lesssim \epsilon_1^2, 
\ee
\[
 {Z}_{\alpha, n }  +  {W}_{\alpha, n } \lesssim    \sum_{0\leq l' \leq n/2+2, |\beta|\leq |\alpha|}  \epsilon_1\big( \tilde{Z}_{\beta, l' }  + \tilde{W}_{\beta, l' } \big)   +   \sum_{0\leq l  \leq  n,|\beta|\leq |\alpha|   } \epsilon_1\big( {Z}_{\beta, l }  +  {W}_{\beta, l } \big) +\epsilon_1^2 
\]
\be
  \Longrightarrow  \sum_{n\in[0, N_0  ]\cap \Z, |\alpha|\in [0, N_0-n]\cap \Z}  {Z}_{\alpha, n   }  +  {W}_{\alpha, n  }   \lesssim \epsilon_1^2. 
\ee
Hence finishing the proof of our desired estimate 
 (\ref{june29eqn71}).  
 \end{proof}

\begin{lemma}\label{highorderterm5}
 Under the bootstrap  assumptions  \eqref{BAmetricold} and \eqref{BAvlasovori},
  the following estimate holds for any  $\gamma\in\{0,1,2,3\}$,  $t\in [2^{m-1}, 2^m]\subset[0, T], m\in \Z_+,$ and any triple  $ \beta\in \Z_+^3 $,  for any $i\in\{1,2,3\}, \gamma\in\{0,1,2,3\}$, $t\in [2^{m-1}, 2^m]\subset[0, T], m\in \Z_+,$
  \[
      2^{m/2- d( |L_1|+1, \|L_1\|-|L_1| ) \delta_0m    }   \big[ \big(\|\langle v \rangle^{|\beta|+1 } {L}_1\nabla_v^{\beta}(\Lambda_{\geq 2}[ (v^0)^{-1} ])\|_{L^2_x L^\infty_v}  + \|\langle v \rangle^{|\beta|-1 }{L}_1\nabla_v^{\beta}(\Lambda_{\geq 2}[ {v^0 } ])\|_{L^2_x L^\infty_v}\]
 \[
   + \|\langle v \rangle^{|\beta|-1 }\mathcal{L}\nabla_v^{\beta}(\Lambda_{\geq 2}[ {v_0 } ])\|_{L^2_x L^\infty_v }\big)+  2^{ m/6}  \big(\|\langle v \rangle^{|\beta|+1 } \p_{\gamma} {L}_1\nabla_v^{\beta}(\Lambda_{\geq 2}[ (v^0)^{-1} ])\|_{L^2_x L^\infty_v}   
   \]
   \[
     + \|\langle v \rangle^{|\beta|-1 }\p_{\gamma} {L}_1\nabla_v^{\beta}(\Lambda_{\geq 2}[ {v^0 } ])\|_{L^2_x L^\infty_v}\]
     \[
     + \|\langle v \rangle^{|\beta|-1 }\p_{\gamma}\mathcal{L}\nabla_v^{\beta}(\Lambda_{\geq 2}[ {v_0 } ])\|_{L^2_x L^\infty_v }\big)   \big]   +        2^{ 2m-  d( |L_2|+2, \|L_2\|-|L_2| ) \delta_0 m   }  \big(\|\langle v \rangle^{|\beta|+1 } {L}_2\nabla_v^{\beta}(\Lambda_{\geq 2}[ ({v^0 })^{-1} ])\|_{L^\infty_{x,v}} 
      \]
  \be\label{july2eqn2}
 + \|\langle v \rangle^{|\beta|-1 } {L}_2\nabla_v^{\beta}(\Lambda_{\geq 2}[ {v^0 } ])\|_{L^\infty_{x,v}} + \|\langle v \rangle^{|\beta|-1 } {L}_2\nabla_v^{\beta}(\Lambda_{\geq 2}[ {v_0 } ])\|_{L^\infty_{x,v}} \big)\lesssim 2^{2\delta_0 m }\epsilon_1^2.
  \ee
\end{lemma}
\begin{proof}
Recall (\ref{june16eqn51}). Note that
 \[
 ({v^0})^{-1}= ({\sqrt{1+|v|^2} + \Lambda_{1}[v^0] + \Lambda_{\geq 2}[v^0]})^{-1}.\]
Let
\[
\quad x:=\Lambda_{1}[v^0], \quad y:=\Lambda_{\geq 2}[v^0], \quad z:=\Lambda_{\geq 2}[  ({v^0})^{-1} ], \quad w:= \Lambda_{1}[  ({v^0})^{-1} ]=-\langle v \rangle^{-2}\Lambda_{1}[v^0]=-\langle v \rangle^{-2}x.
 \]
 Hence, the following equality holds,
 \be\label{july2eqn1}
\big(\langle v\rangle^{-1} +z+w \big)\big(\langle v\rangle+ x +y \big)=1, \quad \Longrightarrow \langle v\rangle z =- (x+y) z -\langle v \rangle^{-1} y -x w - y w .
 \ee
From the fixed point type formulations for $\Lambda_{\geq 2}[v_0]$, $\Lambda_{\geq 2}[v^0]$, and $\Lambda_{\geq 2}[\big(v^0\big)^{-1}]$ in (\ref{july15e1qn98}), (\ref{june30eqn86}), (\ref{june30eqn87}), and (\ref{july2eqn1}) and the $L^2$-type estimate  and the $L^\infty$-type estimate    (\ref{june29eqn71}) in Lemma  \ref{highorderterm3},  our desired estimate  (\ref{july2eqn2}) holds after using the same argument used in the proof of the estimates (\ref{june29eqn71})    Lemma \ref{highorderterm3} and doing induction on the size of  ``$|\beta|$''.   
\end{proof}

\subsection{Estimates  of the nonlinearities  in Vlasov part}

In the following lemma, we estimate the energy of  the cubic and higher order terms of the Vlasov part. 
\begin{lemma}\label{estimateofremaindervlasov}
Under the bootstrap assumptions \textup{(\ref{BAmetricold})} and \textup{(\ref{BAvlasovori})}, the following estimate holds for any   $t\in [2^{m-1}, 2^m]$, $\mathcal{L}\in \cup_{n\leq N_0  } \mathcal{P}_n, \rho\in \mathcal{S},  |\rho|+|\mathcal{L}| \leq N_0$,
\[
  \sum_{i=1,3,4 }  \| \omega^{\mathcal{L}}_{\rho}(x,v)  \Lambda_{\geq 3}[  \mathfrak{N}_{\rho,i }^{\mathcal{L
 }} ](t,x,v)   \|_{L^2_{x,v}} +  \| \langle v \rangle^{-1} \omega^{\mathcal{L}}_{\rho}(x,v)  \Lambda_{\geq 3}[  \mathfrak{N}_{\rho,2 }^{\mathcal{L
 }} ](t,x,v)   \|_{L^2_{x,v}}
\]
\be\label{aug9eqn87}
\lesssim   2^{-7m/6+3d( |\mathcal{L}|+|\tilde{c}(\rho)| +3, |\rho|- |\tilde{c}(\rho)| )\delta_0m} \epsilon_1^3.
\ee
Moreover, for any   $\mathcal{L}\in \cup_{n\leq N_0  } \mathcal{P}_n, \rho,\iota\in \mathcal{S}, $ s.t.,  $|\iota|=1$, $  |\rho|+|\mathcal{L}| \leq N_0-1$,  we have
\[
\|\omega^{\mathcal{L}}_{\iota\circ\rho}(x,v) \Lambda^{\iota} \Lambda_{\geq 3}[  \mathfrak{N}_{\rho,1 }^{\mathcal{L
 }} ](t,x,v)   \|_{L^2_{x,v}} + \|\langle  v \rangle^{-1} \omega^{\mathcal{L}}_{\iota\circ\rho}(x,v)\Lambda^{\iota} \Lambda_{\geq 3}[  \mathfrak{N}_{\rho,2 }^{\mathcal{L
 }} ](t,x,v)   \|_{L^2_{x,v}}
\]
\be\label{2020april16eqn1}
\lesssim 2^{-2m+d( |\mathcal{L}|+|\tilde{c}(\rho)| +4, |\rho|- |\tilde{c}(\rho)| )\delta_0 m + \delta_0 m}\epsilon_1^3. 
\ee
\end{lemma}
\begin{proof}
Recall (\ref{2020april8eqn32}--\ref{2020april8eqn35}).  Note that,  from the $L^\infty_{x,v}-L^2_{x,v}$ type bilinear estimate,  the estimate (\ref{june30eqn1}) in Lemma \ref{highorderterm2} and the estimate (\ref{june29eqn71}) in Lemma \ref{highorderterm3}, we have
\[
 \| \omega^{\mathcal{L}}_{\rho}(x,v)  \Lambda_{\geq 3}[  \mathfrak{N}_{\rho,1 }^{\mathcal{L
 }} ](t,x,v)   \|_{L^2_{x,v}} +  \| \langle v \rangle^{-1} \omega^{\mathcal{L}}_{\rho}(x,v)  \Lambda_{\geq 3}[  \mathfrak{N}_{\rho,2 }^{\mathcal{L
 }} ](t,x,v)   \|_{L^2_{x,v}}
\]
\[
 \lesssim \big(\| v\Lambda_{\geq 2}[(v^0)^{-1}](t,x,v)\|_{L^\infty_{x,v}} +  \| \langle v \rangle^{-1}  (v^0)^{-1}  \Lambda_{\geq 2}[ v_{\alpha} v_{\beta} \nabla_x g^{\alpha\beta}](t,x,v)\|_{L^\infty_{x,v}} \big)
 \]
 \be
 \times  \| \omega^{\mathcal{L}}_{\rho}(x,v)   \Lambda^{\rho} u^{\mathcal{L}}(t,x,v)\|_{L^2_{x,v}} \lesssim 2^{-2m+ d( |\mathcal{L}|+|\tilde{c}(\rho)| +4, |\rho|- |\tilde{c}(\rho)| )\delta_0 m + \delta_0 m}\epsilon_1^3. 
\ee
By using the same strategy,  our desired estimate (\ref{2020april16eqn1}) holds from the equality (\ref{sepeqn610}) and the estimate of corresponding coefficients (\ref{sepeqn88})
 in Lemma \ref{decompositionofderivatives}.

It remains to estimate $  \mathfrak{N}_{\rho,3}^{\mathcal{L
 }} $ and $  \mathfrak{N}_{\rho,4 }^{\mathcal{L
 }} $.  We use different strategies based on the possible size of the number of vector fields acting on the Vlasov part. 

\noindent $\bullet$\qquad  If $|\rho_1|+ |\mathcal{L}_1|>  |\rho|+|\mathcal{L}|-5$\qquad Without further explanation, we let the vector fields be restricted in the range listed in (\ref{2020april8eqn35}).  Since there are at most six derivatives act on the metric component, we use the decay of the perturbed metrics. As a result, from the $L^\infty_{x,v}-L^2_{x,v}$ type bilinear estimate,  the estimates of coefficients in (\ref{2020july30eqn1}), the estimate (\ref{june30eqn1}) in Lemma \ref{highorderterm2}, the estimate (\ref{june29eqn71}) in Lemma \ref{highorderterm3}, and the estimate (\ref{july2eqn2}) in Lemma \ref{highorderterm5},  the following estimate holds for the 
\be\label{oct3reqn64}
\begin{split}
\|   c_{\mathcal{L}_1\mathcal{L}_2\mathcal{L}_3;\tilde{L}_1, \tilde{L}_2}^{  {\mathcal{L}};\Gamma_1,\Gamma_2;\rho, \rho_1, \rho_2}(x,v) \Lambda_{\geq 2}[   \tilde{L}_1 \Gamma_2 \mathcal{L}_2 g^{\alpha\beta}&(t,x+t\hat{v})  \tilde{L}_2 \mathcal{L}_3\big(  ({2v^0})^{-1} { v_{\alpha} v_{\beta}} \big](t,x+t\hat{v}, v)  \Lambda^{\rho_1} u^{\Gamma_1 \mathcal{L}_1}(t,x,v) \\
\times  \omega^{\mathcal{L}}_{\rho}(x,v) \|_{L^2_{x,v}}  + \|\omega^{\mathcal{L}}_{\rho}(x,v)    \tilde{c}_{\mathcal{L}_1\mathcal{L}_2;\tilde{L}}^{\mathcal{L};\rho,\rho_1,\rho_2}(x,v)  \cdot &\nabla_x \Lambda^{\rho_1}  u^{\mathcal{L}_1}(t,x,v) \Lambda_{\geq 2}[\tilde{L}\mathcal{L}_2(v^0)^{-1}](t, x+t\hat{v},v) \|_{L^2_{x,v}}\\
\lesssim \sum_{|\tilde{\rho}|+|\tilde{\mathcal{L}}|\leq |\rho|+|\mathcal{L}|} \big( \| 
 \Lambda_{\geq 2}[\tilde{L}_1 &\Gamma_2 \mathcal{L}_2 g^{\alpha\beta}] \|_{  L^\infty_{x,v}} +  \|  \langle v\rangle^{-1}\Lambda_{\geq 1}[\tilde{L}_2\mathcal{L}_3\big(  ({2v^0})^{-1} { v_{\alpha} v_{\beta}} \big) ](t ) \|_{ L^\infty_{x,v}} \\
\times  \|\tilde{L}_1\Gamma_2 \mathcal{L}_2 g^{\alpha\beta}(t, x  )   \|_{ L^\infty_{x,v}} &+\| \Lambda_{\geq 2}[\tilde{L}\mathcal{L}_2(v^0)^{-1}](t, x+t\hat{v})\|_{L^\infty_{x,v}}\\
 +\|   \langle v\rangle^{-1}\Lambda_{\geq 2}[\tilde{L}_2\mathcal{L}_3\big(&  ({2v^0})^{-1} { v_{\alpha} v_{\beta}} \big) ](t ) \|_{ L^\infty_{x,v}}  \big)  \|   \omega^{ \tilde{\mathcal{L}}}_{\tilde{\rho}}  (x, v) \Lambda^{\tilde{\rho}}u^{ \tilde{\mathcal{L}}} (t,x , v) \|_{  L^2_{x,v}}\\
  \qquad \qquad  \lesssim &2^{-2m  + 3 d( |\mathcal{L}|+|\tilde{c}(\rho)| +3, |\rho|- |\tilde{c}(\rho)| ) \delta_0 m}\epsilon_1^3. \qquad \qquad 
\end{split}
 \ee

\noindent $\bullet$\qquad If $|\rho_1|+ |\mathcal{L}_1|\leq |\rho|+|\mathcal{L}|-5$. \qquad 

For this case, we use the decay of the density of the Vlasov part. From the $L^2_x L^\infty_v-L^\infty_x L^2_v$ type estimate, the decay estimate of density function type (\ref{densitydecay}) in Lemma \ref{decayestimateofdensity}, and  the estimate (\ref{july2eqn2}) in Lemma \ref{highorderterm5},     we have
 \[
\| \omega^{\mathcal{L}}_{\rho}(x,v)    c_{\mathcal{L}_1\mathcal{L}_2\mathcal{L}_3;\tilde{L}_1, \tilde{L}_2}^{  {\mathcal{L}};\Gamma_1,\Gamma_2;\rho, \rho_1, \rho_2}(x,v) \Lambda_{\geq 2}[   \tilde{L}_1 \Gamma_2 \mathcal{L}_2 g^{\alpha\beta}(t,x+t\hat{v})  \tilde{L}_2 \mathcal{L}_3\big(  ({2v^0})^{-1} { v_{\alpha} v_{\beta}} \big)](t,x+t\hat{v}, v) \]
\[
\times  \Lambda^{\rho_1} u^{\Gamma_1 \mathcal{L}_1}(t,x,v)\|_{L^2_{x,v}}
\]
\[
+ \|\omega^{\mathcal{L}}_{\rho}(x,v)    \tilde{c}_{\mathcal{L}_1\mathcal{L}_2;\tilde{L}}^{\mathcal{L};\rho,\rho_1,\rho_2}(x,v)  \cdot \nabla_x \Lambda^{\rho_1}  u^{\mathcal{L}_1}(t,x,v) \Lambda_{\geq 2}[\tilde{L}\mathcal{L}_2(v^0)^{-1}](t, x+t\hat{v},v) \|_{L^2_{x,v}}
\]
\[ 
 \lesssim \sum_{| \tilde{\rho}|+|\tilde{\mathcal{L}} |\leq N_0-5  } \big(  \|\Lambda_{\geq 2}[   \tilde{L}_1 \Gamma_2 \mathcal{L}_2 g^{\alpha\beta}(t,x+t\hat{v})  \tilde{L}_2 \mathcal{L}_3\big(  ({2v^0})^{-1} { v_{\alpha} v_{\beta}} \big](t,x , v)] \|_{L^2_x L^\infty_v} 
\] 
\[
+\|  \Lambda_{\geq 2}[\tilde{L}\mathcal{L}_2(v^0)^{-1}](t, x )  \|_{L^2_x L^\infty_v}   \big) \| \langle v \rangle^{-20} \omega^{ \tilde{\mathcal{L}}}_{\tilde{\rho}}   (x-t\hat{v},v)  \Lambda^{ \tilde{\rho}}u^{ \tilde{\mathcal{L}}}_{} (t,x-t\hat{v}, v) \|_{L^\infty_x L^2_v}
\]
\be\label{oct3reqn3}
  \lesssim 2^{-2m+ 3 d( |\mathcal{L}|+|\tilde{c}(\rho)| +3, |\rho|- |\tilde{c}(\rho)| ) \delta_0 m}\epsilon_1^3.
\ee

By using a similar strategy, the following estimate holds from the decay estimate of density function type (\ref{densitydecay}) in Lemma \ref{decayestimateofdensity}, the estimate (\ref{june30eqn1}) in Lemma \ref{highorderterm2}, the estimate (\ref{june29eqn71}) in Lemma \ref{highorderterm3}, the estimate (\ref{july2eqn2}) in Lemma \ref{highorderterm5},   and the $L^2$-type bilinear estimate (\ref{aug9eqn84}) in Lemma \ref{aug9bilinearLemma},
 \[
 \|   \omega^{\mathcal{L}}_{\rho}(x,v)   c^{\tilde{L}}_\rho(x,v) \Lambda_{\geq 2}[\big( (v^0)^{-1}   v_{\alpha} v_{\beta} \nabla_x \tilde{L} \mathcal{L} g^{\alpha\beta}\big)] (t, x+t\hat{v},v) \cdot D_v u^{ }(t,x,v)   \|_{L^2_{x,v}}
 \]
\[
 \lesssim 2^{m}\big( \|     \nabla_x \tilde{L} \mathcal{L} g^{\alpha\beta} \|_{L^2_x  } \|  \langle v\rangle^{-2}\Lambda_{\geq 1}[ \big(  ({2v^0})^{-1}   { v_{\alpha} v_{\beta}} \big) ](t, x  , v)\|_{ L^\infty_{x,v}}    +\|  \Lambda_{\geq 2}[   \nabla_x \tilde{L} \mathcal{L} g^{\alpha\beta}](t, x  , v)\|_{L^2_{x} }    \big)   
\]
\[
\times \|  \langle v\rangle^{ 2}   \omega^{\mathcal{L}}_{\rho}(x,v)   \nabla_{x,v}^{\alpha}  \Lambda^{\rho_1}u^{\mathcal{L}}(t,x-t\hat{v}, v) \|_{L^\infty_x L^2_v}
\]
\be\label{oct3reqn63}
 \lesssim 2^{-7m/6+ 3d( |\mathcal{L}|+|\tilde{c}(\rho)| +3, |\rho|- |\tilde{c}(\rho)| )\delta_0 m}\epsilon_1^3.
\ee
   To sum up,   our desired estimate (\ref{aug9eqn87}) holds from the estimates (\ref{oct3reqn3}--\ref{oct3reqn63}).
 
\end{proof}
Now, we estimate  the quadratic terms of the Vlasov part. 
\begin{lemma}\label{estimateofremaindervlasov3}
Under the bootstrap assumptions \textup{(\ref{BAmetricold})} and \textup{(\ref{BAvlasovori})}, the following estimate holds for any  $t\in [2^{m-1}, 2^m]$,   $ m\in \mathbb{Z}_{+}$,  $\mathcal{L}, \tilde{\mathcal{L}}\in   \cup_{n\leq N_0  } \mathcal{P}_n, \rho, \iota, \kappa\in \mathcal{S}, $ s.t.,     $|\iota|=1   $, $|\rho|+|\mathcal{L}|\leq N_0$,  $|\kappa|+|  \tilde{\mathcal{L}}|\leq |\rho|+|\mathcal{L}|-1,$
\be\label{2020april9eqn21}
\begin{split}
    &  \sum_{i=1,   4  }  \| \omega^{\mathcal{L}}_{\rho}(x,v)  \Lambda_{2}[  \mathfrak{N}_{\rho,i }^{\mathcal{L
 }} ](t,x,v)   \|_{L^2_{x,v}}  + \| \omega^{\tilde{\mathcal{L}} }_{\iota\circ\kappa}(x,v)  \Lambda_{2}[ \Lambda^{\iota} \mathfrak{N}_{  \kappa,i }^{ \tilde{\mathcal{L}}} ](t,x,v)   \|_{L^2_{x,v}} \\
&+\| \langle v \rangle^{-1} \omega^{\mathcal{L}}_{\rho}(x,v)  \Lambda_{2}[  \mathfrak{N}_{\rho,2 }^{\mathcal{L
 }} ](t,x,v)   \|_{L^2_{x,v}}+  \| \langle v \rangle^{-1} \omega^{\tilde{\mathcal{L}} }_{\iota\circ\kappa}(x,v)  \Lambda_{2}[ \Lambda^{\iota} \mathfrak{N}_{  \kappa,2 }^{ \tilde{\mathcal{L}}} ](t,x,v)   \|_{L^2_{x,v}}   \\ 
   & \lesssim   2^{-m +   H(\mathcal{L}, \rho) \delta_0m + H(4 )\delta_0  m } \epsilon_1^2.\\ 
   &   \| \omega^{\mathcal{L}}_{\rho}(x,v)  \Lambda_{2}[  \mathfrak{N}_{\rho,3 }^{\mathcal{L
 }} ](t,x,v)   \|_{L^2_{x,v}}  + \| \omega^{\tilde{\mathcal{L}} }_{\iota\circ\kappa}(x,v)  \Lambda_{2}[ \Lambda^{\iota} \mathfrak{N}_{  \kappa,3  }^{ \tilde{\mathcal{L}}} ](t,x,v)   \|_{L^2_{x,v}} \\ 
 & \lesssim   2^{-m/2 +   H(\mathcal{L}, \rho) \delta_0m + H(4 )\delta_0  m } \epsilon_1^2 .
\end{split}
\ee
 Moreover, for the second part of non-linearity, $ \mathfrak{N}_{\rho,2 }^{\mathcal{L
 }}(t,x,v) $,  the following improved estimate holds if we allow some loss of decay in $(x,v),$
\[
  \| \langle v \rangle^{-3}\langle |x|\rangle^{-1} \omega^{\mathcal{L}}_{\rho}(x,v)  \Lambda_{2}[  \mathfrak{N}_{\rho,2 }^{\mathcal{L
 }} ](t,x,v)   \|_{L^2_{x,v}} +  \| \langle v \rangle^{-3}\langle |x|\rangle^{-1}\omega^{\tilde{\mathcal{L}} }_{\iota\circ\kappa}(x,v)  \Lambda_{2}[ \Lambda^{\iota} \mathfrak{N}_{  \kappa,2 }^{ \tilde{\mathcal{L}}} ](t,x,v)   \|_{L^2_{x,v}} 
\]
\be\label{2020april14eqn3}
  \lesssim2^{-2m+    H(\mathcal{L}, \rho) \delta_0 m + H(4) \delta_1 m}\epsilon_1^2.
\ee

\end{lemma}
\begin{proof}

Recall (\ref{2020april8eqn32}-\ref{2020april8eqn33}). From the $L^\infty_{x,v}-L^2_{x,v}$ type bilinear estimate, the estimate (\ref{2020julybasicestiamte}) in Lemma \ref{basicestimates},  the estimate of corresponding coefficients (\ref{sepeqn88})
 in Lemma \ref{decompositionofderivatives},  we have
\be\label{2020april14eqn1}
\begin{split}
 &   \| \omega^{\tilde{\mathcal{L}} }_{\iota\circ\kappa}(x,v)  \Lambda_{2}[ \Lambda^{\iota} \mathfrak{N}_{  \kappa,1 }^{ \tilde{\mathcal{L}}} ](t,x,v)   \|_{L^2_{x,v}}    +  \| \langle v \rangle^{-1}  \omega^{\tilde{\mathcal{L}} }_{\iota\circ\kappa}(x,v)  \Lambda_{2}[ \Lambda^{\iota} \mathfrak{N}_{  \kappa,2 }^{ \tilde{\mathcal{L}}} ](t,x,v)   \|_{L^2_{x,v}} \\
 &   +  \| \omega^{\mathcal{L}}_{\rho}(x,v)  \Lambda_{ 2}[  \mathfrak{N}_{\rho,1 }^{\mathcal{L
 }} ](t,x,v)   \|_{L^2_{x,v}} +  \| \langle v \rangle^{-1} \omega^{\mathcal{L}}_{\rho}(x,v)  \Lambda_{2}[  \mathfrak{N}_{\rho,2 }^{\mathcal{L
 }} ](t,x,v)   \|_{L^2_{x,v}}\\ 
 &\lesssim \sum_{\Gamma\in P_1, |\Gamma|+|\rho|+|\mathcal{L}|\leq N_0} \| \omega^{\mathcal{L}}_{\rho}(x,v)   \Lambda^{\rho} u^{\mathcal{L}}(t,x,v)\|_{L^2_{x,v}}  \big( \|  \langle v \rangle \Lambda_{1}[\Gamma (v^0)^{-1}](t,x,v)\|_{L^\infty_{x,v}}\\
  & +\| \langle v \rangle^{-1}  (v^0)^{-1}  \Lambda_{1}[ v_{\alpha} v_{\beta} \nabla_x \Gamma g^{\alpha\beta}](t,x,v)\|_{L^\infty_{x,v}}\big) \lesssim 2^{-m+  \mathcal{H}(\mathcal{L}, \rho) \delta_0 m +  (H(1)\delta_1 +2\gamma)  m}\epsilon_1^2.
\end{split}
  \ee
 
 Next, we estimate $  \mathfrak{N}_{\rho,3 }^{\mathcal{L
 }}(t,x,v) $. Recall \eqref{2022may16eqn1}. From the $L^2_{x}L^\infty-L^\infty_x L^2_v$ type bilinear estimate,  the decay estimate of the density type function (\ref{densitydecay}) in Lemma \ref{decayestimateofdensity} and the estimate of coefficients in (\ref{2020july30eqn1}), we have
\[
\| \omega^{\mathcal{L}}_{\rho}(x,v)  \Lambda_{2}[  \mathfrak{N}_{\rho,3}^{\mathcal{L
 }} ](t,x,v)   \|_{L^2_{x,v}}  + \| \omega^{\tilde{\mathcal{L}} }_{\iota\circ\kappa}(x,v)  \Lambda_{2}[ \Lambda^{\iota} \mathfrak{N}_{  \kappa,3 }^{ \tilde{\mathcal{L}}} ](t,x,v)   \|_{L^2_{x,v}} 
\]
\be\label{2022may16eqn2}
\lesssim 2^{-m/2+     H(\mathcal{L}, \rho)\delta_0m + H(4)\delta_0    m } \epsilon_1^2. 
\ee

Lastly,  we  estimate $  \mathfrak{N}_{\rho,4 }^{\mathcal{L
 }}(t,x,v) $. Recall   (\ref{2020april8eqn35}).  If $|\rho_1|+|\tilde{\mathcal{L}}_1 |\geq 10$ or $|\rho|+|\mathcal{L}|\leq N_0-10$, then the following estimate holds from the $L^\infty_{x,v}-L^2_{x,v}$ type bilinear estimate,    the   estimate of coefficients in (\ref{2020july30eqn1}) and the decay estimate in  (\ref{2020julybasicestiamte}) in Lemma \ref{basicestimates},
  \[
\|    c_{\mathcal{L}_1\mathcal{L}_2\mathcal{L}_3;\tilde{L}_1, \tilde{L}_2}^{  {\mathcal{L}};\Gamma_1,\Gamma_2;\rho, \rho_1, \rho_2}(x,v) \Lambda_{1}[   \tilde{L}_1 \Gamma_2 \mathcal{L}_2 g^{\alpha\beta}(t,x+t\hat{v})  \tilde{L}_2 \mathcal{L}_3\big(  ({2v^0})^{-1} { v_{\alpha} v_{\beta}} \big)(t,x+t\hat{v}, v)]  \Lambda^{\rho_1} u^{\Gamma_1 \mathcal{L}_1}(t,x,v)
\] 
  \[
 \times \omega^{\mathcal{L}}_{\rho}(x,v) \|_{L^2_{x,v}} +  \| \omega^{\mathcal{L}}_{\rho}(x,v)      \tilde{c}_{\mathcal{L}_1\mathcal{L}_2;\tilde{L}}^{\mathcal{L};\rho,\rho_1,\rho_2}(x,v)  \cdot \nabla_x \Lambda^{\rho_1}  u^{\mathcal{L}_1}(t,x,v) \Lambda_{  1}[\tilde{L}\mathcal{L}_2(v^0)^{-1}](t, x+t\hat{v},v) \|_{L^2_{x,v}}
  \]
\be\label{2020april9eqn47}
\lesssim 2^{-m+  \mathcal{H}(\mathcal{L}, \rho) \delta_0 m +  (H(1)\delta_1 +2\gamma)  m}  \epsilon_1^2. 
 \ee
If  $|\rho_1|+|\tilde{\mathcal{L}}_1 |\leq 10$ and $|\rho|+|\mathcal{L}|\geq N_0-10$, we do dyadic decomposition for the metric component. As an example, we have 
 \begin{multline}
 \nabla_x \Lambda^{\rho_1}  u^{\mathcal{L}_1}(t,x,v) \Lambda_{  1}[\tilde{L}\mathcal{L}_2(v^0)^{-1}](t, x+t\hat{v},v) = \sum_{k\in \mathbb{Z}} H_k(t, x+t\hat{v}, v), \\ 
  H_k(t, x+t\hat{v}, v)=   \nabla_x \Lambda^{\rho_1}  u^{\mathcal{L}_1}(t,x,v) P_k( \Lambda_{  1}[\tilde{L}\mathcal{L}_2(v^0)^{-1}])(t, x+t\hat{v},v). 
 \end{multline}
On one hand, the following estimate holds from the $L^\infty_{x,v}-L^2_{x,v}$ type bilinear estimate and the $L^\infty\longrightarrow L^2$ type Sobolev embedding, 
\be\label{2020april9eqn49}
 \| \omega^{\mathcal{L}}_{\rho}(x,v)     \tilde{c}_{\mathcal{L}_1\mathcal{L}_2;\tilde{L}}^{\mathcal{L};\rho,\rho_1,\rho_2}(x,v)   H_k(t, x+t\hat{v}, v)\|_{L^2_{x,v}} \lesssim 2^{(1-\gamma)k+  H(\mathcal{L}, \rho) \delta_0 m}\epsilon_1^2. 
 \ee
 On the other hand,  form the $L^2_{x}L^\infty_v-L^\infty_x L^2_v$ type bilinear estimate,  the decay estimate of the density type function (\ref{densitydecay}) in Lemma \ref{decayestimateofdensity}, the first estimate in (\ref{2020julybasicestiamte}) in Lemma \ref{basicestimates} and the estimate of coefficients in (\ref{2020july30eqn1}), we have
\[
  \| \omega^{\mathcal{L}}_{\rho}(x,v)      \tilde{c}_{\mathcal{L}_1\mathcal{L}_2;\tilde{L}}^{\mathcal{L};\rho,\rho_1,\rho_2}(x,v)   H_k(t, x+t\hat{v}, v)\|_{L^2_{x,v}} \lesssim \sum_{ \iota\in \mathcal{S},|\iota|\leq 4 } 2^{-k/2+\gamma k-3m/2+ H(\mathcal{L}_1, \iota\circ \rho_1)  \delta_0 m} 
 \]
\be\label{2020april9eqn50}
 \times \min\{2^{   H(\mathcal{L}_2,  \rho_2)  \delta_0m}, 2^{H(  |\mathcal{L}_2|+| \rho_2|)\delta_1m}+2^{k+   H(\mathcal{L}_2,  \rho_2)   \delta_0m}+ 2^{m+k} \mathbf{1}_{|\rho_2|+|\mathcal{L}_2|=N_0}\} \epsilon_1^2  . 
 \ee
 Combining the above two estimates (\ref{2020april9eqn49}) and (\ref{2020april9eqn50}), we have
 \[
   \|    \omega^{\mathcal{L}}_{\rho}(x,v)      \tilde{c}_{\mathcal{L}_1\mathcal{L}_2;\tilde{L}}^{\mathcal{L};\rho,\rho_1,\rho_2}(x,v) \cdot  \nabla_x \Lambda^{\rho_1}  u^{\mathcal{L}_1}(t,x,v) \Lambda_{  1}[\tilde{L}\mathcal{L}_2(v^0)^{-1}](t, x+t\hat{v},v)   \|_{L^2_{x,v}}
 \]
\be\label{2020april9eqn51}
\lesssim 2^{-m+    H(\mathcal{L}, \rho)   \delta_0 m+   (H(1)\delta_1 +2\gamma)  m     } \epsilon_1^2. 
 \ee
 Following the same strategy, we have
   \[
\|   c_{\mathcal{L}_1\mathcal{L}_2\mathcal{L}_3;\tilde{L}_1, \tilde{L}_2}^{  {\mathcal{L}};\Gamma_1,\Gamma_2;\rho, \rho_1, \rho_2}(x,v) \Lambda_{1}[   \tilde{L}_1 \Gamma_2 \mathcal{L}_2 g^{\alpha\beta}(t,x+t\hat{v})  \tilde{L}_2 \mathcal{L}_3\big(  ({2v^0})^{-1} { v_{\alpha} v_{\beta}} \big)(t,x+t\hat{v}, v)]  \Lambda^{\rho_1} u^{\Gamma_1 \mathcal{L}_1}(t,x,v)
\] 
 \be\label{2020july30eqn2}
 \times \omega^{\mathcal{L}}_{\rho}(x,v) \|_{L^2_{x,v}}  \lesssim 2^{-m+    H(\mathcal{L}, \rho)   \delta_0 m+   (H(1)\delta_1 +2\gamma)  m    } \epsilon_1^2. 
 \ee
 To sum up, our desired estimate (\ref{2020april9eqn21}) holds from the estimates (\ref{2020april14eqn1}--\ref{2020april9eqn47})  and (\ref{2020april9eqn51}--\ref{2020july30eqn2}). 

 $\bullet$\qquad Proof of the desired estimate (\ref{2020april14eqn3}).  

 Recall (\ref{2020april8eqn33}).  The main idea  is to exploit the  decay rate of the distance with respect to the light cone. More precisely, the following identity holds,  
\be\label{definitionofcoefficient1}
(\p_t, \nabla_x) = \sum_{\Gamma\in P_1}  a_{\Gamma}(t,x) \Gamma,\quad  a_{\Gamma}(t,x) = c_{\Gamma}(t,x) (|t|^2-|x|^2)^{-1}
\ee
where  $c_{\Gamma }(t,x)$ are some uniquely determined  coefficients and satisfy the following estimate 
\be\label{2022may16eqn21}
 |c_{\Gamma}(t,x)|\leq |t|+|x|, \quad |a_{\Gamma}(t,x)|\leq \big|  |t|-|x| \big|^{-1}.
\ee
 Note that the following point-wise estimate holds, 
\be\label{2020aug31eqn31}
\langle |t|-|x+t\hat{v}|\rangle^{-1} \langle |x |\rangle^{-1} \langle|v|\rangle^{-2} \lesssim 2^{-m}. 
\ee
 From the above estimate, the $L^\infty_{x,v}-L^2_{x,v}$ type bilinear estimate, and the decay estimate in  (\ref{2020julybasicestiamte}) in Lemma \ref{basicestimates},  we have 
\[
  \| \langle v \rangle^{-3}\langle |x|\rangle^{-1} \omega^{\mathcal{L}}_{\rho}(x,v)  \Lambda_{2}[  \mathfrak{N}_{\rho,2 }^{\mathcal{L
 }} ](t,x,v)   \|_{L^2_{x,v}} +  \| \langle v \rangle^{-3}\langle |x|\rangle^{-1}\omega^{\tilde{\mathcal{L}} }_{\iota\circ\kappa}(x,v)  \Lambda_{2}[ \Lambda^{\iota} \mathfrak{N}_{  \kappa,2 }^{ \tilde{\mathcal{L}}} ](t,x,v)   \|_{L^2_{x,v}} 
\]  
\[
 \lesssim   \sum_{\Gamma_1, \Gamma_2\in P_1 , |\Gamma_1|+ |\rho|+|\mathcal{L}|\leq N_0} 2^{-m} \| \langle v \rangle^{-1}  (v^0)^{-1}  \Lambda_{1}[ v_{\alpha} v_{\beta} \Gamma_2 \Gamma_1 g^{\alpha\beta}](t,x,v)\|_{L^\infty_{x,v}} \| \omega^{\mathcal{L}}_{\rho}(x,v)   \Lambda^{\rho} u^{\mathcal{L}}(t,x,v)\|_{L^2_{x,v}}
\]
\be\label{2020april14eqn2}
 \lesssim 2^{-2m+   H(\mathcal{L}, \rho)   \delta_0 m + H(4) \delta_1 m}\epsilon_1^2. 
\ee
Hence finishing the proof of desired estimate (\ref{2020april14eqn3}).  
\end{proof}

\subsection{Estimates of cubic and higher order terms of the metric part}
\begin{lemma}\label{cubicandhighermetric}
Under the bootstrap  assumptions  \eqref{BAmetricold} and \eqref{BAvlasovori}, the following estimate holds for any $  \tilde{h}\in \{F, \underline{F}, \omega_j, \vartheta_{mn}\}$, $t\in [2^{m-1}, 2^m]\subset [0, T], m\in \mathbb{Z}_{+},L\in   P_n, n\in[0, N_0], $ $l \in [0, N_0-n]\cap \Z$, 
\[
 \|   L( 
 \Lambda_{\geq 2}[ H_{i\gamma}] \p_i\p_{\gamma} h_{\alpha\beta} ) -  \Lambda_{\geq 2}[ H_{i\gamma}] \p_i\p_{\gamma} L h_{\alpha\beta}  \|_{E_{ l}} + \|     L \Lambda_{\geq 3}[\mathcal{N}_{\alpha\beta}^{wa}]  \|_{ E_{ l}}  +  \|   R_{mt;s}^{\tilde{h};L } (t) \|_{E_{ l}}
\]
\be\label{oct1eqn31}
 \lesssim  2^{-7m/6+d(n+4,0 )\delta_0 m} \epsilon_1^3.
\ee
 
\end{lemma}
\begin{proof}

Recall (\ref{2020april3eqn1}), (\ref{nullstructure}) and (\ref{weaknull}).  We have
\[
{L}\Lambda_{\geq 3}[(g^{00})^{-1}  P_{\alpha\beta}]= \sum_{  \tilde{L}_1 \circ \tilde{L}_2 \preceq L} \tilde{L}_1(\Lambda_{\geq 1}[ (g^{00})^{-1}]) \tilde{L}_2 P_{\alpha\beta}   + \sum_{  L_1\circ L_2\circ L_3\preceq L}   c_{ {L}_2 {L}_3}^{ {L}_1}\Lambda_{\geq 1}[  {L}_1(g^{\gamma\mu}g^{\lambda\nu})]
\]
\[
\times \big[ \big( -\p_{\alpha} {L}_2h_{\mu\lambda} \p_{\beta} {L}_3 h_{\gamma\nu} \big) 
 +\h  {L}_2\p_{\alpha}h_{\mu\gamma} \p_{\beta} {L}_3h_{\lambda\nu}\big],
\]
\[
 {L} \Lambda_{\geq 3}[(g^{00})^{-1} Q_{\alpha\beta}]= \sum_{  \tilde{L}_1\circ \tilde{L}_2\preceq L} \tilde{L}_1(\Lambda_{\geq 1}[(g^{00})^{-1}] ) \tilde{L}_2 Q_{\alpha\beta} - L(\Lambda_{\geq 3}[Q_{\alpha\beta}]),
\]
\be
  L( 
 \Lambda_{\geq 2}[ H_{i\gamma}] \p_i\p_{\gamma} h_{\alpha\beta} ) -  \Lambda_{\geq 2}[ H_{i\gamma}] \p_i\p_{\gamma} L h_{\alpha\beta} =\sum_{ L_1\circ L_2\preceq L, L_2\prec L }   c_{ {L}_1{L}_2 }^{ {L}}     {L}_1(\Lambda_{\geq 2}[ H_{i\gamma} ])\p_{i}\p_{\gamma}  {L}_2 h_{\alpha\beta}. 
\ee
From the volume of support of the output frequency, the $L^2-L^2-L^\infty$ type trilinear estimate, and  the $L^2-L^\infty-L^\infty$ type trilinear estimate and the estimate (\ref{june30eqn1}) in Lemma \ref{highorderterm2},  we know that the following estimate holds after putting the input with the highest order derivative in $L^2$, 
\be\label{july3eqn91}
\begin{split}
 &\|   L( 
 \Lambda_{\geq 2}[ H_{i\gamma}] \p_i\p_{\gamma} h_{\alpha\beta} ) -  \Lambda_{\geq 2}[ H_{i\gamma}] \p_i\p_{\gamma} L h_{\alpha\beta}  \|_{E_{
 l}} + \|     L \Lambda_{\geq 3}[\mathcal{N}_{\alpha\beta}^{wa}]  \|_{E_{ l}} \\
&\lesssim 2^{- 3m/2+d(n+4,l )\delta_0 m } \epsilon_1^3+ \sum_{k\in \Z, k\leq 0} 2^{ -k_{-}/2+\gamma k_{-} +d(n+4,0 )\delta_0 m} \min\{ 2^{ - 3m/2  }, 2^{ 3k/2  } 2^{-m/2 } \}  \epsilon_1^3  \\
&\lesssim 2^{-7m/6+d(n+4,0 )\delta_0 m} \epsilon_1^3.\\
\end{split}
\ee
Recall  the   formula of  ``$  R_{mt;s}^{\tilde{h};L } $''  in \textup{(\ref{aug11eqn21})}.     Our desired estimate (\ref{oct1eqn31}) follows from    the above estimate (\ref{july3eqn91}) and   the estimate (\ref{june30eqn1})  in Lemma \ref{highorderterm2}.
\end{proof}
 
\begin{lemma}\label{cubicandhigherrough}
Under the bootstrap assumptions \textup{(\ref{BAmetricold})} and \textup{(\ref{BAvlasovori})}, the following estimate holds for  any  $t\in [2^{m-1}, 2^m]\subset [0, T], m\in \mathbb{Z}_{+}, L\in   P_n, n\in[0, N_0], $ $l \in [0, N_0-n]\cap \Z$, 
\be\label{july3eqn61}
   \| \Lambda_{\geq 3}[   L  \mathcal{N}_{\alpha\beta}^{vl} ]\|_{E_{ l}}   
   \lesssim  2^{-3m/2+2d(|L|+4, l )\delta_0 m } \epsilon_1^3.
\ee
\end{lemma}
\begin{proof}
Recall  (\ref{2020april3eqn1}). We have
\[
 \Lambda_{\geq 3}[L  \mathcal{N}_{\alpha\beta}^{vl} ]= \sum_{  \tilde{L}\preceq L} c_{L\tilde{L}}\tilde{L} \big( -  \Lambda_{\geq 3}[ T_{\alpha\beta}(f)] + \h \Lambda_{\geq 3}[ g_{\alpha\beta} T(f)] \big),\quad  T(f)=g^{\mu\nu} T_{\mu\nu}(f),  
\]
\[
   \Lambda_{\geq 3}[(g^{00})^{-1} g_{\alpha\beta}T_{ }(f)] =\Lambda_{\geq 2}[(g^{00})^{-1}g_{\alpha\beta}]T_{ }(f) + \Lambda_{1}[(g^{00})^{-1} g_{\alpha\beta}] \Lambda_{\geq 2}[T_{ }(f) ] 
\]
\be
 + \Lambda_{0}[(g^{00})^{-1}g_{\alpha\beta}] \Lambda_{\geq 3}[T_{ }(f) ].
\ee
 Since $g_{\alpha\beta}T(f)$ and $T_{\alpha\beta}(f)$ can be estimated in the same way, we only estimate $T_{\alpha\beta}(f)$ in details here.  Recall (\ref{march19eqn6}). After doing dyadic decomposition for inputs, for any $\tilde{L}\preceq L$ and $\alpha\in \Z^3$, s.t., $|\alpha|\leq N_0-n$, we have
\be\label{oct3reqn81}
  \Lambda_{\geq 3}[ \nabla_x^\alpha \tilde{L}T_{\alpha\beta}(f) ] =  \sum_{ \begin{subarray}{c}  
  |\mu|+|\mathcal{L}_1| +|\mathcal{L}_2|\leq  |L |\\ 
   |\alpha_1|+|\alpha_2|\leq |\alpha|
   \end{subarray}
 }    c^{ \tilde{L}, \mu }_{\mathcal{L}_1\mathcal{L}_2} \int_{\R^3} \nabla_x^{\alpha_1} \mathcal{L}_1f (t,x,v)   \Lambda_{\geq 2}[\nabla_x^{\alpha_2}\mathcal{L}_2 \big( \nabla_v^{\mu} \big(\frac{v_{\alpha} v_{\beta}}{v^0\sqrt{-\det g}}\big)   \big)] \big)  d v,
\ee
where   $ c^{ \tilde{L}, \mu }_{\mathcal{L}_1\mathcal{L}_2} $   are some uniquely determined absolute constants whose precise values are not pursued here. 

Note that, from the $L^2-L^2$ type bilinear estimate, the estimate (\ref{june30eqn1}) in Lemma \ref{highorderterm2}, and the estimate  (\ref{july2eqn2}) in Lemma \ref{highorderterm5}, we have
\[
\sum_{|\alpha|\leq N_0-n }\| \nabla_x^\alpha \Lambda_{\geq 3}[  \tilde{L}T_{\alpha\beta}(f) ]\|_{L^1_x} \lesssim \sum_{|\alpha_1|+|\alpha_2|\leq N_0-n} \|\langle v \rangle^{-10} \nabla_x^{\alpha_2}  \Lambda_{\geq 2}[\mathcal{L}_2 \big( \nabla_v^{\mu} \big(\frac{v_{\alpha} v_{\beta}}{v^0\sqrt{-\det g}}\big)   \big)] \big) \|_{L^2_x L^\infty_v}
\]
\be\label{2020aug24eqn39}
\times \|\langle v \rangle^{20} \nabla_x^{\alpha_1} \mathcal{L}_1f (t,x,v)\|_{L^2_xL^2_v}\lesssim 2^{-m/2+2e(|L|+4 )\delta_0 m }\epsilon_1^3. 
\ee
 Moreover, from the orthogonality in $L^2$, we have
 \[
  \| \int_{\R^3} \nabla_x^{\alpha_1} \mathcal{L}_1f (t  )   \Lambda_{\geq 2}[\nabla_x^{\alpha_2}\mathcal{L}_2 \big( \nabla_v^{\mu} \big(\frac{v_{\alpha} v_{\beta}}{v^0\sqrt{-\det g}}\big)   \big)] \big)  d v\|_{L^2_x }^2
 \]
 \[
 \lesssim \sum_{k\in \Z}  \| \int_{\R^3}P_{\leq k-10}\big( \nabla_x^{\alpha_2 }\Lambda_{\geq 2}[ \mathcal{L}_2 \big( \nabla_v^{\mu} \big(\frac{v_{\alpha} v_{\beta}}{v^0\sqrt{-\det g}}\big)   \big)]   \big)  P_{k}(\nabla_x^{\alpha_1} \mathcal{L}_1f ) (t )    d v\|_{L^2_x }^2 \]
 \[
  +  \| \int_{\R^3} P_{\leq k-10}(\nabla_x^{\alpha_1} \mathcal{L}_1f ) (t )  P_{k}\big( \Lambda_{\geq 2}[ \nabla_x^{\alpha_2} \mathcal{L}_2 \big( \nabla_v^{\mu} \big(\frac{v_{\alpha} v_{\beta}}{v^0\sqrt{-\det g}}\big)   \big)]   \big)  d v\|_{L^2_x }^2 
 \]
 \be\label{2020aug24eqn31}
 + \big( \sum_{|k_1-k_2|\leq 10}  \| \int_{\R^3} P_{k_1}(\nabla_x^{\alpha_1} \mathcal{L}_1f ) (t )  P_{k_2}\big( \Lambda_{\geq 2}[ \nabla_x^{\alpha_2} \mathcal{L}_2 \big( \nabla_v^{\mu} \big(\frac{v_{\alpha} v_{\beta}}{v^0\sqrt{-\det g}}\big)   \big)]   \big)  d v\|_{L^2_x } \big)^2.
 \ee

Based on the possible size of $|\mathcal{L}_1|$ and $|L|$, we separate into three cases as follow.

\noindent $\bullet$\quad If $ |\mathcal{L}_1|\leq |L|-5$ \qquad  From the estimate (\ref{julyeqn21}) in  Lemma \ref{decayestimateofdensity}, the following estimate holds 
 \[
 \|  \int_{\R^3}   P_{k_1}( \nabla_x^{\alpha_1} \mathcal{L}_1f )(t,x,v)   P_{k_2}
\big(\Lambda_{\geq 2}[\nabla_x^{\alpha_2}\mathcal{L}_2 \big( \nabla_v^{\mu} \big(\frac{v_{\alpha} v_{\beta}}{v^0\sqrt{-\det g}}\big)   \big)] \big) \big) d v\|_{L^2_x }
 \]
 \be\label{2020aug24eqn33}
 \lesssim 2^{-3m +  e(|L|+4 )\delta_0 m +\delta_1 k_{1}  -2k_{1,+}}\epsilon_1  \|  (1+|v|)^{-10}  P_{k_2}
\big(\Lambda_{\geq 2}[ \nabla_x^{\alpha_2}\mathcal{L}_2 \big( \nabla_v^{\mu} \big(\frac{v_{\alpha} v_{\beta}}{v^0\sqrt{-\det g}}\big)   \big)] \big) \big)\|_{L^2_x L^\infty_v}.
 \ee

\noindent $\bullet$\quad If $ |\mathcal{L}_1|\geq |L|-5$ and $|L|\leq N_0-5$.  \qquad  From the estimate (\ref{julyeqn21}) in  Lemma \ref{decayestimateofdensity}, the obtained  estimate (\ref{2020aug24eqn33}) is also valid if $k_2\geq k_1-10.$ It remains to consider the case 
 $k_2 \leq k_1-10$. From the estimate (\ref{june30eqn1}) in Lemma \ref{highorderterm2}, and the estimate  (\ref{july2eqn2}) in Lemma \ref{highorderterm5}, after putting the metric component in $L^\infty_{x,v}$, the following estimate holds if $k_2\leq k_1-10$,
  \[
 \|  \int_{\R^3}   P_{k_1}( \nabla_x^{\alpha_1} \mathcal{L}_1f )(t,x,v)   P_{k_2}
\big(\Lambda_{\geq 2}[\nabla_x^{\alpha_2}\mathcal{L}_2 \big( \nabla_v^{\mu} \big(\frac{v_{\alpha} v_{\beta}}{v^0\sqrt{-\det g}}\big)   \big)] \big) \big) d v\|_{L^2_x } \]
\be\label{2020aug24eqn35}
\lesssim 2^{-2m + e(|L|+4 )\delta_0 m }\epsilon_1^2 \|\langle v\rangle^{10}P_{k_1}( \nabla_x^{\alpha_1} \mathcal{L}_1f )(t,x,v)\|_{L^2_{x,v}}.
 \ee

\noindent $\bullet$\quad If $ |\mathcal{L}_1|\geq |L|-5$ and $|L|\geq N_0-5$.  \qquad Note that we have $|\mathcal{L}_2|\leq 5$  for this case. From the estimate (\ref{june30eqn1}) in Lemma \ref{highorderterm2}, and the estimate  (\ref{july2eqn2}) in Lemma \ref{highorderterm5}, after putting the metric component in $L^\infty_{x,v}$, the following estimate holds if $k_2\leq k_1-10$,
  \[
 \|  \int_{\R^3}   P_{k_1}( \nabla_x^{\alpha_1} \mathcal{L}_1f )(t,x,v)   P_{k_2}
\big(\Lambda_{\geq 2}[\nabla_x^{\alpha_2}\mathcal{L}_2 \big( \nabla_v^{\mu} \big(\frac{v_{\alpha} v_{\beta}}{v^0\sqrt{-\det g}}\big)   \big)] \big) \big) d v\|_{L^2_x } \]
\be\label{2020aug24eqn36}
\lesssim 2^{-2m +d(|\mathcal{L}_2|+4, |\alpha_2|)\delta_0 m -10k_{2,+}}\epsilon_1^2 \|\langle v\rangle^{10}P_{k_1}( \nabla_x^{\alpha_1} \mathcal{L}_1f )(t,x,v)\|_{L^2_{x,v}}.
 \ee
 To sum up, after combining the estimates (\ref{2020aug24eqn31}--\ref{2020aug24eqn36}), we have
\be\label{2020aug24eqn37}
\| \nabla_x^\alpha \Lambda_{\geq 3}[  \tilde{L}T_{\alpha\beta}(f) ]\|_{L^2_x}\lesssim 2^{-2m+2d(|L|+4, l )\delta_0 m } \epsilon_1^3 . 
\ee
Combinging the estimates (\ref{2020aug24eqn39}) and (\ref{2020aug24eqn37}), we have
\[
\|   \Lambda_{\geq 3}[  \tilde{L}T_{\alpha\beta}(f) ]\|_{E_n}\lesssim  \sum_{k\leq -m } 2^{(1+\gamma)k}2^{-m/2+  2d(|L|+4, l )\delta_0 m }\epsilon_1^3 + \sum_{k\in [-m,0]\cap \Z} 2^{-k/2+\gamma k }2^{-2m+ 2d(|L|+4, l )\delta_0 m } \epsilon_1^3
\]
\be
+ 2^{-2m+ 2d(|L|+4, l )\delta_0 m } \epsilon_1^3\lesssim 2^{-3m/2+2d(|L|+4, l )\delta_0 m } \epsilon_1^3.
\ee
 Hence finishing the proof of our desired estimate (\ref{july3eqn61}). 
\end{proof}

\subsection{Estimates of wave-Vlasov type quadratic terms of the metric part}
Recall (\ref{2020june30eqn1}). In this subsection, we mainly estimate the wave-Vlasov type interaction in $L^2$-type space, i.e.,  $ \Lambda_{\geq 2}[N_{vl}^{  {\tilde{h} ;L }}]    $ and $ \mathcal{N}_{vl;m}^{  {\tilde{h} ;L }}$. To this end, we first prove some bilinear estimates of general bilinear operators defined as follows.

For any $\mu\in\{+,-\}$, $a,b,c\in \R$, $k, k_1, k_2\in \mathbb{Z}$, any symbol $m\in \cup_{d\in \R}P_v^d \mathcal{M}^{a;b,c}$, we define the following two operators,
\[
T^k_{k_1,k_2}(h,u)(t,\xi):=\int_{\R^3}\int_{\R^3} e^{it|\xi| -it \hat{v}\cdot \eta} \widehat{h^{ }}(t,\xi-\eta,v)  \widehat{u}(t, \eta, v)
\]
\be\label{sep4eqn1} 
 \times m(\xi-\eta, \eta, v)\psi_k(\xi)\psi_{k_1}(\xi-\eta)\psi_{k_2}(\eta) d \eta d v.
\ee

\begin{lemma}\label{wavevlasovbi1}
Assume that $|t|\geq 1, a,b,c,d\in \R, k,k_1,k_2\in \mathbb{Z}$, $m\in P_v^d\mathcal{M}^{a;b,c} $. The following estimate holds, 
\[
\|T^k_{k_1,k_2}(h,u)(t,\xi)\|_{L^2}  
\lesssim  2^{ak + b k_1 + c k_2 }    2^{ 3\min\{k,k_1,k_2\}/2+3 k_2 /2}\| m\|_{P^d_v\mathcal{M}^{a;b,c}}   \| \widehat{h}(t, \xi, v)\psi_{k_1}(\xi) \|_{ L^\infty_v  L^2_\xi }\]
\be\label{june23eqn2}
\times  \min\{ \| (1+|x|+|v|)^{d+ 10} u(t,x,v)\|_{L^2_{x,v}},  \sum_{|\alpha|\leq 2} |t|^{-2}2^{-2k_2 }  \| (1+|x|+|v|)^{d+ 10}\nabla_v^{\alpha}u(t,x,v)\|_{L^2_{x,v}}\}. 
\ee
\[
 \|T^k_{k_1,k_2}(U,u)(t,\xi)\|_{L^\infty_\xi}  \lesssim 2^{ak+ b k_1 + c k_2 } 2^{ 3\min\{k_1,k_2\}/2}  
 \]
 \[
 \times \min\{ 2^{3\min\{k_1,k_2\}/2} \| \widehat{h}(t, \xi,v)\psi_{k_1}(\xi) \|_{L^\infty_{\xi, v}}, \| \widehat{h}(t, \xi,v)\psi_{k_1}(\xi) \|_{L^\infty_v L^2_\xi}\} \]
 \be\label{2020aug5eqn5}
  \times \min\{\|(1+|v|)^d\widehat{u}(t, \xi, v) \|_{L^1_vL^\infty_\xi},   \sum_{|\alpha|\leq 2} |t|^{-2}2^{-2k_2 }  \| (1+|x|+|v|)^{d+ 10}\nabla_v^{\alpha}u(t,x,v)\|_{L^2_{x,v}} \}.
\ee
\end{lemma}
\begin{proof}
Recall (\ref{june23eqn2}). After using the volume of support of $\xi$,   the $L^2-L^2$ type bilinear estimate,    Minkowski inequality, and the volume of support of $\eta$, the following estimates hold, 
\[
\|T^k_{k_1,k_2}(h,u)(t,\xi)\|_{L^2} \lesssim 2^{ak+ b k_1 + c k_2 } \| m\|_{P^d_v\mathcal{M}^{a;b,c}}  \min\{2^{3k/2 +3\min\{k_1,k_2\}/2}, 2^{3k_2 }\} 
\]
\be\label{june23eqn25}
 \times \| \widehat{h}(t, \xi, v)\psi_{k_1}(\xi) \|_{L^\infty_v L^2_\xi } \|(1+|v|)^d\widehat{u}(t, \xi, v) \|_{L^1_vL^\infty_\xi}.
\ee
\[
\|T^k_{k_1,k_2}(h,u)(t,\xi)\|_{L^\infty_\xi} \lesssim 2^{ak+ b k_1 + c k_2 } 2^{ 3\min\{k_1,k_2\}/2}\|(1+|v|)^d\widehat{u}(t, \xi, v) \|_{L^1_vL^\infty_\xi}
\]
\be\label{2020aug5eqn2}
\times  \min\big\{ 2^{3\min\{k_1,k_2\}/2} \| \widehat{h}(t, \xi, v)\psi_{k_1}(\xi) \|_{L^\infty_{\xi, v}}, \| \widehat{h}(t, \xi, v)\psi_{k_1}(\xi) \|_{L^\infty_v L^2_\xi}\big\}.
\ee
 On the other hand, we can do integration by parts in $v$ twice. As a result, the following estimates hold, 
\[
\|T^k_{k_1,k_2}(h,u)(t,\xi)\|_{L^2} \lesssim \sum_{|\alpha|\leq 2} 2^{a k + b k_1 + ck_2} \min\{2^{3k/2+ 3\min\{k_1,k_2\}/2}, 2^{3\min\{k_1,k_2\}}\}    2^{-2k_2}|t|^{-2} \| m\|_{P^d_v\mathcal{M}^{a;b,c}}\]
\be\label{june23eqn26}
\times\| \widehat{h}(t, \xi, v)\psi_{k_1}(\xi) \|_{L^\infty_v L^2_\xi} \| (1+|x|+|v|)^{d+ 10}\nabla_v^{\alpha}u(t,x,v) \|_{L^2_x L^2_v }.
\ee
\[
\|T^k_{k_1,k_2}(h,u)(t,\xi)\|_{L^\infty_\xi} \lesssim  \sum_{|\alpha|\leq 2} 2^{ak+ b k_1 + c k_2 } 2^{ 3\min\{k_1,k_2\}/2-2m -2k_2} 
 \| (1+|x|+|v|)^{d+ 10}\nabla_v^{\alpha}u(t,x,v) \|_{L^2_x L^2_v } 
\]
\be\label{2020aug5eqn3}
 \times\min\{ 2^{3\min\{k_1,k_2\}/2} \| \widehat{h}(t, \xi, v)\psi_{k_1}(\xi) \|_{L^\infty_{\xi, v}}, \| \widehat{h}(t, \xi, v)\psi_{k_1}(\xi) \|_{L^\infty_v L^2_\xi}\}.
\ee
Therefore, our desired estimate (\ref{june23eqn2}) holds   from the estimates (\ref{june23eqn25}) and (\ref{june23eqn26}) and the desired estimate (\ref{2020aug5eqn5}) holds from the estimates (\ref{2020aug5eqn2}) and (\ref{2020aug5eqn3}).

\end{proof}

By exploiting the distance with respect to the light cone, in the following Lemma, we show that we can trade one derivative  of the wave part for the decay rate $2^{-m}$ over time for the wave-Vlasov type interaction. 
 
 \begin{lemma}\label{wavevlabil6sep}
Assume that $t\in[2^{m-1}, 2^m]\subset[0,T], m\in \mathbb{Z}_{+}, a,b,c,d\in \R, k,k_1,k_2\in \mathbb{Z}$, $m\in P_v^d\mathcal{M}^{a;b,c} $, $ {\min\{k_1,k_2\}}\geq -m +\delta_1 m   $, then the following estimate holds for the bilinear operator defined in \textup{(\ref{sep4eqn1})},
\[
\| T^k_{k_1,k_2}(h,u)(t,\xi)\|_{L^2} \lesssim \sum_{\Gamma\in P}  2^{- m }   2^{3k_{2 }  } 2^{ak + bk_1+c k_2}  \| m\|_{P_v^d\mathcal{M}^{a;b,c}}  \big(2^{-k_1} \|    P_{k_1}\big(\Gamma h\big)   (t,x,v)\|_{  L^2_{x,v} }
\]
 \be\label{2020july8bilinear}
+ 2^{-2k_1} \| P_{k_1}(\p_t  h)\|_{L^2_{x,v}} \big)  \|P_{k_2}(u)(t,x,v)(1+|x|^5+|v|^{ 3/\delta_1 +d})\|_{L^2_{x,v}}.
\ee
\[
\| T^k_{k_1,k_2}(h,u)(t,\xi)\|_{L^2} \lesssim \sum_{\Gamma\in P}  2^{- m }   2^{3\min\{k,k_2\}_{-}/2 } 2^{ak + bk_1+c k_2}  \| m\|_{P_v^d\mathcal{M}^{a;b,c}}  \big(2^{-k_1} \|    P_{k_1}\big(\Gamma h\big)   (t,x,v)\|_{  L^\infty_{x,v} }
\]
 \be\label{bilinearestimate2}
  + 2^{-2k_1} \| P_{k_1}(\p_t  h)\|_{L^\infty_{x,v}} \big)  \|P_{k_2}(u)(t,x,v) (1+|x|^5+|v|^{ 3/\delta_1 +d})\|_{L^2_{x,v}}.
\ee
\end{lemma}

\begin{proof}
 Note that,  after applying the   equality (\ref{definitionofcoefficient1}) for $\mathcal{F}_{\xi}^{-1}[   \widehat{h^{ }}(t, \xi)\psi_{k_1}(\xi) ](t, x , v)$,  we have
 \[
 \mathcal{F}^{-1}_{\xi}[e^{-it |\xi|}T^k_{k_1,k_2}(h,u)(t,\xi)](x) = \sum_{\Gamma\in P } \int_{\R^3}\int_{\R^3} K(x_1,x_2,v)  u_{k_2}(t, x-x_1-x_2-t\hat{v}, v) 
 \]
 \be\label{sep7eqn1}
 \times  a_{\Gamma}(t, x-x_1)\cdot \Gamma \big( \mathcal{F}_{\xi}^{-1}[   -i\xi |\xi|^{-2}\widehat{h^{ }}(t, \xi)\psi_{k_1}(\xi) ]\big)(t, x-x_1, v)    d x_1 x_2 d v, 
 \ee
 where
\[
 K(x,y,v):=\int_{\R^3}\int_{\R^3} e^{ix \cdot\xi+ i y\cdot \eta}  m(\xi-\eta, \eta, v) \psi_{[k_1-4,k_1+4]}(\xi-\eta)\psi_{[k_2-4,k_2+4]}(\eta)\psi_k(\xi) d \xi d \eta.
\]

By doing integration by parts in $\xi$ and $\eta$ many times, we know that the following pointwise estimate holds for the kernel  if $k_2\leq k_1+10$, 
\be\label{july10eqn2}
 |K(x,y,v)| \lesssim \|m\|_{P_v^d\mathcal{M}^{a;b,c}} 2^{ak+b k_1 + c k_2} (1+|v|)^{d}2^{ {3k}+3k_2}(1+2^{k}|x|)^{-  20/\delta_1} (1+2^{\min\{k_1,k_2\} }|y|)^{-2 0/\delta_1}. 
\ee
If $k_2\geq k_1 + 10$, then we first switch the role of $\xi-\eta$ and $\eta$ and then do integration by parts in $\xi$ and $\eta$ many times. As a result, the following estimate holds for the kernel,
\be\label{july10eqn90}
|K(x,y,v)| \lesssim \|m\|_{P_v^d\mathcal{M}^{a;b,c}} 2^{ak+b k_1 + c k_2} (1+|v|)^{d}2^{ {3k}+3k_1}(1+2^{k}|x+y|)^{- 20/\delta_1}(1+2^{ \min\{k_1,k_2\}}|y|)^{- 20/\delta_1}.
 \ee
Note that the following estimate holds if  $ {\min\{k_1,k_2\}}\geq  -m +\delta_1 m    $, 
\be\label{july5eqn23}
(1+|t-(x-x_1)|)(1+| x-x_1-x_2-t\hat{v}|)(1+2^{\min\{k_1,k_2\}}(|x_1|+|x_2|) )^{20/\delta_1}(1+|v|)^{5/(2\delta_1)}\geq |t|. 
\ee
Hence, recall (\ref{sep7eqn1}), our desired estimate (\ref{2020july8bilinear}) follows directly from the above estimate, the estimate \eqref{2022may16eqn21},  and the estimate of kernel in (\ref{july10eqn2}). 

Now, we proceed to the proof of the desired estimate \eqref{bilinearestimate2}.  From the estimates of kernel in (\ref{july10eqn2}) and (\ref{july10eqn90}), we have
 \[
\|    T^k_{k_1,k_2}(h,u)(t,\xi) \|_{L^2}\lesssim  \sum_{\Gamma\in P  }  2^{-m} \|P_{k_2}(u)(t,x,v)(1+|x|+|v|^{5/(2\delta_1)}) \|_{L^2_{x,v}}  \]
\[
\times    \|   \Gamma \big( \nabla \d^{-2}P_{k_1}h \big)(t,x,v)\|_{L^\infty_{x,v}}   \|(1+2^{\min\{k_1,k_2\}}|z_2|)^{10/\delta_1}(1+|v|)^{-d}K(z_1,z_2, v)\|_{L^\infty_v L^1_{z_1,z_2}}  
  \]
  \[
  \lesssim \sum_{\Gamma\in P  }2^{-m } 2^{ak + bk_1+c k_2} \| m\|_{P_v^d\mathcal{M}^{a;b,c}}    \|   \Gamma \big( \nabla \d^{-2}P_{k_1}h \big)(t,x,v)\|_{L^\infty_{x,v}} 
  \]
\be\label{july10eqn5}
  \times \|P_{k_2}(u)(t,x,v)(1+|x|+|v|^{5/(2\delta_1)}) \|_{L^2_{x,v}}.
 \ee
Note that, as a result of direct computation,     for    any $\Gamma\in \{S, \Omega_{ij}, L_i, \p_{x_i}\}$,   we have 
\[
   \| \Gamma \big( \nabla \d^{-2}P_{k_1}h \big) (t,x,v)\|_{  L^\infty_{x,v} } \lesssim   \|    \nabla \d^{-2}P_{k_1}\big(\Gamma h\big)   (t,x,v)\|_{  L^\infty_{x,v} } + \|    [\Gamma, \nabla \d^{-2}P_{k_1}] \big(h\big)   (t,x,v)\|_{  L^\infty_{x,v} }
\]
\be\label{sep13eqn1}
\lesssim 2^{-k_1} \|    P_{k_1}\big(\Gamma h\big)   (t,x,v)\|_{  L^\infty_{x,v} } + 2^{-2k_1} \| P_{k_1}(\p_t  h)\|_{L^\infty_{x,v}}.
\ee

Therefore, from the above two estimates, to prove (\ref{bilinearestimate2}), it would be sufficient to  consider the case $\min\{k,k_2\}\leq 1 $. Based on the relative size of $k_1$ and $k_2$, we split into two cases as follow.

\noindent $\bullet$\quad If $k_2\leq k_1+10$.\qquad From the   pointwise estimate  (\ref{july5eqn23}) and the estimate of kernel $K(z_1,z_2,v)$ in  (\ref{july10eqn2}), we have 
\[
|\mathcal{F}^{-1}_{\xi}[e^{-it|\xi|} T^k_{k_1,k_2} (h,u)(t,\xi)](x) |\lesssim  \sum_{\Gamma\in P }  2^{ak+b k_1 +c k_2}  2^{-m}  \| m\|_{P_v^d\mathcal{M}^{a;b,c}}  \|   \Gamma \big( \nabla \d^{-2}P_{k_1}h \big)(t,x,v)\|_{L^\infty_{x,v}}
\]
\be\label{june23eqn40}
\times \|u(t,x,v) (1+|x|^5+|v|^{ 3/\delta_1 +d})\|_{L^2_{x,v}}  2^{3k+3k_2 }\big(\int_{\R^3} \int_{\R^3} \int_{\R^3} (1+2^{k}|x-y|)^{- 10/\delta_1}  \tilde{m}(t, y, z) d y d z  \big),
\ee
where
\be\label{june23eqn41}
\tilde{m}(t, y,z):=  \int_{\R^3} (1+2^{ \min\{k_1,k_2\}}|y-z-t\hat{v}| )^{-10/\delta_1}(1+|v|^{10}+|z|^{  4 }  )^{-1} d v . 
\ee
From the estimates (\ref{june23eqn40})   , the following estimate holds from the Minkowski inequality, 
\[
\|  T^k_{k_1,k_2}(h,u)(t,\xi)\|_{L^2} \lesssim \sum_{\Gamma\in P^{ }  }  2^{3\min\{k,k_2\}/2}   2^{ak + bk_1+c k_2}   2^{- m } \| m\|_{P_v^d\mathcal{M}^{a;b,c}}  \|   \Gamma \big( \nabla \d^{-2}P_{k_1}h \big)(t,x,v)\|_{L^\infty_{x,v}}  
\]
 \be\label{july10eqn10}
\times    \|P_{k_2}(u)(t,x,v) (1+|x|^5+|v|^{ 3/\delta_1 +d}) \|_{L^2_{x,v}}.
\ee
\noindent $\bullet$\quad If $k_2\geq k_1+10$.\qquad From the estimate (\ref{july10eqn90}) ,    we have 
 \[
|\mathcal{F}^{-1}_{\xi}[e^{-it|\xi|}  T^k_{k_1,k_2}(h,u)(t,\xi)](x) |\lesssim  \sum_{\Gamma\in P^{ }  }   2^{ak+b k_1 +c k_2} \| m\|_{P_v^d\mathcal{M}^{a;b,c}}  2^{-m}   \|   \Gamma \big( \nabla \d^{-2}P_{k_1}h \big)(t,x,v)\|_{L^\infty_{x,v}} 
\]
\be\label{july11eqn2}
\times \|P_{k_2}(u)(t,x,v) (1+|x|^5+|v|^{ 3/\delta_1 +d})\|_{L^2_{x,v}}  2^{3k+3k_1 }\big(\int_{\R^3} \int_{\R^3}\int_{\R^3} (1+2^{k}|x-z-t\hat{v}|)^{-  10/\delta_1}  \tilde{m}(t, y,z) d y  d z  \big).
\ee
Recall (\ref{june23eqn41}). From the above estimate, the following estimate holds, 
\[
\|    T^k_{k_1,k_2}(h,u)(t,\xi) \|_{L^2} \lesssim  \sum_{\Gamma\in P  }   2^{3k/2 }   2^{ak+b k_1 +c k_2} \| m\|_{P_v^d\mathcal{M}^{a;b,c}} 2^{-m}   \|   \Gamma \big( \nabla \d^{-2}P_{k_1}h \big)(t,x,v)\|_{L^\infty_{x,v}}  \]
\be\label{july11eqn31}
\times     \|P_{k_2}(u)(t,x,v) (1+|x|^5+|v|^{ 3/\delta_1 +d})\|_{L^2_{x,v}}.
\ee
To sum up, from the estimate (\ref{july10eqn10}) and the estimate (\ref{july11eqn31}), the following estimate holds regardless the size of $k_1-k_2$,
\[ \|    T^k_{k_1,k_2}(h,u)(t,\xi) \|_{L^2}\lesssim  \sum_{\Gamma\in P  }  2^{3\min\{k,k_2\}/2}   2^{ak + bk_1+c k_2}   2^{-m } \| m\|_{P_v^d\mathcal{M}^{a;b,c}}  \|   \Gamma \big( \nabla \d^{-2}P_{k_1}h \big)(t,x,v)\|_{L^\infty_{x,v}}  
\]
 \be\label{july11eqn21}
\times       \|P_{k_2}(u)(t,x,v) (1+|x|^5+|v|^{ 3/\delta_1 +d})\|_{L^2_{x,v}}.
\ee
To sum up, our desired estimate (\ref{bilinearestimate2}) holds from the estimates    (\ref{july10eqn5}),   \eqref{july10eqn10}, (\ref{sep13eqn1}),  and  (\ref{july11eqn21}).

\end{proof}

In the following Lemma, by using the bilinear estimates obtained in previous Lemmas, we obtain rough energy estimate for the wave-Vlasov type interaction in the Einstein equations. 

\begin{lemma}\label{resonancevlasov}
Under the bootstrap assumptions \textup{(\ref{BAmetricold})} and \textup{(\ref{BAvlasovori})}, the following estimate holds for any      $t\in [2^{m-1}, 2^m]\subset[0,T]$, $m\in \Z_+, L\in   P_n, n\in[0, N_0], $ $l \in [0, N_0-n]\cap \Z$,   
\be\label{july11eqn32}
\begin{split}
  2^{-k_{-}/2+\gamma k_{-} + l k_{+} } \|P_k(\mathcal{N}_{vl;m}^{  {\tilde{h} ;L }} ) (t) \|_{L^2}&\lesssim 2^{ -m+k_{-}+ d(|L|+1,l)\delta_0m }\epsilon_1^2 + 2^{-7m/6+3  d(|L|+1,l)\delta_0m} \epsilon_1^2,\\
\| \mathcal{N}_{vl;m}^{  {\tilde{h} ;L }}  (t) \|_{E_l}&\lesssim 2^{ -m +  d(|L|+1,l)\delta_0m }\epsilon_1^2.  
\end{split}
\ee
\end{lemma}

\begin{proof}
Recall (\ref{2020june30eqn1}) and  the equation satisfied by the $u^{\mathcal{L}}$ in (\ref{2020april7eqn6}) . Note that
\be\label{2020april14eqn23}
\Lambda_2[   \mathfrak{N}_3^{\mathcal{L
 }}](t,x,v)  + \Lambda_2[   \mathfrak{N}_4^{\mathcal{L
 }}](t,x,v) = \nabla_x  \cdot  \mathfrak{Q}_1^{\mathcal{L
 }}  (t,x,v)  +  D_v \cdot  \mathfrak{Q}_2^{\mathcal{L
 }} (t,x,v) + \mathfrak{Q}_3^{\mathcal{L
 }} (t,x,v), 
\ee
where 
\be\label{2020july9eqn1}
\begin{split}
\mathfrak{Q}_1^{\mathcal{L
 }} (t,x,v)& = \sum_{ \mathcal{L}_1 \circ \mathcal{L}_2\preceq  {\mathcal{L}} }   \tilde{c}_{\mathcal{L}_1\mathcal{L}_2}^{\mathcal{L}}(v) \Lambda_{  1}[\mathcal{L}_1(v^0)^{-1}]    \mathcal{L}_2 f, \\  
 \mathfrak{Q}_2^{\mathcal{L
 }} (t,x,v)&= \sum_{
\begin{subarray}{c}
 \mathcal{L}_1 \circ \mathcal{L}_2  \preceq  {\mathcal{L}}
 \end{subarray}} {c}_{\mathcal{L}_1\mathcal{L}_2\mathcal{L}_3 }^{  {\mathcal{L}}; j,l}(v)    \p_{j} \mathcal{L}_1 g^{\alpha\beta  }  \mathcal{L}_2 f    \mathbf{e}_l, \\
 \mathfrak{Q}_3^{\mathcal{L
 }} (t,x,v) & = \sum_{\mathcal{L}_1, \mathcal{L}_2\in \mathcal{P}_n, \mathcal{L}_1 \circ \mathcal{L}_2\preceq  {\mathcal{L}} }   \tilde{c}_{\mathcal{L}_1\mathcal{L}_2}^{\mathcal{L}}(v)\cdot \nabla_x  \Lambda_{  1}[\mathcal{L}_1(v^0)^{-1}]    \mathcal{L}_2 f .
\end{split}
\ee 

By using the same strategy used in obtaining the estimates    (\ref{2020april9eqn21})   and  (\ref{2020april14eqn3}) in Lemma \ref{estimateofremaindervlasov3}, the following estimate holds for any $\rho\in \mathcal{S}$ s.t., $|\rho| +|\mathcal{L}|\leq N_0,$
\be\label{2020april14eqn21}
\begin{split}
\|\langle v \rangle^{-3}\langle |x|\rangle^{-1}  \omega^{\mathcal{L}}_{\rho}(x,v) \Lambda^\rho \mathfrak{Q}_1^{\mathcal{L
 }} (t,x,v)   \|_{L^2_{x,v}} & \lesssim 2^{-m +  d( |\tilde{c}(\rho)|+|\mathcal{L}|+1, |\rho|-|\tilde{c}(\rho)|)\delta_0m    } \epsilon_1^2, \\ 
   \sum_{i=2,3}\| \langle v \rangle^{-3}\langle |x|\rangle^{-1}  \omega^{\mathcal{L}}_{\rho}(x,v) \Lambda^\rho \mathfrak{Q}_i^{\mathcal{L
 }} (t,x,v)   \|_{L^2_{x,v}} & \lesssim 2^{-3m/2 + d(  |\tilde{c}(\rho)|+|\mathcal{L}|+3,|\rho|- |\tilde{c}(\rho)|)\delta_0 m } \epsilon_1^2 .
 \end{split}
\ee
From the equalities (\ref{2020april7eqn6}) and  (\ref{2020april14eqn23}), we have
\[
\int_{\R^3}  e^{ - i t\hat{v}\cdot \xi}  \frac{ {b}^{L\mathcal{L} }_{\alpha\beta}(v)  k^{\tilde{h}}_{\alpha\beta}(\xi)  \p_t \widehat{u^{\mathcal{L}}}(t, \xi, v)}{i2\big(|\xi|-\hat{v}\cdot\xi\big)}\psi_k(\xi) d v = I_k + II_k + III_k,
\]
where 
\be\label{oct5eqn53}
\begin{split}
I_k & = \sum_{i=3,4}\int_{\R^3} e^{-i t |\xi|- i t\hat{v}\cdot \xi}  {b}^{L\mathcal{L} }_{\alpha\beta}(v)  k^{\tilde{h}}_{\alpha\beta}(\xi)    \frac{  \widehat{  \Lambda_{\geq 3}[\mathfrak{N}_i^{\mathcal{L}}] } (t, \xi, v)  + \widehat{  \mathfrak{Q}_3^{\mathcal{L}} } (t, \xi, v) }{i2\big(|\xi|+\hat{v}\cdot\xi\big)} \psi_{k}(\xi) d v, \\
II_k  & = \int_{\R^3} e^{-i t |\xi| }  {b}^{L\mathcal{L} }_{\alpha\beta}(v)  k^{\tilde{h}}_{\alpha\beta}(\xi)    i \xi \cdot \frac{ \widehat{ \mathfrak{N}_1^{\mathcal{L}}} (t, \xi, v) + \widehat{ \mathfrak{Q}_1^{\mathcal{L}}} (t, \xi, v)  }{i\big(|\xi|+\hat{v}\cdot\xi\big)} \psi_{k}(\xi) d v, \\  
III_k: & =   - \int_{\R^3} e^{-i t |\xi| }    \big(  \widehat{ \mathfrak{N}_2^{\mathcal{L}}} (t, \xi, v)  +  \widehat{ \mathfrak{Q}_2^{\mathcal{L}}} (t, \xi, v) \big)   \cdot \nabla_v \big( \frac{ {b}^{L\mathcal{L} }_{\alpha\beta}(v)  k^{\tilde{h}}_{\alpha\beta}(\xi)  }{i\big(|\xi|+\hat{v}\cdot\xi\big)}\big) \psi_{k}(\xi) d v.
\end{split}
\ee

From the estimate (\ref{aug9eqn87}) in Lemma \ref{estimateofremaindervlasov},  the estimate  (\ref{2020april14eqn3}) in Lemma \ref{estimateofremaindervlasov3}, and the estimate (\ref{2020april14eqn21}), the following estimate holds from the volume of support of $\xi$ if $k\leq 0$ and the Cauchy-Schwarz inequality, 
  \be\label{oct5eqn51}
  \begin{split}
  & \sum_{k\in \Z} 2^{-k_{-}/2+\gamma k_{-} + l k_{+} } \big(\big\| I_k\big\|_{L^2} + \big\| III_k\big\|_{L^2} \big) \\
 &\lesssim  \sum_{k\in \Z}  2^{-k_{-}/2+\gamma k_{-}  } 2^{-k+3k_{-}/2}  \big( 2^{-7m/6+3 d(|L|+3,l)\delta_0m} \epsilon_1^3 +   2^{ -2m  + d(|L|+3, l)\delta_0m} \epsilon_1^2\big) \\ 
 & \lesssim 2^{-7m/6+3 d(|L|+3, l) \delta_0m} \epsilon_1^2.
  \end{split}
\ee
Similarly, from the estimate  (\ref{2020april9eqn21}) in Lemma \ref{estimateofremaindervlasov3} and the estimate (\ref{2020april14eqn21}), the following estimates hold  from the volume of support of ``$\xi$'' if $k\leq 0$ and the Cauchy-Schwarz inequality, 
\[
  2^{-k_{-}/2+\gamma k_{-} + l k_{+} }  \big\| II_k\big\|_{L^2}  \lesssim 2^{-k_{-}/2+\gamma k_{-}  } 2^{ 3k_{-}/2} \big( \sum_{\rho\in \mathcal{S}, |\rho|\leq l, \tilde{c}(\rho)=0 } \| \langle v\rangle^{10} \Lambda^\rho  \mathfrak{Q}_1^{\mathcal{L}}(t,x,v)\|_{L^2_{x,v}} 
  \]
  \be\label{oct5eqn52}
  +  \| \langle v\rangle^{10}   \mathfrak{R}_{\rho,1}^{\mathcal{L}}(t,x,v)\|_{L^2_{x,v}}\big)  \lesssim 2^{-m+k_{-}+ d(|L|+1, l) \delta_0 m }\epsilon_1^2.
\ee
 
Our desired estimates in (\ref{july11eqn32}) hold  from the decomposition (\ref{oct5eqn53}) and the estimates (\ref{oct5eqn51}) and (\ref{oct5eqn52}).
\end{proof}

\begin{lemma}\label{quadraticvlasovrough}
Under the bootstrap assumptions \textup{(\ref{BAmetricold})} and \textup{(\ref{BAvlasovori})}, the following estimate holds for any   $k\in \mathbb{Z}$, $t\in[2^{m-1}, 2^m]\subset[0,T]$,  $m\in \Z_+, L\in   P_n, n\in[0, N_0], $ $l \in [0, N_0-n]\cap \Z$,    
\be\label{july5eqn30}
\begin{split}
 2^{-k_{-}/2+\gamma k_{-}+ l k_{+}}   \| \Lambda_{2}[ P_k( L  \mathcal{N}_{\alpha\beta}^{vl})]\|_{L^2}   
  &\lesssim  2^{  -m+k_{-}+  d(|L|+1, l )  \delta_0 m   } \epsilon_1^2 +    2^{-3m/2}\epsilon_1^2,\\ 
   \| \Lambda_{2}[ P_k( L  \mathcal{N}_{\alpha\beta}^{vl})]\|_{E_l} &\lesssim   2^{  -m+       d(|L|+1, l ) \delta_0 m   } \epsilon_1^2. 
\end{split}
\ee
\end{lemma}
\begin{proof}
Recall  (\ref{2020april3eqn1}). Since $g_{\alpha\beta}T(f)$ can be estimated in the same way as the estimate of $T_{\alpha\beta}(f)$, we only handle $T_{\alpha\beta}(f)$ in details here. 

Recall the equality (\ref{oct3reqn81}). For fixed $  \mathcal{L}_i\in \cup_{n\leq N_0}\mathcal{P}_{n} , i\in\{1,\cdots,4\},$  s.t., $   \mathcal{L}_1\circ \mathcal{L}_2\preceq  {L},    \mathcal{L}_3\circ \mathcal{L}_4\preceq \mathcal{L}_2  
 $, we do dyadic decomposition for both inputs and have the decomposition as follows, 
\[
P_k\big( \int_{\R^3} \mathcal{L}_1f(t,x,v) \Lambda_{1}[ \mathcal{L}_3 \big(\frac{1}{\sqrt{-\det g}}\big) \mathcal{L}_4\big(  \nabla_v^{\beta}\big(  \frac{v_{\mu} v_{\nu} }{  v^0}\big)\big)]  d v\big) =\sum_{(k_1,k_2)\in \cup_{i=1,2,3}\chi_k^i} H_{k_1,k_2}^k,
\] 
where
\be\label{july5eqn1}
 H_{k_1,k_2}^k = \int_{\R^3} P_k\big[ P_{k_1}(\mathcal{L}_1f)(t,x,v) P_{k_2}(\Lambda_{1}[\mathcal{L}_3 \big(\frac{1}{\sqrt{-\det g}}\big) \mathcal{L}_4\big(  \nabla_v^{\beta}\big(  \frac{v_{\mu} v_{\nu} }{  v^0}\big)\big)]) \big]d v.
\ee

\noindent $\bullet$\quad  If $|\mathcal{L}_1|\leq |L|-5$.  \qquad From the estimate (\ref{june23eqn2}) in Lemma \ref{wavevlasovbi1}, the following estimate holds, 
\[
\| H^k_{k_1,k_2}\|_{L^2} \lesssim \sum_{ \rho\in \mathcal{S}, |\rho|\leq N_0-(|\mathcal{L}_1|+3), |L|\leq |\mathcal{L}_2|,|\gamma|\leq 3} 2^{3\min\{k,k_1,k_2\}/2-k_2 - (N_0-n+2)k_{1,+}} 2^{3k_1/2 } \min\{1,  2^{-2m -2k_1}\}  
\]
\be\label{july3eqn121}
\times  \| P_{k_2} U^{ {L}h_{\alpha\beta}}\|_{L^2} \|\langle |x|+|v|\rangle^{20}\Lambda^{\rho}\nabla_v^{\gamma}  u^{\mathcal{L}_1f}(t,x,v)\|_{L^2_{x,v}}
\ee

\noindent $\bullet$\qquad    If $|\mathcal{L}_1| > |L|-5$ and $|L|\leq N_0-5$. \qquad   Note that,  the following estimate holds from the $L^2-L^\infty$ type bilinear estimate by putting the perturbed metric component in $L^\infty$ and the estimate (\ref{2020julybasicestiamte}) in Lemma \ref{basicestimates} if $k_2\leq k-10$,
\[
\| H_{k_1, k_2}^k\|_{L^2} \lesssim  \min\{2^{ k_2},  2^{-m  }\} 2^{-\gamma k_2  +H( |\mathcal{L}_2|+1)\delta_1m +1.1\gamma m - 4k_{2,+}} \epsilon_1
\]
\be\label{2020aug25eqn23}
\times  \min\{ \|\langle |v|\rangle^{4}P_{k_1}(\mathcal{L}_1 f) (t,x,v) \|_{L^2_{x,v}}, 2^{3k_1/2} \|\langle |v|\rangle^{4} \widehat{u^{\mathcal{L}_1}}(t, \xi, v)\psi_{k_1}(\xi)  \|_{L^1_{v}L^\infty_{\xi}}\}. 
\ee
If $k_2\geq k -10$, then the following estimate holds from  the estimate (\ref{june23eqn2}) in Lemma \ref{wavevlasovbi1},
\be\label{2020aug25eqn21}
\| H^k_{k_1,k_2}\|_{L^2} \lesssim   2^{3k_1-k_2 - 2k_{1,+} + H(N_0 )\delta_0 m }  \min\{1,  2^{-2m -2k_1}\} \epsilon_1 
   \| P_{k_2} U^{ {L}h_{\alpha\beta}}\|_{L^2}.
\ee

\noindent $\bullet$\qquad    If $|\mathcal{L}_1| > |L|-5$ and $|L|\geq N_0-5$. \qquad For this case we have  $l\leq 5$ and $|\mathcal{L}_2|\leq 5.$ Note that the obtained estimate (\ref{2020aug25eqn23}) is also valid if $k_2\leq k-10$. 

  It would be sufficient to consider the case $k_2\geq k-10$, i.e., $(k_1,k_2)\in \chi_k^1\cup \chi_k^2   $.
  From the volume of support of the output frequency and inputs frequency, the following estimate holds  if $ {k_1}\leq  {-99m/100}$, 
\be\label{july18eqn1}
   \| H_{k_1, k_2}^k\|_{L^2} \lesssim2^{  3\min\{k,k_1\}/2+3k_1/2 - k_2 } \|\widehat{U^{ {L}h_{\alpha\beta}}}(t, \xi)\psi_{k_2}(\xi)\|_{L^2_\xi}   \| \langle  v \rangle^{4}\widehat{u^{\mathcal{L}_1   }}(t, \xi, v)\|_{L^1_v L^\infty_\xi}
\ee  
 
If ${k_1}\geq  {-99m/100}$, then the following estimate holds from the estimate (\ref{bilinearestimate2}) in Lemma \ref{wavevlabil6sep} and the estimate  (\ref{2020julybasicestiamte}) in Lemma \ref{basicestimates}, 
 
\be\label{july5eqn21}
 \| H_{k_1, k_2}^k\|_{L^2} \lesssim  2^{3k/2 -20k_{2,+}} 2^{-2m-k_2+d(|L|+3,0)\delta_0 m } \epsilon_1^2. 
\ee

To sum up, in whichever case, from the above obtained estimates (\ref{july3eqn121}--\ref{july5eqn21}), for any fixed $k\in \Z,  $we have   
\be\label{july5eqn29}
   \sum_{(k_1,k_2)\in \chi_k^1\cup \chi_k^2\cup \chi_k^3}  2^{ -k_{-}/2+\gamma k_{-}+ l k_{+}}\| H_{k_1, k_2}^k\|_{L^2} \lesssim  2^{-m+k_{-} + d(|L|+1, l)\delta_0 m   } \epsilon_1^2 + 2^{-3m/2}\epsilon_1^2. 
\ee
Hence finishing the proof of our desired estimate (\ref{july5eqn30}). 

\end{proof}

 \subsection{Estimates of wave-wave type quadratic terms of the metric part}\label{wavewave}

 In this subsection, we first prove several 
   bilinear estimates for the wave-wave type interaction and then give a fixed time energy estimate for  the wave-wave type    quadratic terms of the metric part. 

We remark that  there are two major differences between the desired bilinear estimates and the general   $L^2-L^\infty$ type rough bilinear estimate. Firstly, we pay attention to the smallest frequency among the output frequency and the frequencies of  the two inputs. Because of the delicate structure of energy spaces at the low frequency part, the gain of the factor $2^{\min\{k,k_1,k_2\}/2}$ is crucial. Secondly, we pay  attention to the hierarchy of growth rates for different orders of vector fields applied, which    is very subtle and also   crucial  to close the bootstrap argument. 

\subsubsection{Set-up of bilinear estimates}
  For any $t\in [2^{m-1}, 2^m]\subset[0, T]$, $m\in \mathbb{Z}_{+}$,  $ L_1, L_2 \in  \cup_{n\leq N_0 } P_{n } , k \in \mathbb{Z}, i\in\{1,2\}, $     we assume that the following estimates hold,
\be\label{oct22eqn1}
\begin{split}
 \sum_{k\in \Z }  (A_i(k,m,L_1;L_2))^{-2} \|R_{\alpha }P_k(L_2   f_i^{L_1}  )(t) \|_{L^2}^2 & \lesssim 1,\quad A_i(k,m,L_1 ):=A_i(k,m,L_1;Id),  \\    \|\widehat{f_i}(t, \xi)\psi_k(\xi)\|_{L^\infty_\xi}  &  \lesssim C(k,m),  
 \end{split}
\ee
where $ f_i^{L_1}$ is some input depends on the input $f_i$ and the vector field $L_1$, e.g., $L_1 h_{\alpha\beta}$, $\widetilde{U^{L_1 h_{\alpha\beta}}}$, and $\widetilde{U^{  \tilde{h}^{L_1} }}$.  Moreover, the following hierarchy holds for   $A_i(k,m,L_1;L_2)$, $i\in\{1,2\},$
\be
A_i(k,m,\tilde{L}_1;\tilde{L}_2)\leq A_i(k,m,L_1;L_2), \quad \textup{if\,\,} \tilde{L}_1\preceq  {L}_1, \quad \tilde{L}_2\preceq  {L}_2. 
\ee

Given any  fixed $m\gg 1,$ $t\in [2^{m-1}, 2^m]\subset[0, T]$, $k \in \mathbb{Z}$,  $(k_1,k_2)\in \chi_k^1\cup\chi_k^2\cup \chi_k^3$,    $L_1, L_2\in \cup_{n\leq N_0 } P_n$,   s.t., $|L_1|+|L_2|\leq N_0  $, and any $f,g$ satisfy the estimates in (\ref{oct22eqn1}), 
 we define the following bilinear operator, 
\be\label{2020june20eqn8}
\begin{split}
 T^{  }_{k,k_1, k_2}(R_\alpha L_1f_1,R_\beta L_2f_2   )(t,x)&= \int_{\R^3}\int_{\R^3} K_{k,k_1,k_2}(x_1, x_2)\\
& \times P_{k_1}R_{\alpha}(L_1f_1 )(x-x_1) P_{k_2}R_{\beta}(L_2 f_2 )(x-x_1-x_2) d x_1 d x_2, 
\end{split}
\ee
where the kernel  $  K_{k,k_1,k_2}(x_1, x_2)$ is defined as follows, 
\be  
 K_{k,k_1,k_2}(x_1, x_2):=\int_{\R^3}\int_{\R^3} e^{i x_1 \cdot \xi + i x_2 \cdot \eta} k(\xi-\eta, \eta) \psi_{k}(\xi) \psi_{[k_1-4, k_1+4]}(\xi-\eta)  \psi_{[k_2 -4, k_2+4]}( \eta) d \xi d \eta, 
\ee
and the following estimate holds for the above defined kernlel 
\be\label{2020june20eqn30}
 \|K_{k,k_1,k_2}(x_1, x_2)\|_{L^1_{x_1,x_2}}\leq C(K).  
\ee

For any  $i\in\{1,2\}$, $L\in \nabla_x^\alpha P_n$,  we also define the following quantities for  convenience in notation, 
\be\label{2020june20eqn4}
\begin{split}
D_i(m,k )&:= \sum_{\Gamma_1, \Gamma_2\in  \{\Omega_{ij}\} }2^{3k/2}   C(k,m) + 2^{-m/3-k/3}A_i(k,m, Id;\Gamma_1\circ\Gamma_2  \big)  , \\
C_{i}(k,m, L)&:=\sum_{\Gamma_1, \Gamma_2\in   \{\Omega_{ij}\} }     (A_i(k,m, L;\Gamma_1  ))^{1-\delta_1^2}(A_i(k,m,  L ;  \Gamma_1 \circ\Gamma_2))^{ \delta_1^2} .
\end{split}
\ee

With the above notation,  from the estimate (\ref{july27eqn4}) in Lemma \ref{superlocalizedaug}, the following estimate holds if $k\geq -m $ and $|L|\leq N_0-2$,
\be\label{2020june20eqn7}
\begin{split}
 \|R_{\alpha} P_k( f_i^L )(t, x)\psi_{\geq m-10}(|x| ) \|_{L^\infty_x} &\lesssim 2^{-m+k/2} C_i(k,m,L ),\\ \|R_{\alpha} P_k(  f_i)(t, x)\psi_{\geq m-10}(|x| ) \|_{L^\infty_x} &\lesssim 2^{-m+k/2} D_i(m, k ). 
\end{split}
\ee
\subsubsection{Bilinear estimates}
In the following Lemma, based on the possible room of allowing  maximum number of vector fields, we provide several $L^2$-type bilinear estimates for the   bilinear operator defined in previous subsubsection. 

\begin{lemma}\label{2020junebilinearestimate}
The following estimate holds for the bilinear form defined in \textup{(\ref{2020june20eqn8})},
\be\label{2020june20eqn20}
 \| T^{  }_{k,k_1, k_2}(R_\alpha f_1^{L_1},R_\beta f_2^{ L_2}   )(t,x) \|_{L^2_x} \lesssim  2^{3\min\{k,k_1,k_2\}/2}   C(K) A_1(k_1, m,L_1) A_2(k_2, m, L_2). 
\ee 
Moreover, if $\min\{k_1, k_2\}\geq -m + \delta_1 m $ and  the kernel $K_{k,k_1,k_2}(x_1, x_2)$ satisfies the following estimate, 
\be\label{2020june20eqn1}
\quad  \|K_{k,k_1,k_2}(x_1, x_2)\psi_{\geq -m+\delta_1 m}(|x_1|+|x_2|)\|_{L^1_{x_1,x_2}\cap L^\infty_{x_1,x_2}}\leq 2^{-10 m}C(K).
\ee
Then  then following estimate holds if $  |L_1|\leq N_0-2$, $\star\in \{\leq m-10, \geq m+10\}$, 
\be\label{2020june20eqn11}
\begin{split}
 &\|  T^{  }_{k,k_1, k_2}(R_\alpha  f_1^{L_1} ,R_\beta  f_2^{L_2}  )(t,x)\psi_{\star}(x)\|_{L^2_x} \lesssim  2^{3\min\{k,k_1,k_2\}/2}   C(K) A_2(k_2, m, L_2)\\
&  \times\min\big\{   \sum_{\Gamma_1, \Gamma_2\in P_1} 2^{-2m-2k_1}  A_1(k_1, m,  L_1;\Gamma_1\circ\Gamma_2)  + 2^{-m-3k_1}\|P_{k_1}(\square  f_1^{L_1})\|_{L^2}\\
& + 2^{-2m-4k_1}\|P_{k_1}(\square \Gamma_1   f_1^{L_1} )\|_{L^2}, \sum_{\Gamma \in P_1} 2^{- m- k_1}  A_1(k_1, m,L_1;\Gamma )  
  + 2^{-m-3k_1}\|P_{k_1}(\square  f_1^{L_1})\|_{L^2}    \big\}.
  \end{split}
\ee
If $|L_1|\leq N_0-1$ and $|L_2|\leq N_0-1$, then the following estimate holds for  $\star\in  \{\leq m-10, \geq m+10\}$,
\be\label{2020june20eqn12}
\begin{split}
& \|  T^{  }_{k,k_1, k_2}(R_\alpha  f_1^{L_1} ,R_\beta  f_2^{L_2}  )(t,x)\psi_{\star}(x)\|_{L^2_x} \\
 &\lesssim\prod_{i=1,2}     2^{3\min\{k,k_1,k_2\}/2}  C(K) \big( \sum_{\Gamma \in P_1} 2^{- m- k_1} A_i(k_i, m, L_i;\Gamma)  +  2^{-m-3k_i}\|P_{k_i}(\square f_i^{ L_i})\|_{L^2} \big).
 \end{split}
 \ee 
 Moreover, for the close to the light cone case, the following rough estimate holds,
\be\label{2020june20eqn25}
\begin{split}
 & \|  T^{  }_{k,k_1, k_2}(R_\alpha  f_1^{L_1} ,R_\beta  f_2^{L_2}  )(t,x) \psi_{> m-10}(x)\|_{L^2_x} \\
& \lesssim   2^{-m+\min\{k,k_1,k_2\}/2 }  C(K) C_1(k_1,m,L_1) A_2(k_2, m, L_2). 
\end{split}
\ee
 In particular, if $|L_1|=0$, then we have 
 \be\label{2020june21eqn2}
  \|  T^{  }_{k,k_1, k_2}(R_\alpha  f_1^{L_1} ,R_\beta  f_2^{L_2}  )(t,x)\psi_{> m-10}(x)\|_{L^2_x} \lesssim   2^{-m+\min\{k,k_1,k_2\}/2 }  C(K) D_1(m, k_1) A_2(k_2, m, L_2)  .
\ee
\end{lemma}
\begin{proof}

Recall (\ref{2020june20eqn8}) and the estimate of kernel in (\ref{2020june20eqn30}). For fixed $x_1,x_2\in \R^3$,   we use the $L^2-L^\infty$ type bilinear estimate and the $L^\infty \longrightarrow L^2$ type Sobolev embedding. As a result, we have
\be\label{oct25eqn1}
\| T^{  }_{k,k_1, k_2}(R_\alpha  f_1^{L_1} ,R_\beta  f_2^{L_2}  )(t,x)\|_{L^2_x} \lesssim 2^{ 3\min\{k_1, k_2\}/2  }C(K)  A_1(k_1, m, L_1)A_2(k_2, m, L_2). 
\ee 
Alternatively, the following estimate holds from first using    the volume of support of $\xi$ and then using the $L^2-L^2$ type bilinear estimate, 
\be\label{oct25eqn2}
 \| T^{  }_{k,k_1, k_2}(R_\alpha  f_1^{L_1} ,R_\beta  f_2^{L_2}  )(t,x)\|_{L^2_x} \lesssim    2^{ 3k/2  }C(K)  A_1(k_1, m, L_1)A_2(k_2, m, L_2). 
\ee 
Hence the desired estimate (\ref{2020june20eqn20}) holds from the estimates (\ref{oct25eqn1}) and  (\ref{oct25eqn2}).  

Now we proceed to prove our desired estimates (\ref{2020june20eqn11}) and (\ref{2020june20eqn12}), which are the  the far away from the light cone cases. 
  From the  estimate of kernel in (\ref{2020june20eqn1}), we can restrict ourself to the case $|x_1|+|x_2|\leq 2^{m-\delta_1 m  }$. More precisely, by using the same strategy used in obtaining the estimate (\ref{2020june20eqn20}), we have   
\[
\big|\int_{|x_1|+|x_2|\geq 2^{m-\delta_1 m}} K_{k,k_1,k_2}(x_1, x_2) P_{k_1}(R_{\alpha}f_1^{L_1} )(x-x_1) P_{k_2}(R_{\beta} f_2^{L_2})(x-x_1-x_2) d x_1 d x_2\big|
\]
\be\label{2020june20eqn21}
 \lesssim  2^{-10m + 3\min\{k,k_1,k_2\}/2} C(K)  A_1(k_1, m,L_1) A_2(k_2, m, L_2). 
\ee
Hence $x-x_1$ and $x-x_1-x_2$ are also localized away from the light cone if $|x_1|+|x_2|\leq 2^{m-\delta_1 m  }$. 
  From the equality (\ref{definitionofcoefficient1}), the following equalities hold if $|L_1|\leq N_0-1$,
  \be\label{2020june20eqn24}  
\begin{split}
 P_{k_1}R_{\alpha}(f_1^{L_1})(t,x)  &=\sum_{\Gamma\in P_1} -  a_{\Gamma}(t,x)  \cdot \Gamma\nabla_x \d^{-2}( P_{k_1}R_{\alpha}(f_1^{L_1}) ) , \quad 
   \\ 
  \Gamma\nabla_x \d^{-2}( P_{k_1}R_{\alpha}(f_1^{L_1} ) )     &= T_{\Gamma,k_1}^{\alpha,1 } \p_{\alpha} f_1^{L_1}+   \nabla_x \d^{-2} P_{k_1}R_{\alpha}(\Gamma f_1^{L_1})   + T_{\Gamma,k_1}^{\alpha,2} \p_{t}^2 f_1^{L_1}   ,  
\end{split}
\ee
where $T_{\Gamma,k_1}^{\alpha,1}\in \mathcal{S}^{ -2}, T_{\Gamma,k_1}^{\alpha,2}\in \mathcal{S}^{ -3}$ are some uniquely determined Fourier multiplier operator such that $ T_{\Gamma,k_1}^{\alpha,2} =0$ if $\alpha \neq 0.$ 

If $|L_1|\leq N_0-2$,  then we can do the process of trading    derivatives for vector fields one more time in (\ref{2020june20eqn24}).   As a result, we have
\be\label{oct25eqn6}
\begin{split}
\nabla_x \d^{-2} P_{k_1}R_{\alpha}(\Gamma f_1^{L_1})  &+ T_{\Gamma,k_1}^{\alpha,1} \p_{\alpha} f_1^{L_1}+ T_{\Gamma,k_1}^{\alpha,2} \Delta f_1^{L_1}= \sum_{\tilde{\Gamma}\in P} - a_{\tilde{\Gamma}}(t,x) \big[ \tilde{T}_{\tilde{\Gamma}, \Gamma,k_1}^{\alpha,1}\p_\alpha (\tilde{\Gamma}\Gamma f_1^{L_1}) \\
&+ \tilde{T}_{\tilde{\Gamma}, \Gamma,k_1}^{\alpha,2}\p_\alpha (\tilde{\Gamma} f_1^{L_1}) + \tilde{T}_{\tilde{\Gamma}, \Gamma,k_1}^{\alpha,3}\p_t^2 ( \Gamma f_1^{L_1}) + \tilde{T}_{\tilde{\Gamma}, \Gamma,k_1}^{\alpha,4}\p_t^2 (  f_1^{L_1})\big],
\end{split}
\ee
where $\tilde{T}_{\tilde{\Gamma}, \Gamma,k_1}^{\alpha,1}, \tilde{T}_{\tilde{\Gamma}, \Gamma,k_1}^{\alpha,2}\in \mathcal{S}^{ -3}, \tilde{T}_{\tilde{\Gamma}, \Gamma,k_1}^{\alpha,3}, \tilde{T}_{\tilde{\Gamma}, \Gamma,k_1}^{\alpha,4}\in \mathcal{S}^{ -4}$.
Therefore,  by using the strategy used in obtaining the estimate (\ref{2020june20eqn20}),   our desired estimate (\ref{2020june20eqn11}) holds  from the equalities in (\ref{2020june20eqn24}) and  (\ref{oct25eqn6})   and the estimates (\ref{definitionofcoefficient1}) and 
(\ref{oct22eqn1}).  
 
If $|L_1|\leq N_0-1$ and $|L_2|\leq N_0-1$,    we are allowed to trade the space regularity for the decay rate of the distance with respect to the light cone for both inputs $ f_1^{L_1}$ and $f_2^{L_2}$. As a result, our desired estimate (\ref{2020june20eqn11}) holds after using the equalities in (\ref{2020june20eqn12})  for both ``$f_1^{L_1}$'' and ``$f_2^{L_2}$'' and  rerun the argument   used in obtaining the estimate (\ref{2020june20eqn20}).

Now, we focus on   proof  of   the desired estimates (\ref{2020june20eqn25})  and (\ref{2020june21eqn2}). Since they can be proved similalry,   we first prove the more difficult estimate (\ref{2020june20eqn25}) and then briefly explain how to obtain the other desired  estimate (\ref{2020june21eqn2}).  

 Note that  the estimate (\ref{2020june20eqn21}) is still valid. Hence, it would be sufficient to restrict ourself to the case $x_1,x_2\in \R^3$ are fixed such that $|x_1|+|x_2|\leq 2^{m-\delta_1 m }$. 
 Note that, the following estimate holds from  the estimate   (\ref{2020june20eqn7}) and the $L^2-L^\infty$ type bilinear estimate  
 \[
 \big|\int_{|x_1|+|x_2|\leq 2^{m-\delta_1 m}} K_{k,k_1,k_2}(x_1, x_2) P_{k_1}(R_\alpha f_1^{L_1} )(x-x_1) P_{k_2}(R_\beta  f_2^{L_2} )(x-x_1-x_2)  \psi_{> m-10}(x) d x_1 d x_2\big|
 \] 
 \be\label{oct24eqn47}
 \lesssim  2^{-m +k_1/2} C(K) C_1(k_1, m, L_1) A_2(k_2, m, L_2).
 \ee
 Hence finishing the proof of desired estimate (\ref{2020june20eqn25}) if $(k_1,k_2)\in \chi_k^2$ or $ (k_1,k_2)\in \chi_k^1, k\geq k_1-10$. 

 It remains to consider the case $(k_1,k_2)\in \chi_k^3$ and  the case   $ (k_1,k_2)\in \chi_k^1, k\leq k_1-10$.  Due to the symmetric structure in $L^2$,   the strategy  used in the case $(k_1,k_2)\in \chi_k^3,$ i.e., $k_2\leq k_1-10$ is also applicable for the case   $ (k_1,k_2)\in \chi_k^1, k\leq k_1-10$, it would be sufficient to consider the case when $(k_1,k_2)\in \chi_k^3$, i.e., $k_2\leq k_1-10$. 

  To take advantage of the orthogonality in $L^2$, we  localize further the frequency of $f_1^{L_1}$ 
  and have the following decomposition, 
  \[
   \int_{\R^3} \int_{\R^3} K_{k,k_1,k_2}(x_1, x_2) P_{k_1}(R_\alpha f_1^{L_1} )(x-x_1) P_{k_2}(R_\beta  f_2^{L_2} )(x-x_1-x_2) \psi_{\geq m-20 }(x-x_1)d x_1 d x_2  
   \]
  \[
   =\sum_{n\in [2^{k_1-k_2  -10}, 2^{k_1-k_2  + 10}]\cap \mathbb{Z}} I_n^1 + I_n^2,
  \]
where
\be\label{2020june20eqn50}
I_n^1:=\sum_{l\in \mathbb{Z}, |l-n|\leq 10}  T^{  }_{k,k_1, k_2}(P_{k_1,n}(R_\alpha P_{k_1,l}f_1^{L_1} \psi_{\geq m-20 }(\cdot) ),R_\beta  f_2^{L_2}   )(t,x),
\ee
\be\label{2022march20eqn1}
I_n^2:=\sum_{l\in \mathbb{Z}, |l-n|> 10}  T^{  }_{k,k_1, k_2}(P_{k_1,n}(R_\alpha P_{k_1,l}f_1^{L_1} \psi_{\geq m-20 }(\cdot) ),R_\beta   f_2^{L_2}   )(t,x)
\ee
where  the Fourier multiplier operator $P_{k_1,n}, n\in \Z$, are  defined by  the Fourier symbol $\psi_{\leq k_2}(|\xi|-2^{k_2}n)\psi_{k_1}(\xi)$.

From the estimate   (\ref{2020june20eqn21}) and the decomposition in (\ref{2020june20eqn50}) and the orthogonality in $L^2$, we have
\be\label{oct7eqn32}
\begin{split}
& \|  T^{  }_{k,k_1, k_2}(R_\alpha f_1^{L_1},R_\beta    f_2^{L_2})(t,x) \psi_{> m-10}(x)\|_{L^2_x}^2\\
& \lesssim   2^{-20m + 3\min\{k,k_1,k_2\} } \big(C(K)  A_1(k_1, m,L_1) A_2(k_2, m, L_2)\big)^2  + \sum_{i,j\in\{1,2\}}  I_{i,j}, 
\end{split}
\ee
where
\be\label{2022march20eqn13}
I_{i,j}:=   \sum_{n,n'\in  [2^{k_1-k_2  -10}, 2^{k_1-k_2  + 10}] \cap \mathbb{Z}, |n-n'|\leq 10 }  \int_{\R^3} \overline{I_n^i} I_{n'}^j d  x.
\ee 
From the $L^2-L^\infty$ type bilinear estimate and the super localized decay estimate (\ref{july27eqn4}) in Lemma \ref{superlocalizedaug}, the following rough estimate holds for $I_n^1$,
\be\label{oct7eqn13}
\begin{split}
\big\| I_n^1   \big\|_{L^2}&\lesssim \sum_{l \in \mathbb{Z}, |l-n|\leq 10}  C(K)\|P_{k_2}R_{\beta}( f_2^{L_2})\|_{L^2}\| P_{k_1,l}R_{\alpha}(f_1^{L_1} ) \psi_{\geq m-20}(|x|)\|_{L^\infty}
\\
&\lesssim  \sum_{|\gamma|\leq 2,|\kappa|\leq 1}  2^{-m   +k_2/2} C(K) A_2(k_2,m, L_2) \big[\| P_{k_1, n}\big(R_\alpha \Omega^\gamma f_1^{L_1}\big) \|_{L^2}^{\delta_1^2} \| P_{k_1, n}\big(R_\alpha \Omega^\kappa f_1^{L_1}\big)  \|_{L^2}^{1-\delta_1^2}   \big]. 
\end{split}
\ee
 
From the orthogonality in $L^2$ and the   estimate (\ref{oct7eqn13}) , we have
\be\label{oct7eqn31}
\begin{split}
\big| I_{1,1}\big|& \lesssim \sum_{n,n'\in  [2^{k_1-k_2  -10}, 2^{k_1-k_2  + 10}]\cap \mathbb{Z}, |n-n'|\leq 10} \| I_n^1\|_{L^2}  \| I_{n'}^1\|_{L^2} \\
& \lesssim \Big(\sum_{|\gamma|\leq 2,|\kappa|\leq 1} C(K) 2^{-m   +k_2/2} A_2(k_2,m, L_2)   \| P_{k_1}\big(R_\alpha \Omega^\gamma  f_1^{L_1} \big) \|_{L^2}^{\delta_1^2}\| P_{k_1}\big(R_\alpha \Omega^\kappa f_1^{L_1} \big)  \|_{L^2}^{1-\delta_1^2} \Big)^2\\
& \lesssim \big(  C(K) 2^{-m   +k_2/2}A_2(k_2,m, L_2) C_1(k_1, m, L_1)\big)^2.
\end{split}
\ee
 
Now, we  estimate the case  $I_{i,j},i,j\neq 1$. Note that    the symbol $ \psi_{\leq k_2}(|\xi|-2^{k_2}l) \psi_{\leq k_2}(|\xi-\sigma|-2^{k_2}n)  $ vanishes if $|\sigma|\leq 2^{ k_2}|l-n|/5$ and $|l-n|\geq 10$.  As a result, we have  
\begin{multline}\label{oct7eqn2}
 P_{k_1,n}(R_\alpha P_{k_1,l}f_1^{L_1} \psi_{\geq m-20 }(\cdot) )(x) =2^m \int_{\R^3} \int_{\R^3} e^{i x\cdot \xi}   \widehat{R_\alpha f_1^{L_1}}(\xi-\sigma)  \psi_{k_1}(|\xi|)\psi_{\leq k_2}(|\xi|-2^{k_2}n) \\
 \times \psi_{k_1}(|\xi-\sigma|)\psi_{\leq k_2}(|\xi-\sigma|-2^{k_2}l)  \hat{\psi}(2^m|\sigma|)\psi_{\geq -1}(2^{-k_2}5|l-n|^{-1}\sigma ) d \xi d \sigma. 
\end{multline}
Since $\hat{\psi}$ is a Schwarz function and $k_2\geq -m +\delta_1 m +10$, from the rapidly decay of Schwarz function, we have
\be\label{oct7eqn3}
\int_{\R^3 }\big| \hat{\psi}(2^m|\sigma|)\big|    \psi_{\geq -1}(2^{-k_2}5|l-n|^{-1} \sigma )  d\sigma \lesssim  2^{-100 m } |l-n|^{-2}. 
\ee
From the above estimate (\ref{oct7eqn3}), the $L^2-L^\infty$ type estimate, and the $L^\infty-L^2$ type Sobolev embedding, the following estimate holds for any $  n \in 2^{k_1-k_2}[2^{-10}, 2^{ 10}]\cap \mathbb{Z} $,
\be\label{oct7eqn11}
\big\| I_n^2 \big\|_{L^2}\lesssim \sum_{l\in  \mathbb{Z}} (1+|l-n|)^{-2} 2^{3k_2/2} 2^{-90m  } C(K)  \|P_{k_1,l}R_\alpha f_1^{L_1} \|_{L^2}\|P_{k_2}R_\beta  f_2^{L_2}\|_{L^2}.
\ee
From the $L^2-L^\infty$ type estimate and the Sobolev embedding,   the following rough estimate  holds of $I_n^1$,  
\be\label{oct7eqn14}
\big\| I_n^1   \big\|_{L^2}\lesssim  \sum_{l\in    \mathbb{Z}} (1+|l-n|)^{-2}   2^{3k_2/2}C(K)  \|P_{k_1,l}R_\alpha  f_1^{L_1} \|_{L^2}\|P_{k_2}R_\beta  f_2^{L_2}\|_{L^2}.
\ee
Recall (\ref{oct7eqn32}). From the above estimates (\ref{oct7eqn11}) and   (\ref{oct7eqn14})   the orthogonality in $L^2$, and  the Cauchy-Schwarz inequality,   we have
\be\label{2020june21eqn1}
\begin{split}
&\sum_{i,j\in\{1,2\},i,j\neq 1}\big|I_{i,j}\big| \lesssim \sum_{n,  n', l_1, l_2\in  [2^{k_1-k_2  -10}, 2^{k_1-k_2  + 10}] \cap \mathbb{Z}, |n-n'|\leq 10}  2^{3k_2 } 2^{-90m  }  (C(K))^2\\
& \times (1+|l_2-n|)^{-2}  (1+|l_1-n|)^{-2}  \|P_{k_1,l_1}R_\alpha f_1^{L_1} \|_{L^2} \|P_{k_1,l_2}R_\alpha  f_1^{L_1} \|_{L^2}\|P_{k_2}R_\beta  f_2^{L_2}\|_{L^2}^2 \\
& \lesssim 2^{3k_2 } (C(K))^2 2^{-90m  } \|P_{k_1 }R_\alpha f_1^{L_1} \|_{L^2}^2 \|P_{k_2}R_\beta  f_2^{L_2}\|_{L^2}^2\\
 & \lesssim\big( 2^{-45 m + 3k_2/2  }C(K) A_1(k_1, m, L_1) A_2(k_2,m, L_2)\big)^2.
 \end{split}
\ee
Hence our desired estimate (\ref{2020june20eqn25}) holds from the estimates  (\ref{oct7eqn32}), (\ref{oct7eqn31}), and (\ref{2020june21eqn1}). 

With minor modification,   our desired estimate (\ref{2020june21eqn2}) holds after rerunning the above argument with the first decay estimate in (\ref{2020june20eqn7}) replaced with the second decay estimate in (\ref{2020june20eqn7}).  
\end{proof}

\subsubsection{Estimate of the wave-wave type quadratic terms}

In this subsection, we use the bilinear estimates obtained in previous subsubsection to estimate the wave-wave type quadratic terms of the metric part. More precisely, we have 
 \begin{lemma}\label{quadratichigherein}
Under the bootstrap assumptions   \eqref{BAmetricold} and \eqref{BAvlasovori},   for any   $ L\in    P_n, n\in [0, N_0]\cap \Z, l\in [0, N_0-n]\cap \Z$,  $t\in[2^{m-1}, 2^m]\subset[0,T]$, $k\in \mathbb{Z},$ $K\in\{ S,Q,P\}$, we have 
\be\label{2022march20eqn11}
2^{-k_{-}/2+\gamma k_{-}+ (l-1)k_{+} }   \| P_k( \Lambda_{2}[{L} K_{\alpha\beta}])(t)\|_{L^2}\lesssim 2^{-m } \epsilon_1^2(2^{ H(|L|)\delta_1 m + \beta_{}(|L|) \delta_1 m}  + 2^{ k_{-}+d(|L|+1,l)\delta_0m  } ),
\ee
\be\label{july3eqn99}
 \| \Lambda_{2}[{L} K_{\alpha\beta}] (t)\|_{E_{l-1} }\lesssim 2^{-m + d(|L|+1,l)\delta_0m -5\delta_0m} \epsilon_1^2. 
\ee
Moreover, for any fixed $k\in \mathbb{Z}$, we have  
\be\label{oct28eqn36}
\begin{split}
2^{-k_{-}/2+\gamma k_{-}+(l-1)k_{+} }  &\| P_k( \Lambda_{\geq 2}[ \square L h_{\alpha\beta}])(t)\|_{L^2}\\ 
 &\lesssim 2^{-m + H(|L|)\delta_1 m +\beta_{}(|L|) \delta_1 m   }\epsilon_1^2+  2^{-m +k_{-} +  d(|L|+1,l)\delta_0 m   }\epsilon_1^2.\\ 
 \end{split}
\ee
\end{lemma}
\begin{proof}
 Recall (\ref{2020april8eqn1}).  For any   $  L\in    P_n, n\in [0, N_0]\cap \Z, l\in [0, N_0-n]\cap \Z$,   $t \subset[2^{m-1},2^m]\subset[0, T]$, we define
 \be\label{oct25eqn60}
\begin{split}
 X_{L}^k(t)&:= \sum_{K\in\{ S,Q,P\}} 2^{-k_{-}/2+\gamma k_{-}+ (l-1)k_{+} }   \| P_k( \Lambda_{2}[{L} K_{\alpha\beta}])(t)\|_{L^2}  , \\  X_L(t)&:= 2^{-d(|L|+1, l)\delta_0 m }\sup_{k\in \Z} X_{L}^k(t). 
\end{split}
\ee

Recall the decomposition of $L(K_{\alpha\beta}), K\in\{S, Q, P\}$,  in  (\ref{2020april8eqn1}), the bootstrap assumption (\ref{BAmetricold}) and the definition of bounds in (\ref{oct22eqn1}) and (\ref{2020june20eqn4}). For  any  $L, L_1, L_2\in   \cup_{n\in [0, N_0]\cap \Z, |\alpha|\in[0, N_0-n]\cap \Z } \nabla_x^\alpha P_n $ s.t., $L_1 \circ L_2\in  \cup_{n\in [0, N_0]\cap \Z, |\alpha|\in[0, N_0-n]\cap \Z } \nabla_x^\alpha P_n $ and $|L|\leq N_0 -5$, from the estimate (\ref{2022march19eqn34}), the corresponding  quantities for  the perturbed metric  $h_{\alpha\beta}$  are defined as follows,  
\be\label{2020june21eqn4}
\begin{split}
A_i(k,m,L_1;L_2)&:= 2^{ k_{-}/2-k-\gamma k_{-}- (\|L_1\|+\|L_2\|-|L_1|-|L_2|)k_{+}+d(|L_1|+|L_2|,\|L_1\|+\|L_2\|-|L_1|-|L_2|)\delta_0 m }\epsilon_1, \\
C_i(k,m, L)&:=  2^{ -k_{}/2-\gamma k_{-}- (N_0-|L|-2) k_{+}}   2^{-(1-\delta)k_{+}}  2^{(1-\delta_1^2)H(|L|+1 )\delta_1 m  + \delta_1^2 H(|L| +2 )\delta_1 m } \epsilon_1, \\
D_i(k,m)&:= \big(2^{k/2-a k} 2^{\delta m} + 2^{-m/3+k/6-\gamma k_{-}+ H(2  )\delta_1 m - (N_0-4)k_{+}} \big)\epsilon_1, 
\end{split}
\ee
where  $\|L\|,$ and $|L|$ are defined in Definition \ref{vectorfieldscla}. 

For $L\in  \cup_{n\leq N_0}P_n, \Gamma\in   \{\Gamma_1, \Gamma_2\Gamma_3: \Gamma_1\in P_1 , \Gamma_2, \Gamma_3 \in \{  \Omega_{ij}\}\}$,  the following  quantities will be associated with  the modified profiles  $ \widetilde{   U^{ \tilde{h}^L } }(t, x)$ and $\widetilde{   U^{  L h_{\alpha\beta} } }(t, x) $ defined in (\ref{modifiedperturbedmetric}), 
\be\label{2020june24eqn14}
\widetilde{A}_i(k,m,L;\Gamma):= 2^{ k_{-}/2-k-\gamma k_{-}-(N_0-|L|- |\Gamma|)k_{+}+H(|L|+|\Gamma|  ))\delta_1 m}\epsilon_1. 
\ee

Again, thanks to  the estimate (\ref{2022march19eqn34}), the modified profiles are not different much with the original profiles if the total vector fields are not very close to the top order.  

Recall (\ref{sep28eqn26}).  From the estimate (\ref{oct1eqn31}) in Lemma \ref{cubicandhighermetric}, the estimate (\ref{july3eqn61}) in Lemma \ref{cubicandhigherrough}, and the estimate (\ref{july5eqn30}) in Lemma \ref{quadraticvlasovrough}, the following estimate holds for any fixed $k\in \mathbb{Z}$,  $L\in   \cup_{n\leq N_0 } P_n$, 
\[
2^{-k_{-}/2+\gamma k_{-}+ (l-1)k_{+} } \| P_k\big( \Lambda_{\geq 2}[ \square L h_{\alpha\beta}]\big)\|_{L^2}
\]
 \be\label{oct26eqn1}
  \lesssim  2^{-7m/6+d(n+4,0 )\delta_0 m} \epsilon_1^3  + X_{L}^k (t) + 2^{  -m+k_{-}    + d(|L|+1, l ) \delta_0 m } \epsilon_1^2. 
\ee

With the above preparation, based on the possible size of $k,k_1,k_2$, we  separate into two cases as follow. 

$\bullet$\quad If $(k_1, k_2)\in \chi_k^2\cup \chi_k^3$, i.e., $|k_1-k_2|\geq 10, |k-\max\{k_1,k_2\}|\leq 10$. 

From the estimate (\ref{2020june20eqn20}) in Lemma \ref{2020junebilinearestimate}, the following estimate holds for any $Q\in \{S, P, Q\}$ if $\min\{k_1,k_2\}\leq -m +\delta_1 m $, 
\be\label{2020aug25eqn41}
 \|  QH_{\alpha\beta;k,k_1,k_2}^{L;L_1,L_2}(t,x)\|_{L^2}  \lesssim 2^{3\min\{k_1,k_2\}/2 + |k_1-k_2| } \| U^{L_1 h_{\alpha\beta}}_{k_1}(t)\|_{L^2}\| U^{L_2 h_{\alpha\beta}}_{k_2}(t)\|_{L^2}. 
\ee

Now we consider the case   $-m + \delta_1 m \leq \min\{k_1, k_2\}\leq k+10.$  From the estimate (\ref{2020june20eqn11}) in Lemma \ref{2020junebilinearestimate} and the estimate (\ref{oct26eqn1}), the following estimate holds for any $H\in \{S, P, Q\}  $ one of the following two scenarios happens: (i) $|L_1|\leq |L|-3$; (ii) $|L_1|\geq |L|-3, |L|\leq N_0-5$, and $k_1\leq k_2$,
\[
   \|     QH_{\alpha\beta;k,k_1,k_2}^{L;L_1,L_2}(t,x)  \psi_{\leq m-10}(x)\|_{L^2}\lesssim \sum_{\Gamma \in P_1}         2^{ 2k-k_2} 2^{3\min\{k_1,k_2\}/2} \| U^{L_2 h_{\alpha\beta}}_{k_2}(t)\|_{L^2} \big[2^{- m- k_1} A_1(k_1, m, \Gamma \circ L_1) 
   \]
\be\label{2020aug25eqn45}
+ 2^{-m+k_{1,-}/2-\gamma k_{1,-}- ( l-1)k_{1,+}} \big(2^{-3k_1} X_{L_1}^{k_1}(t)  + 2^{ -3k_{1}/2}\epsilon_1 + 2^{ -k_{1}/2} 2^{H(|L_1|)\delta_0 m } +  2^{- m/2 +k_{1,-}-3k_1}\epsilon_1^2 \big)\big].
\ee

Moreover, from   the estimate (\ref{2020june20eqn25}) in Lemma \ref{2020junebilinearestimate} ,  if $|L_1|\leq |L|-3$, we have 
\be\label{2020aug25eqn42}
 \|  QH_{\alpha\beta;k,k_1,k_2}^{L;L_1,L_2}(t,x) \psi_{ > m-10}(x)\|_{L^2} \lesssim  2^{2k-k_2} 2^{-m+\min\{k_1,k_2\}/2}   C_1(k_1,m,L_1) \| U^{L_2 h_{\alpha\beta}}_{k_2}(t)\|_{L^2}. 
\ee 

Similarly, if either   $|L_1|\geq |L|-3, |L|\leq N_0-5$,   $k_2\leq k_1$ or $|L_1|\geq |L|-3, |L|\geq N_0-5$, then the following two estimates hold  analogously,
\[
   \|     QH_{\alpha\beta;k,k_1,k_2}^{L;L_1,L_2}(t,x)  \psi_{\leq m-10}(x)\|_{L^2} \lesssim \sum_{\Gamma \in P_1}         2^{ 2k-k_1} 2^{3\min\{k_1,k_2\}/2}\| U^{L_1 h_{\alpha\beta}}_{k_1 }(t)\|_{L^2} \big[2^{- m- k_2} A_1(k_2, m, \Gamma \circ L_2)
   \]
\be\label{2020aug25eqn43}
   + 2^{-m + k_{2 ,-}/2-\gamma k_{2 ,-}- (l-1)k_{2,+}}\big(2^{-3k_2} X_{L_2}^{k_2}(t)  + 2^{ -3k_{2}/2}\epsilon_1+ 2^{ -k_{2}/2} 2^{d(|L_2|,l)\delta_0 m } +  2^{- m/2 +k_{2,-}-3k_2 }\epsilon_1^2 \big)\big] .
\ee
\be\label{2020aug25eqn44}
 \|  QH_{\alpha\beta;k,k_1,k_2}^{L;L_1,L_2}(t,x) \psi_{ > m-10}(x)\|_{L^2} \lesssim  2^{2k-k_1} 2^{-m+\min\{k_1,k_2\}/2}   C_1(k_2,m,L_2) \| U^{L_1 h_{\alpha\beta}}_{k_1}(t)\|_{L^2}.
\ee

From the obtained estimates (\ref{2020aug25eqn41}--\ref{2020aug25eqn44}), we have 
\[
\sum_{k\in \Z} \sum_{H\in \{S, P, Q\}}2^{-k_{-}+2\gamma k_{-} + 2 (l-1)k_{+}} \big(\sum_{(k_1,k_2)\in \chi_k^2\cup \chi_k^3}  \|  QH_{\alpha\beta;k,k_1,k_2}^{L;L_1,L_2}(t,x)\|_{L^2_x}  \big)^2
\]
\be\label{2020aug25eqn52} 
\lesssim  \big( \sup_{k_1\in \Z}   2^{-\delta_1 m/2  }   2^{  d(|L_2|,l) \delta_0 m} \epsilon_1   X^{k_1}_{L_1}(t) \big)^2+2^{-2m +2d (|L|+1,l)\delta_0 m -10\delta_0m }\epsilon_1^4. 
\ee
\[
\sum_{(k_1,k_2)\in \chi_k^2\cup \chi_k^3} 2^{-k_{-}/2+\gamma k_{-} + (l-1)k_{+}}  \|  QH_{\alpha\beta;k,k_1,k_2}^{L;L_1,L_2}(t,x)\|_{L^2_x} 
\]
\be\label{2020aug25eqn51}
   \lesssim 2^{- m +k_{-} +  d (|L|+1,l) \delta_0 m -5\delta_0m }\epsilon_1^2+   \sup_{k_1\in \Z}   2^{-\delta_1 m/2 +k_{-}}   2^{ d(|L_2|,l)\delta_0 m} \epsilon_1   X^{k_1}_{L_1}(t).
\ee

$\bullet$\quad If $(k_1, k_2)\in \chi_k^1$, i.e., $|k_1-k_2|\leq 10 $. 

For the High $\times$ High type interaction,   due to symmetry, we assume that $| {L}_1|\leq | {L}_2|$. Note that $| {L}_1|\leq N_0/2+5.$  Recall the decomposition of $QH_{\alpha\beta;k,k_1,k_2}^{L;L_1,L_2}$ in   (\ref{nov7eqn21}) and  we use formulations in \eqref{2020june18eqn1} and \eqref{2020june18eqn2}. Since the error type is zero in this formulation, it would be sufficient to estimate $  QH_{\alpha\beta;k,k_1,k_2}^{ess;L,L_1,L_2}(t,x)$ and $QH_{\alpha\beta;k,k_1,k_2}^{vl;L,L_1,L_2}, H\in \{S, P, Q\}.$

$\oplus$\qquad The estimate of $    QH_{\alpha\beta;k,k_1,k_2}^{vl;L,L_1,L_2}$.

   Recall (\ref{2020june18eqn2}). Note that, if $k_1\leq -m/10$, then we can gain extra $2^{-m/10}$ from the symbol of $  {THVl}_{k,k_1,k_2}^{\alpha\beta;\tilde{L}_1,  \tilde{L}_2 }$, which is more than sufficient to cover the gap of difference between $2^{H(|L|)\delta_1 m }$ and $2^{d (|L|+1,l)\delta_0 m }$. After we redo the argument used in  obtaining the  estimates     (\ref{2020aug25eqn41}--\ref{2020aug25eqn42}), we obtain the following estimate,
\be\label{nov8eqn73}
  \sum_{k\in \mathbb{Z}, (k_1,k_2)\in \chi_k^1, k_1\leq -m/10}     2^{-k_{-}/2+\gamma k_{-}+ (l-1) k_{+} }  \|    QH_{\alpha\beta;k,k_1,k_2}^{vl;L,L_1,L_2}(t,x)\|_{L^2} \lesssim  2^{-  m/40 }  X_{L_1}(t)  \epsilon_1 +  2^{-m  - m/40 }\epsilon_1^2.
 \ee

It remains to consider the case   $k_1\geq -m/10$. From the estimate (\ref{bilinearestimate2}) in Lemma \ref{wavevlabil6sep}  and the estimate (\ref{oct26eqn1}), we have
\be\label{nov8eqn74}
  \sum_{k\in \mathbb{Z}, (k_1,k_2)\in \chi_k^1, k_1\geq -m/10}      2^{-k_{-}/2+\gamma k_{-}+ (l-1)k_{+} }  \|    QH_{\alpha\beta;k,k_1,k_2}^{vl;L,L_1,L_2}(t,x)\|_{L^2}
\lesssim  2^{-3m/2} \epsilon_1^2 + 2^{-  m/40 } X_{L_1}(t) \epsilon_1. 
\ee

$\oplus$\qquad The estimate of $QH_{\alpha\beta;k,k_1,k_2}^{ess;L,L_1,L_2} $.

Recall \eqref{2020june18eqn1}.  We can first rule out the case   $k\leq -m +\delta_1 m $ or $k_1\geq m/10$.   Note that, the upper bounds in (\ref{2020june24eqn14}) will be applied for this case. From the estimate (\ref{2020june20eqn20}) in Lemma \ref{2020junebilinearestimate}, the following estimate holds, 
\[
\sum_{k\in \mathbb{Z}, k\leq -m +\delta_1 m } \sum_{(k_1,k_2)\in \chi_k^1}     2^{-k_{-}/2+\gamma k_{-}+ (l-1) k_{+} }   \|   QH_{\alpha\beta;k,k_1,k_2}^{ess;L,L_1,L_2}(t,x)    \|_{L^2}
\]
\[
+\sum_{k\in \mathbb{Z} } \sum_{(k_1,k_2)\in \chi_k^1, k_1\geq m/10}      2^{-k_{-}/2+\gamma k_{-}+ (l-1) k_{+} }  \|   QH_{\alpha\beta;k,k_1,k_2}^{ess;L,L_1,L_2} (t,x)    \|_{L^2}
\]
 \be\label{nov8eqn62}
 \lesssim 2^{-m   +H(|L|)\delta_1 m   +(H(1)\delta_1 + 2\gamma)m  }\epsilon_1^2.
 \ee

Now,   it would be sufficient to let $k$ and $k_1$ be fixed s.t.,    $   -m +\delta_1m \leq k\leq k_1+10\leq m/10 $. We first rule out the case $|x|$ is relatively small.   From  the estimate (\ref{2020june20eqn11}) in Lemma \ref{2020junebilinearestimate} and the estimate (\ref{oct26eqn1}), we have
\[
    2^{-k_{-}/2+\gamma k_{-}+ (l-1)k_{+} }   \|   QH_{\alpha\beta;k,k_1,k_2}^{ L;L_1,L_2} (t,x)  \psi_{\leq m-10 }(x)\|_{L^2}\lesssim  \sum_{\Gamma_1, \Gamma_2\in P_1} 2^{2k_1 +k+\gamma k_{-}+ (l-1)k_{+}}\]
\[
\times  \big[2^{-2m-2k_1} A_1(k_1, m, \Gamma_1\circ \Gamma_2 \circ  \tilde{L}_1)    + 2^{k_{1,-}/2-\gamma k_{1,-}- (l-1)k_{1,+}}\big(2^{-m-3k_1} X_{ \tilde{L}_1}(t) + 2^{-2m -4k_1}  X_{\Gamma_1 \circ\tilde{L}_1}(t) )  
\]
\[
+ 2^{-2m-3k_1+\gamma m}(1+2^{-m-k_1}) \epsilon_1 +  2^{-3m/2 +k_{1,-}-3k_1}\epsilon_1^2
   +    2^{-5m/2 +k_{1,-}-4k_1}\epsilon_1^2\big)\big]A_2(k_2, m, L_2)
   \]
\be\label{oct26eqn42}
\lesssim \sum_{\Gamma\in P_1} 2^{-m/3} (X_{\tilde{L}_1}(t) + X_{\Gamma \circ \tilde{L}_1}(t))\epsilon_1 + 2^{-3m/2 + H(|L|+2)\delta_1 m }\epsilon_1^2. 
\ee

Recall the decomposition   (\ref{nov7eqn21}). From the above estimate and the obtained estimates \eqref{nov8eqn73} and \eqref{nov8eqn74}, we have 
\[
2^{-k_{-}/2+\gamma k_{-}+ (l-1)k_{+} }   \|    QH_{\alpha\beta;k,k_1,k_2}^{ess;L,L_1,L_2}(t,x)  \psi_{\leq m-10 }(x)\|_{L^2}
\]
\be\label{2022may17eqn31}
\lesssim \sum_{\Gamma\in P_1} 2^{-m/3} (X_{\tilde{L}_1}(t) + X_{\Gamma \circ \tilde{L}_1}(t))\epsilon_1+   2^{-  m/40 }  X_{L_1}(t)  \epsilon_1 +  2^{-m  - m/40 }\epsilon_1^2. 
\ee

Now, we consider the the case $|x|$ is relatively large.  From    the estimate (\ref{2020june20eqn25}) in Lemma \ref{2020junebilinearestimate}, the following estimate holds if $|\tilde{L}_1|\geq 1$, which implies that $|L|\geq 2|\tilde{L}_1|\geq 2$, 
\[
   2^{-k_{-}/2+\gamma k_{-}+ (l-1)k_{+} }   \|    QH_{\alpha\beta;k,k_1,k_2}^{ess;L,L_1,L_2} (t,x)   \psi_{ > m-10 }(x)\|_{L^2}\lesssim 2^{-m -k_{-}/2+\min\{k,k_1\}/2+ 2k_1+\gamma k_{-} +  (l-1)k_{+}}
\]
\be\label{nov8eqn70}
\times   {C}_1(k_1, m, L_1) \widetilde{A}_2(k_2, m, L_2) \lesssim 2^{-m + H(|L|)\delta_1 m   +(H(1)\delta_1 + 2\gamma)m }\epsilon_1^2. 
\ee
From the estimate (\ref{2020june21eqn2})  in Lemma \ref{2020junebilinearestimate}, the following estimate holds if $| \tilde{L}_1|=0$,
\[
    2^{-k_{-}/2+\gamma k_{-}+  (l-1)k_{+} }   \|    QH_{\alpha\beta;k,k_1,k_2}^{ess;L,L_1,L_2} (t,x)  \psi_{>m-10}(x)\|_{L^2}\lesssim 2^{-m -k_{-}/2+\min\{k,k_1\}/2 + 2k_1+\gamma k_{-} +   (l-1) k_{+}}
\]
\be\label{nov8eqn71}
\times   {D}_1(k_1, m) \widetilde{A}_2(k_2, m, L_2) \lesssim 2^{-m + H(|L|)\delta_1 m   +\delta_1 m }\epsilon_1^2. 
\ee 
Hence, from the above estimates (\ref{nov8eqn70}) and (\ref{nov8eqn71}) and the definition of $\beta(|L|)$, we have
 \be\label{nov8eqn72}
2^{-k/2+\gamma k_{-}+  (l-1)k_{+}}  \|    QH_{\alpha\beta;k,k_1,k_2}^{ess;L,L_1,L_2} (t,x)   \psi_{>m-10}(x) \|_{L^2}\lesssim  2^{-m + H(|L|)\delta_1 m + \beta(|L|)\delta_1 m }\epsilon_1^2. 
 \ee

Recall the decomposition  (\ref{nov7eqn21}). To sum up, from the estimates (\ref{2022may17eqn31}), (\ref{nov8eqn72}),  (\ref{nov8eqn73}),  and (\ref{nov8eqn74}), we have
\[
 \sum_{k\in \Z}\sum_{(k_1,k_2)\in \chi_k^1}  2^{-k/2+\gamma k_{-}+   (l-1) k_{+}}   \|  QH_{\alpha\beta;k,k_1,k_2}^{L;L_1,L_2}(t,x)\|_{L^2}
\]
\be\label{nov8eqn77}
 \lesssim  2^{-m + H(|L|)\delta_1 m + \beta(|L|)\delta_1 m }\epsilon_1^2 +   2^{-  m/40 } X_{L_1}(t) \epsilon_1+  \sum_{\Gamma\in P_1} 2^{-m/3} (X_{\tilde{L}_1}(t) + X_{\Gamma \circ \tilde{L}_1}(t))\epsilon_1. 
\ee
Recall (\ref{oct25eqn60}) and   (\ref{2020april8eqn1}).  From the estimates  (\ref{2020aug25eqn51}) and   (\ref{nov8eqn77}), we have
\be\label{oct26eqn84} 
\sum_{  |L|\leq N_0} X_{L}(t) \lesssim  \epsilon_1 \big( \sum_{  |L|\leq N_0}  X_{L}(t)\big)  +     2^{-m   -5\delta_0 m }\epsilon_1^2, \Longrightarrow \sum_{ |L|\leq N_0} X_{L}(t) \lesssim   2^{-m    -5\delta_0 m }\epsilon_1^2 . 
\ee 
Combining the above estimate and the estimates  (\ref{2020aug25eqn51}) and   (\ref{nov8eqn77}), we have
\be\label{nov8eqn80}
X_{L}^k(t) \lesssim    2^{-m +k_{-}+ d(|L|+1, l)\delta_0 m -5\delta_0 m }\epsilon_1^2+ 2^{-m + H(|L|)\delta_1 m + \beta(|L|)\delta_1 m }\epsilon_1^2.
\ee

To sum up,  our desired estimate (\ref{july3eqn99}) holds from the above estimate, the estimates(\ref{2020aug25eqn52}) and (\ref{nov8eqn77}).  Moreover, our desired estimate (\ref{oct28eqn36}) follows from the estimates (\ref{oct26eqn1}) and the estimate (\ref{nov8eqn80}).  
\end{proof}

 \subsection{Proof of Proposition \ref{fixedtimenonlinarityestimate}}\label{proofpropfixed}

 Recall (\ref{aug11eqn151}) and  (\ref{2020april1eqn4}). The  desired estimate  (\ref{oct4eqn41})  holds from the estimate (\ref{oct1eqn31}) in Lemma \ref{cubicandhighermetric}, the estimate (\ref{july3eqn61}) in Lemma  \ref{cubicandhigherrough}, the estimate (\ref{july5eqn30}) in Lemma \ref{quadraticvlasovrough}, the estimate (\ref{2022march20eqn11}) in Lemma \ref{quadratichigherein},  and the estimate (\ref{july11eqn32}) in Lemma \ref{resonancevlasov}.

Now, we proceed to prove the desired estimate (\ref{2020july17eqn1}).  
Recall the   estimate  (\ref{2022march19eqn34}) in Lemma \ref{basicestimates}, the basic energy estimate is improved for the case $|\tilde{L}|\leq N_0-10$. The following improved estimate holds after rerunning the proof of (\ref{july3eqn99}) in  Lemma \ref{quadratichigherein}, 
\be\label{2020aug4eqn1}
X_{\tilde{L}}^k(t)\lesssim  2^{-m +H(| \tilde{L}|+1)\delta_1m  } \epsilon_1^2. 
\ee

Moreover, as $|\tilde{L}|\leq N_0-10$, the total number vector fields act on the Vlasov part is far away from $N_0$ and we are allowed to lose one derivative in the desired estimate (\ref{2020july17eqn1}). Recall the estimate (\ref{oct5eqn51}). After rerunning the proof of  the estimate (\ref{july3eqn121}) and doing integration by parts in $v$ once in the estimate (\ref{2020aug25eqn23}), for any $l\in[0, N_0-|L|]\cap\Z,$ we have
\be\label{2020aug4eqn2}
 2^{-k_{-}/2+\gamma k_{-}+ (l-1)k_{+}} \big( \| \Lambda_{2}[ P_k( \tilde{L}  \mathcal{N}_{\alpha\beta}^{vl})]\|_{L^2}  + \|P_k(\mathcal{N}_{vl;m}^{  {\tilde{h} ;\tilde{L} }} ) \|_{L^2} 
\big) \lesssim 2^{-7m/6+3d(|L|+3, l)\delta_0m} \epsilon_1^2. 
\ee
 From  the above estimates (\ref{2020aug4eqn1}--\ref{2020aug4eqn2}),  the estimate (\ref{oct1eqn31}) in Lemma \ref{cubicandhighermetric}, and the estimate (\ref{july3eqn61}) in Lemma  \ref{cubicandhigherrough},   we have
\be\label{2020july17eqn4}
2^{-k_{-}/2+\gamma k_{-}+ (l-1)k_{+}}\big(\|\p_t \widehat{ \widetilde{   V^{  \tilde{L}\tilde{h} } }}(t, \xi) \psi_k(\xi)\|_{L^2}   + \| \Lambda_{\geq 2}[ \p_t \widehat{   {   V^{ \tilde{L} h_{\alpha\beta} } }}(t, \xi) ]\psi_k(\xi)\|_{L^2} \big)\lesssim   2^{-m +H(| \tilde{L}|+1)\delta_1m  } \epsilon_1^2.  
\ee
As a result of the above  estimate (\ref{2020july17eqn4}), we have  
\[
   2^{ k_{-}/2+\gamma k_{-}+ (l-1) k_{+} }   \|\nabla_\xi   \widehat{   {  V^{ \tilde{L} {h}_{\alpha\beta} } }}(t, \xi)   \psi_k(\xi)\|_{L^2}  \lesssim \sum_{\begin{subarray}{c}
 \Gamma\in  P_1
  \end{subarray}}  2^{ -k+k_{-}/2+\gamma k_{-}+ (l-1) k_{+} } \big(\| P_k\big(\Gamma { {  e^{-it\d}  V^{ \tilde{L} {h}_{\alpha\beta} } }} \big)  \|_{L^2} \]
   \[
  +  2^m \|\p_t \widehat{ {   V^{ \tilde{L} {h}_{\alpha\beta} } }}(t, \xi) \psi_k(\xi)\|_{L^2}\big)
  \lesssim  \sum_{\Gamma\in P_1}   2^{ -k_{-}/2+\gamma k_{-}+ (l-1) k_{+} }     \| \widehat{ {   U^{\Gamma \tilde{L} {h}_{\alpha\beta} } }}(t, \xi) \psi_k(\xi)\|_{L^2}
   \]
  \be\label{2020july17eqn10}
  +  2^{-m +H(|\tilde{L}|+1)\delta_1m -5\delta_1m} \epsilon_1^2   +  2^{ -k+m+k_{-}/2+\gamma k_{-}+ (l-1) k_{+} }\|\Lambda_1[\p_t \widehat{ {   V^{ \tilde{L}{h}_{\alpha\beta} } }}] (t, \xi) \psi_k(\xi)\|_{L^2}.
\ee
By using the   decomposition of the density type function in (\ref{oct11eqn11}), the following estimate holds for the linear part of $\p_t \widehat{ {   V^{L\tilde{h} } }}$, 
 \be\label{2020july17eqn15} 
 \begin{split}
  &2^{ -k +k_{-}/2+\gamma k_{-}+ (l-1) k_{+} }\|\Lambda_1[\p_t \widehat{ {   V^{ \tilde{L} {h}_{\alpha\beta} } }}] (t, \xi) \psi_k(\xi)\|_{L^2} \\ 
  & \lesssim  \min \{2^{k} Z^{vl}(t)   +\sum_{|\mathcal{L}|+|\rho|\leq N_0} 2^{2k  }   E_{\mathcal{L}, \rho}^{vl}(t)  ,  2^{-k-2m}  Z^{vl}(t)  +\sum_{|\mathcal{L}|+|\rho|\leq N_0} 2^{-2m }   E_{\mathcal{L}, \rho}^{vl}(t)  \big\} \\ 
& \lesssim  2^{-m}  Z^{vl}(t)   +\sum_{|\mathcal{L}|+|\rho|\leq N_0} 2^{-3m/2  }  E_{\mathcal{L}, \rho}^{vl}(t) .
\end{split}
\ee
The desired estimate (\ref{2020july17eqn1}) follows from the estimates  (\ref{2020july17eqn10})  and (\ref{2020july17eqn15}).  Hence finishing the proof.

 \subsection{Estimates of the error terms in the double Hodge decomposition}\label{auxiliary}
For the commutator term $E_{\mu}$, which was defined in (\ref{sep28eqn79}), the following basic $L^\infty$-type  estimate and the $L^2$-type estimate hold . 
\begin{lemma}\label{harmonicgaugehigher}
 Under the bootstrap assumptions \textup{(\ref{BAmetricold})} and \textup{(\ref{BAvlasovori})}, the following estimate holds for any  $  L\in    P_n, n\in [0, N_0]\cap \Z, l\in [0, N_0-n]\cap \Z   $, $\tilde{L}\in  \cup_{n\leq N_0-5} P_n$, $t\in [2^{m-1}, 2^m]\subset[0,T]$, $k\in \mathbb{Z}$,
  \be\label{oct9eqn11}
 \|\tilde{L} E_\mu\|_{L^\infty}\lesssim 2^{-2m+H(|\tilde{L}|+4)\delta_1 m}\epsilon_1^2,
 \ee
\be\label{oct7eqn73} 
   2^{-k_{-}/2+\gamma k_{-} + (l-1)k_{+} }  \| P_k ( \p_{\nu} ( L E_{\mu}  ))
   (t) \|_{L^2}  \lesssim   2^{-m+ H(|L|)\delta_1 m + \beta_{}(|L|) \delta_1 m} \epsilon_1^2 + 2^{-m+k_{-}+d(|L|+1,l)\delta_0m  } \epsilon_1^2.
 \ee 

\end{lemma}
\begin{proof}
Recall (\ref{sep28eqn79}). From the harmonic gauge condition (\ref{wavecoordinatecondition}), we have
\begin{multline}\label{oct7eqn77} 
E_{\mu}= \Lambda_1[\Gamma_{\mu}] = -\Lambda_{\geq 2}[\Gamma_{\mu}] = -\Lambda_{\geq 1}[ g^{\alpha\beta}]\big(\p_{\alpha} h_{\beta\mu}   - \h \p_{\mu} h_{\alpha \beta} \big)= E_\mu^1  + E_\mu^2,\\ 
E_\mu^1= -\Lambda_{1}[ g^{\alpha\beta}]\big(\p_{\alpha} h_{\beta\mu}   - \h \p_{\mu} h_{\alpha \beta} \big), \qquad E_\mu^2= -\Lambda_{\geq 2}[ g^{\alpha\beta}]\big(\p_{\alpha} h_{\beta\mu}   - \h \p_{\mu} h_{\alpha \beta} \big).
\end{multline}
 From the above decomposition, our desired estimate (\ref{oct9eqn11}) holds from the $L^\infty-L^\infty$ type bilinear estimate, the $L^\infty$-type estimate in Lemma \ref{basicestimates}, and the estimate (\ref{june30eqn1}) in Lemma \ref{highorderterm2}.

 Note that $\p_{\nu}(L E_\mu^1)$ is a linear combination of same type of quadratic terms $L H_{\alpha\beta}, H\in \{S, P, Q\}$. Similar to the obtained estimate (\ref{july3eqn99}), by using the same argument, the following estimate holds for any $k\in \Z$,
 \[
     2^{-k_{-}/2+\gamma k_{-} + (l-1) k_{+} }  \| P_k ( \p_{\nu} ( L E_{\mu}^1  ))
   (t) \|_{L^2}
 \]
 \be\label{2020aug4eqn3}
 \lesssim2^{-m+ H(|L|)\delta_1 m + \beta_{}(|L|) \delta_1 m} \epsilon_1^2 + 2^{-m+k_{-}+  d(|L|+1,l)\delta_0m -5\delta_0m} \epsilon_1^2.
 \ee

Note that the following estimates from the $L^\infty-L^2$ type estimate and the estimate (\ref{june30eqn1}) in Lemma  \ref{highorderterm2}, 
\[
\sum_{|\alpha|\leq  l-1 } \| \nabla_x^\alpha \p_{\nu}(L E_{\mu}^2) \|_{L^2} \lesssim 2^{-3m/2 + 2  d(n +3,l)\delta_0 m } \epsilon_1^3,  
\]
\[ \sum_{|\alpha|\leq l-1 }
\|\nabla_x^\alpha   \p_{\nu}(L E_{\mu}^2) \|_{L^1} \lesssim 2^{- m/2 + 2  d(n+3,l)\delta_0 m } \epsilon_1^3.
\]
From the above estimate, we have
\[
   2^{-k_{-}/2+\gamma k_{-} + (l-1) k_{+} }  \| P_k\big( L E_{\mu}^2\big)\|_{L^2} \lesssim \sup_{k\leq -2m/3}  2^{k+\gamma k_{ }   } 2^{- m/2 + 2   d(n +3,l) \delta_0 m } \epsilon_1^3    \]
\be\label{oct7eqn75}  
 + \sup_{k\geq -2m/3 } 2^{-k_{-}/2+\gamma k_{-}}2^{-3m/2 + 2  d(n +3,l) \delta_0 m } \epsilon_1^3  \lesssim 2^{-7m/6+ 2  d(n +3,l)\delta_0 m } \epsilon_1^3.
\ee
To sum up, our desired estimate (\ref{oct7eqn73})  holds from (\ref{2020aug4eqn3}) and (\ref{oct7eqn75}). 
\end{proof}

 In the following Lemma, we show that  commutators $E_{L,0}^{comm}$ and $R_k E_{L,k}^{comm}$, $L\in P_1  $, which appears in \eqref{aug25eqn1} and \eqref{aug25eqn2},  decay sharply over time. 
\begin{lemma}\label{firstordercommutator}
Under the bootstrap assumption\textup{(\ref{BAmetricold})}, the following estimate holds for any $L\in P, |L|=1$, $k\in \mathbb{Z}, t\in [2^{m-1}, 2^m]\subset[0,T],$
\[
\sum_{|\gamma|\leq 1}\|  \nabla_{x_\alpha}^\gamma P_k(E^{comm}_{ {L}, 0})(t)\|_{L^\infty_x} + \|  \nabla_{x_\alpha}^\gamma P_k(R_i E^{comm}_{ {L}, i})(t)\|_{L^\infty_x}
\]
\be\label{april26eqn35}
 \lesssim  2^{-m+0.999k-5k_{+}} \epsilon_1 + 2^{-2m +d( 4,0)\delta_0 m -k}. 
\ee
\end{lemma}
\begin{proof}
 We  first show that the commutators $E_{L,0}^{comm}$ and $R_k E_{L,k}^{comm}$, $L\in P_1  $,  don't depend on the bad component $\underline{F}.$ Recall  (\ref{sep28eqn79}). As a result of direct computations, we have
\[
E^{comm}_{ {S}, 0}= -[\p_t, S] h_{00} + [\p_i, S] h_{0i} +\h [\p_t, S] h_{00} - \h [\p_t, S]  h_{ii}= -\h \p_t h_{00} +   \p_i h_{0i}- \h \p_t h_{ii}
\]
\[
=-\h\p_t(F+\underline{F})  +   \p_i(-R_i\rho + \epsilon_{ikl}R_k \omega_l) -\h \p_t R_iR_i(F-\underline{F}) \]
\[
- \h \p_t \big(- \big( \varepsilon_{ilm} R_i + \varepsilon_{ilm}R_i \big) R_l \Omega_m^{  } + \varepsilon_{ipm}\varepsilon_{  iqn} R_p R_q \vartheta_{mn}^{  } \big) \]
\be\label{april26eqn31}
=    \p_i \epsilon_{ikl}R_k \omega_l + \d(\rho-R_0 \underline{F})- \h \p_t \big(- \big( \varepsilon_{ilm} R_i + \varepsilon_{ilm}R_i \big) R_l \Omega_m^{  } + \varepsilon_{ipm}\varepsilon_{  iqn} R_p R_q \vartheta_{mn}^{  } \big),
\ee

\[
E^{comm}_{ L_j, 0}=-[\p_t, L_j ] h_{00} + [\p_i,  L_j ] h_{0i} +\h [\p_t,  L_j] h_{00} - \h [\p_t,  L_j]  h_{ii}= -\h \p_j h_{00} + \p_t h_{0j}   - \h \p_j h_{ii}
\]
\[
= -\h \p_j ( F+\underline{F}) + \p_t(-R_j \rho + \varepsilon_{jkl} R_k \omega_l^{  } ) - \h \p_j R_iR_i(F-\underline{F}) 
\]
\[
- \h \p_j\big( - \big( \varepsilon_{ilm} R_i + \varepsilon_{ilm}R_i \big) R_l \Omega_m^{  } + \varepsilon_{ipm}\varepsilon_{  iqn} R_p R_q \vartheta_{mn}^{  }  \big)
\]
\[
= - \p_j \underline{F} - \p_t R_j(R_0 \underline{F}) - \p_t R_j(\rho - R_0 \underline{F})+   \varepsilon_{jkl} \p_t  R_k \omega_l^{  }    
\]
\[
- \h \p_j\big( - \big( \varepsilon_{ilm} R_i + \varepsilon_{ilm}R_i \big) R_l \Omega_m^{  } + \varepsilon_{ipm}\varepsilon_{  iqn} R_p R_q \vartheta_{mn}^{  }  \big)= -\d^{-1} R_j \square \underline{F} - \p_t R_j(\rho - R_0 \underline{F})
\]
\be\label{april26eqn32}
 + \varepsilon_{jkl} \p_t  R_k \omega_l^{  } - \h \p_j\big( - \big( \varepsilon_{ilm} R_i + \varepsilon_{ilm}R_i \big) R_l \Omega_m^{  } + \varepsilon_{ipm}\varepsilon_{  iqn} R_p R_q \vartheta_{mn}^{  }  \big),
\ee
 \[
E^{comm}_{ \Omega_{ij}, 0}=-\h[\p_t, \Omega_{ij} ] h_{00} + [\p_k, \Omega_{ij} ] h_{0k}    - \h [\p_t,  \Omega_{ij}]  h_{kk}= [\p_k, \Omega_{ij} ] h_{0k} = \p_j h_{0i}-\p_i h_{0j}
 \]
\be\label{april26eqn33}
 =\p_j( h_{0i}+R_i R_0\underline{F}) - \p_i( h_{0 j}+R_j R_0\underline{F})
 \ee
 \[
 E^{comm}_{ {S}, k}= -\h [\p_{x_k}, S] h_{00} + [\p_i, S] h_{ki}  - \h [\p_{x_k}, S]  h_{ii} = -\h \p_{x_k} h_{00} + \p_{x_i} h_{ki}- \h \p_{x_k} h_{ii}
 \]
 \[
 = -\h\p_{x_k}( F+ \underline{F})-\p_{x_i} R_kR_i\underline{F} + \h \p_{x_k}R_iR_i\underline{F} +\p_{x_i}\big( h_{ki}+ R_kR_i\underline{F}\big) - \h \p_{x_k} \big( h_{ii}+ R_iR_i\underline{F}\big)
 \]
 \[
 = -\h\p_{x_k}  F+\p_{x_i}\big( h_{ki}+ R_kR_i\underline{F}\big) - \h \p_{x_k} \big( h_{ii}+ R_iR_i\underline{F}\big)
 \]
 \[
  E^{comm}_{ L_j , k}= -\h [\p_{x_k},L_j] h_{00} + [\p_i, L_j] h_{ki}  - \h [\p_{x_k},L_j]  h_{ii}= -\h \delta_{kj} \p_t h_{00} + \p_t h_{kj}- \h \delta_{kj} \p_t h_{ii}
 \]
 \[
 =  -\h \delta_{kj} \p_t ( F +\underline{F})  - \p_t R_k R_j \underline{F} + \h\delta_{kj} R_iR_i \underline{F} + \p_t( h_{kj} + R_kR_j \underline{F})- \h \delta_{kj} \p_t ( h_{ii} + R_iR_i \underline{F})
 \]
 Therefore, 
 \[
 R_kE^{comm}_{ L_j , k}=  R_k\big(  -\h \delta_{kj} \p_t   F  + \p_t( h_{kj} + R_kR_j \underline{F})- \h \delta_{kj} \p_t ( h_{ii} + R_iR_i \underline{F})\big).
 \]
 Similarly, we have 
 \[
   E^{comm}_{ \Omega_{mn} , k}= -\h [\p_{x_k},\Omega_{mn}] h_{00} + [\p_i,  \Omega_{mn} ] h_{ki}  - \h [\p_{x_k},\Omega_{mn}]  h_{ii} \]
   \[
   =-\h\big(\delta_{mk}\p_{x_n}-\delta_{nk}\p_{x_m} \big)(h_{00}-h_{ii}) + \p_{x_n}h_{km}- \p_{x_m} h_{kn} =- \big(\delta_{mk}\p_{x_n}-\delta_{nk}\p_{x_m} \big)\underline{F}
 \]
 \[
 -\h\big(\delta_{mk}\p_{x_n}-\delta_{nk}\p_{x_m} \big)(F-h_{ii}+R_iR_i \underline{F}) +  \p_{x_n}(h_{km}-R_kR_m \underline{F})- \p_{x_m} ( h_{kn} -R_kR_n \underline{F})).
 \]
 Therefore, 
 \[
 R_kE^{comm}_{ \Omega_{mn} , k}= R_k\big(-\h\big(\delta_{mk}\p_{x_n}-\delta_{nk}\p_{x_m} \big)(F-h_{ii}+R_iR_i \underline{F}) +  \p_{x_n}(h_{km}-R_kR_m \underline{F})- \p_{x_m} ( h_{kn} -R_kR_n \underline{F})) \big).
 \]
 Therefore, our desired estimate (\ref{april26eqn35}) follows from the decay estimate   of the metric componet in (\ref{2020julybasicestiamte})   in Lemma \ref{basicestimates} for the far away from the light cone case, the estimate (\ref{oct9eqn11}) in Lemma \ref{harmonicgaugehigher}, the decay estimates (\ref{july27eqn4}) and (\ref{linearwavedecay2}) in Lemma \ref{superlocalizedaug} and the estimate (\ref{oct4eqn41}) in Proposition \ref{fixedtimenonlinarityestimate}. 
\end{proof}

\subsection{An improved estimate for the modified profile at low frequency}\label{lowfreq}
In the following Lemma, we show   relation between $P_{k}\big[\widetilde{   U^{\tilde{h}^{\Gamma L} } }(t) \big]$ and $ \Gamma\big(P_{k}\big[\widetilde{    U^{\tilde{h}^{ L} }  }\big]\big)(t)$, which helps us to rule out the very low frequency case in the energy estimate of the metric part. 
  
 \begin{lemma}\label{vecmodicomm}
For any $t\in[2^{m-1}, 2^m],  $ $\Gamma\in P, L\in \cup_{n\leq N_0-1}P_n,  k\in \Z$,    the following equality holds for the modified perturbed metric defined in \textup{(\ref{modifiedperturbedmetric})},
\be\label{oct1eqn51}
P_{k}\big[\widetilde{   U^{\tilde{h}^{\Gamma L} } }(t) \big]= \Gamma\big(P_{k}\big[\widetilde{    U^{\tilde{h}^{ L} }  }\big]\big)(t) + \sum_{ L_1\preceq L } P_k R_{\Gamma;\alpha,\beta}^{\mu} \big(  {   U^{  L h_{\alpha\beta} } }(t)\big)^{\mu}+ Vlasov_{\Gamma;k}^{L;\tilde{h}} (t) + Highmetric_{\Gamma;k}^{L;\tilde{h}}(t) ,
\ee
where $R_{\Gamma;\alpha,\beta}^{\mu}$ are some zero order Fourier multiplier operator,  ``$Vlasov_{\Gamma;k}^{L;\tilde{h}}(t)$'' represents the error terms depend only with respect to  the vlasov part and $Highmetric_{\Gamma;k}^{L;\tilde{h}}(t)$ is a quadratic and higher order terms involve the perturbed meteric, see \textup{(\ref{sep24eqn1})} for their detailed formulas. 
Moreover, the following general formulation holds for  $Vlasov_{\Gamma;k}^{L;\tilde{h}} (t) $, 
\be\label{sep20eqn83}
Vlasov_{\Gamma;k}^{L;\tilde{h}} (t) = \sum_{  \mathcal{L} \preceq L\circ \Gamma } \mathcal{F}_{\xi}^{-1}\big[ \int_{\R^3} e^{-it\hat{v}\cdot \xi}\widehat{u^{\mathcal{L}}}(t, \xi, v) m_{\mathcal{L}}^{L;\tilde{h}}(t, \xi, v) \psi_k(\xi)    d v\big](x), 
\ee
where $m_{\mathcal{L}}^{L;\tilde{h}}(t, \xi, v)\in \langle v \rangle^{ 10} L^\infty_{v}\mathcal{S}^{-1}_{ }$.
\end{lemma}
\begin{proof}
Recall the definitions of profiles in (\ref{2020april1eqn4}) and (\ref{aug11eqn151}) and the double Hodge decompositions in (\ref{doublehodge1}). From  (\ref{modifiedperturbedmetric}), we have 
\be\label{oct1eqn1}
\widetilde{ \tilde{h}^L}= R_{\alpha\beta}^{\tilde{h}}(Lh_{\alpha\beta})+  \sum_{\mu\in\{+,-\}}\sum_{  \mathcal{L} \preceq L }      \big(   \mathcal{F}_{\xi}^{-1}[     \int_{\R^3} e^{  - i t\hat{v}\cdot \xi} {b}^{L\mathcal{L};\tilde{h}}_{\alpha\beta}(v) \frac{\mu k_{\alpha\beta}(\xi) \widehat{u^{\mathcal{L}}}(t, \xi, v)}{2|\xi|\big(|\xi|-\hat{v}\cdot\xi\big)} d v ]\big)^{\mu},
\ee
\be\label{oct1eqn2}
\widetilde{\p_t \tilde{h}^L}=   R_{\alpha\beta}^{\tilde{h}}(\p_tLh_{\alpha\beta})+  \sum_{\mu\in\{+,-\}}\sum_{   \mathcal{L} \preceq L }       \big(   \mathcal{F}_{\xi}^{-1}[     \int_{\R^3} e^{  - i t\hat{v}\cdot \xi} {b}^{L\mathcal{L};\tilde{h}}_{\alpha\beta}(v) \frac{  k_{\alpha\beta}(\xi) \widehat{u^{\mathcal{L}}}(t, \xi, v)}{i2\big(|\xi|-\hat{v}\cdot\xi\big)} d v ]\big)^{\mu},
\ee
where $R_{\alpha\beta}^{\tilde{h}}\in \{ R_i,   R_iR_j \}\times \C$.  Therefore, from the equality (\ref{oct1eqn1}), we have
 \be\label{sep28eqn46}
 \begin{split}
 P_{k}(\widetilde{\tilde{h}^{\Gamma L}})&= \Gamma\big(P_k(\widetilde{  \tilde{h}^L })\big)+ [P_k R_{\alpha\beta}^{\tilde{h}}, \Gamma ]( Lh_{\alpha\beta})\\
 &-   \sum_{\mu\in\{+,-\},\mathcal{L}\preceq L }      \Gamma \big(   \mathcal{F}_{\xi}^{-1}[     \int_{\R^3} e^{  - i t\hat{v}\cdot \xi} {b}^{L\mathcal{L};\tilde{h}}_{\alpha\beta}(v) \frac{\mu k_{\alpha\beta}(\xi)\psi_k(\xi) \widehat{u^{\mathcal{L}}}(t, \xi, v)}{i2\big(|\xi|-\hat{v}\cdot\xi\big)} d v ]\big)^{\mu}   \\
 &+\sum_{\mu\in\{+,-\}, \tilde{\mathcal{L}}\preceq \Gamma \circ L }      \big(   \mathcal{F}_{\xi}^{-1}[     \int_{\R^3} e^{  - i t\hat{v}\cdot \xi} {b}^{\Gamma L \tilde{\mathcal{L}};\tilde{h}}_{\alpha\beta}(v) \frac{\mu k_{\alpha\beta}(\xi)\psi_k(\xi) \widehat{u^{ \tilde{\mathcal{L}}}}(t, \xi, v)}{2|\xi|\big(|\xi|-\hat{v}\cdot\xi\big)} d v ]\big)^{\mu}. 
 \end{split}
 \ee
 Note that the following equality holds for some uniquely determined operators $R_{\alpha\beta;0}^{\tilde{h};\Gamma}\in \mathcal{S}^0$ and  $R_{\alpha\beta;-1}^{\tilde{h};\Gamma}\in \mathcal{S}^{-1}$,
 \be\label{sep28eqn56}
 [P_k R_{\alpha\beta}^{\tilde{h}}, \Gamma ] = P_k R_{\alpha\beta;0}^{\tilde{h};\Gamma} + P_k R_{\alpha\beta;-1}^{\tilde{h};\Gamma}\p_t,\quad [\p_t, \Gamma]= c_{\Gamma,\alpha}\p_{\alpha}, 
 \ee
 where 
 \be\label{oct5eqn1}
 R_{\alpha\beta;0}^{\tilde{h};\Gamma}\in \mathcal{S}^0, \quad R_{\alpha\beta;-1}^{\tilde{h};\Gamma}\in \mathcal{S}^{-1}, \quad R_{\alpha\beta;-1}^{\tilde{h};\Gamma}=0, \quad \textup{if}\quad \Gamma\notin \{L_i\}.
 \ee
Following a similar procedure, from (\ref{oct1eqn2}),  we have
 \be\label{sep28eqn41}
 \begin{split}
 P_k(\widetilde{\p_t\tilde{h}^{ \Gamma L}}) & =\Gamma P_k \big(  \widetilde{\p_t \tilde{h}^L }\big)+   P_k R_{\alpha\beta}^{\tilde{h}} [\p_t, \Gamma ](Lh_{\alpha\beta})  + [P_k R_{\alpha\beta}^{\tilde{h}} , \Gamma ](\p_t Lh_{\alpha\beta}) \\
  & -    \sum_{
\mu\in\{+,-\} , \mathcal{L}_1 \preceq L }       \Gamma \big(   \mathcal{F}_{\xi}^{-1}[     \int_{\R^3} e^{  - i t\hat{v}\cdot \xi} {b}^{L\mathcal{L}_1;\tilde{h}}_{\alpha\beta}(v) \frac{  k_{\alpha\beta}(\xi)\psi_k(\xi) \widehat{u^{\mathcal{L}_1}}(t, \xi, v)}{i2\big(|\xi|-\hat{v}\cdot\xi\big)} d v ]\big)^{\mu} \\ 
 &+   \sum_{\mu\in\{+,-\} , \mathcal{L}_2\preceq \Gamma\circ L }       \big(   \mathcal{F}_{\xi}^{-1}[     \int_{\R^3} e^{  - i t\hat{v}\cdot \xi} {b}^{L\mathcal{L}_2;\tilde{h}}_{\alpha\beta}(v) \frac{  k_{\alpha\beta}(\xi) \widehat{u^{\mathcal{L}_2}}(t, \xi, v)}{i2\big(|\xi|-\hat{v}\cdot\xi\big)} d v ]\big)^{\mu}.
 \end{split}
 \ee
We now consider the case when the vector field hits the density type term. Recall (\ref{sep28eqn41}). 
Note that
\be\label{2020april14eqn31}
\begin{split}
 \mathcal{F}_{\xi}^{-1}[     \int_{\R^3} e^{  - i t\hat{v}\cdot \xi} {b}^{L\mathcal{L};\tilde{h}}_{\alpha\beta}(v) \frac{  k_{\alpha\beta}(\xi)\psi_k(\xi) \widehat{u^{\mathcal{L}}}(t, \xi, v)}{i2\big(|\xi|-\hat{v}\cdot\xi\big)} d v ]   & =\int_{\R^3} \int_{\R^3} K^{L\mathcal{L};\tilde{h}}(x-y, v) \mathcal{L}f(t, y, v) d y d v,\\ 
K^{L\mathcal{L};\tilde{h}}(x,v)&:={b}^{L\mathcal{L};\tilde{h}}_{\alpha\beta}(v)  \int_{\R^3} e^{i x\cdot \xi}  \frac{ k_{\alpha\beta}(\xi)\psi_k(\xi) }{i2\big(|\xi|-\hat{v}\cdot\xi\big)} d \xi.
\end{split}
\ee
As a result of direct computations, we have
\[
\Omega_{ij}\big(\int_{\R^3} \int_{\R^3} K^{L\mathcal{L};\tilde{h}}( y, v) \mathcal{L}f(t, x-y, v) d y d v\big)  = \int_{\R^3} \int_{\R^3}  (y_i\p_{j} - y_j\p_{i}) K^{L\mathcal{L};\tilde{h}}( y, v) \mathcal{L}f(t, x-y, v) 
\]
 \be\label{sep28eqn42}
+  K^{L\mathcal{L};\tilde{h}}( y, v) (\Omega_{ij}\mathcal{L})f(t, x-y, v) +  \big(v_i \p_{v_j}  - v_j\p_{v_i} \big) K^{L\mathcal{L};\tilde{h}}( y, v)   \mathcal{L} f(t, x-y, v)\big)      d y d v.
\ee
 \[
 S\big(\int_{\R^3} \int_{\R^3} K^{L\mathcal{L};\tilde{h}}( y, v) \mathcal{L}f(t, x-y, v) d y d v\big) 
 \]
 \be\label{sep28eqn43}
 =  \int_{\R^3} \int_{\R^3} K^{L\mathcal{L};\tilde{h}}( y, v) S\mathcal{L}f(t, x-y, v)  + y\cdot \nabla_y K^{L\mathcal{L};\tilde{h}}( y, v) \mathcal{L}f(t, x-y, v)d y d v . 
\ee
\[
L_i \big(\int_{\R^3} \int_{\R^3} K^{L\mathcal{L};\tilde{h}}( y, v) \mathcal{L}f(t, x-y, v) d y d v\big) 
\]
\[
=\int_{\R^3} \int_{\R^3} K^{L\mathcal{L};\tilde{h}}( y, v) \mathcal{L}_i\mathcal{L}f(t, x-y, v)  
+ \p_{v_i}\big(\sqrt{1+|v|^2} K^{L\mathcal{L};\tilde{h}}( y, v)\big)     \mathcal{L}f(t, x-y, v)  
\]
 \be\label{sep28eqn44} 
 -  y_i K^{L\mathcal{L};\tilde{h}}( y, v) \hat{v}\cdot\nabla_x \mathcal{L}f(t, x-y, v)   +   y_i K^{L\mathcal{L};\tilde{h}}( y, v)  (\p_t + \hat{v}\cdot \nabla_x) \mathcal{L}f(t, x-y, v) dy d v.  
\ee
To sum up, from the equalities (\ref{sep28eqn42}), (\ref{sep28eqn43}), and (\ref{sep28eqn44}), we have
\[
\Gamma\mathcal{F}_{\xi}^{-1}[     \int_{\R^3} e^{  - i t\hat{v}\cdot \xi} {b}^{L\mathcal{L};\tilde{h}}_{\alpha\beta}(v) \frac{  k_{\alpha\beta}(\xi)\psi_k(\xi) \widehat{u^{\mathcal{L}}}(t, \xi, v)}{i2\big(|\xi|-\hat{v}\cdot\xi\big)} d v ] =\sum_{\tilde{\Gamma}\in \mathcal{P}\cup \{Id\} }\int_{\R^3} \int_{\R^3}  K^{L\mathcal{L};\tilde{h}}_{\Gamma \tilde{\Gamma}}(y, v) 
\]
 \be\label{sep28eqn47} 
\times \tilde{\Gamma} \mathcal{L}f(t, x-y, v) d y d v +  C_{\Gamma}^{L_i}\big(\int_{\R^3} \int_{\R^3} y_i K^{L\mathcal{L};\tilde{h}}( y, v)  (\p_t + \hat{v}\cdot \nabla_x) \mathcal{L}f(t, x-y, v) d y d v\big), 
\ee
where $\mathcal{F}_{x}[K^{L\mathcal{L};\tilde{h}}_{\Gamma \tilde{\Gamma}}](\xi, v) \in L^{\infty}_{\langle v\rangle^{-10}}\mathcal{S}^{-1} $ and $ C_{\Gamma}^{L_i}\in \R$, $i\in\{1,2,3\}$, are some uniquely determined constants. Similarly, we have
\[
\Gamma\mathcal{F}_{\xi}^{-1}[     \int_{\R^3} e^{  - i t\hat{v}\cdot \xi} {b}^{L\mathcal{L};\tilde{h}}_{\alpha\beta}(v) \frac{  k_{\alpha\beta}(\xi)\psi_k(\xi) \widehat{u^{\mathcal{L}}}(t, \xi, v)}{i2|\xi|\big(|\xi|-\hat{v}\cdot\xi\big)} d v ] =\sum_{\tilde{\Gamma}\in \mathcal{P}\cup \{Id\} }\int_{\R^3} \int_{\R^3}  \tilde{K}^{L\mathcal{L};\tilde{h}}_{\Gamma \tilde{\Gamma}}(y, v)  \tilde{\Gamma} \mathcal{L}f(t, x-y, v)
\]
\be \label{2020june23eqn11}
 +   C_{\Gamma}^{L_i} y_i \tilde{K}^{L\mathcal{L};\tilde{h}}( y, v)  (\p_t + \hat{v}\cdot \nabla_x) \mathcal{L}f(t, x-y, v) d y d v, 
\ee
\[
 \tilde{K}^{L\mathcal{L};\tilde{h}}(x,v):={b}^{L\mathcal{L};\tilde{h}}_{\alpha\beta}(v)  \int_{\R^3} e^{i x\cdot \xi}  \frac{ k_{\alpha\beta}(\xi)\psi_k(\xi) }{i2|\xi|\big(|\xi|-\hat{v}\cdot\xi\big)} d \xi,
\]
where $\mathcal{F}_{x}[K^{L\mathcal{L};\tilde{h}}_{\Gamma \tilde{\Gamma}}](\xi, v) \in L^{\infty}_{\langle v\rangle^{-10}}\mathcal{S}^{-2} $. Therefore, from (\ref{sep28eqn46}), (\ref{sep28eqn56}) and   (\ref{2020june23eqn11}), we have
 \be\label{sep28eqn50} 
 P_{k}(\widetilde{\tilde{h}^{\Gamma L}})= \Gamma\big(P_k(\widetilde{ \tilde{h}^L})\big)+  P_k R_{\alpha\beta;0}^{\tilde{h};\Gamma} Lh_{\alpha\beta} + P_k R_{\alpha\beta;-1}^{\tilde{h};\Gamma}\p_t Lh_{\alpha\beta}   +    vl_{\Gamma;k}^{L;\tilde{h}} (t) + hm_{\Gamma;k}^{L;\tilde{h}}(t),
\ee
where
\[
 vl_{\Gamma;k}^{L;\tilde{h}} (t)=  \sum_{\mu\in\{+,-\}}\sum_{  \mathcal{L}_1 \preceq \Gamma\circ L }      \big(   \mathcal{F}_{\xi}^{-1}[     \int_{\R^3} e^{  - i t\hat{v}\cdot \xi} {b}^{\Gamma L\mathcal{L}_1;\tilde{h}}_{\alpha\beta}(v) \frac{\mu k_{\alpha\beta}(\xi)\psi_k(\xi) \widehat{u^{\mathcal{L}_1}}(t, \xi, v)}{2|\xi|\big(|\xi|-\hat{v}\cdot\xi\big)} d v ]\big)^{\mu} \]
 \be\label{sep28eqn51} 
  -   \sum_{\mu\in\{+,-\}}\sum_{ \mathcal{L}_2 \preceq L } \sum_{\tilde{\Gamma}\in \mathcal{P}\cup \{Id\} }\big(\int_{\R^3} \int_{\R^3}  \tilde{K}^{L\mathcal{L}_2;\tilde{h}}_{\Gamma \tilde{\Gamma}}(y, v) 
 \tilde{\Gamma} \mathcal{L}_2f(t, x-y, v) d y d v\big)^{\mu},
\ee
 \be\label{sep28eqn52} 
hm_{\Gamma;k}^{L;\tilde{h}}(t)= -   \sum_{\mu\in\{+,-\}}\sum_{  \mathcal{L} \preceq L } \sum_{\tilde{\Gamma}\in \mathcal{P}\cup \{Id\} }C_{\Gamma}^{L_i}\big(\int_{\R^3} \int_{\R^3} y_i \tilde{K}^{L\mathcal{L};\tilde{h}}( y, v)  (\p_t + \hat{v}\cdot \nabla_x) \mathcal{L}f(t, x-y, v) d y d v\big)^{\mu}.
\ee
Similarly, from (\ref{sep28eqn41}),   (\ref{sep28eqn56}), and (\ref{sep28eqn47}) and the nonlinear wave equation satisfied by $Lh_{\alpha\beta}$ in (\ref{sep28eqn26}), we have
 \be\label{sep28eqn53}
 \begin{split} 
 P_k(\widetilde{\p_t \tilde{h}^{\Gamma L}})& =  \Gamma P_k \big(  \widetilde{\p_t\tilde{h}^L}\big) +  P_k R_{\alpha\beta}^{\tilde{h}} c_{\Gamma,\alpha}\p_{\alpha} Lh_{\alpha\beta}  +P_k R_{\alpha\beta;0}^{\tilde{h};\Gamma}\p_t Lh_{\alpha\beta}\\
  & + P_k R_{\alpha\beta;-1}^{\tilde{h};\Gamma}\Delta Lh_{\alpha\beta} +   ptvl_{\Gamma;k}^{L;\tilde{h}} (t) + pthm_{\Gamma;k}^{L;\tilde{h}}(t),
  \end{split}
\ee
where
\[
 ptvl_{\Gamma;k}^{L;\tilde{h}} (t)= \sum_{\mu\in\{+,-\}}\sum_{   \tilde{\Gamma}\in \mathcal{P}\cup \{Id\},|\mathcal{L}_1\preceq \Gamma \circ L, \mathcal{L}_2\preceq L }       \big(   \mathcal{F}_{\xi}^{-1}[     \int_{\R^3} e^{  - i t\hat{v}\cdot \xi} {b}^{L\mathcal{L}_1;\tilde{h}}_{\alpha\beta}(v) \frac{  k_{\alpha\beta}(\xi) \widehat{u^{\mathcal{L}_1}}(t, \xi, v)}{i2\big(|\xi|-\hat{v}\cdot\xi\big)} d v ]\big)^{\mu}
\]
 
\be\label{sep28eqn71} 
  -   \big(\int_{\R^3} \int_{\R^3}  {K}^{L\mathcal{L}_2;\tilde{h}}_{\Gamma \tilde{\Gamma}}(y, v) 
 \tilde{\Gamma} \mathcal{L}_2f(t, x-y, v) d y d v\big)^{\mu}  + P_k R_{\alpha\beta;-1}^{\tilde{h};\Gamma}   \big(\sum_{  \mathcal{L} \preceq L  }  \int_{\R^3} {b}^{L\mathcal{L} }_{\alpha\beta}(v) \mathcal{L} f   (t,x, v) d v\big),  
 \ee
 \[
  pthm_{\Gamma;k}^{L;\tilde{h}}(t)=   -   \sum_{\begin{subarray}{c}
    \mu\in\{+,-\},  \mathcal{L} \preceq L\\ 
     \tilde{\Gamma}\in \mathcal{P}\cup \{Id\} \\
  \end{subarray}}C_{\Gamma}^{L_i}\big(\int_{\R^3} \int_{\R^3}    y_i  {K}^{L\mathcal{L};\tilde{h}}( y, v)   \big((\p_t + \hat{v}\cdot \nabla_x) \mathcal{L}f(t, x-y, v) \big) d y d v\big)^{\mu}  \]
\be\label{sep28eqn72} 
   + P_k R_{\alpha\beta;-1}^{\tilde{h};\Gamma}  \Lambda_{\geq 2}[\square L h_{\alpha\beta}]  .
 \ee
 Hence, to sum up, from (\ref{modifiedperturbedmetric}),  (\ref{sep28eqn50})  and (\ref{sep28eqn53})  ,  we know that   the following equalities hold, 
\be\label{sep24eqn1}
\begin{split}
Vlasov_{\Gamma;k}^{L;\tilde{h}} (t) & = ptvl_{\Gamma;k}^{L;\tilde{h}} (t) + i \d \big( vl_{\Gamma;k}^{L;\tilde{h}} (t)\big), \\ 
 Highmetric_{\Gamma;k}^{L;\tilde{h}}(t) & = pthm_{\Gamma;k}^{L;\tilde{h}}(t) +  i \d \big(hm_{\Gamma;k}^{L;\tilde{h}}(t)\big).
\end{split}
\ee
Recall the detailed formulas of $ vl_{\Gamma;k}^{L;\tilde{h}} (t)$ in (\ref{sep28eqn51}) and $  ptvl_{\Gamma;k}^{L;\tilde{h}} (t g)$ in (\ref{sep28eqn71}). 
We know that the representation formula for $Vlasov_{\Gamma;k}^{L;\tilde{h}} (t)$ in (\ref{sep20eqn83}) holds  for some uniquely determined symbol $m_{\mathcal{L}}^{L;\tilde{h}}(t, \xi, v)$. 
\end{proof} 
With the decomposition in previous Lemma, now, we show that the error terms appeared in (\ref{oct1eqn51}) is indeed small and the estimate of perturbed metric is improved at very low frequency,  e.g., of size almost like $\langle t \rangle^{-1}.$

 \begin{lemma}\label{lowfrequencyinputestimate}
Under the bootstrap assumptions \textup{(\ref{BAmetricold})} and \textup{(\ref{BAvlasovori})}, the following estimate holds for any    $L \in \cup_{l\leq N_0 }P_l$, $|L|\geq 1$,  $t\in[2^{m-1}, 2^m]\subset[0, T]$, $k\in \Z$, s.t., $k\leq -2H(N_0) \delta_0 m$,  and zero order Fourier multiplier operator $M$ with symbol in $ \mathcal{S}^0$, 
\[
2^{-k/2+\gamma k}\big(\| M P_k(  U^{Lh_{\alpha\beta}})(t,x)\psi_{\leq m+2}(|x|) \|_{L^2} +\| M P_k( \widetilde{U^{Lh_{\alpha\beta}}}  )(t,x)  \psi_{\leq m+2}(|x|) \|_{L^2}\big)   \]
\be\label{nov9eqn20}
\lesssim (1+2^{-k-m}) 2^{   {H}(|L|-1)\delta_1 m + \beta_{}(|L|) \delta_1 m} \epsilon_1    + 2^{  {H}(|L|-1)\delta_1 m + m + k } \epsilon_1.
\ee
\end{lemma}
\begin{proof}
Recall (\ref{modifiedperturbedmetric}). From   the decomposition of density type function in (\ref{oct11eqn11}), we have 
\be\label{nov9eqn22}
 2^{-k/2+\gamma k}\big(\|P_k  \tilde{\rho}^{h_{\alpha\beta}}_{L}(f)(t) \|_{L^2}+\|P_k  \tilde{\rho}^{\tilde{h}}_{L}(f)(t) \|_{L^2}\big) \lesssim \epsilon_1 + 2^{k+m}\epsilon_1. 
\ee

 Since   $L \in \cup_{l\leq N_0 }P_l$, $|L|\geq 1$, we know that $L=\Gamma \circ \tilde{L}$ holds for some $\Gamma\in P_1 , \tilde{L}\in P_{|L|-1}$.  
From the equality (\ref{oct1eqn51}) in Lemma \ref{vecmodicomm}, we have
 \be\label{2020april13eqn23}
 \begin{split}
  & 2^{-k/2+\gamma k} \| M P_k( \widetilde{U^{Lh_{\alpha\beta}}}  )(t,x)  \psi_{\leq m+2}(|x|) \|_{L^2}\\ 
  &\lesssim  2^{-k/2+\gamma k}\big[ \|\Gamma M P_k( \widetilde{U^{\tilde{L}h_{\alpha\beta}}}) \psi_{\leq m+2}(|x|) \|_{L^2} +\|[\Gamma, M] P_k( \widetilde{U^{\tilde{L}h_{\alpha\beta}}})   \|_{L^2} \\
  &+ 2^{-k/2+\gamma k} \big(  \sum_{   L_1 \prec   L }  \|   P_k(  {U^{ {L}_1 h_{\alpha\beta}}}) \|_{L^2} + \|Vlasov_{\Gamma;k}^{\tilde{L};\tilde{h}} (t)\|_{L^2} + \|Highmetric_{\Gamma;k}^{\tilde{L};\tilde{h}}(t) \|_{L^2}\big)\big].
  \end{split}
\ee
 
From     the estimate  (\ref{oct4eqn41}) in Proposition \ref{fixedtimenonlinarityestimate} and the fact that $|\tilde{L}|\leq N_0-1$, with minor modification in the obtained estimate (\ref{2020june24eqn40}), we have
\[
 2^{-k/2+\gamma k} \big(\|[\Gamma, M] P_k( \widetilde{U^{\tilde{L}h_{\alpha\beta}}})   \|_{L^2} + \sum_{ L_1 \prec  L   }   \|   P_k(  {U^{ {L}_1 h_{\alpha\beta}}}) \|_{L^2} \big) \lesssim (1+ 2^{-m-k})2^{ {H}(|L|-1)\delta_1 m + \beta_{}(|\tilde{L}|) \delta_1 m} \epsilon_1    
\]
\be\label{2020april13eqn21}
 + 2^{-m + {d} (|L|+1,0)\delta_0m -5\delta_0m} \epsilon_1   \lesssim  (1+ 2^{-m-k})2^{ {H}(|L|-1)\delta_1 m + \beta_{}(|\tilde{L}|) \delta_1 m} \epsilon_1 . 
\ee

Recall the detailed formulation of $ Highmetric_{\Gamma;k}^{L;\tilde{h}}(t)$  in (\ref{sep24eqn1}).   We first estimate the imaginary part $hm_{\Gamma;k}^{L;\tilde{h}}(t)$.  Recall (\ref{sep28eqn52}),   the equations (\ref{2020april2eqn1}) and (\ref{2020april7eqn6})  and the decomposition in  (\ref{2020april14eqn23}). After doing integration by parts in $v$ once, we have
\[
\int_{\R^3} \int_{\R^3} y_i \tilde{K}^{\tilde{L}\mathcal{L};\tilde{h}}( y, v)  (\p_t + \hat{v}\cdot \nabla_x) \mathcal{L}f(t, x-y, v) d y d v
\]
\[
 =\sum_{i=3,4} \int_{\R^3} \int_{\R^3} y_i \tilde{K}^{\tilde{L}\mathcal{L};\tilde{h}}( y, v) \big(  \Lambda_{\geq 3}[\mathfrak{N}_i^{\mathcal{L}}] (t,x-y-t\hat{v},v) +\mathfrak{Q}_3^{\mathcal{L}}(t,x-y-t\hat{v},v) + \nabla_x\cdot   \mathfrak{N}_1^{\mathcal{L}}(t, x-y-t\hat{v},v) 
\]
\be\label{oct4eqn82}
 + \nabla_x\cdot  \mathfrak{Q}_1^{\mathcal{L
 }}  (t, x-y-t\hat{v},v)  \big)  -   \nabla_v\big(y_i \tilde{K}^{\tilde{L}\mathcal{L};\tilde{h}}( y, v)\big)\cdot    \big(    \mathfrak{N}_2^{\mathcal{L
 }}(t, x-y-t\hat{v},v)+ \mathfrak{Q}_2^{\mathcal{L
 }}(t, x-y-t\hat{v},v)\big) d y d v.
\ee
 
Recall (\ref{2020april14eqn31}). With the above equality, from  the volume of support of $\xi$ if $k\leq 0$, 
   the estimate (\ref{aug9eqn87}) in Lemma \ref{estimateofremaindervlasov}, the estimate  (\ref{2020april14eqn3})  in Lemma \ref{estimateofremaindervlasov3}, the estimate (\ref{2020april14eqn21}), and the Cauchy-Schwarz inequality, we have
\[
2^{- k_{-}/2+\gamma k_{-} } \|\d(hm_{\Gamma;k}^{\tilde{L};\tilde{h}}(t))\|_{L^2} \lesssim 2^{- k_{-}/2+\gamma k_{-}+k+3k_{-}/2} 2^{-3k+  d(|L|+3,0)\delta_0m}  \big(    2^{ -m+k}\epsilon_1^2   +    2^{ -3m/2    }\epsilon_1^2 
\]
\be\label{oct4eqn83}
+  2^{-7m/6+2d(|L|+3,0)\delta_0m}\epsilon_1^3 \big)\lesssim  2^{-m +d(|L|+3,0)\delta_0 m } \epsilon_1^2 +  2^{-k- 7m/6+3d(|L|+3,0)\delta_0m}\epsilon_1^2. 
\ee

Now we estimate the real part $pthm_{\Gamma;k}^{\tilde{L};\tilde{h}}(t)$. Recall (\ref{sep28eqn72}). With minor modification in  the estimate of  $ hm_{\Gamma;k}^{\tilde{L};\tilde{h}}(t)$, the following estimate holds from the estimate  (\ref{oct4eqn41}) in Proposition \ref{fixedtimenonlinarityestimate}, 
\be\label{oct4eqn84}
2^{- k_{-}/2+\gamma k_{-}  } \| pthm_{\Gamma;k}^{\tilde{L};\tilde{h}}(t)  \|_{L^2} \lesssim 2^{-k}    2^{-m+    {H}(|L|-1)\delta_1 m + \beta_{}(|L|) \delta_1 m} \epsilon_1^2 +  2^{- m/6+3d(|L|+3,0)\delta_0m}\epsilon_1^2. 
\ee
Combining the estimates (\ref{oct4eqn83}) and  (\ref{oct4eqn84}), we  have 
\be\label{oct5eqn11}
  2^{- k_{-}/2+\gamma k_{-}   }   \|Highmetric_{\Gamma;k}^{\tilde{L};\tilde{h}}(t) \|_{L^2} \lesssim     2^{- m/6+3d(|L|+3,0)\delta_0m}\epsilon_1^2  + 2^{-k -m+    {H}(|L|-1)\delta_1 m + \beta_{}(|L|) \delta_1 m} \epsilon_1^2. 
\ee
Recall the general formula of  $Vlasov_{\Gamma;k}^{\tilde{L};\tilde{h}} (t)$ in (\ref{sep20eqn83}). By using the same strategy used in   obtaining the  equality (\ref{oct11eqn11}) and the estimate (\ref{nov9eqn22}),   we have
\be\label{2020june22eqn11}
2^{-k/2+\gamma k}  \|Vlasov_{\Gamma;k}^{\tilde{L};\tilde{h}} (t)\|_{L^2}  \lesssim \epsilon_1 + 2^{k+m}\epsilon_1. 
\ee
From the estimate (\ref{oct4eqn41}) in  Proposition \ref{fixedtimenonlinarityestimate}, we have
\be\label{2020april13eqn31}
2^{-k/2+\gamma k}  \|\Gamma P_k( \widetilde{U^{\tilde{L}h_{\alpha\beta}}}) \psi_{\leq m+2}(|x|) \|_{L^2} \lesssim   2^{   {H}(|L|-1)\delta_1 m + \beta_{}(|L|) \delta_1 m} \epsilon_1  + 2^{  {H}(|L|-1)\delta_1 m + m + k } \epsilon_1.  
\ee 
To sum up,   our desired estimate (\ref{nov9eqn20}) holds from the estimates (\ref{nov9eqn22}--\ref{2020april13eqn21}) and (\ref{oct5eqn11}--\ref{2020april13eqn31}).
\end{proof}

\section{Energy estimate of the perturbed metric part}\label{energyestimatewave}

Let  $\tilde{h}\in \{  F,  \underline{F},  \omega_j,  \vartheta_{mn} \}$,  $ L \in   P_n, n\leq N_0 $. To get around the losing derivatives issue, we utilize the symmetric structure of the Einstein equation and  define a modified  energy for the modified perturbed metric as follows, 
\[
E_{modi }^{ \tilde{h};L}(t):=\sup_{k\in \mathbb{Z}} 2^{-k_{-}+2\gamma k_{-}+2  (N_0-|L|)k_{+}}  E_{modi;k}^{ \tilde{h};L}(t), 
\]
where
 \be\label{nov22eqn6}
   E_{modi;k}^{ \tilde{h};L}(t):=  \h \int_{\R^3} \big| P_k({ \widetilde{   V^{ \tilde{h}^L } }})(t  ) \big|^2 +\frac{1}{2} \p_{x_i} P_k ( \widetilde{\tilde{h}^L} )  H_{ij}\p_{x_j} P_k ( \widetilde{\tilde{h}^L}   )d x.
\ee
The first part of  $ E_{modi}^{L\tilde{h};1}(t)$ denotes the energy of the modified profile $\widetilde{   V^{ \tilde{h}^L } }(t)$,  and the second part are   cubic correction terms which aim  to get around the losing derivatives issue  caused from the quasilinear nature of the Einstein equation.  

From (\ref{2020june30eqn1})   and the definition of the modified perturbed metric in (\ref{modifiedperturbedmetric}), we have
\[
\p_t E_{modi;k}^{ \tilde{h};L}(t) = \textup{Re}\big[\int_{\R^3}\overline{ P_k(\widetilde{   V^{ \tilde{h}^L } }  ) } \p_t  P_k(\widetilde{   V^{ \tilde{h}^L } }) d x  \big]
\]
\[
 + \int_{\R^3}  \p_{x_i} P_k ( \p_t \widetilde{\tilde{h}^L} )  H_{ij}\p_{x_j} P_k ( \widetilde{\tilde{h}^L}   )+\frac{1}{2}  \p_{x_i} P_k (  \widetilde{\tilde{h}^L} ) \p_t H_{ij}\p_{x_j} P_k ( \widetilde{\tilde{h}^L}   ) d x 
\]
\[
=  \textup{Re}\big[\int_{\R^3}\overline{ P_k(\widetilde{   U^{ \tilde{h}^L } }  ) }       \big(   H_{i0} \p_{i}P_k \widetilde{\p_{t}\tilde{h}^{L}} +    H_{ij} \p_{i} {\p_{j}P_k\widetilde{\tilde{h}^{L}}}\big)   d x  \big] + \int_{\R^3}   \p_{x_i} P_k ( \widetilde{ \p_t \tilde{h}^L} )  H_{ij}\p_{x_j} P_k ( \widetilde{\tilde{h}^L}   ) d x 
\]
\[
+  \textup{Re}\big[ \int_{\R^3} \overline{ P_k(\widetilde{   U^{ \tilde{h}^L } }  ) }   \big(e^{-it\d}\p_t  P_k \widetilde{   V^{ \tilde{h}^L }} -  \big(   H_{i0} \p_{i}P_k \widetilde{\p_{t}\tilde{h}^{L}} +    H_{ij} \p_{i} {\p_{j}P_k\widetilde{\tilde{h}^{L}}}\big)\big)  d x\big]
\]
\be\label{2020june30eqn31}
  +  \int_{\R^3} \frac{1}{2}  \p_{x_i} P_k (  \widetilde{\tilde{h}^L} ) \p_t H_{ij}\p_{x_j} P_k ( \widetilde{\tilde{h}^L}   )     d x  +\int_{\R^3}    \p_{x_i} P_k ( \p_t \widetilde{\tilde{h}^L} -  \widetilde{\p_t \tilde{h}^L})  H_{ij}\p_{x_j} P_k ( \widetilde{\tilde{h}^L}  ) d x. 
\ee

Recall the definition of the modified half-wave in \eqref{2022march19eqn12}. The real part of $\widetilde{ U^{ \tilde{h}^L }}   $ is $ \widetilde{\p_{t}\tilde{h}^{L}}$. Note that, thanks to the symmetric structure, after doing integration by parts in $x$, we have 
\begin{multline}\label{2020june30eqn39}
   \textup{Re}\big[\int_{\R^3}\overline{ P_k(\widetilde{   U^{ \tilde{h}^L } }  ) }       \big(   H_{i0} \p_{i}P_k \widetilde{\p_{t}\tilde{h}^{L}} +    H_{ij} \p_{i} {\p_{j}P_k\widetilde{\tilde{h}^{L}}}\big)   d x  \big] + \int_{\R^3}   \p_{x_i} P_k ( \widetilde{ \p_t \tilde{h}^L} )  H_{ij}\p_{x_j} P_k ( \widetilde{\tilde{h}^L}   ) d x 
  \\ 
= -    \int_{\R^3} { P_k(\widetilde{   \p_t { \tilde{h}^L } }  ) }   \p_i  H_{i0 } P_k \big( \widetilde{\p_t \tilde{h}^{L}} \big) -   { P_k(\widetilde{   \p_t { \tilde{h}^L } }  ) }     \p_i H_{ij } P_k \big( \p_j \widetilde{ \tilde{h}^{L}} \big)   d x.
\end{multline}
 To sum up, from equalities (\ref{2020june30eqn31}) and (\ref{2020june30eqn39}), we have
\be\label{2020july1eqn1}
\begin{split}
\p_t E_{modi;k}^{ \tilde{h};L}(t)& =   \textup{Re}\big[ \int_{\R^3} \overline{ P_k(\widetilde{   U^{ \tilde{h}^L } }  ) }   \big(e^{-it\d}\p_t  P_k \widetilde{   V^{ \tilde{h}^L }} - \big(   H_{i0} \p_{i}P_k \widetilde{\p_{t}\tilde{h}^{L}} +    H_{ij} \p_{i} {\p_{j}P_k\widetilde{\tilde{h}^{L}}}\big)\big)  d x\big]\\
& -   \int_{\R^3} { P_k(\widetilde{   \p_t { \tilde{h}^L } }  ) }   \p_i  H_{i0 } P_k \big( \widetilde{\p_t \tilde{h}^{L}} \big)   d x-   \int_{\R^3} { P_k(\widetilde{   \p_t { \tilde{h}^L } }  ) }     \p_i H_{ij } P_k \big( \p_j \widetilde{ \tilde{h}^{L}} \big)   d x \\
& +  \int_{\R^3} \frac{1}{2}  \p_{x_i} P_k (  \widetilde{\tilde{h}^L} ) \p_t H_{ij}\p_{x_j} P_k ( \widetilde{\tilde{h}^L}   )     d x+   \int_{\R^3}    \p_{x_i} P_k ( \p_t \widetilde{\tilde{h}^L} -  \widetilde{\p_t \tilde{h}^L})  H_{ij}\p_{x_j} P_k ( \widetilde{\tilde{h}^L}  ) d x.  
\end{split}
\ee

Recall (\ref{nov23eqn2}) and (\ref{2020june30eqn1}). We know that $\Lambda_{1}[H_{0i}] = -2h_{0i}$ and $\Lambda_{1}[H_{ij}] = h_{00}\delta_{ij} + h_{ij} $.  After using the   double Hodge decompositions in (\ref{doublehodge1}), we have
\be\label{nov22eqn13}
\begin{split}
 & -   \int_{\R^3} { P_k(\widetilde{   \p_t { \tilde{h}^L } }  ) }   \p_i \Lambda_{1}[ H_{i0 }] P_k \big( \widetilde{\p_t \tilde{h}^{L}} \big)   d x-   \int_{\R^3} { P_k(\widetilde{   \p_t { \tilde{h}^L } }  ) }     \p_i \Lambda_{1}[ H_{ij }] P_k \big( \p_j \widetilde{ \tilde{h}^{L}} \big)   d x\\
 & +  \int_{\R^3} \frac{1}{2}  \p_{x_i} P_k (  \widetilde{\tilde{h}^L} ) \p_t \Lambda_{1}[ H_{ij}] \p_{x_j} P_k ( \widetilde{\tilde{h}^L}   )     d x +   \int_{\R^3} \overline{ P_k(\widetilde{   U^{ \tilde{h}^L } }  ) }\big[ P_k\big(  \Lambda_{1}[H_{i0}]\p_{i} \widetilde{\p_{t}\tilde{h}^{L}}\\
 &  +   \Lambda_{1}[H_{ij}]\p_{i} {\p_{j}\widetilde{\tilde{h}^{L}}}\big)- \big(  \Lambda_{1}[H_{i0}]\p_{i}P_k \widetilde{\p_{t}\tilde{h}^{L}} +   \Lambda_{1}[H_{ij}]\p_{i} {\p_{j}P_k\widetilde{\tilde{h}^{L}}}\big)\big]  d x \\
& = \sum_{h'\in \{F, \underline{F}, \omega_{j}, \vartheta_{mn}\}} \sum_{\mu, \nu \in\{+,-\}} \textup{Re}\big[ \int_{\R^3}    \widetilde{U^{ \tilde{h}^L }_k }(t,x) \tilde{Q}^{\mu\nu}_{ h' \tilde{h}}( (U^{h'})^{\mu}(t), ( \widetilde{ U^{\tilde{h}^L }_k } )^{\nu}(t))(x) d x\big] + \mathfrak{R}_k ,
 \end{split}
\ee
where, very importantly, as a result of direct computations,   the bilinear operators $\tilde{Q}^{\mu\nu}_{h'\tilde{h}}(\cdot, \cdot)\in \mathcal{M}_{\mu\nu}^{null}, \mu, \nu\in\{+,-\} ,$ see \eqref{nullmutipliers},  have null structure. 

 The quartic and higher order term  $\mathfrak{R}_k$  reads as follows,
\[
\mathfrak{R}_k = \int_{\R^3} -\big(    P_{k}( \widetilde{   \p_t { \tilde{h}^L } }  )     )\big)^2 \p_i R_i\big(\rho-R_0 \underline{F} \big) -   \h \p_{x_i}    P_{k}( \widetilde{ \tilde{h}^{L}} )     \p_{x_j}  P_{k}( \widetilde{ \tilde{h}^{L}}  )\big[\p_t\big(   \big( \varepsilon_{ilm} R_j + \varepsilon_{jlm}R_i \big) R_l (\Omega_m^{ }-R_0\omega_m) \big)
\]
\be\label{aug8eqn41}
 +       \big( \varepsilon_{ilm} R_j + \varepsilon_{jlm}R_i \big) R_l \d^{-1}\square \omega_m \big]   +     P_k(  \widetilde{   \p_t { \tilde{h}^L } } )    \p_{j}  P_{k}(   \widetilde{ \tilde{h}^{L}}  ) \p_i  \big(   \big( \varepsilon_{ilm} R_j + \varepsilon_{jlm}R_i \big) R_l (\Omega_m^{ }-R_0\omega_m) \big) d x. 
\ee

Now, we study the first line in (\ref{nov22eqn13}). Recall (\ref{2020june30eqn1}).  From  (\ref{2020april7eqn6}), and  (\ref{2020april14eqn23}), we classify the   nonlinearity $ \mathcal{N}_{vl;m}^{  {\tilde{h} ;L }}$ of $e^{-it\d}\p_t  P_k \widetilde{   V^{ \tilde{h}^L }}$ based on the order of nonlinearity as follows,
\[
 \textup{Re}\big[ \int_{\R^3} \overline{ P_k(\widetilde{   U^{ \tilde{h}^L } }  ) }  \mathcal{N}_{vl;m}^{  {\tilde{h} ;L }}(t,x) d x \big] =     -\sum_{\mathcal{L}\in \mathcal{P}_{a}^{b },|\mathcal{L}|\leq |L|   }    \textup{Re}\big[\int_{\R^3} \int_{\R^3} \overline{  \widehat{\widetilde{   U^{ \tilde{h}^L } } }  (t, \xi)}
\]
\[
 \times (\psi_k(\xi))^2         e^{ - i t\hat{v}\cdot \xi} {b}^{L\mathcal{L} }_{\alpha\beta}(v) \frac{ k^{\tilde{h}}_{\alpha\beta}(\xi)  \p_t \widehat{u^{\mathcal{L}}}(t, \xi, v)}{i2\big(|\xi|-\hat{v}\cdot\xi\big)} d \xi d  v  \big]  
  =\sum_{i=1,\cdots, 4} \textup{Re}\big[E_{\tilde{h};L}^{k;i}(t)\big],
\] 
where
\be\label{2020july1eqn2}
\begin{split}
   E_{\tilde{h};L}^{k;1}(t)&:=  \int_{\R^3} \int_{\R^3} \overline{  \widehat{   \widetilde{ U^{L\tilde{h} }  } }(t, \xi) }    {b}^{L\mathcal{L};\tilde{h}}_{\alpha\beta}(v) \frac{k_{\alpha\beta}(\xi) \xi\cdot \big( \widehat{\Lambda_{2}[  \mathfrak{N}_1^{\mathcal{L
 }}  ]}(t, \xi, v) +\widehat{   \mathfrak{D}_1^{\mathcal{L
 }}   }(t, \xi, v)\big) }{ 2\big(|\xi|-\hat{v}\cdot\xi\big)} 
  \big(\psi_k(\xi)\big)^2 d v d\xi,\\ 
 E_{\tilde{h};L}^{k;2}(t)&:= \int_{\R^3} \int_{\R^3} \overline{  \widehat{   \widetilde{ U^{L\tilde{h} }  } }(t, \xi) }    \big( \widehat{  \mathfrak{N}_2^{\mathcal{L
 }}}  (t, \xi, v) + \widehat{  \mathfrak{D}_2^{\mathcal{L
 }}}  (t, \xi, v)  \big)\cdot  \nabla_v \Big( \frac{k_{\alpha\beta}(\xi)    {b}^{L\mathcal{L};\tilde{h}}_{\alpha\beta}(v)  }{ 2\big(|\xi|-\hat{v}\cdot\xi\big)} \Big)
  \big(\psi_k(\xi)\big)^2  d v d\xi,\\
 E_{\tilde{h};L}^{k;3}(t)&:=\sum_{i=3,4}  \int_{\R^3} \int_{\R^3} \overline{  \widehat{   \widetilde{ U^{L\tilde{h} }  } }(t, \xi) }   {b}^{L\mathcal{L};\tilde{h}}_{\alpha\beta}(v) \frac{k_{\alpha\beta}(\xi) \big(2 \widehat{ \Lambda_{\geq 3}[\mathfrak{N}_i^{\mathcal{L
 }}] }(t, \xi, v) + \widehat{  \mathfrak{D}_3^{\mathcal{L
 }}  }(t, \xi, v)\big) }{i4\big(|\xi|-\hat{v}\cdot\xi\big)} 
  \big(\psi_k(\xi)\big)^2 d v d\xi,\\
 E_{\tilde{h};L}^{k;4}(t)&:=\int_{\R^3} \int_{\R^3} \overline{  \widehat{   \widetilde{ U^{L\tilde{h} }  } }(t, \xi) }     {b}^{L\mathcal{L};\tilde{h}}_{\alpha\beta}(v) \frac{k_{\alpha\beta}(\xi) \xi\cdot   \big( \widehat{\Lambda_{\geq 3}[  \mathfrak{N}_1^{\mathcal{L
 }}  ]}(t, \xi, v) \big)}{ 2\big(|\xi|-\hat{v}\cdot\xi\big)} 
  \big(\psi_k(\xi)\big)^2 d v d\xi.
\end{split}
\ee

From (\ref{2020july1eqn1}), (\ref{nov22eqn13}), and (\ref{2020july1eqn2}),  based on the types and the orders of inputs,  we  classify   $  \p_t  \tilde{E}_{modi;k}^{  \tilde{h};L }(t)  $ as follows, 
\[
  2    \p_t    \tilde{E}_{modi;k}^{  \tilde{h};L }(t)     = C_{vlasov}^{\tilde{h}^L;k}(t) + R_{vlasov}^{ \tilde{h}^L;k}(t)+\int_{\R^3} \textup{Re}\big[ C_{metric}^{ \tilde{h}^L;k}(t, x)\big] d x    +   R_{metric}^{ \tilde{h}^L;k}(t), 
\]
where $C_{vlasov}^{\tilde{h}^L;k}(t)$ denotes wave-wave-Vlasov type trilinear interaction, $R_{vlasov}^{ \tilde{h}^L;k}(t)$ denotes the quartic and higher order interaction that involves the Vlasov part, $C_{metric}^{ \tilde{h}^L;k}(t, x)$ denotes wave-wave-wave type trilinear interaction, and $R_{metric}^{ \tilde{h}^L;k}(t)$ denotes the quartic and higher order purely wave-type interaction. More precisely, the detailed formulas are given as follows, 
\[
C_{vlasov}^{\tilde{h}^L;k}(t) =   E_{\tilde{h};L}^{k;1}(t) +   \textup{Re}\big[ \int_{\R^3} \overline{ P_k(\widetilde{   U^{ \tilde{h}^L } }  ) }    P_k\big( \Lambda_{  2}[ N_{vl}^{  {\tilde{h} ;L }}  ] \big) d  x\big]
\]
\be\label{2022april30eqn1}
 +    \int_{\R^3} { P_k(\widetilde{   \p_t { \tilde{h}^L } }  ) }   \big[ \Lambda_{1}[ H_{ij }] P_k \big(\p_i\p_j ( {   \tilde{h}^{L}}- \widetilde{  \tilde{h}^{L}})  \big)+      \Lambda_{1}[   H_{i0 } ] P_k \big(\p_i ( {\p_t \tilde{h}^{L}}- \widetilde{\p_t \tilde{h}^{L}})  \big)\big]   d x, 
 \ee
\be\label{2022april30eqn2}
    R_{vlasov}^{ \tilde{h}^L;k}(t) =      \textup{Re}\big[ \int_{\R^3} \overline{ P_k(\widetilde{   U^{ \tilde{h}^L } }  ) }    P_k\big( \Lambda_{ \geq 3}[ N_{vl}^{  {\tilde{h} ;L }}  ]\big) d  x\big] +  \sum_{i=2,3,4} E_{\tilde{h};L}^{k;i}(t) ,
\ee
\[
 C_{metric}^{\tilde{h}^L;k}(t,x) = \sum_{h'\in \{F, \underline{F}, \omega_{j}, \vartheta_{mn}\}} \sum_{\mu, \nu \in\{+,-\}}     \widetilde{U^{ \tilde{h}^L }_k }(t,x) \tilde{Q}^{\mu\nu}_{ h' \tilde{h}}( (U^{h'})^{\mu}(t), ( \widetilde{ U^{\tilde{h}^L }_k } )^{\nu}(t))(x) 
 \]
 \be\label{2022april30eqn3}
+   \overline{ P_k(\widetilde{   U^{ \tilde{h}^L } }  ) }   P_k(    Q_{mt  }^{ \tilde{h};L}) (t,x) ,
\ee
\begin{multline}\label{nov14eqn48}
 R_{metric}^{ \tilde{h}^L;k}(t)=  \mathfrak{R}_k   +    \textup{Re}\big[ \int_{\R^3} \overline{ P_k(\widetilde{   U^{ \tilde{h}^L } }  ) }   P_k\big(   R_{mt;s}^{\tilde{h};L }\big) d x\big]  -   \int_{\R^3} \big({ P_k(\widetilde{   \p_t { \tilde{h}^L } }  ) }\big)^2   \p_i \Lambda_{\geq 2}[ H_{i0 }]      \\
 -     { P_k(\widetilde{   \p_t { \tilde{h}^L } }  ) }     \p_i \Lambda_{\geq 2}[ H_{ij }] P_k \big( \p_j \widetilde{ \tilde{h}^{L}} \big)   d x+   \int_{\R^3} \overline{ P_k(\widetilde{   U^{ \tilde{h}^L } }  ) } \big( P_k( \Lambda_{1}[H_{i\gamma}]\p_{i} {\p_{\gamma}\tilde{h}^{L}}  )-   \Lambda_{1}[H_{i\gamma}] P_k(\p_{i} {\p_{\gamma}\tilde{h}^{L}}) d x \\
 +  \int_{\R^3} \frac{1}{2}  \p_{x_i} P_k (  \widetilde{\tilde{h}^L} ) \Lambda_{\geq 2}[\p_t   H_{ij}] \p_{x_j} P_k ( \widetilde{\tilde{h}^L}   )     d x +   \int_{\R^3} { P_k(\widetilde{   \p_t { \tilde{h}^L } }  ) }\big[\Lambda_{\geq 2}[   H_{i0 } ] P_k \big(\p_i ( {\p_t \tilde{h}^{L}}- \widetilde{\p_t \tilde{h}^{L}})  \big)  \\
       +       \Lambda_{\geq 2}[ H_{ij }] P_k \big(\p_i\p_j ( {   \tilde{h}^{L}}- \widetilde{  \tilde{h}^{L}})  \big)  \big] d x  
+  \int_{\R^3}    \p_{x_i} P_k ( \p_t \widetilde{\tilde{h}^L} -  \widetilde{\p_t \tilde{h}^L})  H_{ij}\p_{x_j} P_k ( \widetilde{\tilde{h}^L}  ) d x.
\end{multline}

In this section, we aim to prove the following proposition. 
\begin{proposition}\label{energyestwave}
Under the bootstrap assumptions \textup{(\ref{BAmetricold}) and (\ref{BAvlasovori})}, the following estimate holds for any $L \in   P_n, n\leq N_0 $, $\tilde{h}\in \{  F,  \underline{F},  \omega_j,  \vartheta_{mn}\},$ and any $t_1, t_2 \in [2^{m-1}, 2^m]\subset [0, T],$
\be\label{aug19eqn70}
 \big| E_{modi }^{ \tilde{h};L}(t_2)- E_{modi }^{ \tilde{h};L}(t_1)\big| \lesssim  2^{2  e(L;\tilde{h})\delta_1 m -\delta_1 m/300   } \epsilon_1^3.   
\ee
\end{proposition}

By assuming the estimates of the wave-wave-wave type interaction and the wave-wave-Vlasov type interaction, which are postponed in next two subsections, we first give a proof of the above Proposition.

\noindent \textit{Proof of Proposition} \ref{energyestwave}. 

Recall (\ref{nov14eqn48}). From the estimates  in Lemma \ref{energycubictermslow},    the estimate (\ref{2020july8eqn91}) in Proposition  \ref{theessentialcubic}, and  the estimate (\ref{2020july9eqn31}) in Proposition \ref{waveVlaovestimate},   we have
\[
  \sum_{k\in \Z}2^{-k_{-}+2\gamma k_{-}+ 2 (N_0-|L|) k_{+}   }\big[\big| \int_{t_1}^{t_2} \int_{\R^3}   C_{metric}^{ \tilde{h}^L;k}(t, x)  d x   d t  \big| + \big| \int_{t_1}^{t_2} \int_{\R^3}  C_{vlasov}^{ \tilde{h}^L;k}(t,x)    d x  d t\big|\big]
\]
\be\label{2020july8eqn92}
\lesssim 2^{ 2 e(L;\tilde{h})\delta_1 m -\delta_1 m/300   } \epsilon_1^3. 
\ee
It remains to estimate  $R_{vlasov}^{L\tilde{h};k}(t)$ and  $R_{metric}^{L\tilde{h};k}(t)$.

\noindent $\bullet$\qquad The estimate of $R_{vlasov}^{L\tilde{h};k}(t)$. \qquad Recall (\ref{nov14eqn48}), (\ref{aug11eqn21}), and (\ref{2020april3eqn1}). From the $L^2-L^2$ estimate and the estimate (\ref{july3eqn61}) in Lemma \ref{cubicandhigherrough}, the following estimate holds,
\[
 \sum_{k\in\Z} 2^{-k_{-}+2\gamma k_{-}+ 2 (N_0-|L|) k_{+}}  \big| \int_{\R^3} \overline{ P_k(\widetilde{   U^{ \tilde{h}^L } }  ) }    P_k\big( \Lambda_{ \geq 3}[ N_{vl}^{  {\tilde{h} ;L }}  ]\big) d  x\big| 
\]
\be\label{aug24eqn30}
\lesssim 2^{  H(|L| )\delta_1 m }  \epsilon_1   \| \Lambda_{\geq 3}[   L  \mathcal{N}_{\alpha\beta}^{vl} ]\|_{E_n}\lesssim 2^{-4m/ 3 }\epsilon_1^4.
\ee

Recall (\ref{2020july1eqn2}). From the $L^2-L^2$ type bilinear estimate, the volume of support of $\xi$, and the estimate (\ref{aug9eqn87}) in Lemma \ref{estimateofremaindervlasov}, and the estimate (\ref{2020april14eqn21}), the estimate (\ref{2020april14eqn3}) in Lemma  \ref{estimateofremaindervlasov3},  the following estimate holds,  
\[
  \sum_{i=2,3,4} \sum_{k\in \Z} 2^{-k_{-}+2\gamma k_{-}+ 2 (N_0-|L|) k_{+}}    |E_{\tilde{h};L}^{k;i}(t)  |  \lesssim \sum_{i=1,3,4} \sum_{j=2,3}   2^{H(|L|)\delta_1 m }\epsilon_1   \big( \|(1+|x|)^2(1+|v|)^{10} \mathfrak{N}_2^{\mathcal{L}} (t,x,v) \|_{L^2_{x,v}}
 \]
 \[
  + \|(1+|x|)^2(1+|v|)^{10}\Lambda_{\geq3}[ \mathfrak{N}_i^{\mathcal{L}} ](t,x,v) \|_{L^2_{x,v}} + \|(1+|x|)^2(1+|v|)^{10}   \mathfrak{D}_j^{\mathcal{L}}  (t,x,v) \|_{L^2_{x,v}} \big)
 \]
 \be\label{2020may15eqn41}
 \lesssim 2^{-7m/6 + 3d(N_0+3,0)\delta_0 m } \epsilon_1^4. 
\ee
Combining the estimates (\ref{aug24eqn30}) and (\ref{2020may15eqn41}), we have
\be\label{2020july1eqn10}
\sum_{k\in \mathbb{Z}} 2^{-k_{-}+2\gamma k_{-}+ 2(N_0-|L|)k_{+}}   | R_{vlasov}^{ \tilde{h}^L;k}(t)|  \lesssim  2^{-7m/6 + 3d(N_0+3,0)\delta_0 m } \epsilon_1^4.
\ee

\noindent $\bullet$\qquad The estimate of $R_{metric}^{L\tilde{h};k}(t)$.\qquad 
 We first estimate $\mathfrak{R}_k^3$. Recall  (\ref{aug8eqn41}). From the $L^2-L^2-L^\infty$ type multilinear estimate, the following estiamte holds for any fixed $k\in \Z$,
 \[
  2^{-k_{-}+2\gamma k_{-}+ 2 (N_0-|L|) k_{+}}  | \mathfrak{R}_k^3| 
 \]
\be\label{aug25eqn9}
 \lesssim  \sum_{ k_1\in \mathbb{Z}}  2^{  2H(|L|)\delta_1 m } \big(\|\p_{i}P_{k_1}(\rho-R_0 \underline{F})\|_{L^\infty_x} + \|\p_{\alpha}P_{k_1}(\Omega_m-R_0 \omega_m )\|_{L^\infty_x} + \| \d^{-1} P_{k_1}(\square \omega_m) \|_{L^\infty_x} \big).
\ee
From  the relations between double Hodge decomposition variable in (\ref{april1eqn11})  and the equalities (\ref{aug25eqn1}) and (\ref{aug25eqn2}), we have
\[
\| \p_{i}P_{k_1}(\rho-R_0 \underline{F})\|_{L^\infty_x} + \| \p_{\alpha}P_{k_1}(\Omega_m-R_0 \omega_m)\|_{L^\infty_x} +  \| \d^{-1} P_{k_1}(\square \omega_m) \|_{L^\infty_x}\]
\[
\lesssim  \|\p_tP_{k_1} \tau(t)\|_{L^\infty_x} + \|E_\mu(t)\|_{L^\infty_x} + \|\d^{-1}\p_t E_\mu (t)\|_{L^\infty_x} \]
\be\label{aug25eqn10}
 \lesssim \min\{2^{-k_1}\big(\|\p_t E_0\|_{L^\infty_x } + \|\square h_{0i}\|_{L^\infty_x} \big)+  \|E_0(t)\|_{L^\infty_x}, 2^{k_1/2} \big(\|\p_t E_0\|_{L^2_x } + \|\square h_{0i}\|_{L^2_x} \big) + 2^{3k_1/2} \|E_0(t)\|_{L^2_x} \}.
\ee

From the estimates (\ref{oct7eqn73}) and (\ref{oct9eqn11}) in Lemma \ref{harmonicgaugehigher} , the linear decay estimate (\ref{july27eqn4}) in Lemma \ref{superlocalizedaug}, which is used in the close to the light cone case, the equality (\ref{definitionofcoefficient1}) in the far away from    the light cone case,   and the estimate (\ref{oct4eqn41}) in Lemma \ref{fixedtimenonlinarityestimate}, we have
\be\label{aug25eqn15}
 (\textup{\ref{aug25eqn10}})\lesssim  \sum_{k_1\in \mathbb{Z}} \min\{  2^{-k_1 -2m  }  + 2^{-2m    } , 2^{k_1/2- m  } + 2^{3k_1/2- m  }  \}2^{ 2H( 1)\delta_0 m}\epsilon_1^2\lesssim 2^{-4m/3  + 2H( 1)\delta_0 m}\epsilon_1^2.
\ee
From the above estimate (\ref{aug25eqn15}) and the estimate (\ref{aug25eqn9}), we have
\be\label{aug25eqn16}
\sum_{k\in \mathbb{Z}}   2^{-k_{-}+2\gamma k_{-}+ 2 (N_0-|L|) k_{+}} | \mathfrak{R}_k^3| \lesssim 2^{-4m/3  + 2H( 1)\delta_0 m +2 H(|L|)\delta_0 m }\epsilon_1^3. 
\ee

Recall (\ref{2022may17eqn33}) and (\ref{2020june30eqn1}). Note that, the imaginary part of $ e^{-it\d} \p_t  \widetilde{   V^{  \tilde{h}^L} }$ contains only $ \mathcal{N}_{vl;m}^{  {\tilde{h} ;L }}$, which is semilinear. From the $L^2$-estimate of  $ \mathcal{N}_{vl;m}^{  {\tilde{h} ;L }}$   (\ref{july11eqn32}) in Lemma  \ref{resonancevlasov} and the $L^2-L^2-L^\infty$ type multilinear estimate, we have
\be\label{2020july2eqn1}
\sum_{k\in \mathbb{Z}}   2^{-k_{-}+2\gamma k_{-}+ 2 (N_0-|L|) k_{+}} \big|  \int_{\R^3}    \p_{x_i} P_k ( \p_t \widetilde{\tilde{h}^L} -  \widetilde{\p_t \tilde{h}^L})  H_{ij}\p_{x_j} P_k ( \widetilde{\tilde{h}^L}  ) d x \big| \lesssim 2^{-2m+ 2d(N_0+3,0)\delta_1 m }\epsilon_1^3. 
\ee

Recall (\ref{nov14eqn48}).   From the above estimates (\ref{aug25eqn16}) and (\ref{2020july2eqn1}), the $L^2-L^2$ type bilinear estimate, the $L^2-L^2-L^\infty$ type multilinear estimate,  the estimate  (\ref{oct1eqn31})  in Lemma \ref{cubicandhighermetric},  and  the estimate (\ref{june30eqn1})  in Lemma \ref{highorderterm2}, we have
\be\label{aug25eqn20}
\sum_{k\in \mathbb{Z}}   2^{-k_{-}+2\gamma k_{-}+ 2 (N_0-|L|) k_{+}} \big| R_{metric}^{L\tilde{h};k}(t)\big| \lesssim  2^{-m-\delta_1 m   } \epsilon_1^4. 
\ee  
 To sum up, our desired estimate (\ref{aug19eqn70}) holds from (\ref{2020july8eqn92}), (\ref{2020july1eqn10}) and (\ref{aug25eqn20}).

\qed

Now, the goal is reduced to estimate  the contribution from the wave-wave-Vlasov type interaction, i.e., the estimate of $C_{vlasov}^{ \tilde{h}^L;k}(t, x)$ and the wave-wave-wave type interaction, i.e., the estimate $C_{metric}^{ \tilde{h}^L;k}(t, x)$. We first do further reduction for  the estimate of $C_{metric}^{ \tilde{h}^L;k}(t, x)$ and then analyze these two types of interaction in the next two subsections in details. 

We first rule out the very low output frequency case and the case when $x$ is  far away from the light cone. The main tool we will use is the Lemma \ref{lowfrequencyinputestimate}, which allows us to exploit the difference of hierarchy we set for the growth rate of energy of the  different order of vector fields.  

Let 
\be\label{outputthre}
\alpha_{out}(|L|):= \left\{\begin{array}{ll}
( H(|L|)-H(|L|-1)-  2 \beta(|L|))/2 - 1/50 & \textup{if\,} |L|\geq 1,\\ 
-10 &  \textup{if\,} |L|=0, 
\end{array}\right. 
\ee
where $\beta(|L|)$ is defined in (\ref{2022may3eqn11}).  

In the following Lemma, we first rule out the very low output frequency case. 

\begin{lemma}\label{energycubictermslow}
Under the bootstrap assumptions \textup{(\ref{BAmetricold})} and \textup{(\ref{BAvlasovori})}, the following estimate holds for any $t\in [2^{m-1}, 2^m]\subset [0,T]$, $\tilde{h}\in \{F, \underline{F}, \omega_j, \vartheta_{mn}\}$, $L\in \cup_{n\leq N_0 }P_n, k\in \mathbb{Z}$,  
\be\label{nov14eqn84}
\sum_{k\in \mathbb{Z}, k\leq-m + \alpha_{out}(|L|)\delta_1 m}  2^{-k_{-}+2\gamma k_{-}+ 2 (N_0-|L|) k_{+}} \big| \int_{\R^3}   C_{metric}^{\tilde{h}^L;k}(t,x) d x \big| \lesssim 2^{-m + 2   e(L;\tilde{h})\delta_1   m   -\delta_1 m/100} \epsilon_1^3, 
\ee
\[
\sum_{k\in \mathbb{Z}   }  2^{-k_{-}+2\gamma k_{-}+ 2 (N_0-|L|) k_{+}} \big| \int_{\R^3}   C_{metric}^{\tilde{h}^L;k}(t,x) \big(\psi_{ < m-4}(x) + \psi_{> m+1}(x)\big) d x \big|
\]
\be\label{2020june24eqn31}
 \lesssim 2^{-m + 2   e(L;\tilde{h})\delta_1 m - \delta_1 m } \epsilon_1^3. 
\ee
\end{lemma}

\begin{proof}
Recall (\ref{2022april30eqn3}) and (\ref{2020june1eqn11}). We split into two cases as follows.

$\bullet$\qquad \textbf{Case $1$}:\qquad  Proof of the desired estimate \eqref{nov14eqn84}.

Similar to the obtained estimate (\ref{2022march20eqn11}), after rerunning the argument and using the volume of support of the output frequency instead of using the $L^2-L^\infty$ type bilinear estimate,    the following  estimate  holds   if $k\leq -m +\alpha_{out}(|L|)\delta_1 m $,
\be\label{nov9eqn1}
\sum_{H\in\{S, P, Q\}} \sum_{(k_1,k_2)\in \chi_k^1\cup \chi_k^2\cup \chi_k^3} 2^{-k/2+\gamma k+  (N_0-|L|) k_{+}}\| QH_{\alpha\beta;k,k_1,k_2}^{L;L_1,L_2}(t)\|_{L^2} \lesssim 2^{ (1+\gamma)k_{-}+ H(|L|)\delta_1 m +\delta_1 m  }\epsilon_1^2.
\ee
 
From the $L^2-L^2$ type estimate,  the above obtained estimates (\ref{nov9eqn1})  
 and the estimate (\ref{nov9eqn20}) in Lemma \ref{lowfrequencyinputestimate} for the case $|L|\geq 1 $,   we have
 \[
   \sum_{k\in \mathbb{Z}, k\leq-m + \alpha_{out}(|L|)\delta_1 m} 2^{-k_{-}+2\gamma k_{-}+ 2 (N_0-|L|) k_{+}}  \big| \int_{\R^3}     \overline{ P_k(\widetilde{   U^{ \tilde{h}^L } }  ) }   P_k(    Q_{mt }^{ \tilde{h};L}) (t,x)  \psi_{\leq m +2}(|x|)  d x\big| 
 \]
 \be\label{nov14eqn61}
\lesssim  2^{-m + 2 e(L;\tilde{h})\delta_1 m -\delta_1 m/100  } \epsilon_1^3.
\ee
Hence finishing the proof of the desired estimate (\ref{nov14eqn84}).

From the proof of the estimates (\ref{2020aug25eqn45}), (\ref{2020aug25eqn44}),  and (\ref{oct26eqn42}),  the following improved estimate holds for the exterior of the light cone region  if $k\leq -m/2$,
\be\label{nov14eqn80}
\sum_{H\in\{S, P, Q\}} \sum_{k\in \Z, k\leq -m/2} \sum_{(k_1,k_2)\in \chi_k^1\cup \chi_k^2\chi_k^3} 2^{-k/2+\gamma k+  (N_0-|L|)k_{+}}\| QH_{\alpha\beta;k,k_1,k_2}^{L;L_1,L_2}(t)  \psi_{> m + 2}(|x|) \|_{L^2} \lesssim 2^{-4m/3}\epsilon_1^3. 
\ee
From the obtained estimates  (\ref{nov14eqn61})  and (\ref{nov14eqn80}),  we have
\[
 \sum_{k\in \mathbb{Z}, k\leq-m + \alpha_{out}(|L|)\delta_1 m} 2^{-k_{-}+2\gamma k_{-}+ 2 (N_0-|L|) k_{+}}  \big| \int_{\R^3}     \overline{ P_k(\widetilde{   U^{ \tilde{h}^L } }  ) }   P_k(    Q_{mt }^{ \tilde{h};L}) (t,x)   d x\big|
\]
\be\label{nov14eqn68}
\lesssim 2^{-m + 2 e(L;\tilde{h})\delta_1 m -\delta_1 m/100  } \epsilon_1^3. 
\ee

Recall again  (\ref{2022april30eqn3}). It remains to estimate the $\tilde{Q}^{\mu\nu}_{ h'\tilde{h}}(U^{h'}, ( U^{ \tilde{h}^L }_k)^{\nu}) $.   Note that, from the support of frequency variables,  the frequency of $U^{h'}$ is bounded from above by $2^{k+10}$, otherwise the integral vanishes.  From the $L^2-L^2-L^\infty$ type estimate and the $L^\infty\rightarrow L^2$ type Sobolev embedding, the following estimate holds,
\[
 \sum_{k\in \mathbb{Z}, k\leq-m + \alpha_{out}(|L|)\delta_1 m}  2^{-k+2\gamma k}\big|\int_{\R^3}   ( U^{ \tilde{h}^L }_k)^{\mu} (t,x) \tilde{Q}^{\mu\nu}_{ h'\tilde{h}}(U^{h'}, ( U^{ \tilde{h}^L }_k)^{\nu}) (t,x)   d x\big|
\]
\be\label{nov14eqn65}
\lesssim   \sum_{k\in \mathbb{Z}, k\leq-m + \alpha_{out}(|L|)\delta_1 m}  2^{-m + 2 H(|L|)\delta_0m+ (1+\gamma) k  +\delta_1 m } \epsilon_1^3 \lesssim 2^{-3m/2}  \epsilon_1^3.
\ee
  To sum up,  our desired estimate (\ref{nov14eqn84}) holds from the estimates (\ref{nov14eqn61}--\ref{nov14eqn65}).

$\bullet$\qquad \textbf{Case $2$}:\qquad  Proof of the desired estimate \eqref{2020june24eqn31}.

 For the semilinear quadratic terms $QK_{\alpha\beta;k,k_1,k_2}^{L;L_1,L_2}(t,x), K\in\{  Q, P\}$ and $\tilde{Q}^{\mu\nu}_{ h'\tilde{h}}(U^{h'}, ( U^{ \tilde{h}^L }_k)^{\nu}) (t,x)$, we use the   bilinear estimate (\ref{2020june20eqn11}) in Lemma    \ref{2020junebilinearestimate} and for the quasilinear  quadratic term $QS_{\alpha\beta;k,k_1,k_2}^{L;L_1,L_2}(t,x)$, we use the the   bilinear estimate (\ref{2020june20eqn11}) in Lemma    \ref{2020junebilinearestimate} for the input $L_2 h_{\alpha\beta}$ if $|L_2|\leq |L|/2+4$ and use the bilinear estimate (\ref{2020june20eqn12}) in  Lemma    \ref{2020junebilinearestimate} if $|L_2| >|L|/2+4$. 

 Here,   we exploit the fact that  $|L_2|< N_0 -1$ for the quasilinear  quadratic term  $QS_{\alpha\beta;k,k_1,k_2}^{L;L_1,L_2}(t,x)$. As a result, the symbols of quadratic terms allow us to gain the smallness $2^{-m}$ from the distance with respect to the light cone, which is more than sufficient to ensure the validity of our desired estimate (\ref{2020june24eqn31}). 
\end{proof}

Now, we consider the case   $k\geq -m +\alpha_{out}(|L|) \delta_1 m$. Recall (\ref{nov14eqn48}).  To distinguish the subtle growth rates of the perturbed profile and the modified perturbed profile for the metric component, from  (\ref{2020june1eqn11}--\ref{nov7eqn21}), based on the size of $\min\{k_1,k_2\}$,   we decompose  $  C_{metric}^{\tilde{h}^L;k}(t,x) $ into several parts as follows,
\be\label{2020june24eqn32}
 C_{metric}^{\tilde{h}^L;k}(t,x) = \sum_{H\in \{S, P, Q\}}\sum_{(k_1,k_2)\in \chi_k^1\cup\chi_k^2\cup\chi_k^3, }  CH_{metric }^{ \tilde{h}^L;k,k_1,k_2}(t,x) + \sum_{k'\in \mathbb{Z}} \sum_{j\in\{1,2\}}   \tilde{C}_{metric;j}^{ \tilde{h}^L;k,  k'}(t,x),
\ee
where  $ CH_{metric }^{ \tilde{h}^L;k,k_1,k_2}(t,x)$ denotes term determined by $ Q_{mt;s}^{ \tilde{h};L} (t,x)  $ and $\tilde{C}_{metric;j}^{ \tilde{h}^L;k,  k'}(t,x)$ are given as follows, 
\be\label{2022april18eqn11}
\begin{split}
\tilde{C}_{metric;1}^{ \tilde{h}^L;k,  k'}(t,x)& = \sum_{ {h}' \in \{F, \underline{F}, \omega_{j}, \vartheta_{mn}\}} \sum_{\mu, \nu \in\{+,-\}}      \widetilde{U^{ \tilde{h}^L }_k }(t,x) \tilde{Q}^{\mu\nu}_{  {h}' \tilde{h}}(\widetilde{ ( U^{ {h}'}_{k'})^{\mu}}, \widetilde{(  U^{ \tilde{h}^L }_k)^{\nu}}) (t,x),\\
\tilde{C}_{metric;2}^{ \tilde{h}^L;k,  k'}(t,x)&= \sum_{ h'\in \{F, \underline{F}, \omega_{j}, \vartheta_{mn}\}} \sum_{\mu, \nu \in\{+,-\}}  \widetilde{U^{ \tilde{h}^L }_k }(t,x) \tilde{Q}^{\mu\nu}_{ h' \tilde{h}}(  { ( \tilde{\rho}^{h'}_{}(P_{k'} f))^{\mu}}, \widetilde{(  U^{ \tilde{h}^L }_k)^{\nu}})(t,x). 
\end{split}
\ee

 $ CH_{metric }^{ \tilde{h}^L;k,k_1,k_2}(t,x)$ is more subtle.  We use  different formulations for the wave-wave-wave type interaction based on the size of the input frequency because, from \eqref{2022march19eqn30}, there is no need  to decompose the profile into the modified profile and the density type function  if the input frequency is very small. 

 More precisely, let 
 \be\label{inputthre}
\alpha_{inp}(|L|):= \left\{\begin{array}{ll}
( H(|L|)-H(|L|-1)-  2 \beta(|L|))/2 - 1/50 & \textup{if\,} |L|\geq 1,\\ 
10 &  \textup{if\,} |L|=0, 
\end{array}\right. 
\ee
We remark that,  $\alpha_{inp}(|L|)$ only differs from $\alpha_{out}(|L|)$ when $|L|=0$, see \eqref{outputthre}.

If $\min\{k_1, k_2\}\geq -m + \alpha_{inp}(|L|)\delta_1 m$, then 
\begin{multline}\label{2020july6eqn65}
 \forall H\in \{S, P, Q\}, \qquad CH_{metric }^{  \tilde{h}^L;k,k_1,k_2}(t,x) = \sum_{  \star\in \{ess, err,vl\}} CH_{metric; \star}^{  \tilde{h}^L;k,k_1,k_2}(t,x),\\
  CS_{metric; \star }^{  \tilde{h}^L;k,k_1,k_2}(t,x) =   \sum_{\begin{subarray}{c}
 \tilde{L}_1\circ\tilde{L}_2\preceq L,\tilde{L}_2\prec  L
 \end{subarray}}   c_{\tilde{L}_1\tilde{L}_2}^{\tilde{L} }   \overline{ P_k(\widetilde{   U^{ \tilde{h}^L } }  ) } (t,x)   R_{\tilde{h}}^{\alpha\beta} \big(   QS_{\alpha\beta;k,k_1,k_2}^{\star;L,\tilde{L}_1,\tilde{L}_2} \big)(t,x), \\ 
 CK_{metric ; \star}^{  \tilde{h}^L;k,k_1,k_2}(t,x) =   \sum_{\begin{subarray}{c}
 \tilde{L}_1\circ\tilde{L}_2\preceq L
 \end{subarray}}   c_{\tilde{L}_1\tilde{L}_2}^{\tilde{L} }   \overline{ P_k(\widetilde{   U^{ \tilde{h}^L } }  ) } (t,x)   R_{\tilde{h}}^{\alpha\beta} \big(   QK_{\alpha\beta;k,k_1,k_2}^{\star;L, \tilde{L}_1, \tilde{L}_2} \big)(t,x),\quad K\in \{P, Q\},
 \end{multline}
 where $QH_{\alpha\beta;k,k_1,k_2}^{\star;L,\tilde{L}_1,\tilde{L}_2}, H\in \{S, P, Q\},  \star\in \{ess, err,vl\}$ are defined in (\ref{nov7eqn21}).

 If  $\min\{k_1, k_2\}\leq -m + \alpha_{inp}(|L|)\delta_1 m$, then 
 \[
 CH_{metric }^{  \tilde{h}^L;k,k_1,k_2}(t,x) =  CH_{metric;wa }^{  \tilde{h}^L;k,k_1,k_2}(t,x) + CH_{metric;vl }^{  \tilde{h}^L;k,k_1,k_2}(t,x), \quad H\in \{S, P, Q\},
 \]
where, for $K\in \{P, Q\}$, we have 
 \be\label{2020july6eqn51}
 \begin{split}
 CK_{metric;wa }^{  \tilde{h}^L;k,k_1,k_2}(t,x) & =   \sum_{\begin{subarray}{c}
 \tilde{L}_1\circ\tilde{L}_2\preceq L
 \end{subarray}}   c_{\tilde{L}_1\tilde{L}_2}^{\tilde{L} }   \overline{ P_k(\widetilde{   U^{ \tilde{h}^L } }  ) } (t,x)   R_{\tilde{h}}^{\alpha\beta} \big(    QK_{\alpha\beta;k,k_1,k_2}^{ L,\tilde{L}_1,\tilde{L}_2}   \big)(t,x), \\
   CK_{metric;vl }^{  \tilde{h}^L;k,k_1,k_2}(t,x)& =0,
 \end{split}
 \ee
 and, for the quaslinear terms, i.e., $H=S$, as two inputs have different number of derivatives, we use different formulas based on sizes of $k_1$ and $k_2$. More precisely, if  $k_2=\min\{k_1, k_2\}\leq -m + \alpha_{inp}(|L|)\delta_1 m$, then 
  \begin{equation}\label{2020july6eqn52}
  CS_{metric;wa }^{  \tilde{h}^L;k,k_1,k_2}(t,x) =   \sum_{\begin{subarray}{c}
 \tilde{L}_1\circ\tilde{L}_2\preceq L,\tilde{L}_2\prec  L
 \end{subarray}}   c_{\tilde{L}_1\tilde{L}_2}^{\tilde{L} }   \overline{ P_k(\widetilde{   U^{ \tilde{h}^L } }  ) } (t,x)   R_{\tilde{h}}^{\alpha\beta} \big(    QS_{\alpha\beta;k,k_1,k_2}^{ L,\tilde{L}_1,\tilde{L}_2}  \big)(t,x),  \ee
 \be
    CS_{metric;vl }^{  \tilde{h}^L;k,k_1,k_2}(t,x) =0.
 \end{equation}
 If $k_1=\min\{k_1,k_2\}\leq   -m + \alpha_{inp}(|L|)\delta_1 m$, then 
 \[
  CS_{metric;wa }^{  \tilde{h}^L;k,k_1,k_2}(t,x) =  \sum_{\begin{subarray}{c}
 \tilde{L}_1\circ\tilde{L}_2\preceq L,\tilde{L}_2\prec  L 
 \end{subarray}}    c_{\tilde{L}_1\tilde{L}_2}^{\tilde{L} }   \overline{ P_k(\widetilde{   U^{ \tilde{h}^L } }  ) } (t,x)  R_{\tilde{h}}^{\alpha\beta} \big[   c^{\mu\nu}_{ i 0} P_k\big(P_{k_1}(  L_1  h_{\mu\nu} ) P_{k_2}(\p_{i} \widetilde{\p_{t}  L_2 h_{\alpha\beta} } )\big)
 \]
 \be\label{2022may3eqn18}
  + c^{\mu\nu}_{ i j} P_k\big(P_{k_1}(  L_1  h_{\mu\nu} ) P_{k_2}(\p_{i}\p_{j} \widetilde{  L_2 h_{\alpha\beta} } )\big)\big](t,x),
 \ee
 \[
  CS_{metric;vl }^{  \tilde{h}^L;k,k_1,k_2}(t,x) =  \sum_{\begin{subarray}{c}
 \tilde{L}_1\circ\tilde{L}_2\preceq L \\ 
 \tilde{L}_2\prec  L 
 \end{subarray}}    c_{\tilde{L}_1\tilde{L}_2}^{\tilde{L} }   \overline{ P_k(\widetilde{   U^{ \tilde{h}^L } }  ) } (t,x)  R_{\tilde{h}}^{\alpha\beta} \big[   c^{\mu\nu}_{ i 0} P_k\big(P_{k_1}(  L_1  h_{\mu\nu} ) P_{k_2}(\p_{i}\big( {\p_{t}  L_2 h_{\alpha\beta} } )- \widetilde{\p_{t}  L_2 h_{\alpha\beta} } )\big) \big)
 \]
 \be\label{2020july6eqn53}
+ c^{\mu\nu}_{ i j} P_k\big(P_{k_1}(  L_1  h_{\mu\nu} ) P_{k_2}(\p_{i}\p_{j}\big( {  L_2 h_{\alpha\beta} } - \widetilde{  L_2 h_{\alpha\beta} } )\big)\big](t,x). 
 \ee
 
With the above preparations, in the following Lemma, we first rule out some easy cases by using the obtained estimates in section \ref{fixedtimeestimate}.
\begin{lemma}\label{lowinputfrequency2}
 Under the bootstrap assumptions \textup{(\ref{BAmetricold})} and \textup{(\ref{BAvlasovori})}, the following estimate holds for any $t\in [2^{m-1}, 2^m]\subset [0,T]$,  $L\in \cup_{l\leq N_0 } P_l$,  the following estimates hold,
 \[
   \sum_{\begin{subarray}{c}
 k,k'\in \Z, k\geq - m + \alpha_{out}(|L|) \delta_1 m \\
 j=1,2,k'\leq -2|\alpha_{out}(|L|)| \delta_1 m\\
 \end{subarray}} 2^{-k_{-}+2\gamma k_{-}+ 2 (N_0-|L|) k_{+}}\big| \int_{\R^3}  \tilde{C}_{metric;j}^{ \tilde{h}^L;k,  k'}(t,x) \psi_{[m-4, m+1]}(x) d x\big|  
 \]
 \be\label{2020june24eqn50}
 \lesssim 2^{-m +   2e(L;\tilde{h})\delta_1 m  - \delta_1 m} \epsilon_1^3, 
 \ee
 \[
 \sum_{\begin{subarray}{c}
 k,k'\in \Z, k\geq - m + \alpha_{out}(|L|) \delta_1 m \\
 k'\geq -2|\alpha_{out}(|L|)| \delta_1 m\\
 \end{subarray}} 2^{-k_{-}+2\gamma k_{-}+ 2 (N_0-|L|) k_{+}} \big| \int_{\R^3}  \tilde{C}_{metric ;2}^{ \tilde{h}^L;k,  k'}(t,x) \psi_{[m-4, m+1]}(x) d x\big|
 \]
  \be\label{2020june24eqn51}
\lesssim 2^{-m +   2e(L;\tilde{h})\delta_1 m  - \delta_1 m } \epsilon_1^3. 
 \ee
 \[
 \sum_{\begin{subarray}{c}
k\in \Z, k\geq -m + \alpha_{out}(|L|) \delta_1 m \\
K\in \{S,P, Q\}, (k_1,k_2)\in \chi_k^1\cup \chi_k^2\cup\chi_k^3\\ 
 \min\{k_1,k_2\}\geq -m + \alpha_{inp}(|L|) \delta_1 m
\end{subarray}  }  2^{-k_{-}+2\gamma k_{-}+ 2 (N_0-|L|) k_{+}}   \big|\int_{\R^3}  CK_{metric;err }^{  \tilde{h}^L;k,k_1,k_2}(t,x)  \psi_{[m-4, m+1]}(x) d x \big|
 \]
 \be\label{2020july6eqn56}
\lesssim 2^{-m +   2e(L;\tilde{h})\delta_1 m  - \delta_1 m} \epsilon_1^3.  
 \ee
 \end{lemma}
\begin{proof}
Recall the detailed formula of $\tilde{C}_{metric;j}^{L\tilde{h};k,  k'}(t)$  in  (\ref{2020june24eqn32}) and the equality  (\ref{nov22eqn13}).  From the  bilinear estimate (\ref{2020june20eqn25}) in  Lemma \ref{2020junebilinearestimate},  the following estimate holds if $k'\leq -2 |\alpha_{out}(|L|)|   \delta_1 m$, 
\[
 2^{-k_{-}+2\gamma k_{-}+ 2(N_0-|L|)k_{+}}  \big| \int_{\R^3}  \tilde{C}_{metric;j }^{ \tilde{h}^L;k,  k'}(t,x) \psi_{[m-4, m+1]}(x) d x\big|    \lesssim   2^{-m+ k' +  2e(L;\tilde{h})\delta_1 m  +\beta(|L|)\delta_1 m } 
\]
\be
   \lesssim  2^{-m +  2e(L;\tilde{h})\delta_1 m- \delta_1 m } \epsilon_1^3. 
\ee
Hence finishing the proof of our desired estimate (\ref{2020june24eqn50}).

Recall (\ref{2022april18eqn11}).  Since    $  k'\geq  -2 |\alpha_{out}(|L|)|\delta_1 m  $, the  desired estimate (\ref{2020june24eqn51}) holds directly from the  estimate (\ref{bilinearestimate2}) in Lemma \ref{wavevlabil6sep}, the estimate (\ref{july27eqn4}) in Lemma \ref{superlocalizedaug},   and the decay estimate  of the density type function (\ref{densitydecay}) in Lemma \ref{decayestimateofdensity}. 

Recall (\ref{2020june18eqn1}), (\ref{2020june17eqn22}), and (\ref{2022may18eqn1}). The desired estimate (\ref{2020july6eqn56}) follows directly  from the   estimate   (\ref{2022march19eqn30}) in Lemma \ref{basicestimates} and  the estimate (\ref{oct4eqn41}) in Proposition \ref{fixedtimenonlinarityestimate} and   the  bilinear estimate (\ref{2020june20eqn25}) in  Lemma \ref{2020junebilinearestimate}.
\end{proof}

To sum up, it remains to estimate two types of interactions: 

For the wave-wave-Vlasov type interaction, it remains to    estimate (i) $  C_{vlasov}^{ \tilde{h}^L;k}(t,x) $ in  (\ref{2022april30eqn1}); (ii) $\tilde{C}_{metric;2}^{ \tilde{h}^L;k,  k'}(t,x)$ in (\ref{2022april18eqn11}) for the case $k'\geq -2|\alpha_{out}(|L|)|\delta_1 m   $; (iii) $ CH_{metric;vl }^{  \tilde{h}^L;k,k_1,k_2}(t,x), H\in \{S, P, Q\} $ in (\ref{2020july6eqn65})  for the case $\min\{k_1,k_2\}\geq -m + \alpha_{inp}(|L|) \delta_1 m $ 
 and $ CS_{metric;vl }^{  \tilde{h}^L;k,k_1,k_2}(t,x)$ in (\ref{2020july6eqn53}) for the case $k_1=\min\{k_1,k_2\}\leq -m + \alpha_{inp}(|L|) \delta_1 m $. We analyze these three cases in details in subsection \ref{wavewaveVlasov}. 

For the wave-wave-wave type interaction, it remains to estimate (i) $\tilde{C}_{metric;1}^{ \tilde{h}^L;k,  k'}(t,x)$ in (\ref{2022april18eqn11}) for the case $k'\geq -2|\alpha_{out}(|L|)|\delta_1 m   $; (ii) $ CH_{metric; ess }^{  \tilde{h}^L;k,k_1,k_2}(t,x), H\in \{S, P, Q\}$ in (\ref{2020july6eqn65}) for the case $\min\{k_1,k_2\}\geq  -m + \alpha_{inp}(|L|) \delta_1 m $; (iii) $ CK_{metric;wa }^{  \tilde{h}^L;k,k_1,k_2}(t,x), K\in\{P, Q\},$  in (\ref{2020july6eqn51}) and $  CS_{metric;wa }^{  \tilde{h}^L;k,k_1,k_2}(t,x)$ in \eqref{2020july6eqn52} and  \eqref{2022may3eqn18}  for the case $\min\{k_1,k_2\}\leq  -m + \alpha_{inp}(|L|) \delta_1 m $. We analyze these three cases in details in subsection \ref{wavewavewave}.

With the above divided plan, for the rest of this section,  we aim to prove the following Proposition.   

\begin{proposition}\label{theessentialcubic}
Under the bootstrap assumptions \textup{(\ref{BAmetricold})} and \textup{(\ref{BAvlasovori})}, the following estimate holds for any $t_1, t_2 \in [2^{m-1}, 2^m]\subset [0,T]$, $L\in \cup_{l\leq N_0 }P_l$,  
\[
\sum_{k\in \mathbb{Z}, k>-m + \alpha_{out}(|L|)\delta_1 m}   2^{-k_{-}+2\gamma k_{-}+ 2 (N_0-|L|) k_{+}} \big| \int_{t_1}^{t_2} \int_{\R^3} C_{metric}^{ \tilde{h}^L;k}(t,x)  \psi_{ [m-4, m+1]}(x)   d x  d t\big| 
\]
\be\label{2020july8eqn91}
\lesssim 2^{ 2e(L;\tilde{h})\delta_1 m - \delta_1 m/300 } \epsilon_1^3. 
\ee
\end{proposition}
\begin{proof}
 
Recall the decomposition in (\ref{2020june24eqn32}).   
 Our desired estimate holds from  the estimates (\ref{2020june24eqn50}--\ref{2020july6eqn56}) in Lemma \ref{lowinputfrequency2},  the estimates (\ref{2020july7eqn2}--\ref{2020july9eqn31}) in Lemma \ref{waveVlaovestimate}, the estimate (\ref{nov15eqn50}) in Lemma \ref{lowinputfrequency}, and the estimate (\ref{2020july6eqn27}) in Lemma \ref{wavewavewave2020ess}.  
 
\end{proof}

 \subsection{The estimate of wave-wave-Vlasov type interaction }\label{wavewaveVlasov}

Recall the detailed formula of $  C_{vlasov}^{ \tilde{h}^L;k}(t,x) $ in  (\ref{2022april30eqn1}) and  the the detailed formulas of corresponding terms in  (\ref{2020july1eqn2}), (\ref{2020july9eqn5}),     (\ref{2020july9eqn1}), and (\ref{2020april14eqn51}). We  can write  $  C_{vlasov}^{ \tilde{h}^L;k}(t,x) $ as a linear combinations of the following trilinear form, 
\be\label{2020july9eqn21}
\int_{t_1}^{t_2} \int_{\R^3}  C_{vlasov}^{ \tilde{h}^L;k}(t,x)    d x  d t=\sum_{\begin{subarray}{c}
(k_1,k_2)\in \chi_k^1\cup \chi_k^2\cup \chi_k^3, \mu\in \{+, -\}\\ 
L, L_1\in P_n,\mathcal{L}_2\in \mathcal{P}_n, \mathcal{L}_1\circ\mathcal{L}_2\preceq \mathcal{L} \\ 
V^L  \in \{\widetilde{V^{L h_{\alpha\beta}}},\widetilde{V^{  \tilde{h}^L }}, {V^{L h_{\alpha\beta}}}, {V^{  \tilde{h}^L }} \}\\
V_1^{L_1} \in \{\widetilde{V^{L_1 h_{\alpha\beta}}},\widetilde{V^{  \tilde{h}^{L_1} }}, {V^{L_1 h_{\alpha\beta}}}, {V^{  \tilde{h}^{L_1} }} \}\\
\end{subarray} }    K_{k,k_1,k_2}^{\mu;L,L_1,\mathcal{L}_2 }(t_1,t_2;V,V_1),
\ee
where
\[
K_{k,k_1,k_2}^{\mu;L,L_1,\mathcal{L}_2 }(t_1,t_2;V,V_1):=    \int_{t_1}^{t_2} \int_{\R^3}  \int_{\R^3} \int_{\R^3} e^{it|\xi|+it\mu |\xi-\eta|-it \hat{v}\cdot \eta}   \overline{\widehat{V^{L } }(t, \xi)}\widehat{(V^{L_1}_1)^{\mu}}(t,\xi-\eta)   \]
\be\label{2020july9eqn22}
\times \tilde{m} (\xi-\eta, \eta, v) \widehat{u^{\mathcal{L}_2}}(t, \eta, v) \psi_k(\xi)\psi_{k_1}(\xi-\eta)\psi_{k_2}(\eta)  d \xi  d \eta d v,
\ee
where the symbol $\tilde{m}(\xi-\eta, \eta,v)$ satisfies the following estimate, 
\be\label{2020july9eqn23}
 \| (1+|v|)^{-100}  \tilde{m}(\xi-\eta, \eta,v) \|_{L^\infty_v \mathcal{S}^\infty_{k,k_1,k_2}} \lesssim 2^{-k_1}. 
\ee
Similarly, we have 
\[
\int_{t_1}^{t_2}  \int_{\R^3}  CK_{metric; vl}^{  \tilde{h}^L;k,k_1,k_2}(t,x)  \psi_{[m-4, m+1]}(x) d x d t
\]
\be
=  \sum_{\begin{subarray}{c}
  \mu\in \{+, -\},
L, L_1\in P_n,\mathcal{L}_2\in \mathcal{P}_n, \mathcal{L}_1\circ\mathcal{L}_2\preceq \mathcal{L}\\ 
V^L  \in \{\widetilde{V^{L h_{\alpha\beta}}},\widetilde{V^{  \tilde{h}^L }}, {V^{L h_{\alpha\beta}}}, {V^{  \tilde{h}^L }} \}\\
V_1^{L_1} \in \{\widetilde{V^{L_1 h_{\alpha\beta}}},\widetilde{V^{  \tilde{h}^{L_1} }}, {V^{L_1 h_{\alpha\beta}}}, {V^{  \tilde{h}^{L_1} }} \}\\
\end{subarray} }       \int_{\R^3} \varphi(\sigma) T^{K,\mu;L;L_1, \mathcal{L}_2}_{k,k_1,k_2}(t_1, t_2;V,V_1, \sigma) d \sigma,
\ee
\[
\int_{t_1}^{t_2}  \int_{\R^3}  CS_{metric; vl}^{  \tilde{h}^L;k,k_1,k_2}(t,x)  \psi_{[m-4, m+1]}(x) d x d t \]
\[
= \sum_{\begin{subarray}{c}
  \mu\in \{+, -\},
L,   \tilde{L}_1, L_1\in P_n,\mathcal{L}_2\in \mathcal{P}_n, \mathcal{L}_1\circ\mathcal{L}_2\preceq \mathcal{L},  
 \tilde{\mathcal{L}}_1\circ\mathcal{L}_2\preceq \mathcal{L},   {L}_1\prec \mathcal{L} \\ 
V^L  \in \{\widetilde{V^{L h_{\alpha\beta}}},\widetilde{V^{  \tilde{h}^L }}, {V^{L h_{\alpha\beta}}}, {V^{  \tilde{h}^L }} \},  
V_1^{L_1} \in \{\widetilde{V^{L_1 h_{\alpha\beta}}},\widetilde{V^{  \tilde{h}^{L_1} }}, {V^{L_1 h_{\alpha\beta}}}, {V^{  \tilde{h}^{L_1} }} \}\\
\end{subarray} }       \int_{\R^3}  \varphi(\sigma)\big[T^{S_1,\mu;L;\tilde{L}_1, \mathcal{L}_2}_{k,k_1,k_2}(t_1, t_2;V,V_1, \sigma) 
\]
\be\label{2020july9eqn61}
+T^{S_2,\mu;L;L_1, \mathcal{L}_2}_{k,k_1,k_2}(t_1, t_2;V,V_1, \sigma)  \big] d \sigma
\ee
where
\be
\varphi(\sigma):=2^{ 3m}\widehat{\psi}_{[-4,1]}(2^m \sigma),
\ee
  \[
T^{\star,\mu;L;L_1, \mathcal{L}_2}_{k,k_1,k_2}(t_1, t_2;V,V_1, \sigma):= \int_{t_1}^{t_2}  \int_{\R^3}  \int_{\R^3}  \int_{\R^3}  e^{it|\xi-\sigma|+it\mu |\xi-\eta|-it \hat{v}\cdot \eta}  \overline{\widehat{V^{L } }(t, \xi-\sigma)}  \widehat{(V^{L_1}_1)^{\mu}}(t,\xi-\eta)  \]
\be\label{2020july8eqn21}
\times  m_{\star}(\xi-\eta, \eta, v) \widehat{u^{\mathcal{L}_2}}(t, \eta, v)\psi_k(\xi)\psi_{k_1}(\xi-\eta)\psi_{k_2}(\eta) d \xi  d \eta d v,\quad \star\in \{P, Q, S_1, S_2\} ,
\ee
and   the symbols $m_{\star}(\xi-\eta, \eta, v),  \star\in \{P, Q, S_1, S_2\} ,$ satisfies the following estimate,
\be\label{symbolestimate2020july} 
\begin{split}
 \sum_{\star\in \{P, Q\}   }\| (1+|v|)^{-100} m_{\star}(\xi-\eta, \eta, v) \|_{L^\infty_v \mathcal{S}^\infty_{k,k_1,k_2}}& \lesssim 2^{-k_2},\\
\| (1+|v|)^{-100} m_{S_1}(\xi-\eta, \eta, v) \|_{L^\infty_v \mathcal{S}^\infty_{k,k_1,k_2}} &\lesssim  2^{-k_1},\\ 
  \| (1+|v|)^{-100} m_{S_2 }(\xi-\eta, \eta, v) \|_{L^\infty_v \mathcal{S}^\infty_{k,k_1,k_2}} &\lesssim  2^{k_1-2k_2}.
\end{split}
\ee

Recall (\ref{2020july9eqn22}) and the estimates   (\ref{2020july9eqn23}) and (\ref{symbolestimate2020july}), we know that  $K_{k,k_1,k_2}^{\mu;L,L_1,\mathcal{L}_2 }(t_1,t_2;V,V_1)$ is a special case of trilinear form $T^{S_1,\mu;L;L_1, \mathcal{L}_2}_{k,k_1,k_2}(t_1, t_2;V,V_1,   \sigma)$.
Recall (\ref{2020june24eqn32}). We know that the trilinear form of $ \tilde{C}_{metric;2}^{ \tilde{h}^L;k,  k'}(t,x)$ is of same type as $CK_{metric; vl}^{  \tilde{h}^L;k,k_1,k_2}(t,x), K\in \{P, Q\} $ with $k'=k_1$ and $k_2=k$. 

With the above preparations, by assuming the validity of a technical Lemma, Lemma \ref{oscillationwavevlasov},  we  state and prove the main result of this subsection as follows.

\begin{proposition}\label{waveVlaovestimate}
 Under the bootstrap assumptions \textup{(\ref{BAmetricold})} and \textup{(\ref{BAvlasovori})}, the following estimates hold  for any $t_1, t_2\in [2^{m-1}, 2^m]\subset [0,T]$, $\tilde{h}\in \{F, \underline{F}, \omega_j, \vartheta_{ij}\},$   $L\in \cup_{l\leq N_0 } P_l   $,  
\[
 \sum_{\begin{subarray}{c} 
 k\in \Z, (k_1,k_2)\in \chi_k^1\cup \chi_k^2\cup  \chi_k^3\\
 k\geq - m + \alpha_{out}(|L|) \delta_1 m \\
   k_1=\min\{k_1,k_2\}\leq -m + \alpha_{inp}(|L|) \delta_1 m-10 
 \end{subarray} }    2^{-k_{-}+2\gamma k_{-}+ 2 (N_0-|L|) k_{+}} \big| \int_{t_1}^{t_2}  \int_{\R^3}  CS_{metric; vl}^{  \tilde{h}^L;k,k_1,k_2}(t,x)  \psi_{[m-4, m+1]}(x) d x d t\big| 
\]
\be\label{2020july7eqn2}
\lesssim   \epsilon_1^3,
\ee
\[
\sum_{ 
 k\in[- m + \alpha_{out}(|L|) \delta_1 m, \infty)\cap \Z
 } 2^{-k_{-}+2\gamma k_{-}+ 2 (N_0-|L|) k_{+}} \Big[
 \sum_{\begin{subarray}{c}
  H\in \{S, P, Q\}, 
  (k_1, k_2)\in \chi_k^1\cup \chi_k^2\cup \chi_k^3,\\ 
   \min\{k_1,k_2\}\geq -m + \alpha_{inp}(|L|) \delta_1 m-10 
 \end{subarray} }   \big| \int_{t_1}^{t_2} \int_{\R^3}  CH_{metric; vl}^{  \tilde{h}^L;k,k_1,k_2}(t,x) 
\]
\be\label{2020july7eqn1}
 \times  \psi_{[m-4, m+1]}(x) d x d t\big| + \sum_{ k'\geq -2 |\alpha_{out}(|L|)\delta_1 m} \big| \int_{t_1}^{t_2} \int_{\R^3}   \tilde{C}_{metric;2}^{ \tilde{h}^L;k,  k'}(t,x) \psi_{[m-4, m+1]}(x) d x d t\big| \Big]\lesssim   \epsilon_1^3,
\ee
\be\label{2020july9eqn31}
\sum_{k\in \mathbb{Z}}      2^{-k_{-}+2\gamma k_{-}+ 2 (N_0-|L|) k_{+}}\big| \int_{t_1}^{t_2} \int_{\R^3}  C_{vlasov}^{ \tilde{h}^L;k}(t,x)    d x  d t\big| \lesssim \epsilon_1^3.
\ee
\end{proposition}
\begin{proof}

As discussed at the beginning of this section, the proof of our desired estimates is reduced to the estimate of trilinear forms $T^{K,\mu;L;\tilde{L}_1, \mathcal{L}_2}_{k,k_1,k_2}(t_1, t_2;V,V_1,  \sigma), K\in \{S_1, S_2, P, Q\}$ for different choices of $k,k_1,k_2$ in different settings. 
 
  Recall (\ref{2020july9eqn22}) and (\ref{2020july8eqn21}). From the estimates of symbols in (\ref{symbolestimate2020july}), the bilinear estimate \eqref{june23eqn2} in Lemma \ref{wavevlasovbi1} for the case $|L_1|, |\tilde{L}_1|\geq |L|-5$,  and the bilinear estimates \eqref{2020july8bilinear} and (\ref{bilinearestimate2}) in Lemma \ref{wavevlabil6sep} for the case $|L_1|, |\tilde{L}_1|\leq |L|-5$, the following estimate holds for any fixed  $k\in \Z,(k_1,k_2)\in \chi_k^1\cup \chi_k^2\cup \chi_k^3$, $L,   \tilde{L}_1, L_1\in P_n,\mathcal{L}_2\in \mathcal{P}_n$, s.t., $ \mathcal{L}_1\circ\mathcal{L}_2\preceq \mathcal{L}, \tilde{\mathcal{L}}_1\circ\mathcal{L}_2\preceq \mathcal{L},  {L}_1\prec \mathcal{L}$,
\[
\sum_{\star\in \{S_1,   P, Q\}} 2^{-k_{-}+2\gamma k_{-}+ 2 (N_0-|L|)k_{+}}\big(\big|K_{k,k_1,k_2}^{\mu;L,L_1,\mathcal{L}_2 }(t_1,t_2;V,V_1)\big|+ \sup_{\sigma \in \R^3}\big| T^{\star,\mu;L;\tilde{L}_1, \mathcal{L}_2}_{k,k_1,k_2}(t_1, t_2;V,V_1, \sigma)\big|
\] 
\be\label{2020july9eqn70}
+ \big| T^{S_2,\mu;L; {L}_1, \mathcal{L}_2}_{k,k_1,k_2}(t_1, t_2;V,V_1, \sigma)\big|\big)  \lesssim 2^{2d(|L|+1, N_0-|L|)\delta_0 m + \max\{k_1, k_2\} }\epsilon_1^3.
\ee
The above estimate allows us to rule out the case $\max\{k_1, k_2\}\leq -2d(|L|+1, N_0-|L|)\delta_0 m$.

 Moreover, we rule out the case $|\sigma|$ is not so small.  Recall (\ref{2020july9eqn61}). From the rapidly decay rate of $\varphi(\sigma)$ if $|\sigma|\geq 2^{-m + \delta_1 m }$ and the rough estimate,  the following estimate  holds,
\[
\sum_{\star\in \{S_1,   P, Q\}} \sum_{k\in \Z, (k_1,k_2)\in \chi_k^1\cup \chi_k^2\cup \chi_k^3         } 2^{-k_{-}+2\gamma k_{-}+ 2 (N_0-|L|)k_{+}} \int_{|\sigma|\geq 2^{-m+\delta_1 m }} |\varphi(\sigma)|\big(|T^{\star,\mu;L;\tilde{L}_1, \mathcal{L}_2}_{k,k_1,k_2}(t_1, t_2;V,V_1, \sigma)|\]
 \[
 + |T^{S_2,\mu;L;\tilde{L}_1, \mathcal{L}_2}_{k,k_1,k_2}(t_1, t_2;V,V_1, \sigma)|\big) d \sigma \lesssim  2^{-m/4}  \epsilon_1^3. 
\]

  It remains to consider the case $\max\{k_1, k_2\}\geq -2d(|L|+1, N_0-|L|)\delta_0 m$ and fixed $\sigma\in \R^3$, s.t., $|\sigma|\leq 2^{-m + \delta_1 m }.$  Based on the possible size of $\mathcal{L}_2$, we separate into two cases as follows.

$\bullet$\quad If $|\mathcal{L}_2|\leq |L|-5$.

For this case, the total number  of vector fields act on the Vlasov part is far away from the top order and we mainly use   the decay estimate for the density type function. More precisely, from the estimate of symbol in (\ref{symbolestimate2020july}) and  the estimate (\ref{densitydecay}) in Lemma \ref{decayestimateofdensity}, the following estimate holds for any fixed $k\in \Z$, s.t.,  $ k\geq - m + \alpha_{out}(|L|) \delta_1 m$, 
\[
 \sum_{\mathcal{L}_1\circ \mathcal{L}_2\preceq \mathcal{L} ,   |\mathcal{L}_2|\leq |L|-5 } \sum_{\star\in \{P, Q\}} \sum_{ (k_1,k_2)\in \chi_k^1\cup \chi_k^2\cup \chi_k^3} \sup_{\sigma \in \R^3} 2^{-k_{-}+2\gamma k_{-}+ 2 (N_0-|L|) k_{+}} \big|T^{ \star, \mu ;L;L_1, \mathcal{L}_2}_{k,k_1,k_2}(t_1, t_2;V,V_1, \sigma)\big|    
 \]
 \[
  \lesssim \sum_{k_2\in \Z} 2^{-k_{-}/2+ m+  2d(|L|, N_0-|L|)\delta_0 m  } \min\{ 2^{2k_2} + 2^{3k_2+ d(|L|, N_0-|L|) \delta_0 m}, 2^{-4m-2 k_2} 
 \]
 \be\label{2020july9eqn41}
+  2^{-4m-k_2+ d(|L|, N_0-|L|) \delta_0 m} \} \lesssim 2^{-m/3}\epsilon_1^3.
 \ee
 Similarly, from the estimate of symbol in (\ref{2020july9eqn23}), the following estimate holds for any fixed $k\in \Z$, 
 \[
  \sum_{\begin{subarray}{c}
   \mathcal{L}_1\circ \mathcal{L}_2\preceq \mathcal{L} ,   |\mathcal{L}_2|\leq |L|-5\\
   k\in \Z, (k_1,k_2)\in \chi_k^1\cup \chi_k^2\cup \chi_k^3\\
  \end{subarray}
  } 
  2^{-k_{-}+2\gamma k_{-}+ 2 (N_0-|L|)  k_{+}}  \big[ \big|K_{k,k_1,k_2}^{\mu;L,L_1,\mathcal{L}_2 }(t_1,t_2;V,V_1)\big| + \sup_{\sigma\in \R^3}\big|T^{ S_1, \mu ;L;L_1, \mathcal{L}_2}_{k,k_1,k_2}(t_1, t_2;V,V_1, \sigma)\big|  \big]\]
  \[ \lesssim \sup_{k_2\in \Z} \sum_{k_1\leq k_2-10 } 2^{  m+  2d(|L|, N_0-|L|)\delta_0 m  }  \min\{2^{k_1}, 2^{-k_1/2}\min\{ 2^{3k_2} + 2^{4k_2+ d(|L|, N_0-|L|) \delta_0 m}, 2^{-4m-  k_2} \]
 \be\label{2020july9eqn42}
 +  2^{-5m-k_2+ d(|L|, N_0-|L|) \delta_0 m} \}\}\epsilon_1^3    +2^{-m/3}\epsilon_1^3 \lesssim 2^{-m/3}\epsilon_1^3. 
 \ee

It remains to consider $T^{ S_1, \mu ;L;L_1, \mathcal{L}_2}_{k,k_1,k_2}(t_1, t_2;V,V_1, \sigma)$. Recall (\ref{symbolestimate2020july}). Since the    symbol $m_{S_2}(\xi-\eta, \eta, v) $ is worse when   $k_2\leq k_1-10$, this term is more delicate. By using  the estimate (\ref{densitydecay}) in Lemma \ref{decayestimateofdensity}, we first  rule out the case $k_2\geq -m +2d(|L|, N_0-|L|)\delta_0 m $ as follows, 
\[
 \sum_{  \begin{subarray}{c} 
 {\mathcal{L}}_1\circ \mathcal{L}_2\preceq \mathcal{L},  {\mathcal{L}}_1\prec  \mathcal{L}, |\mathcal{L}_2|\leq |L|-5 \\ 
 k\in \Z, (k_1,k_2)\in \chi_k^1\cup \chi_k^2\cup \chi_k^3\\ 
  k_2\geq -m +2 d(|L|, N_0-|L|)\delta_0 m
 \end{subarray} }   2^{-k_{-}+2\gamma k_{-}+ 2 (N_0-|L|)  k_{+}}   \sup_{\sigma\in \R^3} \big|T^{ S_2;\mu ;L; {L}_1, \mathcal{L}_2}_{k,k_1,k_2}(t_1, t_2;V,V_1, \sigma)\big|  
\]
\[
\lesssim \sum_{k_2\geq -m +2d(|L|, N_0-|L|)\delta_0 m} 2^{m+  2d(|L|, N_0-|L|)\delta_0 m  } \big( 2^{-3m-  2k_2} +  2^{-3m-k_2+d(|L|, N_0-|L|)\delta_0 m} \big)
\]
 \be\label{2020july9eqn43}
+ 2^{-m/3}\epsilon_1^3\lesssim \epsilon_1^3.
\ee

It remains to estimate $T^{ S_2;\mu ;L; {L}_1, \mathcal{L}_2}_{k,k_1,k_2}(t_1, t_2, \sigma)$ for the case $k_1\geq -2d(|L|+1, N_0-|L|)\delta_0m$ and $k_2\in [-m +\alpha_{input}(|L|)\delta_1 m -10, -m+2d(|L|, N_0-|L|)\delta_0 m ]$. Recall (\ref{2020july9eqn61}). We have  $| {L}_1|< |L|$ for the estimate of  $T^{ S_2;\mu ;L; {L}_1, \mathcal{L}_2}_{k,k_1,k_2}(t_1, t_2, \sigma)$.  From the bilinear estimate (\ref{2020july8bilinear}) in Lemma \ref{wavevlabil6sep} and the estimate (\ref{oct4eqn41}) in Proposition \ref{fixedtimenonlinarityestimate}, we have
\[
 \sum_{  \begin{subarray}{c}
   {\mathcal{L}}_1\circ \mathcal{L}_2\preceq \mathcal{L}, \tilde{\mathcal{L}}_1\prec  \mathcal{L}, |\mathcal{L}_2|\leq N_0-5 \\ 
   k_2 \leq -m+2 d(|L|, N_0-|L|) \delta_0 m \\  k_1\geq -2 d(|L|+1, N_0-|L|) \delta_0 m 
 \end{subarray}
}    2^{-k_{-}+2\gamma k_{-}+ 2 (N_0-|L|) k_{+}} \sup_{\sigma\in \R^3} \big|T^{S_2, \mu ;L;\tilde{L}_1, \mathcal{L}_2}_{k,k_1,k_2}(t_1, t_2;V,V_1,  \sigma)\big|  
\]
 \be\label{2020july8eqn22}
\lesssim   \sum_{k_2 \leq -m+2 d(|L|, N_0-|L|) \delta_0 m  } 2^{m + 2d(|L|+1, N_0-|L|)\delta_0 m +k_1-2k_2 }  2^{-m-k_1 +  3k_2 } \epsilon_1^3  \lesssim 2^{-m/2}\epsilon_1^3. 
 \ee

 $\bullet$\quad If $|\mathcal{L}_2|> |L|  -5$ and $|L|\geq N_0-5$. 

 For this case, there are at most five vector fields act on the metric component. From the bilinear estimate (\ref{bilinearestimate2}) in Lemma \ref{wavevlabil6sep}, the $L^\infty_x$ decay estimate in (\ref{2020julybasicestiamte})  in Lemma \ref{basicestimates}, and the estimate (\ref{oct4eqn41}) in Proposition \ref{fixedtimenonlinarityestimate}, the following estimate holds,  
 \[
 \sum_{\begin{subarray}{c}
  \mathcal{L}_1\circ \mathcal{L}_2\preceq \mathcal{L}  \\ 
     |\mathcal{L}_2|> |L|-5, \star\in \{P, Q\}
 \end{subarray}} \sum_{\begin{subarray}{c}
 k\in \Z, (k_1,k_2)\in  \chi_k^1\cup\chi_k^2\cup \chi_k^3\\
 \min\{k_1,k_2\}\geq -m +\alpha_{inp}(|L|)\delta_1 m
 \end{subarray} }   2^{-k_{-}+2\gamma k_{-}+ 2 (N_0-|L|)  k_{+}}  \sup_{\sigma\in \R^3 }   \big|T^{ \star,\mu;L;L_1, \mathcal{L}_2}_{k,k_1,k_2}(t_1, t_2;V,V_1,  \sigma)\big|   
 \]
 \be\label{2020july8eqn50}
 \lesssim    2^{m   + 3d(|L|+1, N_0-|L|)\delta_0 m } 2^{-2m }\epsilon_1^3 \lesssim  2^{-m/4} \epsilon_1^3.
 \ee
 Similarly, we have
\[
 \sum_{ \begin{subarray}{c}
 \tilde{\mathcal{L}}_1\circ \mathcal{L}_2\preceq \mathcal{L},     |\mathcal{L}_2|> |L|-5, 
 k\in \Z, (k_1,k_2)\in \chi_k^1\cup\chi_k^2\cup \chi_k^3\\ 
 k_1\geq - m + 4 d(|L|+1, N_0-|L|)\delta_0 m 
 \end{subarray}} \sup_{\sigma\in \R^3}  2^{-k_{-}+2\gamma k_{-}+ 2 (N_0-|L|)  k_{+}}    \big(\big|T^{S_1, \mu ;L;\tilde{L}_1, \mathcal{L}_2}_{k,k_1,k_2}(t_1, t_2;V,V_1,\sigma)\big| 
 \]
 \[
   + \big|T^{S_2, \mu ;L;\tilde{L}_1, \mathcal{L}_2}_{k,k_1,k_2}(t_1, t_2;V,V_1, \sigma)\big|  + \big| K_{k,k_1,k_2}^{\mu;L, \tilde{L}_1,\mathcal{L}_2 }(t_1,t_2;V,V_1 )\big| \big)
 \]
 \be\label{2020july8eqn51}
 \lesssim \sum_{k_1\geq - m +  4 d(|L|+1, N_0-|L|)\delta_0 m  } 2^{m   + 3d(|L|+1, N_0-|L|)\delta_0 m } 2^{-2m}  2^{-k_1}  \epsilon_1^3 \lesssim  \epsilon_1^3.
 \ee

From the estimate (\ref{2020july8eqn44}) in Lemma \ref{oscillationwavevlasov}, the following estimate holds if $k_1\leq - m + 4  d(|L|+1, N_0-|L|)\delta_0 m $ and  $  k_2 \geq -2d(|L|+1, N_0-|L|)\delta_0 m$.
\[
  \sum_{ \begin{subarray}{c}
 \tilde{\mathcal{L}}_1\circ \mathcal{L}_2\preceq \mathcal{L},     |\mathcal{L}_2|> |L|-5\\ 
 k\in \Z, (k_1,k_2)\in \chi_k^1\cup\chi_k^2\cup \chi_k^3\\ 
 k_1\leq - m +   4 d(|L|+1, N_0-|L|)\delta_0 m    
 \end{subarray}} \sup_{\sigma\in \R^3,|\sigma|\leq 2^{-m+\delta_1 m }}   2^{-k_{-}+2\gamma k_{-}+ 2(N_0-|L|)k_{+}}      \big(\big|T^{S_1, \mu ;L;\tilde{L}_1, \mathcal{L}_2}_{k,k_1,k_2}(t_1, t_2;V,V_1,\sigma)\big|
 \]
\be\label{2020july8eqn52}
  + \big|T^{S_2, \mu ;L;\tilde{L}_1, \mathcal{L}_2}_{k,k_1,k_2}(t_1, t_2;V,V_1,\sigma)\big|  + \big| K_{k,k_1,k_2}^{\mu;L, \tilde{L}_1,\mathcal{L}_2 }(t_1,t_2;V,V_1)\big| \big)
      \lesssim 2^{-m/4} \epsilon_1^3.
 \ee

$\bullet$\quad If $|\mathcal{L}_2|> |L|  -5$ and $|L|\leq N_0-5$.  

The main difference between this case and the previous case  is that we pay attention to the possible issue of losing derivative.

 If $(k_1,k_2)\in \chi_k^1 \cup \chi_k^3$, i.e., $k_2\leq k_1+10$, then the argument used in the case $|\mathcal{L}_2|\leq |L|-5$ is still applicable since $|L|\leq N_0-5$. As a result, we have
\[
  \sum_{\begin{subarray}{c} 
  \star\in \{S_1,   P, Q\},
 k\in \Z, (k_1,k_2)\in \chi_k^1\cup   \chi_k^3\\
 k\geq - m + \alpha_{out}(|L|) \delta_1 m \\
 \end{subarray} }   2^{-k_{-}+2\gamma k_{-}+ 2 (N_0-|L|) k_{+}}\big(\big|K_{k,k_1,k_2}^{\mu;L,L_1,\mathcal{L}_2 }(t_1,t_2;V,V_1)\big|
\] 
\be\label{2020sep18eqn1}
+ \sup_{\sigma \in \R^3}\big| T^{\star,\mu;L;\tilde{L}_1, \mathcal{L}_2}_{k,k_1,k_2}(t_1, t_2;V,V_1, \sigma)\big| + \big| T^{S_2,\mu;L; {L}_1, \mathcal{L}_2}_{k,k_1,k_2}(t_1, t_2;V,V_1, \sigma)\big|\big) \lesssim  \epsilon_1^3. 
\ee
 
 If $(k_1,k_2)\in \chi_k^2$, i.e., $k_1\leq k_2-10, |k-k_2|\leq 10$, then the argument used in the case $|\mathcal{L}_1|>|L|-5, |L|\geq N_0-5$ is still applicable. As  a result, we have
\[
  \sum_{\begin{subarray}{c} 
  \star\in \{S_1, P, Q\}, 
 k\in \Z, (k_1,k_2)\in \chi_k^2 \\
 k\geq - m + \alpha_{out}(|L|) \delta_1 m \\
 \end{subarray} }   2^{-k_{-}+2\gamma k_{-}+ 2N(|L|)k_{+}} \sup_{\sigma\in \R^3,|\sigma|\leq 2^{-m+\delta_1 m }}  \big(\big|K_{k,k_1,k_2}^{\mu;L,L_1,\mathcal{L}_2 }(t_1,t_2;V,V_1)\big|
\] 
\be\label{2020sep18eqn3}
+  \big| T^{\star,\mu;L;\tilde{L}_1, \mathcal{L}_2}_{k,k_1,k_2}(t_1, t_2;V,V_1, \sigma)\big| + \big| T^{S_2,\mu;L; {L}_1, \mathcal{L}_2}_{k,k_1,k_2}(t_1, t_2;V,V_1, \sigma)\big|   \big)  \lesssim  \epsilon_1^3. 
\ee
Hence finishing the proof of desired estimates (\ref{2020july7eqn2}), (\ref{2020july7eqn1}) and (\ref{2020july9eqn31}).

\end{proof}

  For the rest of this subsection, we prove the assumed Lemma \ref{oscillationwavevlasov} in the proof of the above Proposition. 

\begin{lemma}\label{oscillationwavevlasov}
 Given any $t_1, t_2\in [2^{m-1}, 2^m]\subset [0,T]$, $L, L_1 , L_2\in \cup_{l\leq N_0 } {P}_l, $   s.t., $ \mathcal{L}_1\circ\mathcal{L}_2\preceq \mathcal{L}   $   .   Under the bootstrap assumptions \textup{(\ref{BAmetricold})} and \textup{(\ref{BAvlasovori})},  the following estimate holds for   the trilinear form defined in \textup{(\ref{2020july8eqn21})}, 
\[
 \sum_{\begin{subarray}{c}
\star\in \{S_1,S_2\}, k\in \Z, (k_1,k_2)\in \chi_k^2\\
  k_1\leq -m/2, k\geq - 4d(N_0+1,0)\delta_0 m
  \end{subarray}}\sup_{\sigma\in \R^3, |\sigma|\leq 2^{-m+\delta_1 m}}  2^{-k_{-}+2\gamma k_{-}+ 2(N_0-|L|)k_{+}}  \big|T^{ \star,\mu ;L;L_1, \mathcal{L}_2}_{k,k_1,k_2}(t_1, t_2;V,V_1, \sigma)\big|
\]
 \be\label{2020july8eqn44}
\lesssim 2^{-m/4} \epsilon_1^3. 
 \ee
\end{lemma}
\begin{proof}
Recall  \textup{(\ref{2020july8eqn21})}.    Based on the size of $v$, we decompose $ T^{   \star,\mu ;L;L_1, \mathcal{L}_2}_{k,k_1,k_2}(t_1, t_2)$ into two parts as follows,
\begin{multline}\label{2020july8eqn41}
T^{   \star,\mu  ;L;L_1, \mathcal{L}_2}_{k,k_1,k_2}(t_1, t_2;V,V_1,\sigma)= T^{  \star,\mu  ; L;L_1, \mathcal{L}_2;\geq  }_{k,k_1,k_2}(t_1, t_2,\sigma) +T^{ \star,\mu  ; L;L_1, \mathcal{L}_2;< }_{k,k_1,k_2}(t_1, t_2,\sigma),\\
  T^{ \star,\mu  ; L;L_1, \mathcal{L}_2;\diamondsuit  }_{k,k_1,k_2}(t_1, t_2,\sigma):=\int_{t_1}^{t_2} \int_{\R^3} \overline{\widehat{V^{L } }(t, \xi-\sigma)}e^{it|\xi-\sigma|+it\mu |\xi-\eta|-it \hat{v}\cdot \eta} \widehat{(V^{L_1}_1)^{\mu}}(t,\xi-\eta)      \\ 
\times  \widehat{u^{\mathcal{L}_2}}(t, \eta, v) m_{\star}(\xi-\eta, \eta, v)
\psi_k(\xi)\psi_{k_1}(\xi-\eta)\psi_{k_2}(\eta) \psi_{ \diamondsuit m/10}(v) d \eta d v, \quad \diamondsuit \in \{\geq, <\}.
\end{multline} 

For the  case  $|v|$ is relatively large, we gain sufficient decay rates from the high moment of $v$. More precisely, from the $L^2-L^2-L^\infty$ type multilinear estimate and the estimate of symbol in (\ref{symbolestimate2020july}), we have
\be\label{2020july8eqn45}
\sum_{k\in\Z, (k_1,k_2)\in\chi_k^2} 2^{-k_{-}+2\gamma k_{-}+ 2 (N_0-|L|) k_{+}}  \big| T^{ \star,\mu  ; L;L_1, \mathcal{L}_2;\geq  }_{k,k_1,k_2}(t_1, t_2, \sigma) \big|\lesssim 2^{-m}\epsilon_1^3. 
\ee

Now, we focus on the estimate of  $ T^{ \star,\mu  ; L;L_1, \mathcal{L}_2; <   }_{k,k_1,k_2}(t_1, t_2,\sigma)$. Note that the following estimate holds if $|v|\lesssim  2^{m/10}$ and $|\sigma|\leq 2^{-m+\delta_1 m}$,
 \be\label{july21eqn16}
 |\xi-\sigma|+ \mu |\xi-\eta| - \eta \cdot \hat{v} \gtrsim \frac{|\eta|}{1+|v|^2}-2 |\xi-\eta|-|\sigma|\gtrsim \frac{ |\eta|}{2(1+|v|^2)}\gtrsim 2^{k_2 }(1+|v|^2)^{-1}. 
\ee

With the above observation, to take the advantage of high oscillation in time,  we do integration by parts in time once  for  $ T^{ \star,\mu  ; L;L_1, \mathcal{L}_2; <   }_{k,k_1,k_2}(t_1, t_2,\sigma)$. As a result, we have
\be\label{2020july8eqn42}
T^{ \star,\mu  ; L;L_1, \mathcal{L}_2; <  }_{k,k_1,k_2}(t_1, t_2,\sigma) = End^{\star,\mu  ; L;L_1, \mathcal{L}_2}_{k,k_1,k_2}(t_1, t_2) + J^{\star,\mu  ; L;L_1, \mathcal{L}_2;1}_{k,k_1,k_2}(t_1, t_2)+ J^{\star,\mu  ; L;L_1, \mathcal{L}_2;2}_{k,k_1,k_2}(t_1, t_2),
 \ee
\begin{multline}
End^{ \star,\mu  ; L;L_1, \mathcal{L}_2}_{k,k_1,k_2}(t_1, t_2):=\sum_{i=1,2}(-1)^{i-1} \int_{\R^3} \int_{\R^3}      
 \overline{\widehat{U^{L } }(t_i, \xi-\sigma)}  \widehat{(U^{L_1}_1)^{\mu}}(t_i,\xi-\eta) \widehat{ {\mathcal{L}_2f}}(t_i, \eta, v)\\
 \times  m_{k,k_1,k_2}^{  \star,\mu }(\xi-\eta, \eta, v) d \eta d v,\\
J^{ \star,\mu  ; L;L_1, \mathcal{L}_2;1}_{k,k_1,k_2}(t_1, t_2)= \int_{t_1}^{t_2} \int_{\R^3}  \int_{\R^3} e^{it|\xi-\sigma|+it\mu |\xi-\eta|-it \hat{v}\cdot \eta}\p_t\big(\overline{  \widehat{V^{L } }(t , \xi-\sigma)}  \widehat{(V^{L_1}_1)^{\mu}}(t ,\xi-\eta)\big)\\ 
\times  \widehat{ u^{\mathcal{L}_2 }}(t , \eta, v) m_{k,k_1,k_2}^{ \star,\mu }(\xi-\eta, \eta, v) d \eta d v d t,\\ 
J^{  \star,\mu  ; L;L_1, \mathcal{L}_2;2 }_{k,k_1,k_2}(t_1, t_2)= \int_{t_1}^{t_2} \int_{\R^3}  \int_{\R^3} e^{it|\xi-\sigma|+it\mu |\xi-\eta|-it \hat{v}\cdot \eta}\overline{  \widehat{V^{L } }(t , \xi-\sigma)}  \widehat{(V^{L_1}_1)^{\mu}}(t ,\xi-\eta) \\ 
\times \p_t\big(\widehat{ u^{\mathcal{L}_2 }}(t , \eta, v)\big) m_{k,k_1,k_2}^{  \star,\mu }(\xi-\eta, \eta, v) d \eta d v d t,
\end{multline}
where the symbol $m_{k,k_1,k_2}^{  \star,\mu }(\xi-\eta, \eta, v)$ is defined as follows, 
\[
m_{k,k_1,k_2}^{ \star,\mu }(\xi-\eta, \eta, v):= m_{\star}(\xi-\eta, \eta, v) 
  \psi_k(\xi)\psi_{k_1}(\xi-\eta)\psi_{k_2}(\eta) \psi_{  < m/10}(v)(i |\xi-\sigma|+ \mu |\xi-\eta| - \eta \cdot \hat{v})^{-1}.
\]
From the estimate of symbol    in (\ref{symbolestimate2020july}), the estimate (\ref{july21eqn16}),  the $L^2-L^2-L^\infty$ type multilinear estimate, and the $L^\infty\longrightarrow L^2$ type Sobolev embedding, we have
\[
 \sum_{\star\in \{S_1,S_2\}} \sum_{k\in\Z, (k_1,k_2)\in\chi_k^2, k_1\leq -m/2, k\geq - 4d(N_0+1,0)\delta_0 m}  2^{-k+2\gamma k +2 (N_0-|L|)k_{+}}  \big|End^{\star,\mu  ; L;L_1, \mathcal{L}_2}_{k,k_1,k_2}(t_1, t_2)\big|
\]
\be\label{2020july8eqn46}
\lesssim\sum_{k_1\leq -m/2, k\geq - 4d(N_0+1,0)\delta_0 m}  2^{2d(|L|, N_0-|L|)\delta_0 m  -k_2} 2^{k_1}\epsilon_1^3\lesssim 2^{-m/3} \epsilon_1^3. 
\ee 

Recall the estimate of phase in  (\ref{july21eqn16}). Note that we have $|k-k_2|\leq 10$ for the case we are considering.  After doing integration by parts in time, the symbol is one derivative smoother. Hence,  there is no losing derivatives issue for the estimates of $J^{  L;L_1, \mathcal{L}_2;i}_{k,k_1,k_2}(t_1, t_2), i\in \{1,2\}$. By using the same strategy used in the estimate of endpoint case, the following estimate holds from the estimate (\ref{oct4eqn41}) in Proposition \ref{fixedtimenonlinarityestimate},
\[
 \sum_{\star\in \{S_1,S_2\}}  \sum_{k\in\Z, (k_1,k_2)\in\chi_k^2, k_1\leq -m/2, k\geq - 4 d(N_0+1,0) \delta_0 m} 2^{-k+2\gamma k +2 (N_0-|L|)k_{+}}  \big| J^{ \star,\mu  ; L;L_1, \mathcal{L}_2;1}_{k,k_1,k_2}(t_1, t_2)\big|
\]
\be\label{2020july8eqn47}
  \lesssim 2^{-m/3} \epsilon_1^3. 
\ee

Lastly, we estimate $J^{  L;L_1, \mathcal{L}_2;2}_{k,k_1,k_2}(t_1, t_2).$ Recall the equation satisfied by $ {u}^{\mathcal{L}}(t,x,v)$ in (\ref{2020april7eqn6}) and the decomposition in  (\ref{2020april14eqn23}).  The following equality holds after doing integration by parts in $v$ for the nonlinearity $D_v \cdot\mathfrak{R}_2^{\mathcal{L}_2}(t,x,v)$ and $  D_v \cdot  \mathfrak{Q}_2^{\mathcal{L
 }_2} (t,x,v)$,
\[
J^{\star,\mu  ;  L;L_1, \mathcal{L}_2;2 }_{k,k_1,k_2}(t_1, t_2)=\int_{t_1}^{t_2} \int_{\R^3}  \int_{\R^3} e^{ -it \hat{v}\cdot \eta}\overline{  \widehat{U^{L } }(t , \xi)}  \widehat{(U^{L_1}_1)^{\mu}}(t ,\xi-\eta)
\]
\[
\times \big[  \big(  \widehat{\mathfrak{Q}_3^{\mathcal{L
 }_2}  }(t , \eta, v)  +\sum_{i=3,4} \Lambda_{\geq 3}[\widehat{ \mathfrak{R}_i^{\mathcal{L}_2 }}](t , \eta, v) + i \eta \cdot (\widehat{ \mathfrak{R}_1^{\mathcal{L}_2 }}(t , \eta, v)+\widehat{\mathfrak{Q}_1^{\mathcal{L
 }_2}}(t, \eta, v) \big) m_{k,k_1,k_2}^{ \mu}(\xi-\eta, \eta, v) \]
\[
    +   \big(\widehat{\mathfrak{Q}_2^{\mathcal{L
 }_2}  }(t , \eta, v) + \widehat{ \mathfrak{R}_2^{\mathcal{L}_2 }}(t , \eta, v) \big)  \cdot  \nabla_v m_{k,k_1,k_2}^{ \mu}(\xi-\eta, \eta, v) \big] d \eta d v d t.
\] 

From  the estimate of symbol    in (\ref{symbolestimate2020july}), the estimate (\ref{july21eqn16}),  the $L^2-L^2-L^\infty$ type multilinear estimate,   the $L^\infty\longrightarrow L^2$ type Sobolev embedding,   the estimate (\ref{aug9eqn87}) in Lemma \ref{estimateofremaindervlasov}, and the estimate  (\ref{2020april9eqn21}) in Lemma \ref{estimateofremaindervlasov3}, and the obtained estimate \eqref{2020april14eqn21},  we have 
\be\label{2020july8eqn48}
 \sum_{\star\in \{S_1,S_2\}}  \sum_{k\in\Z, (k_1,k_2)\in\chi_k^2, k_1\leq -m/2, k\geq - 4  d(N_0+1,0) \delta_0 m} 2^{-k+2\gamma k +2  (N_0-|L|) k_{+}}  \big| J^{  L;L_1, \mathcal{L}_2;2}_{k,k_1,k_2}(t_1, t_2)\big|\lesssim 2^{-m/3} \epsilon_1^3.
\ee
Recall the decompositions (\ref{2020july8eqn41}) and  (\ref{2020july8eqn42}). To sum up, our desired estimate (\ref{2020july8eqn44}) holds from the estimates (\ref{2020july8eqn45}) and (\ref{2020july8eqn46}--\ref{2020july8eqn48}).

\end{proof}

\subsection{The estimate of wave-wave-wave type interaction }\label{wavewavewave}

In the following Lemma, we first rule out the very small input frequency case for the wave-wave-wave type interaction.

\begin{lemma}\label{lowinputfrequency}
 Under the bootstrap assumptions \textup{(\ref{BAmetricold})} and \textup{(\ref{BAvlasovori})}, the following estimate holds for any $t\in [2^{m-1}, 2^m]\subset [0,T]$,  $L\in \cup_{n\in[0, N_0]  } P_n, k \in \mathbb{Z}$,   
 \[
   \sum_{\begin{subarray}{c}
 k\in \Z, (k_1, k_2)\in \chi_k^1\cup \chi_k^2\cup \chi_k^3, H\in \{S,P,Q\}\\ 
 \min\{k_1,k_2\}\leq -m + \alpha_{inp}(|L|) \delta_1 m,
 k\geq -m + \alpha_{out}(|L|) \delta_1 m \\
 \end{subarray}}  2^{-k_{-}+2\gamma k_{-}+ 2  (N_0-|L|)  k_{+}} \big| \int_{\R^3}   CH_{metric;wa }^{  \tilde{h}^L;k,k_1,k_2}(t,x)
 \]
 \be\label{nov15eqn50}
\times \psi_{[m-4, m+1]}(x) d x\big|  \lesssim  2^{ 2e(L;\tilde{h})\delta_1 m - \delta_1 m/300 } \epsilon_1^3. 
 \ee
  
 \end{lemma}
\begin{proof}
Recall the detailed formula  of $ CK_{metric;wa }^{  \tilde{h}^L;k,k_1,k_2}(t,x)$ in (\ref{2020july6eqn51}--\ref{2020july6eqn53})  and the   formula of $   QK_{ k,k_1,k_2}^{L;L_1,L_2},$ $  K\in\{S,P,Q\}$  in (\ref{nov5eqn1}--\ref{nov5eqn5}).  
  From  the estimate (\ref{2020june20eqn20}) in Lemma \ref{2020junebilinearestimate} and the estimate (\ref{nov9eqn20}) in Lemma \ref{lowfrequencyinputestimate}, the following estimate holds for $K\in \{S, P, Q\}$ if $|L|\geq 1$, 
\[
   \sum_{\begin{subarray}{c}
 k\in \Z, (k_1, k_2)\in \chi_k^1\cup \chi_k^2\cup \chi_k^3,   k\geq -m + \alpha_{out}(|L|) \delta_1 m
\\
 \min\{k_1,k_2\}\leq -m + \alpha_{inp}(|L|) \delta_1 m \\
 \end{subarray}} 2^{-k_{-}+2\gamma k_{-}+ 2   (N_0-|L|) k_{+}} \big| \int_{\R^3}   CK_{metric;wa }^{  \tilde{h}^L;k,k_1,k_2}(t,x)\psi_{[m-4, m+1]}(x) d x\big|\]
 \[
 \lesssim \sum_{l\leq -m + \alpha_{inp}(|L|) \delta_1 m } 2^{(1-\gamma )l +   e(L;\tilde{h})\delta_1 m  }  2^{\delta_1 m } \epsilon_1^3  \min\{ (1+2^{-  l-m})2^{   {H}(|L|-1)\delta_1 m + \beta_{}(|L|) \delta_1 m}   \]
 \[
    + 2^{  {H}(|L|-1)\delta_1 m + m +  l }  ,    2^{ H(|L|)\delta_1 m +\delta_1 m }\} 
  \lesssim 2^{-m +  2e(L;\tilde{h})\delta_1 m  -\delta_1 m /100 } \epsilon_1^3.
\]

It remains to consider the case $|L|=0$. Note that   $CS_{metric;wa }^{  \tilde{h}^L;k,k_1,k_2}(t,x)$ vanishes for the case we are considering  and the symbol of $CK_{metric;wa }^{  \tilde{h}^L;k,k_1,k_2}(t,x), $ $K\in \{P, Q\}$ contributes the smallness of $2^{\min\{k_1,k_2\}}$. Hence, the following estimate holds from the volume of support and the $Z$-norm estimates of profiles, 
\[
    \sum_{\begin{subarray}{c}
 k\in \Z, (k_1, k_2)\in \chi_k^1\cup \chi_k^2\cup \chi_k^3,   k\geq -m + \alpha_{out}(|L|) \delta_1 m
\\
 \min\{k_1,k_2\}\leq -m + \alpha_{inp}(|L|) \delta_1 m \\
 \end{subarray}}2^{-k_{-}+2\gamma k_{-}+ 2   (N_0-|L|)k_{+}} \big| \int_{\R^3}   CK_{metric;wa }^{  \tilde{h}^L;k,k_1,k_2}(t,x)\psi_{[m-4, m+1]}(x) d x\big| \]
 \be
 \lesssim  \sum_{l\leq -m + \alpha_{inp}(|L|) \delta_1 m }  2^{(2 -2a )l  + 2H(0)\delta_1 m +2\delta_1 m } \epsilon_1^3 
 \lesssim 2^{-3m/2   } \epsilon_1^3.
\ee 
Hence finishing the proof of our desired estimate (\ref{nov15eqn50}). 
\end{proof}

Now,  we  restrict ourselves to the case $k\geq-m + \alpha_{out}(|L|)\delta_1 m $  and $\min\{k_1,k_2\}\geq -m + \alpha_{inp}(|L|)\delta_1 m $  for the estimate of $CK_{metric; ess}^{  \tilde{h}^L;k,k_1,k_2}(t,x)$  in (\ref{2020july6eqn65})    and  the case   $ k \geq -m + \alpha_{out}(|L|)\delta_1 m $ and  $ k'\geq  -2 |\alpha_{out}(|L|)| \delta_1 m$ for  the estimate of  $ \tilde{C}_{metric;1}^{  \tilde{h}^L;k, k'}(t)$ in  (\ref{2020june24eqn32}). 

Recall (\ref{2020july6eqn65}). In the following Lemma, we first rule out the case $L_1\circ L_2\prec L$.
\begin{lemma}\label{wavewavewave2020err}
 Under the bootstrap assumptions \textup{(\ref{BAmetricold})} and \textup{(\ref{BAvlasovori})}, the following estimate holds for any $t\in [2^{m-1}, 2^m]\subset [0,T]$,  $L\in  \cup_{l\leq N_0  } P_l, k \in \mathbb{Z}$,   
 \[
   \sum_{\begin{subarray}{c}
 k\in \Z,k\geq - m + \alpha_{out}(|L|) \delta_1 m, 
  H\in \{S, P, Q\},  (k_1,k_2)\in \chi_k^1\cup \chi_k^2\cup \chi_k^3 \\
L_1\circ L_2\prec L,\min\{k_1,k_2\}\geq  -m + \alpha_{inp}(|L|)\delta_1 m\\  
 \end{subarray}}   2^{-k_{-}+2\gamma k_{-}+ 2   (N_0-|L|) k_{+}}
 \]
 \[
 \times \big[\big| \int_{\R^3}  \overline{ P_k(\widetilde{   U^{ \tilde{h}^L } }  ) }  (t,x)R_{\tilde{h}}^{\alpha\beta}   \big(      QK_{\alpha\beta;k,k_1,k_2}^{ess;L, \tilde{L}_1, \tilde{L}_2}  \big)(t,x) \psi_{[m-4, m+1]}( x )    d x\big|  +\big| \int_{\R^3}  \overline{ P_k(\widetilde{   U^{ \tilde{h}^L } }  ) }  (t,x) 
 \]
 \be\label{2020june19eqn1}
\times R_{\tilde{h}}^{\alpha\beta} P_k QH_{\alpha\beta;\mu\nu}^{ {L}_1, {L}_2;\gamma_1\kappa_1,\gamma_2\kappa_2  } (    (\widetilde{   U^{  {L}_1 h_{\gamma_1\kappa_1} }_{k_1} }(t))^{\mu} , ( \widetilde{   U_{k_2}^{  {L}_2 h_{\gamma_2\kappa_2} } }(t))^{\nu} ) )(x) d x \big|\big]\lesssim  2^{-m+  2e(L;\tilde{h})\delta_1 m -\delta_1 m }\epsilon_1^3.
 \ee
 \end{lemma}
\begin{proof}
Note that for the case we are considering, we have $|L|\geq 1$. Without loss of generality, we assume that $|L_1|\leq |L_2|$.   Note that, from the     bilinear estimates (\ref{2020june20eqn25}) and \eqref{2020june21eqn2} in Lemma \ref{2020junebilinearestimate},  the following estimate holds, 
\[
\textup{L.H.S. of (\ref{2020june19eqn1})} \lesssim   \sum_{\begin{subarray}{c}
 k\in \Z,  (k_1,k_2)\in \chi_k^1\cup \chi_k^2\cup \chi_k^3, 
L_1\circ L_2\prec L 
 \\ 
  k\geq - m + \alpha_{out}(|L|) \delta_1 m, \min\{k_1,k_2\}\geq  -m + \alpha_{inp}(|L|)\delta_1 m\\  
 \end{subarray}}     2^{-k_{-}+2\gamma k_{-}+ 2N(|L|)k_{+}} \|P_k(\widetilde{   U^{ \tilde{h}^L } }  )(t)\|_{L^2} 
\]
 \[  
 \times \big[\|  \psi_{[m-4, m+1]}( x )P_k( R_{\alpha}(\widetilde{U_{k_1}^{L_1 h_{\iota\tau}}})^{\mu} R_{\beta}  (\widetilde{U_{k_2}^{L_2 h_{\gamma\kappa}}})^{\nu} )(t,x)\|_{L^2_x} \]
 \[
 + \|P_k( \d^{-1}R_{\alpha}(\widetilde{U_{k_1}^{L_1 h_{\iota\tau}}})^{\mu}(t,x)R_{\beta} \d (\widetilde{U_{k_2}^{L_2 h_{\gamma\kappa}}})^{\nu} ) (t,x)\psi_{[m-4, m+1]}( x )  \|_{L^2_x}\]
\[
   +  \|P_k( \d  R_{\alpha}(\widetilde{U_{k_1}^{L_1 h_{\iota\tau}}})^{\mu} R_{\beta} \d^{-1} (\widetilde{U_{k_2}^{L_2 h_{\gamma\kappa}}})^{\nu} )   \psi_{[m-4, m+1]}( x )\|_{L^2_x} \big]
\]
\[
\lesssim 2^{-m+e(L;\tilde{h})\delta_1 m + H(|L_2|)\delta_1 m +  \delta_1 m  }(2^{  H(|L_1|+1)\delta_1 m}\mathbf{1}_{|L_1|\geq 1} + 2^{\delta_1 m}\mathbf{1}_{|L_1|=0})\epsilon_1^3 
\]
\[
 \lesssim 2^{-m+  2e(L;\tilde{h})\delta_1 m - \delta_1 m  }\epsilon_1^3. 
\]
Hence finishing the proof of our desired estimate (\ref{2020june19eqn1}). 
\end{proof}
 
Now, it remains to consider the case $L_1\circ L_2= L$, in which   the null structure play   essential roles.   We summarize the    effect of  null structure in the following Lemma.

\begin{lemma}\label{2022maybilinearest}
Given  $\mu, \nu\in\{+, -\}$, $\tilde{h}_1, \tilde{h}_2 \in \{F, \underline{F},\rho, \Omega_i,  \omega_{i}, \vartheta_{ij}\}, $ $t_1, t_2\in [2^{m-1}, 2^m]$,$L, L_1, L_2, \tilde{L}_1\in \cup_{n\leq N_0 } P_n$, s.t.,  $|L_1|+|L_2|=|L|$,  $|\tilde{L}_1|+|\tilde{L}_2|=|L|$, $|\tilde{L}_1|< |L|$. Let
\be\label{2022may3eqn32} 
  A:=\{(\textup{quasi}, \tilde{L}_1, \tilde{L}_2),(\textup{semi},   {L}_1,   {L}_2): \tilde{L}_1\circ \tilde{L}_2\preceq \tilde{L}, L_1\circ L_2\preceq L, |\tilde{L}_1|< |L| \}. 
\ee
  For any given  bilinear operators $T_{\mu\nu}^{   \textup{quas}}$ and  $T_{\mu\nu}^{  \textup{semi}}$ with symbols $  m_{\mu\nu}^{\textup{quas}}(\xi-\eta, \eta)|\xi-\eta|^{-1}|\eta|\in \mathcal{M}^{null}_{\mu\nu}\cup \tilde{\mathcal{M}}^{null}_{- \nu}$ and  $  m_{\mu\nu}^{\textup{semi}}(\xi-\eta, \eta)\in  \mathcal{M}^{null}_{\mu\nu}\cup \tilde{\mathcal{M}}^{null}_{- \nu}$ respectively,  the following estimate  holds, 
\[
\sum_{\begin{subarray}{c}
k\in \Z, (k_1,k_2)\in \chi_k^1\cup \chi_k^2\cup \chi_k^3, (\star, L_1, L_2)\in A  \\ 
\min\{k,k_1,k_2\} \geq -m +  \alpha_{inp}(|L|)\delta_1 m 
\end{subarray}} 2^{-k_{-}+2\gamma k_{-}+ 2   (N_0-|L|) k_{+}} \big[ \big|  \int_{t_1}^{t_2} \int_{\R^3}   T_{\mu\nu}^{ {\star}}(    (\widetilde{   U^{ \tilde{h}_1^{L_1} }_{k_1} }(t))^{\mu} , ( \widetilde{   U_{k_2}^{ \tilde{h}_2^{L_2} } }(t))^{\nu} ) )(x)
\]
\be\label{2022may3eqn31} 
  \times \widetilde{   U^{ \tilde{h}^{L  } }_{k  } } (t,x)  \psi_{[m-4, m+1]}(x)  d x d t \big| \big] \lesssim 2^{2e(L;\tilde{h})\delta_1 m - \delta_1 m/100 }\epsilon_1^3.
\ee
\end{lemma}
\begin{proof}
Postponed to subsection \ref{proofofmaybilinear} for better presentation. 
\end{proof}

By using the above Lemma, in the following Lemma,  we finish the remaining cases of estimating $CK_{metric; ess}^{  \tilde{h}^L;k,k_1,k_2}(t,x)$  in (\ref{2020july6eqn65})  for the case   $ k \geq -m + \alpha_{out}(|L|)\delta_1 m $ and  $ \tilde{C}_{metric;1}^{  \tilde{h}^L;k, k'}(t)$ in  (\ref{2020june24eqn32}) for the case $ k'\geq  -2 |\alpha_{out}(|L|)| \delta_1 m$.

\begin{lemma}\label{wavewavewave2020ess}
Under the bootstrap assumptions \textup{(\ref{BAmetricold})} and \textup{(\ref{BAvlasovori})}, the following estimate holds for any $t\in [2^{m-1}, 2^m]\subset [0,T]$, $\tilde{h}\in \{F, \underline{F}, \omega_j, \vartheta_{ij}\},$  $L\in\cup_{n\leq N_0} P_n, $ 
 \[
 \sum_{ 
 k\in[- m + \alpha_{out}(|L|) \delta_1 m, \infty) 
 } 2^{-k_{-}+2\gamma k_{-}+ 2   (N_0-|L|) k_{+}} \Big[
 \sum_{\begin{subarray}{c}
  H\in \{S, P, Q\}, 
  (k_1, k_2)\in \chi_k^1\cup \chi_k^2\cup \chi_k^3,\\ 
   \min\{k_1,k_2\}\geq -m + \alpha_{inp}(|L|) \delta_1 m-10 
 \end{subarray} }   \big|  \int_{t_1}^{t_2} \int_{\R^3}  CH_{metric; ess}^{  \tilde{h}^L;k,k_1,k_2}(t,x)   
\]
\[
\times \psi_{[m-4, m+1]}(x) d x d t\big| + \sum_{k'\geq  -2|\alpha_{out}(|L|)|\delta_1 m }  \big| \int_{t_1}^{t_2}  \int_{\R^3}   \tilde{C}_{metric ;1}^{ \tilde{h}^L;k,  k'}(t,x)  \psi_{[m-4, m+1]}(x) d x d t\big|\Big] 
\]
\be\label{2020july6eqn27}
\lesssim  2^{  2 e(L;\tilde{h})\delta_1 m -\delta_1 m/300}\epsilon_1^3.
\ee

 \end{lemma}
\begin{proof}
Recall (\ref{2020june24eqn32}) and  (\ref{2020july6eqn65}). From the estimate (\ref{2020june19eqn1}) in Lemma \ref{wavewavewave2020err}, to prove   our desired estimate (\ref{2020july6eqn27}), it would be sufficient to the case $ |L_1|+|L_2|=|\tilde{L}_1|+|\tilde{L}_2|=|L|, |\tilde{L}_2|< |L|$. Based on the possible sizes of $|L|$ and $k$, we separate into two cases as follow. 

\noindent $\bullet$\quad If $ k \geq -m + \alpha_{inp}(|L|)\delta_1 m $, i.e., the assumptions in Lemma \ref{2022maybilinearest} are valid.

Recall (\ref{2020june17eqn21}) and the statement about null structure in Lemma \ref{nullstructurefact}.   Due to the presence of null structure in   all the  quadratic terms of wave-wave type interaction   in the equation satisfied by $\tilde{h}^L, \tilde{h}\in \{F, \omega_j, \vartheta_{ij}\},$  the following estimate follows directly from the estimate (\ref{2022may3eqn31}) in Lemma \ref{2022maybilinearest},  
  \begin{multline}\label{2020june19eqn2}
     \sum_{\begin{subarray}{c}
 H\in \{ P, Q\}, k\in \Z,  (k_1,k_2)\in \chi_k^1\cup \chi_k^2\cup \chi_k^3 \\
   \min\{k_1,k_2\}\geq  -m + \alpha_{inp}(|L|)\delta_1 m\\  
k \geq -m + \alpha_{out}(|L|)\delta_1 m\\ 
 \end{subarray}}   2^{-k_{-}+2\gamma k_{-}+ 2   (N_0-|L|) k_{+}}\big[\big|\int_{t_1}^{t_2} \int_{\R^3}  \overline{ P_k(\widetilde{   U^{  \tilde{h}^L } }  ) }  (t,x)R_{\tilde{h}}^{\alpha\beta}  \big(    QH_{\alpha\beta;k,k_1,k_2}^{ess;L, {L}_1, {L}_2}   \big)(t,x)  \\
 \times \psi_{[m-4, m+1]}(x)  d x d t\big| +\big|\int_{t_1}^{t_2}  \int_{\R^3}  \overline{ P_k(\widetilde{   U^{  \tilde{h}^L } }  ) }  (t,x)R_{\tilde{h}}^{\alpha\beta}  \big(   QS_{\alpha\beta;k,k_1,k_2}^{ess;L,\tilde{L}_1,\tilde{L}_2}  \big) (t,x)\psi_{[m-4, m+1]}(x)  d x d t\big|\\
+ \sum_{k'\geq -2|\alpha_{out}(|L|)|\delta_1 m}   \big| \int_{\R^3}  \tilde{C}_{metric ;1}^{ \tilde{h}^L;k,  k'}(t,x) \psi_{[m-4, m+1]}(x) d x\big|\big] \lesssim  2^{ 2 e(L;\tilde{h})\delta_1 m -\delta_1 m/100 }\epsilon_1^3.
\end{multline}

 It remains to consider the case when $\tilde{h}=\underline{F}$.     Recall (\ref{2020june17eqn21}). We first consider the quadratic terms without null structure, i.e., the interaction of $\vartheta_{mn}$ and $\vartheta_{pq}$.    From the bilinear estimates   (\ref{2020june20eqn25}) and (\ref{2020june21eqn2}) in Lemma \ref{2020junebilinearestimate},  the following estimate holds for any $\tilde{h}_1,  \tilde{h}_2\in \{\vartheta_{ij}\}  $,
 \[
      \sum_{\begin{subarray}{c}
  k\in \Z,  (k_1,k_2)\in \chi_k^1\cup \chi_k^2\cup \chi_k^3 \\
   \min\{k,k_1,k_2\}\geq  -m + \alpha_{inp}(|L|)\delta_1 m 
 \end{subarray}} 2^{-k_{-}+2\gamma k_{-}+ 2  (N_0-|L|)k_{+}} \big|\int_{t_1}^{t_2} \int_{\R^3}  \overline{ P_k(\widetilde{   U^{  \underline{F}^L } }  ) }  (t,x)  
 \]
 \[
   \times R_{\underline{F}}^{\alpha\beta}   P_k QP_{\alpha\beta;\mu\nu}^{\tilde{h}_1, \tilde{h}_2} (    (\widetilde{   U^{ \tilde{h}_1^{L_1} }_{k_1} }(t))^{\mu} , ( \widetilde{   U_{k_2}^{ \tilde{h}_2^{L_2} } }(t))^{\nu} ) )(x)   \psi_{[m-4, m+1]}(x)  d x d t\big| 
 \]
\be\label{2020june22eqn5}
  \lesssim 2^{   e(L;\underline{F})\delta_1 m    +   e(L;\vartheta_{mn})\delta_1 m }  \big[ 2^{\beta(|L|)\delta_1 m -\delta_1 m/200} \mathbf{1}_{|L|\geq 2} +   \mathbf{1}_{|L|\leq 1} \big]\epsilon_1^3   \lesssim  2^{ 2  e(L;\underline{F})\delta_1 m   -\delta_1 m /200  }\epsilon_1^3.
\ee

Now, we consider the cases with null structure, i.e., all other quadratic terms except the terms considered in the above integral. From  the estimate (\ref{2022may3eqn31}) in Lemma \ref{2022maybilinearest}, we have
\[
\sum_{\begin{subarray}{c}
\tilde{h}_1, \tilde{h}_2\in \{F, \underline{F}, \omega_j, \vartheta_{ij}\}, (\tilde{h}_1, \tilde{h}_2)\notin \{\vartheta_{ij}\}\times \{\vartheta_{ij}\}\\ 
k\in \Z, (k_1,k_2)\in \chi_k^1\cup \chi_k^2\cup \chi_k^3, 
 \min\{k,k_1,k_2\}\geq  -m + \alpha_{inp}(|L|)\delta_1 m 
\end{subarray}
 }   2^{-k_{-}+2\gamma k_{-}+ 2 (N_0-|L|)k_{+}}\big|\int_{t_1}^{t_2} \int_{\R^3} \overline{ P_k(\widetilde{   U^{  \underline{F}^L } }  ) }  (t,x)   \]
\be\label{2020july6eqn21}
\times  \psi_{[m-4, m+1]}(x)   R_{\underline{F}}^{\alpha\beta}   P_k QP_{\alpha\beta;\mu\nu}^{\tilde{h}_1, \tilde{h}_2} (    (\widetilde{   U^{ \tilde{h}_1^{L_1} }_{k_1} }(t))^{\mu} , ( \widetilde{   U_{k_2}^{ \tilde{h}_2^{L_2} } }(t))^{\nu} ) )(x)  d x d t\big|\lesssim  2^{ 2  e(L;\underline{F})\delta_1 m   -\delta_1 m /100   }\epsilon_1^3. 
\ee

\[
\sum_{\begin{subarray}{c}
\tilde{h}_1, \tilde{h}_2\in \{F, \underline{F}, \omega_j, \vartheta_{ij}\}, k\in \Z, (k_1,k_2)\in \chi_k^1\cup \chi_k^2\cup \chi_k^3\\ 
 \min\{k,k_1,k_2\}\geq  -m + \alpha_{inp}(|L|)\delta_1 m \\
\end{subarray} }    2^{-k_{-}+2\gamma k_{-}+ 2 (N_0-|L|) k_{+}}\big[ \big|\int_{t_1}^{t_2} \int_{\R^3}   \overline{ P_k(\widetilde{   U^{  \underline{F}^L } }  ) }  (t,x)  \psi_{[m-4, m+1]}(x)  
\]
\[
\times R_{\underline{F}}^{\alpha\beta}   P_k QQ_{\alpha\beta;\mu\nu}^{\tilde{h}_1, \tilde{h}_2} (    (\widetilde{   U^{ \tilde{h}_1^{L_1} }_{k_1} }(t))^{\mu} , ( \widetilde{   U_{k_2}^{ \tilde{h}_2^{L_2} } }(t))^{\nu} ) )(x) d x d t\big|   + \big|\int_{t_1}^{t_2} \int_{\R^3} \overline{ P_k(\widetilde{   U^{  \underline{F}^L } }  ) }  (t,x)
\]
\[
\times    R_{\underline{F}}^{\alpha\beta}   P_k QS_{\alpha\beta;\mu\nu}^{\tilde{h}_1, \tilde{h}_2} (    (\widetilde{   U^{ \tilde{h}_1^{\tilde{L}_1} }_{k_1} }(t))^{\mu} , ( \widetilde{   U_{k_2}^{ \tilde{h}_2^{\tilde{L}_2} } }(t))^{\nu} ) )(x)  \psi_{[m-4, m+1]}(x) d x d t\big| \big]
\]
\be\label{2020july6eqn22}
\lesssim   2^{ 2  e(L;\underline{F})\delta_1 m   -\delta_1 m /100   }\epsilon_1^3. 
\ee
 To sum up,  the following estimate holds   from the estimates (\ref{2020june22eqn5}), (\ref{2020july6eqn21}), and (\ref{2020july6eqn22}),
    \[
  \sum_{\begin{subarray}{c}
 H\in \{ P, Q\}, k\in \Z,  (k_1,k_2)\in \chi_k^1\cup \chi_k^2\cup \chi_k^3 \\
   \min\{k,k_1,k_2\}\geq  -m + \alpha_{inp}(|L|)\delta_1 m  \\ 
 \end{subarray}}    2^{-k_{-}+2\gamma k_{-}+ 2 (N_0-|L|) k_{+}}\big[\big|\int_{t_1}^{t_2} \int_{\R^3}   \psi_{[m-4, m+1]}(x)   \overline{ P_k(\widetilde{   U^{ \underline{F}^L } }  ) }  (t,x) \]
 \[
 \times  R_{\underline{F}}^{\alpha\beta} \big(   QH_{\alpha\beta;k,k_1,k_2}^{ess;L, {L}_1, {L}_2}    \big)(t,x)  d x d t\big| 
  +\big|\int_{t_1}^{t_2}  \int_{\R^3}   \overline{ P_k(\widetilde{   U^{  \underline{F}^L } }  ) }  (t,x)  R_{ \underline{F} }^{\alpha\beta}  \big(    QS_{\alpha\beta;k,k_1,k_2}^{ess;L, \tilde{L}_1, \tilde{L}_2}     \big)(t,x) 
 \]
  \be\label{2020june19eqn4}
 \times  \psi_{[m-4, m+1]}(x) d x d t\big|\big] \lesssim   2^{ 2 e(L; \underline{F})\delta_1 m -\delta_1 m/300  }\epsilon_1^3.
   \ee
\noindent $\bullet$\quad If   $  k \in [ -m + \alpha_{out}(|L|)\delta_1 m,  -m + \alpha_{inp}(|L|)\delta_1 m) $. 

Recall (\ref{outputthre}) and (\ref{inputthre}). For this case, it would be sufficient to consider the case $|L|=0.$

Note that, from the rough bilinear estimate   (\ref{2020june20eqn20}) in Lemma \ref{2020junebilinearestimate}, due to the smallness of at least $2^{\max\{k_1,k_2\}}$ from symbols, we can rule out the case when $\max\{k_1, k_2\}\leq -2 \alpha_{inp}(0)\delta_1 m$. It remains to consider the case $\max\{k_1, k_2\}\geq -2 \alpha_{inp}(0)\delta_1 m$, i.e.,  $(k_1, k_2)\in \chi_k^1.$

From the bilinear estimate (\ref{2020june21eqn2}) in Lemma \ref{2020junebilinearestimate} and the $Z$-norm estimate of the good component, the following estimate holds for the quadratic term without null structure, 
\begin{multline}
 \sum_{\begin{subarray}{c}
  \tilde{h}_1,  \tilde{h}_2\in \{\vartheta_{ij}\}, H\in \{S, P, Q\}\\
   k\in \Z, (k_1,k_2)\in \chi_k^1
 \end{subarray}} 2^{-k_{-}+2\gamma k_{-}+2N_0 k_{+} } \big|\int_{t_1}^{t_2} \int_{\R^3}    R_{  \tilde{h} }^{\alpha\beta}   P_k QH_{\alpha\beta;\mu\nu}^{\tilde{h}_1, \tilde{h}_2} (    (\widetilde{   U^{ \tilde{h}_1^{  } }_{k_1} }(t))^{\mu} , ( \widetilde{   U_{k_2}^{ \tilde{h}_2^{ } } }(t))^{\nu} ) )(x) \\ 
  \times \overline{ P_k(\widetilde{   U^{  \tilde{h} } }  ) }  (t,x)   \psi_{[m-4, m+1]}(x)  d x d t\big|\lesssim 2^{m+  e(L;  \tilde{h})\delta_1 m    } \epsilon_1^3\lesssim 2^{m+ 2 e(L;  \tilde{h})\delta_1 m  -\delta_1 m /100  } \epsilon_1^3
 \end{multline}

It remains to consider the case $ (\tilde{h}_1, \tilde{h}_2)\notin \{\vartheta_{ij}\}\times \{\vartheta_{ij}\}$. We first out the non-time-resonance case   $\mu=\nu$. Note that the phase is relatively large for the case we are considering. More precisely, we have
\[
\big||\xi|+\mu |\xi-\eta|+ \nu |\eta|\big|\gtrsim 2^{k_1}\gtrsim 2^{-2\alpha_{inp}(0)\delta_1 m}.
\]
 
 To take advantage of high oscillation in time, we do integration by parts in time once. Note that, due to the fact that the output frequency is very small, there is no losing derivative
issue. After using the bilinear estimate       (\ref{2020june20eqn20}) in Lemma \ref{2020junebilinearestimate}, the following estimate holds after doing integration by parts in time once, 
\begin{multline}
 \sum_{\begin{subarray}{c}
 \tilde{h}_1, \tilde{h}_2 \in \{F, \underline{F}, \omega_j, \vartheta_{mn}\},
  (\tilde{h}_1, \tilde{h}_2)\notin \{\vartheta_{ij}\}\times \{\vartheta_{ij}\}\\ 
   H\in \{S, P, Q\}, k\in \Z,  (k_1,k_2)\in \chi_k^1, \mu\in\{+,-\}
   \end{subarray}}2^{-k_{-}+2\gamma k_{-} +2N_0 k_{+}  } \big|\int_{t_1}^{t_2} \int_{\R^3}  \overline{ P_k(\widetilde{   U^{  \tilde{h} } }  ) }  (t,x) \psi_{[m-4, m+1]}(x)     \\ 
  \times     R_{  \tilde{h} }^{\alpha\beta} P_k QH_{\alpha\beta;\mu\mu}^{\tilde{h}_1, \tilde{h}_2} (    (\widetilde{   U^{ \tilde{h}_1^{  } }_{k_1} }(t))^{\mu} , ( \widetilde{   U_{k_2}^{ \tilde{h}_2^{ } } }(t))^{\mu} ) )(x)  d x d t\big|\lesssim 2^{-m/2  } \epsilon_1^3.  
 \end{multline}

 Lastly, we consider the case $\nu=-\mu, \mu\in \{+,-\}$. Since the symbol of quadratic terms with null structure belongs to $\mathcal{M}^{null}_{\mu(-\mu)}=\mathcal{M}^{null}_{-}$, see Lemma \ref{nullstructurefact}, we gain the smallness of $(\xi-\eta)/|\xi-\eta|+\eta/|\eta|$ from the symbol, which is less than $2^{k-k_1}$. Therefore, from the rough estimate (\ref{2020june20eqn20}) in Lemma \ref{2020junebilinearestimate}, the following estimate holds, 
 \[
  \sum_{\begin{subarray}{c}
  \tilde{h}_1,  \tilde{h}_2\in \{\vartheta_{ij}\}\\
    H\in \{S, P, Q\}\\ 
     k\in \Z, (k_1,k_2)\in \chi_k^1
  \end{subarray}
}2^{-k_{-}+2\gamma k_{-} +2N_0 k_{+} } \big|\int_{t_1}^{t_2} \int_{\R^3}  \overline{ P_k(\widetilde{   U^{  \tilde{h} } }  ) }  (t,x)  R_{  \tilde{h} }^{\alpha\beta}   P_k QH_{\alpha\beta;\mu(-\mu)}^{\tilde{h}_1, \tilde{h}_2} (    (\widetilde{   U^{ \tilde{h}_1^{  } }_{k_1} }(t))^{\mu} , ( \widetilde{   U_{k_2}^{ \tilde{h}_2^{ } } }(t))^{-\mu} ) )(x) 
 \]
 \be
  \times    \psi_{[m-4, m+1]}(x)  d x d t\big|\lesssim 2^{m + 2k + 4\alpha_{inp}(0)\delta_1 m } \epsilon_1^3\lesssim 2^{-m/2}\epsilon_1^3.
 \ee
Hence finishing the proof of our desired estimate (\ref{2020july6eqn27}). 
\end{proof}

\subsection{Proof of Lemma \ref{2022maybilinearest} }\label{proofofmaybilinear}

 In this subsection, we provide detailed proof of the Lemma \ref{2022maybilinearest}.

  Firstly,   we  prove several $L^2$-type bilinear estimates  in the following Lemma,  which are analogue of bilinear estimates in Lemma \ref{2020junebilinearestimate}.   
 \begin{lemma}\label{2020junebilinearangular}
Assume that $f_1,f_2$ satisfy assumptions defined in \textup{(\ref{oct22eqn1})}. Given $L_1, L_2\in \cup_{n\leq N_0 } P_n, $ $  |L_1|\leq N_0-3$, $m\in \mathbb{Z}_{+}$,  $t\in[2^{m-1}, 2^m]\subset [0,T]$,   $\alpha\in \R,\mu,\nu\in\{0,1,2,3\}, l,k, k_1, k_2\in \mathbb{Z} $, s.t.,  $l\leq 2 $, $k_2\leq k_1$,    we define 
\[
K_{\mu, \nu}(  f_1^{L_1},  f_2^{ L_2} )(t,\xi):=    \int_{\R^3}   \psi_{k}(\xi) \psi_{k_1}(\xi-\eta)\psi_{k_2}(\eta) \psi_{l}(\frac{\xi}{|\xi|}  +\nu  \frac{\eta}{|\eta|} ))\big(\frac{\xi}{|\xi|}\times \frac{\eta}{|\eta|} \big)^{\alpha} 
\]
\be\label{nov23eqn36}
  \times m(\xi-\eta, \eta)\widehat{R_\mu  f_1^{L_1} }(\xi-\eta) \widehat{R_\nu f_2^{L_2}}(\eta) d \eta,
\ee
where the symbol $m(\xi-\eta, \eta)\in \mathcal{S}^\infty_{k,k_1,k_2} $. 
Then the following bilinear estimate holds,
\be\label{nov27eqn31}
  \|   K_{\mu, \nu}(  f_1^{L_1},  f_2^{ L_2} )(t,\xi)\|_{L^2} \lesssim 2^{\alpha l +3\min\{k,k_1,k_2\}/2}   \|m (\xi-\eta, \eta)\|_{\mathcal{S}^\infty_{k,k_1,k_2}}  A_1(k_1, m,L_1) A_2(k_2, m, L_2).
\ee
If  $l+k_1,l+k_2\geq -m+ \delta_1 m$, then the following estimate  holds for $\star\in \{\leq m-10, \geq m+10\}$, 
\[
   \|\mathcal{F}^{-1}[ K_{\mu, \nu}(  f_1^{L_1},  f_2^{ L_2} )](t,x)\psi_{\star}(x)\|_{L^2} \lesssim  2^{\alpha l +3\min\{k,k_1,k_2\}/2}  \|m (\xi-\eta, \eta)\|_{\mathcal{S}^\infty_{k,k_1,k_2}}  A_2(k_2, m, L_2)\]
  \[
  \times \min\big\{  \sum_{\Gamma_1, \Gamma_2\in P_0^1\cup P_1^0} 2^{-2m-2k_1}  A_1(k_1, m,  L_1;\Gamma_1\circ\Gamma_2) + 2^{-m-3k_1}\|P_{k_1}(\square  f_1^{L_1})\|_{L^2} \]
\be\label{nov24eqn2}
  + 2^{-2m-4k_1}\|P_{k_1}(\square \Gamma_1   f_1^{L_1} )\|_{L^2}, \sum_{\Gamma \in P_0^1\cup P_1^0} 2^{- m- k_1}  A_1(k_1, m,L_1;\Gamma )+ 2^{-m-3k_1}\|P_{k_1}(\square  f_1^{L_1})\|_{L^2} ,   \big\}    , 
\ee
\[
  \|  \mathcal{F}^{-1}[ K_{\mu, \nu}(  f_1^{L_1},  f_2^{ L_2} )](t,x) \psi_{>m-10}(x)\|_{L^2_x}
\]
\be\label{2020june25eqn11}
 \lesssim   2^{-m+\min\{k,k_1,k_2\}/2 }   2^{\alpha l  } \|m (\xi-\eta, \eta)\|_{\mathcal{S}^\infty_{k,k_1,k_2}}C_1(k_1,m,L_1) A_2(k_2, m, L_2). 
\ee
If  $l+k_1,l+k_2\geq -m+ \delta_1 m$ and   $|L_1|=0$, then the following estimate holds, 
\[
  \|  \mathcal{F}^{-1}[ K_{\mu, \nu}(  f_1^{L_1},  f_2^{ L_2} )](t,x) \psi_{>m-10}(x)\|_{L^2_x}
\]
  \be\label{nov24eqn4}
 \lesssim   2^{-m+\min\{k,k_1,k_2\}/2 }   2^{\alpha l  }  \|m (\xi-\eta, \eta)\|_{\mathcal{S}^\infty_{k,k_1,k_2}}D_1(m, k_1)  A_2(k_2, m, L_2). 
\ee  
 
\end{lemma}

\begin{proof}
To prove our desired estimates, it would be sufficient to consider the case   $l< -10$, otherwise,  the angular localization function plays little role and desired estimates follow from the estimates in Lemma \ref{2020junebilinearestimate}.

 Let $\{\theta_1,\cdots, \theta_K\}\subset \mathbb{S}^2$ be the centers of the   ordered   sectors s.t., 
\[
1=\sum_{i=1, \cdots, K} \psi_{\leq l}(\xi/|\xi|-\theta_i), \quad  \forall i,j\in\{1,\cdots, K\},\quad  |\theta_i-\theta_j|\geq |i-j|2^{l-2}. 
\] 
 We do partition of unity on sphere for the angular variables    $\xi/|\xi|$ and $\eta/|\eta|$ with the sector of size $2^l$ and have the following decomposition, 
\[
\mathcal{F}_{\xi}^{-1}[K_{\mu, \nu}( f_1^{L_1},   f_2^{L_2})](t,x):= \sum_{i,j\in\{1,\cdots, K\}, |i-j|\leq 3} K_i^j ( f_1^{L_1},   f_2^{L_2}), \]
where
\[
  K_i^j ( f_1^{L_1},   f_2^{L_2}):=  \int_{\R^2}\widehat{R_\mu  f_1^{L_1} }(\xi-\eta) \widehat{R_\nu f_2^{L_2}}(\eta)   m(\xi, \eta)  \big(\frac{\xi}{|\xi|}\times \frac{\eta}{|\eta|} \big)^{\alpha} \psi_{k}(\xi) \psi_{k_1}(\xi-\eta)\psi_{k_2}(\eta)     \]
\[
\times   \psi_{l}(\angle(\xi, \nu \eta)) \psi_{\leq l}(\xi/|\xi|-\theta_i)\psi_{\leq l}(\nu\eta/|\eta|-\theta_j) d \eta.
\]
From the orthogonality in $L^2$, we have 
\be\label{2020july6eqn1}
\| K_{\mu, \nu}( f_1^{L_1},   f_2^{L_2}) \|_{L^2}^2 \sim  \sum_{i,j\in\{1,\cdots, K\}, |i-j|\leq 3}\|  K_i^j ( f_1^{L_1},   f_2^{L_2})\|_{L^2}^2.
\ee
Note that, in terms of kernel, we have
\[
 K_i^j ( f_1^{L_1},   f_2^{L_2})(t,x):=  \int_{\R^3} \int_{\R^3}   \tilde{K} (  x_1,  x_2) P_{k_1}^{l,i}( R_\mu  f_1^{L_1} )(t, x-x_1) P_{k_1}^{l,j}( R_\nu f_2^{L_2} )(t, x-x_1-x_2) d x_1 d x_2, 
\]
\[
\tilde{K} (  x_1,  x_2):=\mathcal{F}_{\xi,\eta}^{-1}[m_{k_1k_2}^k(\xi, \eta)\psi_{l}(\angle(\xi, \nu \eta))  \psi_{\leq l+2}(\nu\eta/|\eta|-\theta_j)\psi_{\leq l+2}(\xi/|\xi|-\theta_i)] (  x_1,  x_2), 
\]
\[
m_{k_1k_2}^k(\xi, \eta):=m(\xi, \eta)\psi_{[k-3,k+3]}(\xi)\psi_{[k_2-4,k_2+4]}(\eta)\psi_{[k_1-4,k_1+4]}(\xi-\eta)\big(\frac{\xi}{|\xi|}\times \frac{\eta}{|\eta|} \big)^{\alpha},
\]
\[
P_{k_1}^{l,i}(f)(x):=\int_{\R^3} e^{i x\cdot \xi } \hat{f}(\xi)\psi_{k_1}(\xi)\psi_{\leq l}( \angle(\xi, \theta_i)) d \xi 
\]
\be\label{july29eqn21}
=  \int_{\R^3} \mathcal{F}_{\xi}^{-1}[ \psi_{[k_1-4,k_1+4]}(\xi)\psi_{\leq l}( \angle(\xi, \theta_i))](y) P_{k_1}(x-y) d y  .
\ee
 
Note that, the following estimate for the kernel after doing integration by parts in $\theta_i\cdot\xi$ and $\theta_i\times \xi$ directions and $\theta_j\cdot\eta$ and $\theta_i\times \eta$, 
\[
|\tilde{K} (  x_1,  x_2)| \lesssim 2^{\alpha l } \| m(\xi-\eta, \eta)\|_{\mathcal{S}^\infty_{k,k_1,k_2}} 2^{3k+2l+3k_2+2l} \big(1+2^{k}|x_1\cdot \theta_i| \big)^{-1000}\big(1+2^{k+l}|x_1\times  \theta_i| \big)^{-1000}
\]
\be\label{2020july6eqn2}
\times \big(1+2^{k_2}|x_2\cdot \theta_j| \big)^{-1000}\big(1+2^{k_2+l}|x_2\times  \theta_j| \big)^{-1000}.
\ee
From the above estimate of kernel, the volume of support,   the estimate (\ref{2020july6eqn1}), and the Cauchy-Schwarz inequality,   we have
\[
\| K_{\mu, \nu}( f_1^{L_1},   f_2^{L_2}) \|_{L^2}^2 \lesssim   \sum_{i,j\in\{1,\cdots, K\}, |i-j|\leq 3} \big( 2^{\alpha l + 3\min\{k,k_1, k_2\}/2 }  \| m(\xi-\eta, \eta)\|_{\mathcal{S}^\infty_{k,k_1,k_2}}    \]
\[
\times \|P_{k_1}^{l,i}( R_\mu  f_1^{L_1} )(t, x ) \|_{L^2}  \|P_{k_2}^{l,j}( R_\nu  f_2^{L_2 } )(t, x ) \|_{L^2}\big)^2  \]
\[
\lesssim  \big( 2^{\alpha l  } 2^{3\min\{k,k_1,k_2\}/2}  \| m(\xi-\eta, \eta)\|_{\mathcal{S}^\infty_{k,k_1,k_2}}   A_1(k_1, m,L_1) A_2(k_2, m, L_2)\big)^2.
\]
Hence finishing the proof of our 
  desired estimate (\ref{nov27eqn31}) holds.

If we restrict ourself to the case $  k,k_1,k_2\in [-m +\delta_1 m, \infty), l\in \Z,$ s.t.,  $l+k_1,l+k_2\geq -m+ \delta_1 m$, from the estimate of kenel in (\ref{2020july6eqn2}), we have 
\be\label{2020april12eqn11}
\| \tilde{K} (  x_1,  x_2) \|_{L^1_{x_1,x_2}  } + 2^{20 m} \|  \tilde{K} (  x_1,  x_2)\psi_{\geq m-100}(|x_1|+|x_2|)\|_{L^1_{x_1,x_2}}  \leq 2^{\alpha l }  \| m(\xi-\eta, \eta)\|_{\mathcal{S}^\infty_{k,k_1,k_2}}. 
\ee
Moreover, 
\be\label{2020april12eqn12}
\begin{split}
\|\mathcal{F}_{\xi}^{-1}[ \psi_{[k_1-4,k_1+4]}(\xi)\psi_{\leq l}( \angle(\xi, \theta_i))](y)\|_{L^1_y} &\lesssim 1, \\  \|\mathcal{F}_{\xi}^{-1}[ \psi_{[k_1-4,k_1+4]}(\xi)\psi_{\leq l}( \angle(\xi, \theta_i))](y)\psi_{\geq m-100}(y)\|_{L^1_y} &\lesssim 2^{-20m }. 
\end{split}
\ee
 Hence, our desired estimates (\ref{nov24eqn2}--\ref{nov24eqn4}) hold  from   the estimate (\ref{2020june20eqn20}) in Lemma 
\ref{2020junebilinearestimate} and the estimate of kernels in (\ref{2020april12eqn11}--\ref{2020april12eqn12})  for the case $|x_1|+|x_2|\geq 2^{m-100} $ or $|y|\geq  2^{m-100}$ and the same argument used in the proof of the estimates (\ref{2020june20eqn11}--\ref{2020june21eqn2}) in Lemma \ref{2020junebilinearestimate} for the case $|x_1|+|x_2|\leq 2^{m-100} $ and $|y|\leq  2^{m-100}. $

\end{proof}

 With the bilinear estimates in above Lemma,   now  we prove the postponed   Lemma \ref{2022maybilinearest}.

\noindent \textit{Proof of Lemma} \ref{2022maybilinearest}. 

Note that, in the sense of potentially losing derivatives in the process of doing normal form transformation,  the High $\times$ High type interaction is similar and also much easier than the Low $\times$ High and the High $\times$ High type interaction, we restrict ourselves to the Low $\times$ High type interaction and the High $\times$ Low type interaction. Moreover,  due to the symmetry between two inputs, without loss of generality, we assume that $(k_1, k_2)\in \chi_k^3$, i.e., $k_2\leq k_1-10. $

Based on the size of angle between $\xi-\eta$ and $\eta$, for $\star\in \{\textup{semi}, \textup{quas}\}$, we decompose the integrals  in  (\ref{2022may3eqn31}) into two parts as follows,
 \[
  \int_{t_1}^{t_2} \int_{\R^3} \widetilde{   U^{ \tilde{h}^{L  } }_{k_1} } (t,x)   T_{\mu\nu}^{ \star} (    (\widetilde{   U^{ \tilde{h}_1^{L_1} }_{k_1} }(t))^{\mu} , ( \widetilde{   U_{k_2}^{ \tilde{h}_2^{L_2} } }(t))^{\nu} ) )(x)  \psi_{[m-4, m+1]}(x)  d x d t
 \] 
\be\label{2020june25eqn2}
 =  J_{k,k_1,k_2}^{ \mu\nu;\star;\leq}(t_1, t_2) + J_{k,k_1,k_2}^{ \mu\nu;\star;> }(t_1, t_2), 
\ee
where for $   \diamondsuit\in \{\leq, >\}, $ $J_{k,k_1,k_2}^{ \mu\nu;\star;\diamondsuit  }(t_1, t_2) $ are defined as follows, 
\[
J_{k,k_1,k_2}^{\mu\nu; \star;\diamondsuit}(t_1, t_2)(t_1, t_2):=  \int_{t_1}^{t_2}\int_{ \R^3 } \int_{ \R^3 }\int_{ \R^3 } e^{it\Phi^{\mu;\nu}(\xi, \eta, \sigma)} 2^{3m}\widehat{\psi_{[-4,1]}}(2^{ m}\sigma) m_{\mu\nu}^{\star}(\xi-\eta, \eta)\psi_k(\xi) \psi_{k_1}(\xi-\eta) \psi_{k_2}(\eta) 
  \]
\be\label{2020june25eqn3}
\times \psi_{\diamondsuit \bar{l}_{\mu;\nu}}(\angle(\xi, -\nu \eta))   \widehat{\widetilde{V^{ \tilde{h}^{L}}}}(t, -\xi-\sigma)  \widehat{(\widetilde{   V^{ \tilde{h}_1^{L_1} }_{ } }(t))^{\mu}}(t, \xi-\eta )   \widehat{(\widetilde{   V^{ \tilde{h}_2^{L_2} }_{ } }(t))^{\nu}}(t, \eta )  d \sigma  d\eta d \xi d  t, 
\ee 
where the phase $\Phi^{\mu;\nu}(\xi, \eta, \sigma)$ and the threshold of angular cutoff function are  given as follows, 
\be\label{2020june25eqn4}
\begin{split}
\Phi^{\mu;\nu}(\xi, \eta, \sigma)&:= |\xi+\sigma| +\mu |\xi-\eta| +\nu |\eta|,\quad  \bar{l}_{+;+}=\bar{l}_{+;-}:= 2, \\ 
 \bar{l}_{-;+}=\bar{l}_{-;-}&:=  - 2\beta(|L|)\delta_1 m -\delta_1 m/100 .
\end{split} 
\ee

We remark that, for the non-time-resonance case, i.e., $\mu=+$, there is no need to do angular cutoff and we  did this anyway merely for     consistence in notation. For this case, we have  $J_{k,k_1,k_2}^{\mu\nu; \leq}(t_1, t_2)=0$. Note that, on the physical space, we have
\[
J_{k,k_1,k_2}^{\mu\nu;\star;\leq}(t_1, t_2) :=\int_{t_1}^{t_2} \int_{\R^3} \widetilde{   U^{ \tilde{h}^{L  } }_{k } } (t,x)   T_{\mu\nu}^{ \star;\leq } (    (\widetilde{   U^{ \tilde{h}_1^{L_1} }_{k_1} }(t))^{\mu} , ( \widetilde{   U_{k_2}^{ \tilde{h}_2^{L_2} } }(t))^{\nu} ) )(x)  \psi_{[m-4, m+1]}(x)  d x d t, 
\]
where the bilinear operator $  T_{\mu\nu}^{ \star;\leq } $ is defined by the symbol  $  m_{\mu\nu}^{\star}(\xi-\eta, \eta)  \psi_{\leq \bar{l}_{\mu;\nu}}(\angle(\xi,  -\nu \eta))$. 

Since the null structure of symbol contributes at least one degree of smallness of the angle between $\xi$ and $-\nu \eta$,   from the bilinear estimates (\ref{2020june25eqn11}) and (\ref{nov24eqn4}) in Lemma \ref{2020junebilinearangular}, we have
\[
\sum_{k\in \Z, (k_1,k_2)\in \chi_k^3} 2^{-k_{-}+2\gamma k_{-}+ 2(N_0-|L|) k_{+}} | J_{k,k_1,k_2}^{\mu\nu;\star;\leq}(t_1, t_2)|\lesssim \sum_{\tilde{L}_1\circ \tilde{L}_2\preceq L}   2^{e(L;\tilde{h})\delta_1 m   + \bar{l}_{\mu;\nu}   +H(|\tilde{L}_1|)\delta_1m   } 
\]
\be\label{2020june27eqn7}
  \times \min\{ 2^{H(|\tilde{L}_2|+1)\delta_1 m+\delta_1 m/100 }\mathbf{1}_{|\tilde{L}_2|\geq 1} + 2^{\delta_1 m } \mathbf{1}_{|\tilde{L}_2|=0}\}\epsilon_1^3  \lesssim 2^{2 e(L;\tilde{h})\delta_1 m - \delta_1 m/100 }\epsilon_1^3. 
\ee

Now, we proceed to estimate $J_{k,k_1,k_2}^{\mu\nu;\star; >}(t_1, t_2)$. Note that, for $l\in [2, \bar{l}_{\mu\nu})\cap \Z$, the following estimate holds for the   phase $\Phi^{\mu\nu}(\xi, \eta, \sigma)$ when the angle is relatively large,
\be\label{2020june25eqn5} 
\textup{if\,\,} \angle( \xi,- \nu \eta) \sim  2^{l},|\sigma|\leq 2^{-m +\delta_1 m/10}\quad \Longrightarrow \quad |\Phi^{\mu;\nu}(\xi, \eta, \sigma)| \sim  2^{k_2+ 2l}. 
\ee

To exploit the oscillation of phase in time, we want to do integration by parts in time once for the case $|\sigma|\leq 2^{-m +\delta_1 m/10}$. However, due to the quasilinear nature of the Einstein's equations,    there is an issue of losing derivative after doing integration by parts in time. This issue has been resolved by  Ionescu-Pausader in \cite{IP} by using the paralinearization and symmetrization method.

For the sake of completeness, we   redo this process here. As in the modified profile of the perturbed metric, see \eqref{2020april1eqn4} and \eqref{2022march19eqn11}, we modify the quasilinear variable $\mathcal{U}^{\tilde{h}^{L}}$ introduced in \cite{IP} as well. More precisely, we let
\be\label{2020sep9eqn51}
\widetilde{\mathcal{U}^{\tilde{h}^{L}}}(t)  := {\mathcal{U}^{\tilde{h}^{L}}}(t)  -\widetilde{\rho}^{\tilde{h} }_{L}(f)(t)= \widetilde{U^{\tilde{h}^L}}-i T_{\Lambda_{\geq 1}[\sigma_{wa}]} \tilde{h}^L, \quad \widetilde{\mathcal{V}^{\tilde{h}^{L}}}:= e^{i t\d}\widetilde{\mathcal{U}^{\tilde{h}^{L}}}(t).  
\ee
\be\label{2020aug9eqn1}
\mathcal{U}^{\tilde{h}^{L}} := (\p_t -i T_{\sigma_{wa}}) \tilde{h}^{L}, \quad \sigma_{wa}:=\sqrt{ (\tilde{g}^{0j}\zeta_j)^2 +(\tilde{g}^{jk}\zeta_j\zeta_k) } +\tilde{g}^{0j}\zeta_j, \quad    \tilde{g}^{\mu\nu}= -(g^{00})^{-1} g^{\mu\nu}. 
  \ee
 From  \cite{IP}[Proposition 4.5], the following paralinearization holds, 
 \be\label{2020july21eqn30}
\begin{split}
 (\p_t + i T_{\Sigma_{wa}})\mathcal{U}^{\tilde{h}^{L}}&= - \tilde{g}^{\alpha\beta}\p_{\alpha}\p_{\beta}\tilde{h}^{L} + \sum_{h\in \{h_{\alpha\beta\}}, \iota_1, \iota_2\in \{+,-\}} I_{m_{\iota_1, \iota_2}^h}((U^{h})^{\iota_1}, (U^{\tilde{h}^{L}})^{\iota_2 }),\\
\Sigma_{wa}&= \sqrt{ (\tilde{g}^{0j}\zeta_j)^2 +(\tilde{g}^{jk}\zeta_j\zeta_k) } -\tilde{g}^{0j}\zeta_j,
 \end{split}
 \ee  
 where the Weyl quantization operator $T_a$ is defined in (\ref{weylquantization}) and   the symbols $m_{\iota_1, \iota_2}^h\in \mathcal{M}$, $h\in \{h_{\alpha\beta\}}, \iota_1, \iota_2\in \{+,-\}$, satisfy the following estimate, 
 \be\label{2020july21eqn31}
 \| \mathcal{F}^{-1}[ m_{\iota_1, \iota_2}^h(\xi-\eta, \eta)\psi_{k_1}(\xi-\eta)\psi_{k_2}(\eta)]\|_{L^1_{x,y}} \lesssim \min\{1, 2^{k_2-k_1}\}, \quad \forall k_1, k_2\in \Z.
 \ee 

  Note that $\widetilde{\mathcal{U}^{\tilde{h}^{L}}}(t) $ only differs from $ \widetilde{U^{\tilde{h}^L}}$ at the quadratic and higher order. As a result of $L^2-L^\infty$ type bilinear estimate, we have
  \be\label{2020july22eqn1}
 \begin{split} 
  \sum_{\begin{subarray}{c}
    k\in \Z, (k_1,k_2)\in \chi_k^3\\
     k_2\geq -m +  \alpha_{inp}(|L|)\delta_1 m 
  \end{subarray}
} & 2^{-k_{-}+2\gamma k_{-}+ 2 (N_0-|L|) k_{+}} \big| J_{k,k_1,k_2}^{\mu\nu;\star;\leq}(t_1, t_2) -\mathcal{J}_{k,k_1,k_2}^{\mu\nu;\star;\leq}(t_1, t_2)  \big|\lesssim  \epsilon_1^3.  \\
  &\mathcal{J}_{k,k_1,k_2}^{\mu\nu;\star;\leq}(t_1, t_2) := \int_{t_1}^{t_2} \int_{\R^3}   {   \mathcal{U}^{\tilde{h}^{L}}_{k } } (t,x) T_{-\nu}^{\star;>} (    (\widetilde{   \mathcal{U}^{ \tilde{h}_1^{L_1} }_{k_1} }(t))^{\mu}  , (  {   U_{k_2}^{ \tilde{h}_2^{L_2 } } }(t))^{\nu} ) )(x) d x d t.
 \end{split}
 \ee 
Now the goal is reduced   to the estimate of  $\mathcal{J}_{k,k_1,k_2}^{\mu\nu;\star;\leq}(t_1, t_2).$ After doing  integration by parts in time once for $\mathcal{J}_{k,k_1,k_2}^{\mu\nu;\star;\leq}(t_1, t_2)$ for the case $|\sigma|\leq 2^{-m +\delta_1 m/10}$, we have  
\be\label{2020june27eqn9}
\mathcal{J}_{k,k_1,k_2}^{\mu\nu;\star;\leq}(t_1, t_2) = \sum_{i=1,2} (-1)^{i}   End_{k,k_1,k_2}^{\mu\nu;\star  }(t_i)  +  \widetilde{Err}_{k,k_1,k_2}^{\mu\nu;\star  }(t_1, t_2)  - \widetilde{\mathcal{J}}_{k,k_1,k_2}^{\mu\nu;\star ;i }(t_1, t_2)(t_1, t_2) ,
 \ee
\[
End_{k,k_1,k_2}^{\mu\nu;\star  }(t_i) := \int_{ \R^3 } \int_{ \R^3 }\int_{ \R^3 } e^{it_i\Phi^{\mu;\nu}(\xi, \eta, \sigma)} 2^{3m}\widehat{\psi_{[-4,1]}}(2^{ m}\sigma) \frac{ m_{\mu\nu}^{\star}(\xi-\eta, \eta)\psi_k(\xi) \psi_{k_1}(\xi-\eta) \psi_{k_2}(\eta) }{(i \Phi^{\mu;\nu}(\xi, \eta, \sigma))}  
  \]
\be\label{2020june25eqn33}
\times \psi_{ > \bar{l}_{\mu;\nu}}(\angle(\xi,   -\nu \eta))   \widehat{\widetilde{\mathcal{V}^{ \tilde{h}^{L}}}}(t_i, -\xi-\sigma)  \widehat{(\widetilde{   \mathcal{V}^{ \tilde{h}_1^{L_1} }_{ } } )^{\mu}}(t_i, \xi-\eta )   \widehat{(\widetilde{   V^{ \tilde{h}_2^{L_2} }_{ } } )^{\nu}}(t_i, \eta )  \psi_{\leq -m +\delta_1 m/10 }(\sigma) d \sigma  d\eta d \xi.
\ee
\[
\widetilde{Err}_{k,k_1,k_2}^{\mu\nu;\star  }(t_1, t_2)  :=   \int_{t_1}^{t_2}\int_{ \R^3 } \int_{ \R^3 }\int_{ \R^3 } e^{it\Phi^{\mu;\nu}(\xi, \eta, \sigma)} 2^{3m}\widehat{\psi_{[-4,1]}}(2^{ m}\sigma) m_{\mu\nu}^{\star}(\xi-\eta, \eta)\psi_k(\xi) \psi_{k_1}(\xi-\eta) \psi_{k_2}(\eta) 
  \]
\be\label{2020june25eqn35}
\times \psi_{>  \bar{l}_{\mu;\nu}}(\angle(\xi,   -\nu \eta))   \widehat{\widetilde{\mathcal{V}^{ \tilde{h}^{L}}}}(t, -\xi-\sigma)  \widehat{(\widetilde{  \mathcal{V}^{ \tilde{h}_1^{L_1} }_{ } } )^{\mu}}(t, \xi-\eta )   \widehat{(\widetilde{   V^{ \tilde{h}_2^{L_2} }_{ } } )^{\nu}}(t, \eta ) \psi_{ > -m +\delta_1 m/10 }(\sigma) d \sigma  d\eta d \xi d  t.
\ee
\[
\widetilde{\mathcal{J}}_{k,k_1,k_2}^{\mu\nu;\star;1 }(t_1, t_2)  := \int_{t_1}^{t_2}  \int_{ \R^3 } \int_{ \R^3 }\int_{ \R^3 } e^{it \Phi^{\mu;\nu}(\xi, \eta, \sigma)} 2^{3m}\widehat{\psi_{[-4,1]}}(2^{ m}\sigma)\frac{ m_{\mu\nu}^{\star} (\xi-\eta, \eta)\psi_k(\xi) \psi_{k_1}(\xi-\eta) \psi_{k_2}(\eta) }{ (i \Phi^{\mu;\nu}(\xi, \eta, \sigma)) }
  \]
\be\label{2020june25eqn34}
\times \psi_{ > \bar{l}_{\mu;\nu}}(\angle(\xi,   -\nu \eta))  \widehat{\widetilde{\mathcal{V}^{ \tilde{h}^{L}}}}(t , -\xi-\sigma)  \widehat{(\widetilde{  \mathcal{V}^{ \tilde{h}_1^{L_1} }_{ } } )^{\mu}}(t, \xi-\eta )  \p_t  \widehat{(\widetilde{   V^{ \tilde{h}_2^{L_2} }_{ } } )^{\nu}}(t, \eta )  \psi_{\leq -m +\delta_1 m/10 }(\sigma) d \sigma  d\eta d \xi d t,
\ee 
\[
\widetilde{\mathcal{J}}_{k,k_1,k_2}^{\mu\nu;\star; 2 }(t_1, t_2)  := \int_{t_1}^{t_2}  \int_{ \R^3 } \int_{ \R^3 }\int_{ \R^3 } e^{it \Phi^{\mu;\nu}(\xi, \eta, \sigma)} 2^{3m}\widehat{\psi_{[-4,1]}}(2^{ m}\sigma)\frac{ m_{\mu\nu}^{\star} (\xi-\eta, \eta)\psi_k(\xi) \psi_{k_1}(\xi-\eta) \psi_{k_2}(\eta) }{ (i \Phi^{\mu;\nu}(\xi, \eta, \sigma)) }
  \]
\be\label{2022may19eqn51}
\times \psi_{ > \bar{l}_{\mu;\nu}}(\angle(\xi,   -\nu \eta))  \p_t\big( \widehat{\widetilde{\mathcal{V}^{ \tilde{h}^{L}}}}(t , -\xi-\sigma)  \widehat{(\widetilde{  \mathcal{V}^{ \tilde{h}_1^{L_1} }_{ } } )^{\mu}}(t, \xi-\eta ) \big)   \widehat{(\widetilde{   V^{ \tilde{h}_2^{L_2} }_{ } } )^{\nu}}(t, \eta ) \psi_{\leq -m +\delta_1 m/10 }(\sigma) d \sigma  d\eta d \xi d t.
\ee 
$\bullet$\qquad We first rule out the error type term $ \widetilde{Err}_{k,k_1,k_2}^{\mu\nu;\star  }(t_1, t_2)$. 

From the rapidly decay rate of $\widehat{\psi_{[-4,1]}}(2^{ m}\sigma)$ when $|\sigma|\geq 2^{ -m +\delta_1 m/10 -1}$, from the $L^2-L^2$ type bilinear estimate and the volume of support of $\xi$, the following rough estimate holds, 
\be\label{2020june27eqn3}
  \sum_{k\in \Z, (k_1,k_2)\in \chi_k^3, k_2\geq -m +  \alpha_{inp}(|L|)\delta_1 m } 2^{-k_{-}+2\gamma k_{-}+ 2 (N_0-|L|)k_{+}} \big|\widetilde{Err}_{k,k_1,k_2}^{\mu\nu;\star  }(t_1, t_2) \big| \lesssim 2^{-m } \epsilon_1^3. 
\ee

$\bullet$\qquad Next, we estimate the endpoint terms $End_{k,k_1,k_2}^{\mu\nu;\star  }(t_i), i\in\{1,2\}$ and  $\widetilde{\mathcal{J}}_{k,k_1,k_2}^{\mu\nu;\star;1 }(t_1, t_2).$

Due to the presence of  singular symbol because of the normal form transformation, we need to carefully localize the directions of frequencies. 

  Let $l\in [\bar{l}_{\mu\nu}, 2]\cap \Z$   be fixed and  $\{\theta_1,\cdots, \theta_{K_l}\}\subset \mathbb{S}^2$ be a corresponding ordered  selection of centers on sphere such that $|\theta_i-\theta_j|\gtrsim |p-q|2^{l-2}$, $\forall p,q\in\{1,\cdots, K_l\}$. Based on the possible size of $|L_2|$, we split into two cases as follows.

$\oplus$\qquad If   $|L_2|\geq N_0-10$. 

For this case, we have $|L|\geq N_0-10$ and $|L_1|\leq 10.$  After doing dyadic decomposition for the angle between $\xi$ and $\mu\eta$,  super localizing $ U^{ \tilde{h}_1^{L_1} }_{k_1 }$, 
  and doing  angular decomposition for $\xi/|\xi|$, $(\xi-\eta)/|\xi-\eta|$ and $\eta/|\eta|$ with the sector of size $2^l$ in (\ref{2020june25eqn33}), we have 
  \[
   End_{k,k_1,k_2}^{\mu\nu;\star  }(t_i)=\sum_{l > \bar{l}_{\mu;\nu}, p,q\in\{1, \cdots, K_l\},|p-q|\leq 3, e,f\in2^{k_1-k}[2^{-10}, 2^{10}], |e-f|\leq 10}  End^{e,f}_{l;p,q}(t_i),
  \]
\[
 End^{e,f}_{l;p,q}(t_i):= \int_{\R^3} \int_{\R^3} \int_{\R^3} \int_{\R^3}\widetilde{U^{ \tilde{h}^{L}}_{k,e;p} (t) }(  x-x_1-x_3)(\widetilde{   U^{ \tilde{h}_1^{L_1} }_{k_1,f } }(t_i))^{\mu}(x- x_3)
\]
\be
   \times   (\widetilde{   U^{ \tilde{h}_2^{L_2} }_{ k_2;q} }(t_i ))^{\nu}(x-x_1-x_2-x_3)    \tilde{K}_{k,k_1,k_2}^{\mu\nu;l,p,q}(x_1, x_2,x_3) \psi_{[m-4, m+1]}(x) d x_1 d x_2 d x_3 d x,
\ee
  where
  \[
   \widetilde{U^{ \tilde{h}^{L}}_{k,e;p} }(t, x):= \int_{\R^3} e^{i x\cdot \xi}  \widehat{\widetilde{U^{ \tilde{h}^{L} }}}(t ,  \xi ) \psi_{k}(\xi) \psi_{}( 2^{-k_2}|\xi|-e)  \psi_{\leq l+3}(\xi/|\xi|-\theta_{p}) d \xi, 
  \]
  \[
  \widetilde{   U^{ \tilde{h}_1^{L_1} }_{k_1,f } }(t )  (x ):= \int_{\R^3} e^{i x\cdot \xi}  \widehat{\widetilde{U^{ \tilde{h}_1^{L_1} }}}(t ,  \xi ) \psi_{k_1}(\xi) \psi_{}(  2^{-k_2}|\xi|-f)  d \xi,
\]
 \[
  \widetilde{   U^{ \tilde{h}_2^{L_2} }_{k_2; q} }(t )  (x ):= \int_{\R^3} e^{i x\cdot \xi}  \widehat{\widetilde{U^{ \tilde{h}_2^{L_2} }}}(t ,  \xi )  \psi_{k_2}(\xi) \psi_{\leq l+3}(\xi/|\xi|-\theta_{q})    d \xi,
\]
\[
\tilde{K}_{k,k_1,k_2}^{\mu\nu;l,p,q}(x_1, x_2,x_3):= \int_{\R^3} e^{i x_1\cdot\xi + ix_2\cdot \eta + i x_3 \cdot \sigma} \psi_{ \leq  -m +\delta_1 m/10 }(\sigma)  \frac{ m_{\mu\nu}^{\star}(\xi-\eta, \eta) \psi_{[k-2, k+2]}(\xi) \psi_{k_1}(\xi-\eta)}{(\Phi^{\mu\nu}(\xi, \eta, \sigma))}  
\]
\be\label{2020june27eqn10}
\times \psi_{[k_2-2, k_2+2]}(\eta) \psi_{\leq l+3}(\xi/|\xi|-\theta_{p})\psi_{\leq l+3}(\mu\eta/|\eta|-\theta_{q})
  \psi_{ l}(\angle(\xi,  -\nu \eta)) d \sigma d \eta d \xi,
\ee
 
Note that, by doing integration by parts in $\xi, \eta, \sigma$, from the estimate \eqref{2020june25eqn5},  the following rough estimate holds for the kernel $\tilde{K}_{k,k_1,k_2}^{\mu\nu;l,p,q}(x_1, x_2,x_3)$, 
\[
|\tilde{K}_{k,k_1,k_2}^{\mu\nu;l,p,q}(x_1, x_2,x_3)|\lesssim 2^{-k_2-l} 2^{-3m+3\delta_1 m /10}  2^{3k+2l}  2^{3k_2+2l}(1+2^{-m+\delta_1 m /10}|x_3|)^{-N_0^3}(1+2^{k+l}|x_1\times \theta_p|)^{-N_0^3}
\]
\be
\times (1+2^{k }|x_1\cdot\theta_p|)^{-N_0^3}(1+2^{k_2 }|x_2\cdot \theta_q|)^{-N_0^3}(1+2^{k_2+l}|x_2\times \theta_q |)^{-N_0^3}.
\ee

 From the super localized decay estimates (\ref{july27eqn4}) and (\ref{linearwavedecay2}) in Lemma \ref{superlocalizedaug} and the $L^2-L^2-L^\infty$ type multilinear estimate, if $|L_2|\leq N_0-10$, then  we have
\[
 2^{-k_{-}+2\gamma k_{-}+ 2 (N_0-|L|) k_{+}}\big|End^{e,f}_{l;p,q}(t_i)\big| \lesssim \sum_{\Gamma, \Gamma_1, \Gamma_2\in \{\Omega_{ij}\}} 2^{-k_2-l} 2^{k_1-k_2} 2^{-k_{-}+2\gamma k_{-}+ 2 (N_0-|L|) k_{+}}\|\widetilde{U^{ \tilde{h}^{L}}_{k,e;p}(t) }( x ) \|_{L^2} \]
 \[
 \times\| \widetilde{U^{ \tilde{h}_1^{L_1}}_{k_2;q} (t)}( x)\|_{L^2} 2^{-m+k_2/2}\big( \min\{2^{k_1+k_2/2} \| \widehat{  \widetilde{   U^{ \tilde{h}_1^{L_1} }_{k_1,f } }}(t,\xi)\|_{L^\infty_\xi} ,  \| \Gamma  \widetilde{   U^{ \tilde{h}_1^{L_1} }_{k_1,f } }(t )   \|_{L^2}^{1-\delta_1^2}\| \Gamma_1\Gamma_2  \widetilde{   U^{ \tilde{h}_1^{L_1} }_{k_1,f } }(t )   \|_{L^2}^{\delta_1^2}\}  \]
 \[
 + 2^{-m/3-k_1/3} \| \Gamma_1\Gamma_2  \widetilde{   U^{ \tilde{h}_1^{L_1} }_{k_1,f } }(t )   \|_{L^2}\big)+ 2^{-20m - 10k_{1,+}}\epsilon_1^3.
\]
From the orthogonality in $L^2$, H\"older inequalties, and the above estimate, we have
\[
  \sum_{\begin{subarray}{c}
l > \bar{l}_{\mu;
\nu}, k\in \Z, (k_1,k_2)\in \chi_k^3, k_2\geq -m +  \alpha_{inp}(|L|)\delta_1 m \\ 
p,q\in\{1, \cdots, K_l\}, ,|p-q|\leq 3\\ 
e,f\in2^{k_1-k_2}[2^{-10}, 2^{10}], |e-f|\leq 10
  \end{subarray}
  }   2^{-k_{-}+2\gamma k_{-}+ 2(N_0-|L|)k_{+}} \big|End^{e,f}_{l;p,q}(t_i)\big|
\]
\[
  \lesssim \sum_{k_2\geq  -m +  \alpha_{inp}(|L|)\delta_1 m }  2^{-m-k_2-\bar{l}_{\mu\nu} + e(L;\tilde{h})\delta_1 m + H(|L_2|)\delta_1 m   } \big(2^{  H(|L_1|+1)\delta_1m+ 2\gamma   m }\mathbf{1}_{|L_1|\geq 1}
\] 
\be\label{2020june27eqn1}
  + 2^{\delta_1 m } \mathbf{1}_{|L_1|=0} \big) \epsilon_1^3\lesssim 2^{2 e(L;\tilde{h})\delta_1 m  -  \delta_1 m/100  }\epsilon_1^3.
\ee

$\oplus$\qquad  If   $|L_2|\leq N_0-10$. 

As in the previous caase, we also doing dyadic decomposition for the angle between $\xi$ and $\mu\eta$. However, we don't super-localize $ U^{ \tilde{h}^{L} }_{k }$  and $ U^{ \tilde{h}_1^{L_1} }_{k_1 }$,   and just put $ U^{ \tilde{h}_2^{L_2} }_{k_2 }$  in $L^\infty$. With minor modification in the previous case,  we have 
\[
  \sum_{k\in \Z, (k_1,k_2)\in \chi_k^3, k_2\geq -m +  \alpha_{inp}(|L|)\delta_1 m } 2^{-k_{-}+2\gamma k_{-}+ 2 (N_0-|L|)k_{+}}  \big| End_{k,k_1,k_2}^{\mu\nu;\star  }(t_i)  \big|
\] 
\[
    \lesssim \sum_{k_2\geq  -m +  \alpha_{inp}(|L|)\delta_1 m }  2^{-m-k_2-\bar{l}_{\mu\nu} + e(L;\tilde{h})\delta_1 m + H(|L_1|)\delta_1 m   } \big(2^{  H(|L_2|+1)\delta_1m+ 2\gamma   m }\mathbf{1}_{|L_2|\geq 1} + 2^{\delta_1 m } \mathbf{1}_{|L_2|=0} \big) \epsilon_1^3
\]
\be\label{2022may19eqn41}
\lesssim 2^{2 e(L;\tilde{h})\delta_1 m  -  \delta_1 m/100  }\epsilon_1^3.
  \ee

With minor modifications in the estimate of endpoint terms, by using  the estimate \eqref{oct4eqn41} in Proposition \ref{fixedtimenonlinarityestimate} for $ \p_t  { \widetilde{   V^{ \tilde{h}_2^{L_2} }_{ } }  }(t, \cdot )$, we have 
\[
  \sum_{k\in \Z, (k_1,k_2)\in \chi_k^3, k_2\geq -m +  \alpha_{inp}(|L|)\delta_1 m } 2^{-k_{-}+2\gamma k_{-}+ 2 (N_0-|L|)k_{+}}  \big|\widetilde{\mathcal{J}}_{k,k_1,k_2}^{\mu\nu;\star;1 }(t_1, t_2)\big|
\]
\be
\lesssim 2^{2 e(L;\tilde{h})\delta_1 m  -  \delta_1 m/100  }\epsilon_1^3.
\ee

  $\bullet$\qquad Lastly, we estimate  $\widetilde{\mathcal{J}}_{k,k_1,k_2}^{\mu\nu;\star;2 }(t_1, t_2)    $, see  \eqref{2022may19eqn51}. 

We first rule out    the non-resonance case, $\mu=+$.   Note that, $\forall \sigma \in \R^3, |\sigma|\leq 2^{-m+\delta_1 m}$, we have $|\Phi^{+;\nu}(\xi, \eta, \sigma)|\sim 2^{k_1}$. Hence, we gain one derivative during the normal form transformation, which implies that it allows the one derivative loss for the rough estimate of $ \p_t  { \widetilde{   \mathcal{V}^{ \tilde{h}_2^{L_2} }_{ } }  }(t, \cdot )$. By using the rough    estimate \eqref{oct4eqn41} in Proposition \ref{fixedtimenonlinarityestimate} and the same argument used in the estimate of the endpoint case,  we have 
\be
  \sum_{k\in \Z, (k_1,k_2)\in \chi_k^3, k_2\geq -m +  \alpha_{inp}(|L|)\delta_1 m } 2^{-k_{-}+2\gamma k_{-}+ 2 (N_0-|L|)k_{+}}  \big|\widetilde{\mathcal{J}}_{k,k_1,k_2}^{+\nu;\star;2 }(t_1, t_2)\big|\lesssim 2^{2 e(L;\tilde{h})\delta_1 m  -  \delta_1 m/100  }\epsilon_1^3.
\ee

Now, we focus on the resonance case $\mu=-.$  Note that, from the equality (\ref{2020july21eqn30}) and the equation satisfied by $\tilde{h}^L$ in (\ref{aug11eqn21}), we have 
\be\label{2020july22eqn4}
\begin{split}
e^{-it\d}  \p_t\mathcal{V}^{\tilde{h}^{L}}_{  } & = (\p_t + i\d) \mathcal{U}^{\tilde{h}^{L}}_{ } =  (\p_t + i T_{\Sigma_{wa}})\mathcal{U}^{\tilde{h}^{L}}- i T_{\Lambda_{\geq 1}[\Sigma_{wa}}] \mathcal{U}^{\tilde{h}^{L}}\\ 
& =- i T_{\Lambda_{\geq 1}[\Sigma_{wa}}] \mathcal{U}^{\tilde{h}^{L}} + R_{ }^{\tilde{h}^L  } + L_{ }^{\tilde{h}^L }, \quad \quad L_{ }^{\tilde{h}^L }= - \Lambda_{1}[ N_{vl}^{  {\tilde{h} ;L }}],\\
 R_{ }^{\tilde{h}^L }&= \sum_{h\in \{h_{\alpha\beta\}}, \iota_1, \iota_2\in \{+,-\}} I_{m_{\iota_1, \iota_2}^h}((U^{h})^{\iota_1}, (U^{\tilde{h}^{L}})^{\iota_2 }) 
-\big(\Lambda_{\geq 2}[ N_{vl}^{  {\tilde{h} ;L }}]   + R_{mt;s}^{\tilde{h};L }  + Q_{mt;s}^{ \tilde{h};L}\big).   
\end{split}
\ee

Recall (\ref{2020sep9eqn51}).  Since $\widetilde{\mathcal{V}^{\tilde{h}^{L}}}(t) $ only differs from $ \widetilde{V^{\tilde{h}^L}}$ at the quadratic and higher order and  the linear part  vanishes in the equation satisfied by $ \widetilde{V^{\tilde{h}^L}}$, see \eqref{2020june30eqn1}, we know that $e^{-it\d}  \p_t\mathcal{V}^{\tilde{h}^{L}}_{  }$ consists only quadratic and higher order terms.

From  \eqref{2020july22eqn4},    by using the decomposition of metric into modified perturbed metric and the density type function in (\ref{modifiedperturbedmetric}) if the total number of vector fields    is more than ``$|L|-5$'', otherwise the energy estimate of the perturbed metric is improved, see \eqref{2022march19eqn34},   we decompose $ e^{-it\d}  \p_t\widetilde{\mathcal{V}^{\tilde{h}^{L}}_{k }}$  into three parts as follows, 
\be\label{2020aug10eqn1}
 e^{-it\d}  \p_t\widetilde{\mathcal{V}^{\tilde{h}^{L}}_{k }}  =  (\p_t + i T_{\Lambda_{0}[\Sigma_{wa}}])\widetilde{\mathcal{U}^{\tilde{h}^{L}}}=- i T_{\Lambda_{\geq 1}[\Sigma_{wa}}] \widetilde{\mathcal{U}^{\tilde{h}^{L}}}    +  Q_{vl}^{\tilde{h}^L} + Good^{\tilde{h}^L} ,
\ee
where    ``$Good^{\tilde{h}^L}$'' denotes all good type nonlinearities such that the following estimate holds,
\be\label{2020aug10eqn7}
  \|   Good^{\tilde{h}^L} \|_{E_{|L|}} \lesssim 2^{-m + H(|L|)\delta_1 m + \beta(|L|)\delta_1 m } \epsilon_1^2,
\ee 
and $Q_{vl}^{\tilde{h}^L}  $ denotes the    wave-Vlasov type quadratic terms in which the total number of vector fields act on the Vlasov part is more than ``$|L|-5$'' because otherwise it's good type.  
 Roughly speaking,   $Q_{vl}^{\tilde{h}^L}  $ has the following general formulation, 
\be\label{2020aug10eqn21}
\begin{split}
 \mathcal{F}[Q_{vl}^{\tilde{h}^L}](t, \xi):= \sum_{\begin{subarray}{c}
   L_1\circ L_2 \preceq L, |L_2|\geq |L|-5\\ 
   h\in\{L_1 h_{\alpha\beta}, \widetilde{L_1 h_{\alpha\beta}}\}
 \end{subarray}
 }   \int_{\R^3} \int_{\R^3}& e^{-i t\hat{v}\cdot (\xi-\kappa)}m_{L_1, L_2}^L(\xi, v) \\ 
 & \times  \widehat{ h }(t, \kappa) \widehat{u^{\mathcal{L}_2}}(t, \xi-\kappa, v) \psi_{\leq -3m/4} (\kappa) d \kappa  d v,
 \end{split}
\ee
where the zero order symbol  symbol $m_{L_1, L_2}^L(\xi, v)\in \langle v \rangle^{10}L^\infty_v \mathcal{S}^0$, see \eqref{symbolclas}. We remark that, thanks to the bilinear estimate \eqref{2020july8bilinear} in Lemma \ref{wavevlabil6sep}, the energy estimate of the wave-Vlasov type interaction can be improved (hence, it's good type) if the frequency of the metric part is not so small, e.g., greater than $  2^{-4m/5}.$

Based on the decomposition (\ref{2020aug10eqn1}), we decompose $\widetilde{\mathcal{J}}_{k,k_1,k_2}^{-\nu;\star;2 }(t_1, t_2)$ into three parts as follows, 
\[
 \widetilde{\mathcal{J}}_{k,k_1,k_2}^{-\nu;\star;2 }(t_1, t_2)=\sum_{i=1, 2, 3,4}   \mathcal{H}^{ \nu;\star;i}_{L;L_1,L_2}(k;k_1,k_2), 
\]
where
\be\label{2022may20eqn61}
\begin{split}
 \mathcal{H}^{  \nu;\star;1}_{L;L_1,L_2}(k;k_1,k_2)&=  \int_{t_1}^{t_2}  \int_{ \R^3 } \int_{ \R^3 }\int_{ \R^3 }    \widehat{(\widetilde{   U^{ \tilde{h}_2^{ L_2 } }_{ } } )^{\nu}}(t, \eta ) \big[ \overline{\widehat{ \widetilde{   \mathcal{U}^{ \tilde{h}_1^{L_1 } }_{ } }  }}(t, \eta-\xi )     \widehat{  Good^{\tilde{h}^{L}}  }(t, -\xi-\sigma )\\
 &   +  \widehat{ \widetilde{  \mathcal{U}^{ \tilde{h}^{L } }_{ } } }(t, -\xi-\sigma )     \overline{\widehat{  Good^{\tilde{h}_1^{L_1}}  }(t, \eta -\xi)}   
\big]   m^{\star}(\xi, \eta, \sigma) d \sigma  d\eta d \xi d t,
\end{split}
\ee
\be\label{2022may20eqn63}
\begin{split}
  \mathcal{H}^{  \nu;\star;2}_{L;L_1,L_2}(k;k_1,k_2)&=  \int_{t_1}^{t_2}  \int_{ \R^3 } \int_{ \R^3 }\int_{ \R^3 }   m^{\star}(\xi, \eta, \sigma)  \widehat{(\widetilde{   U^{ \tilde{h}_2^{L_2 } }_{ } } )^{\nu }}(t, \eta )  \big[ \overline{\widehat{ \widetilde{   \mathcal{U}^{ \tilde{h}_1^{L_1 } }_{ } }  }}(t, \eta-\xi ) \widehat{  Q_{vl}^{\tilde{h}^{L}}  }(t, -\xi-\sigma )d \sigma  d\eta d \xi d t,  \\
  \mathcal{H}^{  \nu;\star;3}_{L;L_1,L_2}(k;k_1,k_2)&=  \int_{t_1}^{t_2}  \int_{ \R^3 } \int_{ \R^3 }\int_{ \R^3 }   m^{\star}(\xi, \eta, \sigma)  \widehat{(\widetilde{   U^{ \tilde{h}_2^{L_2 } }_{ } } )^{\nu }}(t, \eta )        \widehat{ \widetilde{  \mathcal{U}^{ \tilde{h}^{L } }_{ } } }(t, -\xi-\sigma )  \overline{ \widehat{ Q_{vl}^{\tilde{h}_1^{L_1}} }(t, \eta -\xi)}
\big]   d \sigma  d\eta d \xi d t,
\end{split}
\ee
\be\label{2020sep9eqn71}
\begin{split}
 &\mathcal{H}^{  \nu;\star;4}_{L;L_1,L_2}(k;k_1,k_2)=  \int_{t_1}^{t_2}  \int_{ \R^3 } \int_{ \R^3 }\int_{ \R^3 } i  m^{\star}(\xi, \eta, \sigma)  \widehat{(\widetilde{   U^{ \tilde{h}_2^{L_2 } }_{ } } )^{\nu}}(t, \eta )  \big[ \overline{\widehat{ \widetilde{   \mathcal{U}^{ \tilde{h}_1^{L_1 } }_{ } }  }}(t, \eta-\xi )    \\
 &\times   \widehat{ \widetilde{  \mathcal{U}^{ \tilde{h}^{L } }_{ } } }  (t, -\xi-\sigma-\kappa ) \widehat{\Lambda_{\geq1}[\Sigma_{wa}]}(t, \kappa, -\xi-\sigma-\kappa) \psi_{\leq -10}(|\kappa|/|\xi+\sigma+\kappa|)- \widehat{ \widetilde{  \mathcal{U}^{ \tilde{h}^{L } }_{ } } }(t, -\xi-\sigma ) \\
 &\times  \overline{ \widehat{ \widetilde{   \mathcal{U}^{ \tilde{h}_1^{L_1 } }_{ } }  }}(t, \eta -\xi-\kappa)   \overline{\widehat{\Lambda_{\geq1}[\Sigma_{wa}]}}(t, \kappa, \eta-\xi-  \kappa)  \psi_{\leq -10}(|\kappa|/|\eta-\xi +\kappa|)
\big]   d \sigma  d\eta d \xi d t,
\end{split}
\ee
where
\[
m^{\star}(\xi, \eta, \sigma) = 2^{3m}\widehat{\psi_{[-4,1]}}(2^{ m}\sigma)  (i \Phi^{-;\nu}(\xi, \eta, \sigma))^{-1} m_{-\nu}^{\star} (\xi-\eta, \eta)\psi_k(\xi) \psi_{k_1}(\xi-\eta) \psi_{k_2}(\eta) 
\]
\be
\times \psi_{ > \bar{l}_{+;-}}(\angle(\xi,     \eta))        \psi_{\leq -m +\delta_1 m/10 }(\sigma).
\ee

 $\oplus$\qquad The estimate of $\mathcal{H}^{   \nu;1}_{L;L_1,L_2}(k;k_1,k_2)$.

Recall (\ref{2022may20eqn61}). Similar to the estimate of endpoint case  $End_{k,k_1,k_2}^{\mu\nu;\star  }(t_i)$ in (\ref{2020june27eqn1}) and \eqref{2022may19eqn41}, by using the same strategy,  the following estimate holds from the estimates (\ref{2020aug10eqn7}),  
\be\label{2020sep10eqn1}
 \sum_{\begin{subarray}{c}
 k\in \Z, (k_1,k_2)\in \chi_k^3 \\
  k_2\geq-m+    \alpha_{inp}(|L|) \delta_1 m  \end{subarray}} 2^{-k_{-}+2\gamma k_{-}+ 2 (N_0-|L|) k_{+}} |\mathcal{H}^{   \nu;1}_{L;L_1,L_2}(k;k_1,k_2)|\lesssim 2^{2 e(L;\tilde{h})\delta_1 m  -  \delta_1 m/100 }\epsilon_1^3. 
\ee 
 
 $\oplus$\qquad The estimate of $\mathcal{H}^{   \nu;\star;2}_{L;L_1,L_2}(k;k_1,k_2)$ and $\mathcal{H}^{   \nu;\star;3}_{L;L_1,L_2}(k;k_1,k_2)$ .

 Recall   (\ref{2022may20eqn63}). Since there is little difference between the estimate of $\mathcal{H}^{   \nu;\star;2}_{L;L_1,L_2}(k;k_1,k_2)$ and $\mathcal{H}^{   \nu;\star;3}_{L;L_1,L_2}(k;k_1,k_2)$, we only provide the estimate of $\mathcal{H}^{   \nu;\star;2}_{L;L_1,L_2}(k;k_1,k_2)$ in details here. 

 From the multi-linear estimates used in the endpoint case and the rough estimate for $   { Q_{vl}^{\tilde{h}_1^{L_1}}} $, we first rule out the case $k_2\geq -m+5H(N_0+3)\delta_0 m $ or $k\leq -5H(N_0+3)\delta_0 m$. It would be sufficient to consider the case  $k_2\leq -m+5H(N_0+3)\delta_0 m $ and $k\geq -5H(N_0+3)\delta_0 m$. Moreover, the rough estimate is only off from the desired estimate by at most $2^{3H(N_0+3)\delta_0 m}$.

 Recall  (\ref{2020aug10eqn21}). On the Fourier side, in terms of the profiles,  we have
 \[
  \big|\mathcal{H}^{  \nu;\star;2}_{L;L_1,L_2}(k;k_1,k_2)\big|\lesssim \sum_{\begin{subarray}{c}
   \tilde{L}_1\circ \tilde{L}_2 \preceq L, |\tilde{L}_2|\geq |L|-5\\
   h\in\{\tilde{L}_1 h_{\alpha\beta}, \widetilde{\tilde{L}_1 h_{\alpha\beta}}\}, \mu\in\{+,-\}
 \end{subarray}
 }  \big|   \int_{t_1}^{t_2}     \int_{ (\R^3)^5 }   e^{-it|\xi-\eta|+i\nu |\eta|+i t\hat{v}\cdot (\xi+\sigma+\kappa)}    \overline{\widehat{ \widetilde{   \mathcal{U}^{ \tilde{h}_1^{L_1 } }_{ } }  }}(t, \eta-\xi ) 
 \]
\[
\times    \widehat{(\widetilde{   U^{ \tilde{h}_2^{L_2 } }_{ } } )^{\nu }}(t, \eta )    \widehat{ V^{h} }(t, \kappa) \widehat{u^{\mathcal{\tilde{L}}_2}}(t, -\xi-\sigma-\kappa, v) 
  m^{\star}(\xi, \eta, \sigma)  m_{\tilde{L}_1, \tilde{L}_2}^L(-\xi-\sigma, v)\psi_{\leq -3m/4} (\kappa) d \kappa   d\eta d \sigma d \xi d v\big|.
\]

Note that, for the range of frequencies we are considering, we have 
\[
\forall v \in B(0, 2^{m/10})\subset \R^3, \quad \big||\xi-\eta|-\nu |\eta| -  \hat{v}\cdot (\xi+\sigma+\kappa)\big| \gtrsim 2^{k_1} \langle v \rangle^{-2}.
\]
Therefore,  either the corresponding phases are all relatively large or the size of $v$  is large. As a result, after doing normal form 
transformation once for the case $|v|\leq 2^{m/10}$ and using the rapidly decay rate of $|v|$ for the case $|v|\geq 2^{m/10}$, we have
\be\label{2020sep10eqn2}
  2^{-k_{-}+2\gamma k_{-}+ 2  (N_0-|L|)  k_{+}}  \big|\mathcal{H}^{  \nu;\star;2}_{L;L_1,L_2}(k;k_1,k_2)\big|\lesssim \epsilon_1^3. 
\ee

 $\oplus$\qquad The estimate of $\mathcal{H}^{    \nu;\star;3}_{L;L_1,L_2}(k;k_1,k_2)$.
 
  Recall (\ref{2020sep9eqn71}). After changing coordinates $(\xi,\kappa)\longrightarrow(\xi+\kappa, -\kappa)$ for the first integral and using the fact that $ \Sigma_{wa}(x, \xi) $ is a real valued symbol,  we have
\[
 | \mathcal{H}^{   \nu;\star;3}_{L;L_1,L_2}(k;k_1,k_2)| \lesssim   \big|\int_{t_1}^{t_2}  \int_{ \R^3 } \int_{ \R^3 } \int_{ \R^3 }\int_{ \R^3 }\overline{\widehat{ \widetilde{   \mathcal{U}^{ \tilde{h}_1^{L_1 } }_{ } }  }}(t, \eta-\xi-\kappa )     \widehat{ \widetilde{  \mathcal{U}^{ \tilde{h}^{L } }_{ } } }  (t, -\xi-\sigma  )
\]
\be
\times \widehat{(\widetilde{   U^{ \tilde{h}_2^{L_2 } }_{ } } )^{\nu}}(t, \eta ) M^{\nu;quasi}_{k;k_1,k_2}(\xi, \eta, \sigma, \kappa) d\kappa d \sigma d \eta d \xi d t\big|,
\ee
where
\[
 M^{\nu;quasi}_{k;k_1,k_2}(\xi, \eta, \sigma, \kappa):=im (\xi+\kappa, \eta, \sigma)\widehat{\Lambda_{\geq1}[\Sigma_{wa}]}(t, -\kappa, -\xi-\sigma )   \psi_{\leq -10}(|\kappa|/|\xi+\sigma |) 
\]
\be
- im (\xi, \eta, \sigma)  {\widehat{\Lambda_{\geq1}[\Sigma_{wa}]}}(t, -\kappa, \eta-\xi-  \kappa)  \psi_{\leq -10}(|\kappa|/|\eta-\xi +\kappa|).
\ee

 From the above formula of symbol, we know that, after symmetrization, the leading order symbol vanishes and it is a zero order symbol, i.e., 
  the above integral doesn't lose derivatives, see also \cite{IP}[Lemma 4.9] for this fact and  more detailed computations. 
 
  Hence, the estimate of $\mathcal{H}^{   \nu;\star;3}_{L;L_1,L_2}(k;k_1,k_2)$ is reduced to the estimate of semilinear type quadratic term with modified perturbed metric as inputs, which are good type. Therefore,  similar to  the estimate of   $\mathcal{H}^{   \nu;1}_{L;L_1,L_2}(k;k_1,k_2)$, we have 
 \be\label{2020sep10eqn11}
 \sum_{\begin{subarray}{c}
 k\in \Z, (k_1,k_2)\in \chi_k^3 \\
  k_2\geq-m+    \alpha_{inp}(|L|) \delta_1 m \end{subarray}} 2^{-k_{-}+2\gamma k_{-}+ 2  (N_0-|L|) k_{+}} |\mathcal{H}^{   \nu;\star;3}_{L;L_1,L_2}(k;k_1,k_2)|\lesssim  2^{2 e(L;\tilde{h})\delta_1 m  -  \delta_1 m/100 }\epsilon_1^3. 
\ee
Hence finishing the proof of our desired estimate \eqref{2022may3eqn31}. 

\qed

  \section{ Z-norm estimate of the perturbed metric component}\label{Znormestimate}

In this section, we will show that, same as in the Einstein-Klein-Gordon system, the perturbed metric components enjoy the modified scattering property. We first briefly introduce the correction of the phase function for the profile of the wave part, which is exactly same as in \cite{IP} and then show that  the Vlasov part only plays a minor role in the $Z$-norm analysis, which is the main goal of this section. Due to the almost same setting of the profile and the bootstrap assumption in \cite{IP} is satisfied, for simplicity, the increment of  $Z$-norm over time  contributes from the Einstein part will be taken as granted. The difference of the $Z$-normed space at the high frequency doesn't play much role in the analysis, only minor modifications are required.  Readers are refer to \cite{IP}[Chapter 5, section 3] for more details. 

For $h\in \{h_{\alpha\beta}\}$, we define the low frequency component of $h$, the phase correction function $\Theta(t,\xi)$, and the phase corrected profile $\widehat{V}_{\ast}^{h_{\alpha\beta}}(t, \xi)$ as follows,
\[
h^{low}(t,x):=\int_{\R^3} e^{i x\cdot\xi} \widehat{h}(t,\xi)\psi_{\leq 0}(\langle t\rangle^{0.68}\xi) d \xi, 
\] 
where
\be
\Theta(t,\xi)= \int_{0}^{t} h_{00}^{low}(s,s\frac{\xi}{|\xi|})\frac{|\xi|}{2}+ h_{0j}^{low}(s,s\frac{\xi}{|\xi|})\xi_j + h_{jk}^{low}(s,s\frac{\xi}{|\xi|}) \frac{\xi_j\xi_k}{2|\xi|} d s,
\ee
\be\label{phasecorrection}
  \widehat{V}_{\ast}^{h_{\alpha\beta}}(t, \xi) := e^{-i \Theta(t,\xi)} \widehat{V}_{ }^{h_{\alpha\beta}}(t, \xi),\quad \widehat{V}_{\ast}^{h_{ }}(t, \xi)=R_{h}^{\alpha\beta}(\xi)\widehat{V}_{\ast}^{h_{\alpha\beta}}(t, \xi), \quad h\in\{F, \omega_j, \vartheta_{mn}\}. 
\ee
In particular, from \cite{IP}[equality (5.3.8)], the following equality holds, 
\[
\p_t \Theta(t, \xi)= \frac{1}{(2\pi)^3} \int_{\R^3} e^{i t\xi\cdot \eta/|\xi|}\big( \widehat{h_{00}^{low}}(t, \eta)\frac{|\xi|}{2} +   \widehat{h_{0j}^{low}}(t, \eta) \xi_j +   \widehat{h_{jk}^{low}}(t, \eta)\frac{\xi_j\xi_k}{2|\xi|}\big) d \eta 
\]
 
As a result, from the equation satisfied by $h_{\alpha\beta}$ in (\ref{april3eqn2}), we have
\[
 \p_t \widehat{V}_{\ast}^{h_{\alpha\beta}}(t, \xi)=  e^{-i \Theta(t,\xi)}\big(\p_t \widehat{V}_{ }^{h_{\alpha\beta}}(t, \xi)- i \p_t \Theta(t, \xi)\widehat{V}_{ }^{h_{\alpha\beta}}(t, \xi) \big)  
=\sum_{i=1, 2,3,4} \mathcal{R}^{h_{\alpha\beta}}_i(t,\xi), 
\]
where
\[
 \mathcal{R}^{h_{\alpha\beta}}_1(t,\xi):=  e^{-i \Theta(t,\xi)}e^{it|\xi|}\big( \mathcal{F}[h_{00}\Delta h_{\alpha\beta}-2 h_{0i}\p_i \p_t h_{\alpha\beta} + h_{ij}\p_i\p_j h_{\alpha\beta}](t,\xi)- i \p_t \Theta(t, \xi)\widehat{V}_{ }^{h_{\alpha\beta}}(t, \xi) \big),
\]
\[
  \mathcal{R}^{h_{\alpha\beta}}_2(t,\xi):=e^{-i \Theta(t,\xi)}e^{it|\xi|}\mathcal{F}[\Lambda_{2}[\mathcal{N}_{\alpha\beta}^{wa}](t, \xi),
\]
\[
 \mathcal{R}^{h_{\alpha\beta}}_3(t,\xi):=e^{-i \Theta(t,\xi)}e^{it|\xi|}\mathcal{F}[\Lambda_{\geq 3}[H_{i\gamma}\p_i\p_{\gamma} h_{\alpha\beta}+ \mathcal{N}_{\alpha\beta}^{wa}](t, \xi),
\]
\be\label{correctedprofile}
  \mathcal{R}^{h_{\alpha\beta}}_4(t,\xi):=e^{-i \Theta(t,\xi)}e^{it|\xi|}\mathcal{F}[\mathcal{N}_{\alpha\beta}^{vl}](t, \xi). 
\ee

Due to the phase correction, the low frequency part $ {h_{\mu\nu}^{low}}$ of $h_{\mu\nu}$ is removed in the quasilinear type quadratic term $h_{\mu\nu}\p_{\mu}\p_{\nu} h_{\alpha\beta}$ in $\mathcal{R}^{h_{\alpha\beta}}_1(t,\xi)$.

\begin{proposition}
Under the bootstrap assumptions \textup{(\ref{BAmetricold})} and \textup{(\ref{BAvlasovori})}, the following estimate holds for any $t\in [2^{m-1}, 2^m]\subset [0,T], m\in \Z_+$, $k\in \Z$
\be\label{2020july27eqn20}
\sum_{h\in \{F, \omega_j, \vartheta_{mn}\}}    2^{1.001k_{-} + 40 k_{+}}\| \psi_k(\xi)\hat{V}^{h}_{\ast}(t, \xi)\|_{L^\infty_\xi}\lesssim   \epsilon_0, \quad     2^{1.001k_{-} + 40k_{+}}\| \psi_k(\xi)\hat{V}^{h_{\alpha\beta}}_{\ast}(t, \xi)\|_{L^\infty_\xi}\lesssim 2^{\delta_1 m  } \epsilon_0. 
\ee
\end{proposition}
\begin{proof}
Recall (\ref{phasecorrection}) and (\ref{correctedprofile}). For any $t_1, t_2\in [2^{n-1}, 2^n]\subset[0,T], n\leq m$, as obtained in \cite{IP}[estimate 5.3.17], the following estimate holds, 
\be\label{2020july27eqn19}
    2^{1.001k_{-} + 40k_{+}}\big( \| \psi_k(\xi)\int_{t_1}^{t_2} \mathcal{R}_1^{h_{\alpha\beta}}(s,\xi) d s  \|_{L^\infty_\xi}+\| \psi_k(\xi)\int_{t_1}^{t_2} \mathcal{R}_3^{h_{\alpha\beta}}(s,\xi) d s  \|_{L^\infty_\xi}\big)\lesssim 2^{-\delta_1 n/2  } \epsilon_1^2.
\ee
As obtained in \cite{IP}[Lemma 5.20], the following estimates hold, 
\be\label{2020july27eqn10}
\begin{split}
    2^{1.001k_{-} + 40k_{+}}  \| \psi_k(\xi)\int_{t_1}^{t_2} \mathcal{R}_2^{h_{\alpha\beta}}(s,\xi) d s  \|_{L^\infty_\xi} &\lesssim 2^{ \delta_1 n/2   } \epsilon_1^2, \\ 
 \sum_{h\in \{F, \omega_j, \vartheta_{mn}\}}     2^{1.001k_{-} + 40k_{+}} \| \psi_k(\xi)\int_{t_1}^{t_2} \mathcal{R}_2^{h_{ }}(s,\xi) d s  \|_{L^\infty_\xi} &\lesssim 2^{ -\delta_1 n/2  } \epsilon_1^2.
\end{split}
\ee
It remains to estimate $\mathcal{R}^{h_{\alpha\beta}}_4(t,\xi)$. Recall (\ref{2020april3eqn1}). From the decay estimate of density type function (\ref{densitydecay}) in Lemma \ref{decayestimateofdensity}, and the improved bootstrap assumption for the Vlasov part  obtained in section \ref{energyestimateVlasov}, the following estimate holds for any $t\in[2^{n-1}, 2^n],$
\[
\big| \Lambda_{1}[\mathcal{F}[\mathcal{N}_{\alpha\beta}^{vl}]](t, \xi)\psi_k(\xi)\big|\lesssim    \min\{2^{-190k_{+}}  2^{ {d}(0, 190)\delta_0 n }, 2^{-3k-3n+ {d}(3, 0)\delta_0 n }\}\epsilon_0. 
\]
From the above estimate, for any $t_1, t_2\in[2^{n-1}, 2^n],$ we have
\be\label{2020july27eqn1}
 2^{1.001k_{-} + 40k_{+}}  \| \psi_k(\xi)\int_{t_1}^{t_2} \Lambda_{1}[ \mathcal{R}_4^{h_{\alpha\beta}}(s,\xi) ] d s  \|_{L^\infty_\xi} \lesssim 2^{- 1.001 n + {d}(3, 0) \delta_0  n} \epsilon_0.
\ee
From the $L^\infty_{x,v}-L^1_{x,v}$ type bilinear estimate and the $L^\infty_x$-type estimates of basic quantities in (\ref{june30eqn1}), (\ref{june29eqn71}), the following estimate holds for any $t\in[2^{n-1}, 2^n],$ 
\be\label{2020july27eqn2}
  2^{1.001k_{-} + 40k_{+}} \| \psi_k(\xi)  \Lambda_{\geq 3}[ \mathcal{R}_4^{h_{\alpha\beta}}(t,\xi) ]   \|_{L^\infty_\xi} \lesssim \sum_{\alpha\in \Z_{+}^3, |\alpha|\leq 42} \| \nabla_x^\alpha \Lambda_{\geq 3}[ \mathcal{N}_{\alpha\beta}^{vl}](t,x)\|_{L^1_{x}} \lesssim 2^{-3n/2}\epsilon_1^3. 
\ee
From the bilinear estimate (\ref{octe364}) in Lemma \ref{bilinearinXnormed}, the following estimate holds for any $t\in[2^{n-1}, 2^n],$
\be\label{2020july27eqn3}
   2^{1.001k_{-} +  40 k_{+}} \| \psi_k(\xi)  \Lambda_{2}[ \mathcal{R}_4^{h_{\alpha\beta}}(t,\xi) ]   \|_{L^\infty_\xi} \lesssim 2^{-3n/2}\epsilon_1^3.
\ee
After combining the estimates (\ref{2020july27eqn1}--\ref{2020july27eqn3}), we have
\be\label{2020july27eqn4}
   2^{1.001k_{-} +  40 k_{+}} \| \psi_k(\xi)\int_{t_1}^{t_2} \mathcal{R}_4^{h_{\alpha\beta}}(s,\xi)  d s  \|_{L^\infty_\xi} \lesssim 2^{- 1.001 n +  {d}(3, 0) \delta_0  n} \epsilon_0.
\ee
To sum up, recall the decomposition of $ \p_t \widehat{V}_{\ast}^{h_{\alpha\beta}}(t, \xi)$ in (\ref{correctedprofile}), our desired estimate (\ref{2020july27eqn1}) holds from the estimates   (\ref{2020july27eqn19}), (\ref{2020july27eqn10}), and (\ref{2020july27eqn4}). 
\end{proof}

Let $\mu\in\{+,-\}, $ $h\in \{U^{h_{\alpha\beta}}, U^{\tilde{h}}, \tilde{h}\in \{F, \underline{F}, \omega_j, \vartheta_{ij} \}\}$ and symbol $m(\xi-\eta,\eta, v) \in   P^{ 100}_v \mathcal{M}^{a;b,c} $, $(a,b,c)\in \{(-1,0,0), (0,-1,0),(0,0,-1)\}$, see \eqref{2022may22eqn31} for the definition of the symbol class $P^{d}_v \mathcal{M}^{a;b,c}$,  we consider the following bilinear form, 
\be\label{octeqn1923}
T^{\mu}  (h, u)(t, \xi):= \int_{\R^3}\int_{\R^3} e^{it|\xi|-i\mu t |\xi-\eta| - i  t \hat{v}\cdot \eta} m(\xi-\eta,\eta, v)\widehat{h^\mu}(t, \xi-\eta) \widehat{u}(t,\eta, v) d \eta d v. 
\ee
For the above defined wave-Vlasov type interaction, in the following Lemma, we prove a $L^\infty_\xi$-type bilinear estimate. 
  \begin{lemma}\label{bilinearinXnormed}
Let $t\in [2^{m-1}, 2^m]\subset [0, T]$ be fixed. Under the bootstrap assumptions \textup{(\ref{BAmetricold})} and \textup{(\ref{BAvlasovori})},  the following estimate holds for the bilinear operator defined in \textup{(\ref{octeqn1923})},
\be\label{octe364}
 \sup_{k\in \mathbb{Z}}    2^{1.001k_{-} + 40 k_{+}} \|   T^{\mu} (h, f)(t, \xi) \psi_{k}(\xi)\|_{L^\infty_\xi}\lesssim 2^{-3m/2}\epsilon_1^2. 
\ee
\end{lemma}
\begin{proof}
After doing dyadic decompositions  for two inputs, we have
\[
 T_{\mu} (h, f)(t, \xi) \psi_{k}(\xi)=\sum_{(k_1, k_2)\in \chi_k^1\cup \chi_k^2\cup \chi_k^3}  T_{k,k_1,k_2}^{\mu} (h, f)(t, \xi),
\]
where
\[
  T_{k,k_1,k_2}^{\mu} (h, f)(t, \xi)=  \int_{\R^3}\int_{\R^3} e^{it|\xi|-i\mu t |\xi-\eta| - i  t \hat{v}\cdot \eta}  \widehat{h^\mu}(t, \xi-\eta) \widehat{u}(t,\eta, v)
\]
\be\label{2020july9eqn91}
\times m(\xi-\eta,\eta, v) \psi_{k}(\xi) \psi_{k_1}(\xi-\eta)\psi_{k_2}(\eta) d \eta d v.
\ee
Note that, from the volume of support of $\eta$, the following estimate holds 
\[
  2^{1.001k_{-} + 40k_{+}}  \| T_{k,k_1,k_2}^{\mu} (h, f)(t, \xi)\|_{L^\infty_\xi} 
\]
\be\label{2020july9eqn92}
\lesssim  2^{ 1.001k_{-}  + 3\min\{k_1,k_2\}-\min\{k,k_1,k_2\}-1.001k_{1,-}+2 H(  0)\delta_0 m }    2^{40(k_{+}- k_{1,+})-40k_{2,+}}\epsilon_1^2. 
\ee
Alternatively, the following estimate holds if we do integration by parts in $v$ five times,  
\[
   2^{1.001k_{-} + 40k_{+}} \| T_{k,k_1,k_2}^{\mu} (h, f)(t, \xi)\|_{L^\infty_\xi}  
\]
\be\label{2020july9eqn93}
  \lesssim   2^{ 1.001k_{-} + 3\min\{k_1,k_2\} -\min\{k,k_1,k_2\} -1.001k_{1,-}+2 {d}(5,160)\delta_0 m } 2^{-5m - 5k_2}  2^{40(k_{+}- k_{1,+})-150k_{2,+} }\epsilon_1^2.  
\ee
From the above two estimates and the $L^2-L^2$ type bilinear estimate at high frequency, we have
\[
\sum_{ (k_1, k_2)\in \chi_k^1\cup \chi_k^2\cup \chi_k^3 }      2^{1.001k_{-} + 40k_{+}}  \| T_{k,k_1,k_2}^{\mu} (h, f)(t, \xi)\|_{L^\infty_\xi} 
 \]
 \be
 \lesssim \sum_{k \leq -m} 2^{2k +2 H( 0)\delta_0 m } \epsilon_1^2  +\sum_{k \geq -m } 2^{-5m-3k   +2 {d }(5,160)\delta_0 m} \epsilon_1^2 \lesssim 2^{-3m/2}\epsilon_1^2. 
\ee
Hence finishing the proof of our desired estimate (\ref{octe364}). 
\end{proof}

\section{Energy estimate   for the Vlasov part}\label{energyestimateVlasov}
 
Recall (\ref{BAvlasovori}) and (\ref{2020april8eqn31}). Due to the small data regime, it would be sufficient to consider the Cauchy problem at time $t=1$. Note that,  for any $T^{\ast}\in[1, T)$, we have 
\[  
 \| \omega^{\mathcal{L}}_{\rho} (T^{\ast},x, v) \Lambda^{\rho}{u}^{\mathcal{L}}(T^{\ast} ,x,v)\|_{L^2_{x,v}}^2 -  \| \omega^{\mathcal{L}}_{\rho} (1,x, v) \Lambda^{\rho}{u}^{\mathcal{L}}(1 ,x,v)\|_{L^2_{x,v}}^2
 \]
 \[
  = \int_{1}^{T^{\ast}} \int_{\R^3} \int_{\R^3} 2\big( \omega^{\mathcal{L}}_{\rho} ( t, x, v)\big)^2 \Lambda^{\rho}{u}^{\mathcal{L}}(t,x,v)\p_t \Lambda^{\rho}{u}^{\mathcal{L}}(t ,x, v)  
 \]
\be\label{2020may2eqn1}
+   2 \omega^{\mathcal{L}}_{\rho} ( t, x, v)\big(\Lambda^{\rho}{u}^{\mathcal{L}}(t,x,v)\big)^2 \p_t \omega^{\mathcal{L}}_{\rho} ( t , x, v) d x d  v dt.
\ee
From the estimate (\ref{2020april27eqn61}), we have
\be\label{2020may2eqn2}
\int_{1}^{T^{\ast}} \int_{\R^3} \int_{\R^3} 2 \omega^{\mathcal{L}}_{\rho} ( t, x, v)\big(\Lambda^{\rho}{u}^{\mathcal{L}}(t,x,v)\big)^2 \p_t \omega^{\mathcal{L}}_{\rho} ( t, x, v) d x d  v d t\leq 0. 
\ee
Therefore, to control the weighted energy at time $T^{\ast}$, it would be sufficient to control the first integral in (\ref{2020may2eqn1}).

 For any $t_1, t_2\in [2^{m-1}, 2^m]\subset [0, T]$, from the estimate (\ref{aug9eqn87}) in Lemma \ref{estimateofremaindervlasov} and the equation satisfied by $\Lambda^\rho u^{\mathcal{L}}(t,x,v)$ in (\ref{2020april8eqn31}), after doing dyadic decomposition for the metric component and  classifying the nonlinearities based on the distribution of vector fields,  we have 
 \[
  \big| \int_{t_1}^{t_2} \int_{\R^3} \int_{\R^3} \big( \omega^{\mathcal{L}}_{\rho} (t, x, v)\big)^2 \Lambda^{\rho}{u}^{\mathcal{L}}(t,x,v)\p_t \Lambda^{\rho}{u}^{\mathcal{L}}(t,x, v) d x d  v d t\big|
 \]
\be\label{may7eqn1}
 \lesssim  \epsilon_1^3 + \sum_{i=1,\cdots,5} \big| \sum_{k\in\mathbb{Z}} I_{\rho;i}^{\mathcal{L}}(k,t_1, t_2)\big|  , 
\ee
 where
 \[
 I_{\rho;1}^{\mathcal{L}}(k,t_1, t_2)=  \int_{t_1}^{t_2} \int_{\R^3} \int_{\R^3}    \big(\Lambda^{\rho}{u}^{\mathcal{L}}(t,x,v) \big)^2 \omega^{\mathcal{L}}_{\rho} (t, x, v) v\cdot \nabla_x  \omega^{\mathcal{L}}_{\rho} (t, x, v)  
 \]
\be\label{2020april15eqn11}
 \times P_{k}\big(\Lambda_{ 1}[(v^0)^{-1}]\big)(t, x+t\hat{v})    d x d  v d t, 
\ee
\[
 I_{\rho;2}^{\mathcal{L}}(k,t_1, t_2)=  \int_{t_1}^{t_2} \int_{\R^3} \int_{\R^3}   \big(\Lambda^{\rho}{u}^{\mathcal{L}}(t,x,v) \big)^2 \omega^{\mathcal{L}}_{\rho} ( t,x, v) \langle v \rangle \hat{v}_\alpha \hat{v}_\beta
\]
\be\label{2020april15eqn12}
 \times \nabla_x P_k(h_{\alpha\beta})(t,x+t\hat{v})\cdot     D_v  \omega^{\mathcal{L}}_{\rho} (t, x, v)    d x d  v d t , 
\ee
\[
I_{\rho;3}^{\mathcal{L}}(k,t_1, t_2)=  \sum_{\begin{subarray}{c}
\tilde{L}\in\cup_{l +|\alpha|\leq N_0  }\nabla_x^\alpha P_l   \\ 
  |\tilde{L}|\leq \tilde{c}(\rho),   |\alpha| \leq |\rho|-\tilde{c}(\rho)\\
\end{subarray}}     \int_{t_1}^{t_2} \int_{\R^3} \int_{\R^3}   \big(\omega^{\mathcal{L}}_{\rho} (t, x, v)   \big)^2   \Lambda^{\rho}{u}^{\mathcal{L}}(t,x,v) C^{\tilde{L} }_{\mu\nu;\rho}(x,v) 
\]
\be\label{2020april15eqn13}
 \times \nabla_x P_k(  \tilde{L} \mathcal{L} h_{\mu\nu}) (t, x+t\hat{v}) \cdot D_v u^{ }(t,x,v) \mathbf{1}_{|  \rho |+|\mathcal{L}|\geq 2} d x d v d t, 
\ee
 \[
  I_{\rho;4}^{\mathcal{L}}(k,t_1, t_2)= \sum_{ \begin{subarray}{c}  ( \Gamma_1, \Gamma_2)\in P\times\{\p_{\alpha}\} \cup \{\p_{\alpha}\}\times \mathcal{P}\\ 
  (\tilde{L}, \rho_1, \rho_2, \mathcal{L}_1, \mathcal{L}_2)\in S_{\rho, \mathcal{L}}^1
  \end{subarray} }   \int_{t_1}^{t_2} \int_{\R^3} \int_{\R^3}  \big(\omega^{\mathcal{L}}_{\rho} (t,  x, v)   \big)^2   \Lambda^{\rho}{u}^{\mathcal{L}}(t,x,v)
 \]
 \[
  \times \big[P_k( \tilde{L}\mathcal{L}_2  h_{\mu\nu} ) (t, x+t\hat{v} ){C}_{\mathcal{L}_1\mathcal{L}_2;\tilde{L}, \mu, \nu }^{\mathcal{L};\rho,\rho_1,\rho_2}(x,v)  \cdot \nabla_x \Lambda^{\rho_1}  u^{\mathcal{L}_1}(t,x,v) 
 \]
 \be\label{2020april15eqn15}
  + C_{ \tilde{\mathcal{L}}_1 \tilde{\mathcal{L}}_2 ;\tilde{L},  \mu,\nu}^{  {\mathcal{L}};\Gamma_1,\Gamma_2;\rho, \rho_1, \rho_2}(x,v)    P_k(  \tilde{L}  \Gamma_2  {\mathcal{L}}_2) h_{\mu\nu}(t,x+t\hat{v})\Lambda^{\rho_1} u^{\Gamma_1  {\mathcal{L}}_1}(t,x,v)\big] d x d v d t, 
 \ee
  \[
  I_{\rho;5}^{\mathcal{L}}(k,t_1, t_2)=  \sum_{ \begin{subarray}{c}  ( \Gamma_1, \Gamma_2)\in P\times\{\p_{\alpha}\} \cup \{\p_{\alpha}\}\times \mathcal{P}\\ 
  (\tilde{L}, \rho_1, \rho_2, \mathcal{L}_1, \mathcal{L}_2)\in S_{\rho, \mathcal{L}}^2
  \end{subarray} }    \int_{t_1}^{t_2} \int_{\R^3} \int_{\R^3}  \big(\omega^{\mathcal{L}}_{\rho} ( t,  x, v)   \big)^2   \Lambda^{\rho}{u}^{\mathcal{L}}(t,x,v)
  \]
  \[
  \times \big[P_k( \tilde{L}\mathcal{L}_2  \Lambda_{1}[(v^0)^{-1}] ) (t, x+t\hat{v} )  {C}_{\mathcal{L}_1\mathcal{L}_2;\tilde{L} }^{\mathcal{L};\rho,\rho_1,\rho_2}(x,v)  \cdot \nabla_x \Lambda^{\rho_1}  u^{\mathcal{L}_1}(t,x,v) 
  \]
  \be\label{2020april15eqn51}
  + C_{ \tilde{\mathcal{L}}_1 \tilde{\mathcal{L}}_2 ;\tilde{L},  \mu,\nu}^{  {\mathcal{L}};\Gamma_1,\Gamma_2;\rho, \rho_1, \rho_2}(x,v)  P_k(  \tilde{L}  \Gamma_2  {\mathcal{L}}_2) h_{\mu\nu}(t,x+t\hat{v}) \Lambda^{\rho_1} u^{\Gamma_1  {\mathcal{L}}_1}(t,x,v)  \big]d x d v d t,
  \ee 
 where $\hat{v}_0:=1$ and the sets $S_{\rho, \mathcal{L}}^i, i\in\{1,2\},$ are defined as follows,  
\begin{multline}\label{2022may6eqn1}
S_{\rho, \mathcal{L}}^1:=\{( \tilde{L},   \rho_1, \rho_2, \mathcal{L}_1, \mathcal{L}_2):  \rho_1\circ\rho_2\preceq \rho, \mathcal{L}_1\circ\mathcal{L}_2\preceq \mathcal{L},      | \rho_1 | + |\mathcal{L}_1|< |\mathcal{L}| +  | \rho | -1, \\ 
   | \rho_2 | + |\mathcal{L}_2|< |\mathcal{L}|+| \rho |, \tilde{L}  \in\cup_{ l\in [0, \tilde{c}(\rho_2)]\cap \Z, |\alpha|\in [0, |\rho_2|-\tilde{c}(\rho_2)]\cap \Z } \nabla_x^\alpha P_l \},
 \end{multline}
 \begin{multline}\label{2022may6eqn2}
 S_{\rho, \mathcal{L}}^2:=\{( \tilde{L},   \rho_1, \rho_2, \mathcal{L}_1, \mathcal{L}_2): \rho_1\circ\rho_2\preceq \rho, \mathcal{L}_1\circ\mathcal{L}_2\preceq \mathcal{L},     | \rho_1 | + |\mathcal{L}_1|= |\mathcal{L}| +  | \rho  |-1, \\ 
   | \rho_2 | + |\mathcal{L}_2|\leq 1,  \tilde{L}  \in \cup_{ l\in [0, \tilde{c}(\rho_2)]\cap \Z, |\alpha|\in [0, |\rho_2|-\tilde{c}(\rho_2)]\cap \Z }  \nabla_x^\alpha P_l     \}. 
\end{multline}

The following estimate holds for the coefficients,
\be\label{2020sep10eqn51} 
\begin{split}
| C^{\tilde{L} }_{\mu\nu;\rho}(x,v) | & \lesssim (1+|x|)^{|\rho |-|\tilde{L}|}(1+|v|)^{|\rho |-|\tilde{L}|- c(\rho )},\\ 
|   {C}_{\mathcal{L}_1\mathcal{L}_2;\tilde{L} }^{\mathcal{L};\rho,\rho_1,\rho_2}(x,v) | & \lesssim (1+|x|)^{|\rho_2|-|\tilde{L}|}(1+|v|)^{ 1+ |\rho_2|-|\tilde{L}|- c(\rho_2)}, \\
  |   {C}_{\mathcal{L}_1\mathcal{L}_2;\tilde{L}, \mu, \nu }^{\mathcal{L};\rho,\rho_1,\rho_2}(x,v)  |+     |   C_{ \tilde{\mathcal{L}}_1 \tilde{\mathcal{L}}_2 ;\tilde{L},  \mu,\nu}^{  {\mathcal{L}};\Gamma_1,\Gamma_2;\rho, \rho_1, \rho_2}(x,v)  | & \lesssim (1+|x|)^{|\rho_2|-|\tilde{L}|}(1+|v|)^{ |\rho_2|-|\tilde{L}|- c(\rho_2)}. 
\end{split}
\ee

The main goal of this section is devoted to prove the following proposition. 

\begin{proposition}\label{energyvlasovest}
 Under the bootstrap assumptions \textup{(\ref{BAmetricold})} and \textup{(\ref{BAvlasovori})}, the following estimate holds for any   $t \in [2^{m-1}, 2^m]\subset [0, T]$, $\mathcal{L}\in \cup_{l\leq N_0}  \mathcal{P}_l$, $\rho\in \mathcal{S},$ s.t.,  $ | \mathcal{L}| +   | \rho |  \leq N_0,$  
 \be\label{may7eqn21}
  \| \omega^{\mathcal{L}}_{\rho} (t,x, v) \Lambda^{\rho}{u}^{\mathcal{L}}(t ,x,v)\|_{L^2_{x,v}}  \lesssim 2^{  H(\mathcal{L}, \rho)\delta_0 m  }\epsilon_0 . 
\ee
\end{proposition}
\begin{proof}
Recall the equality (\ref{2020may2eqn1}), the estimates (\ref{2020may2eqn2}) and  (\ref{may7eqn1}). From the estimate (\ref{2020may2eqn11}) in Lemma \ref{vlasovenergy1}, the estimate  (\ref{2020april15eqn31}) in Lemma \ref{vlasovenergy2}, the estimate (\ref{2020april21eqn71}) in Lemma \ref{notbulkcase}, the estimate (\ref{2020april22eqn1}) in Lemma \ref{vlasovenergy4}, the estimate (\ref{2020june1eqn3}) in Lemma \ref{vlasovenergy5},  we have
\[
  \| \omega^{\mathcal{L}}_{\rho} (t,x, v) \Lambda^{\rho}{u}^{\mathcal{L}}(t ,x,v)\|_{L^2_{x,v}}^2\lesssim \epsilon_0^2 + \sum_{n\in [1,m]\cap \mathbb{Z}}   2^{ 2 H(\mathcal{L}, \rho)\delta_0 n  } \epsilon_1^3  \lesssim  2^{  2 H(\mathcal{L}, \rho) \delta_0 m }\epsilon_0^2.
\]
Hence finishing the proof of our desired estimate (\ref{may7eqn21}).
\end{proof}

 \begin{lemma}\label{vlasovenergy1}
 Under the bootstrap assumptions \textup{(\ref{BAmetricold})} and \textup{(\ref{BAvlasovori})}, the following estimate holds for any   $t_1, t_2\in [2^{m-1}, 2^m]\subset [0, T]$,   $\mathcal{L}\in \cup_{l\leq N_0 }\mathcal{P}_l$, $\rho\in \mathcal{S},$ s.t., $ | \mathcal{L}| +   | \rho |  \leq N_0,$  
 \be\label{2020may2eqn11}
 \big|\sum_{k\in \mathbb{Z}}  I_{\rho;1}^{\mathcal{L}}(k,t_1, t_2)  \big|   \lesssim  2^{2  H(\mathcal{L}, \rho)\delta_0 m- \delta_0 m }\epsilon_1^3.
 \ee
 \end{lemma}
 \begin{proof}
 Recall (\ref{2020april15eqn11}) and the choice of weight functions in (\ref{weightfunctions2}).   As a result of direct computations, we have
 \be\label{2020may3eqn11}
\big|\frac{  \nabla_x  \omega^{\mathcal{L}}_{\rho} (t, x, v)  }{ \omega^{\mathcal{L}}_{\rho} (t, x, v) }\big|\lesssim  \mathbf{1}_{|x|\geq 2^{  4H( 4)\delta_1 m-3}}\big[\frac{\phi(|x|/(1+|t|)^{ 4H( 4)\delta_1 }) +  2^{-4H(4)\delta_1 m }|x|\phi'(|x|/(1+|t|)^{4H( 4)\delta_1 })}{1+|v|^{1/\delta_1 }+|x|\phi(|x|/(1+|t|)^{4H( 4)\delta_1 })}  \big].
 \ee
 Recall (\ref{2020april27eqn61}). Note that the following estimate holds if $||x|/(1+|t|)^{4H(  4)\delta_1} -1|\leq 2^{- 2H(  4)\delta_1 m }, $
 \begin{multline}
|\phi'(|x|/(1+|t|)^{4H(4)\delta_1 })|\lesssim 2^{-2H( 4)\delta_1 m },\quad \\
 \Longrightarrow   \big|\frac{  \nabla_x  \omega^{\mathcal{L}}_{\rho} (t, x, v)  }{  \omega^{\mathcal{L}}_{\rho} (t, x, v) }\big|\lesssim \frac{1}{|x|}\mathbf{1}_{|x|\geq 2^{  4H( 4)\delta_1 m-3}}+ |\phi'(|x|/(1+|t|)^{ 4H( 4)\delta_1 })|\lesssim  2^{- 2H( 4 )\delta_1 m   }.
 \end{multline}
 If $||x|/(1+|t|)^{4H( 4)\delta_1 } -1|\geq 2^{-2H( 4)\delta_1 m }, $ then we have
 \[
  \big|\frac{  \nabla_x  \omega^{\mathcal{L}}_{\rho} (t, x, v)  }{  \omega^{\mathcal{L}}_{\rho} (t, x, v) }\big|\lesssim \mathbf{1}_{|x|\geq 2^{  4H( 4)\delta_1 m-3}}  \frac{1}{|x|} + 2^{-4H( 4)\delta_1 m }  |\phi (|x|/(1+|t|)^{4H( 4)\delta_1 })|^{-1} \lesssim 2^{- 2H( 4)\delta_1 m  }.
 \]
 Therefore, in whichever case, we have 
 \be\label{2020may2eqn15}
 \big|\frac{ \nabla_x  \omega^{\mathcal{L}}_{\rho} (t, x, v)  }{   \omega^{\mathcal{L}}_{\rho} (t, x, v) }\big|\lesssim 2^{-2H( 4)\delta_1 m  }.
 \ee

 The above    smallness of $ { \nabla_x  \omega^{\mathcal{L}}_{\rho} (t, x, v)  }{(     \omega^{\mathcal{L}}_{\rho} (t, x, v))^{-1} }$ fills the gap of the sharp decay rate for the perturbed metric. As a result, our desired estimate (\ref{2020may2eqn11}) holds from the above estimate, the $L^\infty_{x,v}-L^2_{x,v}-L^2_{x,v}$ type multilinear estimate, and the  decay estimates (\ref{2020julybasicestiamte}) in Lemma \ref{basicestimates}. 
 \end{proof}
   \begin{lemma}\label{vlasovenergy5}
 Under the bootstrap assumptions \textup{(\ref{BAmetricold})} and \textup{(\ref{BAvlasovori})}, the following estimate holds for any   $t_1, t_2\in [2^{m-1}, 2^m]\subset [0, T]$,  $\mathcal{L}\in \cup_{l\leq N_0}\mathcal{P}_l, \rho\in \mathcal{S}, $  s.t.,  $ | \mathcal{L}| +   | \rho |  \leq N_0,$  
\be\label{2020june1eqn3}
\big|  \sum_{ k\in \Z} I_{\rho;2}^{\mathcal{L}}(k,t_1, t_2)\big|  \lesssim 2^{2H( \mathcal{L},\rho )\delta_0 m }\epsilon_1^3 .
\ee
 \end{lemma}

 \begin{proof}
 Recall (\ref{2020april15eqn12}). Let $\alpha :=  H(4)\delta_1    -2\delta_1.$ From the obtained  estimate (\ref{2020may2eqn15}),   the $L^\infty_{x,v}-L^2_{x,v}-L^2_{x,v}$ type multilinear estimate, and the $L^\infty\longrightarrow L^2$ type Sobolev embedding, we have
 \be\label{2020may2eqn20}
 \sum_{k\in \mathbb{Z}, k\leq -m +  \alpha m } \big| I_{\rho;2}^{\mathcal{L}}(k,t_1, t_2)  \big| \lesssim 2^{m+ 2 H( \mathcal{L},\rho )\delta_0m } 2^{2k+2\delta_1m } \big(1+2^{m- 2H(4)\delta_1 m } \big)\epsilon_1^3    \lesssim  2^{   2 H( \mathcal{L},\rho )\delta_0m }\epsilon_1^3. 
 \ee
 
 Now we restrict ourselves  to the case $k\geq -m +  \alpha m, k\in \Z$. 
Recall   the definition of the  weight function $  \omega^{\mathcal{L}}_{\rho} (t, x, v) $ in (\ref{weightfunctions2}).
Motivated from the   estimate (\ref{2020may3eqn11}), we decompose $D_v  \omega^{\mathcal{L}}_{\rho} (t, x, v)$ into three parts as follows, 
\be
D_v  \omega^{\mathcal{L}}_{\rho} (t, x, v)=   \sum_{i= 0,1,2 }   \widetilde{\omega^{\mathcal{L}}_{\rho;i}} (t, x, v),  
\ee
where
\be\label{2020may3eqn12}
 \widetilde{\omega^{\mathcal{L}}_{\rho;0}} (t, x, v):=  \nabla_v  \omega^{\mathcal{L}}_{\rho} (t, x, v)- \psi_{\geq m/10}(|v|) t \nabla_v \hat{v}\cdot \nabla_x \omega^{\mathcal{L}}_{\rho} (t, x, v),
\ee
\be\label{2020may3eqn14}
\widetilde{\omega^{\mathcal{L}}_{\rho;1}} (t, x, v):=-\psi_{< m/10}(|v|) \psi_{\geq 3m/5}(|x|) t \nabla_v \hat{v}\cdot \nabla_x \omega^{\mathcal{L}}_{\rho} (t, x, v),
\ee
\be\label{2020may3eqn15}
\widetilde{\omega^{\mathcal{L}}_{\rho;2 }} (t, x, v):=-\psi_{<  m/10}(|v|) \psi_{<  3m/5}(|x|)t \nabla_v \hat{v}\cdot \nabla_x \omega^{\mathcal{L}}_{\rho} (t, x, v).
\ee
As a result of direct computations and the estimate (\ref{2020may3eqn11}), we have
\be\label{2020may3eqn16}
 | \widetilde{\omega^{\mathcal{L}}_{\rho;0}} (t, x, v)| \lesssim (1+|v|)^{-1}\omega^{\mathcal{L}}_{\rho} (t, x, v),\qquad  | \widetilde{\omega^{\mathcal{L}}_{\rho;1 }} (t, x, v)| \lesssim 2^{ m/2} (1+|v|)^{-5}\omega^{\mathcal{L}}_{\rho} (t, x, v),
 \ee
 \be\label{2020may3eqn17}
   | \widetilde{\omega^{\mathcal{L}}_{\rho; 2 }} (t, x, v)| (1+|t-|x+t\hat{v}||)^{-2} \lesssim 2^{-m} (1+|v|)^{-5}\omega^{\mathcal{L}}_{\rho} (t, x, v). 
\ee

Correspondingly, we have
 \be\label{2020may3eqn18}
 I_{\rho;2  }^{\mathcal{L}}(k,t_1, t_2)=\sum_{i=0,1,2}  \tilde{I}_{\rho;i}^{\mathcal{L}}(k,t_1, t_2),
\ee
 \be\label{2020may3eqn19}
\tilde{I}_{\rho;i}^{\mathcal{L}}(k,t_1, t_2)=\int_{t_1}^{t_2} \int_{\R^3} \int_{\R^3}   \big(\Lambda^{\rho}{u}^{\mathcal{L}}(t,x,v) \big)^2 \omega^{\mathcal{L}}_{\rho} ( t,x, v) \langle v \rangle \hat{v}_\alpha \hat{v}_\beta \nabla_x P_k(h_{\alpha\beta})(t,x+t\hat{v})\cdot       \widetilde{\omega^{\mathcal{L}}_{\rho;i}} (t, x, v)     d x d  v d t .
\ee

\noindent $\bullet$\quad The estimate of $\tilde{I}_{\rho;0}^{\mathcal{L}}(k,t_1, t_2)$.

Recall the estimate (\ref{2020may3eqn16}). Due to the extra derivative $\nabla_x$ of the metric part $\nabla_x P_k(h_{\alpha\beta})(\cdot, \cdot)$ in \eqref{2020april15eqn12}, we  first rule out the   case $k\notin [-2H(4)\delta_1 m, H(4)\delta_1 m ]$ by using the $L^2_{x,v}-L^2_{x,v}-L^\infty_{x,v}$ type multilinear estimate. More precisely, the following estimate holds, 
\be\label{2020may3eqn41}
\sum_{k\notin [-2H(4)\delta_1 m, H(4)\delta_1 m ]} |\tilde{I}_{\rho;0}^{\mathcal{L}}(k,t_1, t_2)| \lesssim 2^{k-4k_{+}+2 H( \mathcal{L},\rho )\delta_0 m + H(2)\delta_1m } \epsilon_1^3 \lesssim 2^{ 2 H( \mathcal{L},\rho )\delta_0 m } \epsilon_1^3 .
\ee 

It remains to consider the case $k\in [-2H(4)\delta_1 m, H(4)\delta_1 m  ]$.  For this case, 
 we utilize the hidden null structure of $\hat{v}_\alpha \hat{v}_\beta  P_k( L h_{\alpha\beta})(\cdot), L\in \cup_{l\leq N_0}P_l.$ 

To better see the hidden null structure, we use the double null decomposition variables. Recall (\ref{june16eqn51}). From the equalities in  (\ref{april1eqn1}) and   (\ref{april1eqn11}), $\forall L\in \cup_{l\leq N_0}P_l$,  we have 
 \be\label{april26eqn1}
   \hat{v}_{\mu}\hat{v}_{\nu} P_k( L h_{\mu\nu}) = \sum_{\tilde{h}\in\{F, \underline{F}, \omega_l,  \vartheta_{mn}\}}
  \sum_{\mu\in\{+,-\}} \big(T_{k;\mu}^{\tilde{h}}(v) U^{\tilde{h}^{ L }  }\big)^{\mu}+ \mathcal{R}_k,
  \ee
where   the symbol $T_{k;\mu}^{\tilde{h}}(v)(\xi), \tilde{h}\in \{F, \underline{F}, \omega_l, \vartheta_{mn}\},   $ of the Fourier multiplier operator $T_{k;\mu}^{\tilde{h}}(v)$  and the quadratic and higher order error term ``$\mathcal{R}_k$'' are given as follows, 
 \be\label{april26eqn2}
 \begin{split}
T_{k;\mu}^{F}(v)(\xi) & = c_{\mu}|\xi|^{-1}( 1-(\hat{v}\cdot \xi/|\xi|)^2)\psi_k(\xi), \\ 
T_{k;\mu}^{\underline{F}}(v)(\xi)&= |\xi|^{-1}(1- 2\mu \hat{v}\cdot \xi/|\xi| + (\hat{v}\cdot \xi/|\xi|)^2)\psi_k(\xi),\\ 
  T_{k;\mu}^{\omega_l }(v)(\xi)&=  c_{\mu}\epsilon_{ikl} (1-\mu\hat{v}\cdot \xi/|\xi|) \psi_k(\xi)(\hat{v}_i \xi_k-\hat{v}_k \xi_i)/|\xi|^2,  \\  T_{k;\mu}^{\vartheta_{ab}}(v)(\xi)&=\frac{\epsilon_{mpa}\epsilon_{nqb}}{4|\xi|^3} ( \hat{v}_m \xi_p -\hat{v}_p\xi_m) (\hat{v}_n   \xi_q-\hat{v}_q   \xi_n) \psi_k(\xi),
\end{split}
\ee
\be\label{2022may7eqn61}
 \mathcal{R}_k= - 2\hat{v}_i P_k R_i\big(\rho^{\tilde{L}\tilde{\mathcal{L}}_2}   -R_0\underline{F}^{\tilde{L}\tilde{\mathcal{L}}_2}   \big) -\hat{v}_m \hat{v}_n  \big( \varepsilon_{mpq} R_n + \varepsilon_{npq}R_m \big) R_p P_k \big( \Omega_q^{\tilde{L}\tilde{\mathcal{L}}_2 }- R_0\omega_q^{\tilde{L}\tilde{\mathcal{L}}_2 }\big). 
 \ee
Moreover,   from (\ref{april26eqn1}), we have
 \be\label{april26eqn11}
 \hat{v}_{\mu}\hat{v}_{\nu}  P_k(\p_{t} L  h_{\mu\nu}) =    \sum_{\tilde{h}\in\{F, \underline{F}, \omega_k,  \vartheta_{mn}\},\mu\in\{+,-\}} - \big(i T_{k;\mu}^{\tilde{h} }(v)  \d U^{\tilde{h}^{ L }}\big)^{\mu}+ \p_t \mathcal{R}_k + \big(T_{k;\mu}^{\tilde{h}}(v)  \square     L h_{\mu\nu}\big)^{\mu}.
\ee

 After using the decomposition (\ref{april26eqn1}) for $ \hat{v}_\alpha \hat{v}_\beta \nabla_x P_k(h_{\alpha\beta})$ and have
\be\label{2022may7eqn42}
|\tilde{I}_{\rho;0}^{\mathcal{L}}(k,t_1, t_2)|\lesssim     \sum_{\tilde{h}\in\{F, \underline{F}, \omega_k,  \vartheta_{mn}\}}
  \sum_{\mu\in\{+,-\}}   J_k^{\tilde{h}, \mu} (t_1, t_2)  + \tilde{R}_k(t_1, t_2),
\ee
where
\be
J_k^{\tilde{h}, \mu} (t_1, t_2)=   \int_{t_1}^{t_2} \int_{\R^3} \int_{\R^3}   \big(\Lambda^{\rho}{u}^{\mathcal{L}}(t,x,v) \big)^2 \langle v \rangle  \omega^{\mathcal{L}}_{\rho} ( t,x, v)    \widetilde{\omega^{\mathcal{L}}_{\rho;0}} (t, x, v) \cdot \nabla_x    \big(T_{k;\mu}^{\tilde{h}}(v) U^{\tilde{h}^{ }  }\big)^{\mu} (t,x+t\hat{v}) d x d v d t,
\ee
\be\label{2020may3eqn21}
\tilde{R}_k(t_1, t_2) =  \int_{t_1}^{t_2} \int_{\R^3} \int_{\R^3}   \big(\Lambda^{\rho}{u}^{\mathcal{L}}(t,x,v) \big)^2   \langle v \rangle  \omega^{\mathcal{L}}_{\rho} ( t,x, v)      \widetilde{\omega^{\mathcal{L}}_{\rho;0}} (t, x, v) \cdot \nabla_x    \mathcal{R}_k (t, x+t\hat{v})    d x d v d t.
\ee

Recall (\ref{2022may7eqn61})  and  (\ref{april1eqn11}). Since there is no vector field act on the perturbed metric for the case we are considering, the low order commutator $E_{L,\alpha}^{comm}$ vanishes, the following improved decay estimate holds for $\mathcal{R}_k^m(t, x+t\hat{v})$ from the estimate (\ref{oct9eqn11}) in Lemma \ref{harmonicgaugehigher}, and the estimate (\ref{july27eqn4}) in Lemma \ref{superlocalizedaug},  
\be\label{2020may3eqn53}
\|\nabla_x \mathcal{R}_k (t, x ) \|_{L^\infty_{x,v}} \lesssim 2^{-2m+H(4)\delta_0m - 2k_{-}} \epsilon_1 \lesssim 2^{-2m+5 H(4)\delta_0m}\epsilon_1, \Longrightarrow |\tilde{R}_k(t_1, t_2)| \lesssim 2^{-m/2}\epsilon_1^3. 
\ee

  For $ J_k^{\tilde{h}, \mu} (t_1, t_2)$, $\tilde{h}\in \{F,\underline{F} ,  \omega_j, \vartheta_{mn}\}$,  we  do integration by parts in time  to utilize the null structure. As a result, on the Fourier side,  we have
  \[
  \big|J_k^{\tilde{h}, \mu} (t_1, t_2)\big|\lesssim \big|EndJ_k^{\tilde{h}, \mu} (t_1, t_2)\big|+ \big|\widetilde{J}_{k;1}^{\tilde{h}, \mu} (t_1, t_2)\big|+ \big|\widetilde{J}_{k;2}^{\tilde{h}, \mu} (t_1, t_2)\big|,
  \]
  where
   \[
   \begin{split}
    EndJ_k^{\tilde{h}, \mu} (t_1, t_2)&=\sum_{i=1,2}(-1)^{i} \int_{\R^3} \int_{\R^3}       \mathcal{H}_{\rho, \mathcal{L}}(t_i, x, v) \cdot    \big(\widetilde{T}_{k;\mu}^{\tilde{h}}(v) U^{\tilde{h}^{ }  }\big)^{\mu} (t_i,x+t_i\hat{v}) d x d v ,\\
    \widetilde{J}_{k;1}^{\tilde{h}, \mu} (t_1, t_2)& =   \int_{t_1}^{t_2} \int_{\R^3} \int_{\R^3}   \mathcal{H}_{\rho, \mathcal{L}}(t, x, v)\cdot    \big(\widetilde{T}_{k;\mu}^{\tilde{h}}(v) e^{-it  \d}\p_t V^{\tilde{h}^{ }  }\big)^{\mu} (t ,x+t  \hat{v})   d x d v dt ,\\
      \widetilde{J}_{k;2}^{\tilde{h}, \mu} (t_1, t_2)& = \int_{t_1}^{t_2} \int_{\R^3} \int_{\R^3}  \p_t  \mathcal{H}_{\rho, \mathcal{L}}(t, x, v)\cdot    \big(\widetilde{T}_{k;\mu}^{\tilde{h}}(v) U^{\tilde{h}^{ }  }\big)^{\mu} (t ,x+t  \hat{v})   d x d v dt,    
   \end{split}
\]
  where
  \[
  \mathcal{H}_{\rho, \mathcal{L}}(t, x, v):= \big(\Lambda^{\rho}{u}^{\mathcal{L}}(t ,x,v) \big)^2 \langle v \rangle  \omega^{\mathcal{L}}_{\rho} ( t ,x, v)    \widetilde{\omega^{\mathcal{L}}_{\rho;0}} (t , x, v),
  \]
and the operators $\widetilde{T}_{k;\mu}^{\tilde{h}}(v), \tilde{h}\in \{F,\underline{F} ,  \omega_j, \vartheta_{mn}\},$ are defined as follows, 
\[
\widetilde{T}_{k;\mu}^{\tilde{h}}(v)  f(x): =\int_{\R^3} e^{i x\cdot \xi} \frac{ i \xi  T_{k;\mu}^{\tilde{h} }(v)(\xi)}{ \hat{v}\cdot \xi -   \mu |\xi|}\widehat{f}(\xi) d\xi, 
\]
where the symbols $  T_{k;\mu}^{ \tilde{h} }(v)(\xi),  \tilde{h}\in \{F,\underline{F} ,  \omega_j, \vartheta_{mn}\},$ are  defined in \eqref{april26eqn2}. Thanks to the null structure of $T_{k;\mu}^{\tilde{h} }(v)(\xi)$, the symbols  of  the operators $\widetilde{T}_{k;\mu}^{\tilde{h}}(v), \tilde{h}\in \{F,\underline{F} ,  \omega_j, \vartheta_{mn}\},$ are regular.

From  the decay estimates in Lemma \ref{decayestimateofdensity} for the estimate of $e^{-it  \d}\Lambda_{1}[\p_t V^{\tilde{h}^{ }}]$ and Lemma \ref{superlocalizedaug} for the estimate of wave part, the estimate \eqref{oct4eqn41} in Proposition  \ref{fixedtimenonlinarityestimate} for the estimate of $e^{-it  \d}\Lambda_{\geq 2}[\p_t V^{\tilde{h}^{ }}]$, we have 
 \be\label{2020may3eqn54}
 |EndJ_k^{\tilde{h}, \mu} (t_1, t_2)| + | \widetilde{J}_{k;1}^{\tilde{h}, \mu} (t_1, t_2)| \lesssim 2^{-m/4}\epsilon_1^3. 
\ee

Recall the equation satisfied by $\Lambda^{\rho}{u}^{\mathcal{L}}(t ,x,v)$ in  \eqref{2020april8eqn31} and the detailed formulas of $\mathfrak{N}_{\rho,1}^{\mathcal{L}}(t,x,v)$ and $\mathfrak{N}_{\rho,2}^{\mathcal{L}}(t,x,v)$ in \eqref{2020april8eqn32} and \eqref{2020april8eqn33}. We have
\[
 \Lambda^{\rho}{u}^{\mathcal{L}}(t ,x,v) \p_t  \Lambda^{\rho}{u}^{\mathcal{L}}(t ,x,v)= \Lambda^{\rho}{u}^{\mathcal{L}}(t ,x,v)\big( \mathfrak{N}_{\rho,3}^{\mathcal{L}}(t,x,v) +  \mathfrak{N}_{\rho,4}^{\mathcal{L}}(t,x,v)\big) 
\]
\[
+ \nabla_x \cdot \widetilde{\mathfrak{N}}_{\rho,1}^{\mathcal{L}}(t,x,v)+ D_v \cdot \widetilde{\mathfrak{N}}_{\rho,2}^{\mathcal{L}}(t,x,v) + \widetilde{\mathfrak{N}}_{\rho,3}^{\mathcal{L}}(t,x,v),
\]
where
\[
\widetilde{\mathfrak{N}}_{\rho,1}^{\mathcal{L}}(t,x,v)=- \h  \Lambda_{\geq 1}[(v^0)^{-1}](t, x+t\hat{v}) v \big( \Lambda^{\rho} u^{\mathcal{L}}(t,x,v)\big)^2,
\]
\[
 \widetilde{\mathfrak{N}}_{\rho,2}^{\mathcal{L}}(t,x,v)= -  \frac{1}{4} \big( (v^0)^{-1}   v_{\alpha} v_{\beta} \nabla_x g^{\alpha\beta}\big)(t, x+t\hat{v})  \big( \Lambda^{\rho} u^{\mathcal{L}}(t,x,v)\big)^2,
\]
\[
\begin{split}
\widetilde{\mathfrak{N}}_{\rho,3}^{\mathcal{L}}(t,x,v)& = 
- \big[ \h v\cdot  \nabla_x  \big(   \Lambda_{\geq 1}[(v^0)^{-1}](t, x+t\hat{v}) \big)\\  &+
  \frac{1}{4}  D_v \cdot \big( \big( (v^0)^{-1}   v_{\alpha} v_{\beta} \nabla_x g^{\alpha\beta}\big)(t, x+t\hat{v})\big)\big]  \big( \Lambda^{\rho} u^{\mathcal{L}}(t,x,v)\big)^2. 
  \end{split}
\]

From the above equalities,  after   doing integration by parts in $x,v$  for   $ \widetilde{\mathfrak{N}}_{\rho,i}^{\mathcal{L}}(t,x,v),i\in\{1,2\}$, the following estimate holds from  the estimate of the weight function $\widetilde{\omega^{\mathcal{L}}_{\rho;0}} (t, x, v) $   in (\ref{2020may3eqn16}),  the estimate \eqref{aug9eqn87} in Lemma \ref{estimateofremaindervlasov}
, and the estimate \eqref{2020april9eqn21} in Lemma \ref{estimateofremaindervlasov3}, 
 \be\label{2022may22eqn51}
  | \widetilde{J}_{k;2}^{\tilde{h}, \mu} (t_1, t_2)| \lesssim 2^{-m/4}\epsilon_1^3. 
\ee
To sum up, from the estimates (\ref{2020may3eqn41}), (\ref{2020may3eqn53}),   (\ref{2020may3eqn54}), and \eqref{2022may22eqn51},    the following estimate holds, 
\be\label{2020may3eqn55}
\sum_{k\in \mathbb{Z}, k\geq -m + \alpha m } |\tilde{I}_{\rho;0}^{\mathcal{L}}(k,t_1, t_2)| \lesssim 2^{ 2 H( \mathcal{L},\rho )\delta_0 m } \epsilon_1^3 .
\ee
 
\noindent $\bullet$\quad The estimate of $\tilde{I}_{\rho;1}^{\mathcal{L}}(k,t_1, t_2)$.

From the estimate of the weight function $\widetilde{\omega^{\mathcal{L}}_{\rho;1}} (t, x, v) $   in (\ref{2020may3eqn16}) and the $L^2_{x,v}-L^2_{x,v}-L^\infty_{x,v}$ type multilinear estimate, we have
\[
\sum_{k\notin[-m/2-2\delta_1 m, m/10]} |\tilde{I}_{\rho;1}^{\mathcal{L}}(k,t_1, t_2)| \lesssim \sum_{k\notin[-m/2-2\delta_1 m, m/10]}  2^{2  H( \mathcal{L},\rho )\delta_0 m + 2\delta_1 m + m/2+k-10k_{+}} \epsilon_1^3
\]
\be\label{2020may3eqn56}
\lesssim  2^{2 H( \mathcal{L},\rho )\delta_0 m  }\epsilon_1^3 . 
\ee

It remains to consider the case $k\in[-m/2-2\delta_1 m, m/10]$. For this case, we do integration by parts in time once. As  a result, we have
 \[
 \tilde{I}_{\rho;1}^{\mathcal{L}}(k,t_1, t_2)= \sum_{l=1,2} 
 \int_{\R^3} \int_{\R^3}   \big(\Lambda^{\rho}{u}^{\mathcal{L}}(t_l,x,v) \big)^2 \omega^{\mathcal{L}}_{\rho} ( t_l ,x, v) \langle v \rangle \hat{v}_\alpha \hat{v}_\beta \widetilde{h_{\alpha\beta}^k}(t_l, x+t_l\hat{v}, v)\cdot       \widetilde{\omega^{\mathcal{L}}_{\rho;1}} (t_l, x, v)     d x d  v 
 \]
 \[
 + \int_{t_1}^{t_2} \int_{\R^3} \int_{\R^3}  \langle v \rangle \hat{v}_\alpha \hat{v}_\beta  \widetilde{h_{\alpha\beta}^k}(t_l, x+t \hat{v},v) \cdot  \p_t \big[(\Lambda^{\rho}{u}^{\mathcal{L}}(t,x,v)  )^2 \omega^{\mathcal{L}}_{\rho} ( t,x, v)       \widetilde{\omega^{\mathcal{L}}_{\rho;1}} (t, x, v)\big]     d x d  v d t, 
 \]
 \[
 + \int_{t_1}^{t_2} \int_{\R^3} \int_{\R^3}  \langle v \rangle \hat{v}_\alpha \hat{v}_\beta  \widetilde{h_{\alpha\beta}^{k;h}}(t_l, x+t \hat{v}, v) \cdot   (\Lambda^{\rho}{u}^{\mathcal{L}}(t,x,v)  )^2 \omega^{\mathcal{L}}_{\rho} ( t,x, v)       \widetilde{\omega^{\mathcal{L}}_{\rho;1}} (t, x, v)      d x d  v d t
 \]
where
\[
\widetilde{h_{\alpha\beta}^k}(t , x+t\hat{v}, v ):= \sum_{\mu\in\{+,-\}} \int_{\R^3} e^{i  (x+t\hat{v})\cdot \xi - i t\mu |\xi|} \frac{  c_{\mu}\xi}{|\xi|(\mu|\xi|- \hat{v}\cdot \xi) } \widehat{(V^{h_{\alpha\beta}})^{\mu}}(t, \xi ) \psi_{k}(\xi) d \xi , \quad c_{\mu}:=i\mu/2,
\]
\[
\widetilde{h_{\alpha\beta}^{k;h}}(t ,  x+t\hat{v}, v ):= \sum_{\mu\in\{+,-\}} \int_{\R^3} e^{i  (x+t\hat{v})\cdot \xi - i t\mu |\xi|} \frac{  c_{\mu}\xi}{|\xi|(\mu|\xi|- \hat{v}\cdot \xi) } \widehat{(\p_t V^{h_{\alpha\beta}})^{\mu}}(t, \xi ) \psi_{k}(\xi) d \xi .
\]
From the proof of the linear decay estimate (\ref{july27eqn4}) in Lemma \ref{superlocalizedaug},   the estimate (\ref{july27eqn4}) in Lemma \ref{superlocalizedaug}, and the estimate \eqref{2022march19eqn34} in Lemma \ref{basicestimates},  we have
\[
\|\langle v \rangle^{-2}\widetilde{h_{\alpha\beta}^k}(t_l, x, v )\|_{L^\infty_{x,v}} \lesssim 2^{-m+ H(2)\delta_1 m } \epsilon_1, \quad \|\langle v \rangle^{-2}\widetilde{h_{\alpha\beta}^{k;h}}(t_l, x , v)\|_{L^\infty_{x,v}} \lesssim 2^{-2m+ H(2)\delta_1 m } \epsilon_1. 
\]
From the above estimate, the estimate of the weight function $\widetilde{\omega^{\mathcal{L}}_{\rho;1}} (t, x, v) $   in (\ref{2020may3eqn16}) and the $L^2_{x,v}-L^2_{x,v}-L^\infty_{x,v}$ type multilinear estimate,  and the integration by parts in $v$ process used in (\ref{2020aug31eqn11}), we have 
\be\label{2020june1eqn5}
  |\tilde{I}_{\rho;1}^{\mathcal{L}}(k,t_1, t_2)| \lesssim 2^{-m+ H(2)\delta_1 m + m/2 + 2 H( \mathcal{L},\rho )\delta_0 m  } \epsilon_1^3 \lesssim 2^{-m/3}\epsilon_1^3. 
\ee
After combining the above estimate   and the estimate (\ref{2020may3eqn56}), we have
  \be\label{2020may3eqn58}
\sum_{k\in \mathbb{Z}, k\geq -m +  \alpha m } |\tilde{I}_{\rho;1}^{\mathcal{L}}(k,t_1, t_2)| \lesssim 2^{ 2 H( \mathcal{L},\rho )\delta_0 m } \epsilon_1^3 .
\ee
\noindent $\bullet$\quad The estimate of $\tilde{I}_{\rho;2}^{\mathcal{L}}(k,t_1, t_2)$.

Recall (\ref{2020may3eqn19}), the equalities in (\ref{definitionofcoefficient1}) and (\ref{2020june20eqn24}), and the estimate (\ref{2020aug31eqn31}). For this case, we trade spatial derivative for the vector fields and gain the decay with respect to the light cone twice. As a result, the following estimate holds from the decay  estimate of the perturbed metric    (\ref{2020julybasicestiamte}) in Lemma \ref{basicestimates} and the  estimate (\ref{oct4eqn41}) in Proposition \ref{fixedtimenonlinarityestimate},  
\[
(1+|t-|x+t\hat{v}||)^2 |\nabla_x P_k(h_{\alpha\beta})(t,x+t\hat{v})| \lesssim 2^{-m-k + H(3)\delta_1 m +\delta_1 m }\epsilon_1+ 2^{-2m-k+H(4)\delta_0 m}\epsilon_1. 
\]
Recall that $k\geq -m+ ( H(4)\delta_1    -2\delta_1)m$. From the above estimate,  
 the estimate of the weight function $\widetilde{\omega^{\mathcal{L}}_{\rho; 2 }} (t, x, v)$ in (\ref{2020may3eqn17}), and   the $L^2_{x,v}-L^2_{x,v}-L^\infty_{x,v}$ type multilinear estimate, we have
 \[
  \sum_{k\in \mathbb{Z}, k\geq -m + \alpha  m } |\tilde{I}_{\rho;3}^{\mathcal{L}}(k,t_1, t_2)| \lesssim \sum_{k\in \mathbb{Z}, k\geq -m +   \alpha m }  2^{-m-k + H( 3)\delta_1 m + 2 H( \mathcal{L},\rho )\delta_0 m +\delta_1 m  }\epsilon_1^3
 \]
\be\label{2020may3eqn61}
+ 2^{-2m-k+H(4)\delta_0 m+ 2 H( \mathcal{L},\rho )\delta_0 m}\epsilon_1^3\lesssim  2^{2 H( \mathcal{L},\rho )\delta_0 m  }\epsilon_1^3 . 
\ee
Recall the decomposition (\ref{2020may3eqn18}). To sum up, our desired estimate (\ref{2020june1eqn3}) holds from the estimates (\ref{2020may2eqn20}), (\ref{2020may3eqn55}), (\ref{2020may3eqn58}), and  (\ref{2020may3eqn61}).

 \end{proof}
 
 \begin{lemma}\label{vlasovenergy2}
Under the bootstrap assumptions \textup{(\ref{BAmetricold})} and \textup{(\ref{BAvlasovori})}, the following estimate holds for any   $t_1, t_2\in [2^{m-1}, 2^m]\subset [0, T]$,   $\mathcal{L}\in \cup_{l\leq N_0 }\mathcal{P}_l$, $\rho\in \mathcal{S},$   s.t., $ | \mathcal{L}| +   | \rho  |  \leq N_0,$    
\be\label{2020april15eqn31}
\sum_{k\leq -m +   c(\rho, \mathcal{L}) \delta_0 m , k\in \Z}\sum_{i= 3,4  }\big|  I_{\rho;i}^{\mathcal{L}}(k,t_1, t_2)\big|  \lesssim 2^{ 2 H(\mathcal{L}, \rho)\delta_0 m }  \epsilon_1^3,
\ee
where $c(\rho, \mathcal{L}):=\min\{(H(\mathcal{L}, \rho)-H(1))/3, \big(H(|\mathcal{L}+|\tilde{c}(\rho)| )-H(|\mathcal{L}+|\tilde{c}(\rho)|-1))/2 \}    $.
 
 \end{lemma}
\begin{proof}

We first estimate $I_{\rho;3}^{\mathcal{L}}(k,t_1, t_2)$.
Recall (\ref{2020april15eqn13}). Note that we have $| \rho |+|\mathcal{L}|\geq 2$ for this case.  From the decomposition of profile in (\ref{modifiedperturbedmetric}) and the decomposition of the density type function in  (\ref{oct11eqn11}). From the $L^2_{x,v}-L^\infty_{x,v}$ type bilinear estimate and the $L^\infty\rightarrow L^2$ type Sobolev embedding, we have
\[
\sum_{k\leq -m +  c(\rho, \mathcal{L}) \delta_0 m, k\in \Z} | I_{\rho;3}^{\mathcal{L}}(k,t_1, t_2)| \lesssim
\sum_{k\leq -m +   c(\rho, \mathcal{L}) \delta_0 m , k\in \Z}\sum_{\iota_1, \iota_2\in \mathcal{S},|\iota_1|=|\iota_2|=1, \tilde{c}(\iota_1)=0,\tilde{c}(\iota_2)=1 } 2^{ m +2k}  2^{    H(\mathcal{L}, \rho)\delta_0 m  }
\]
 \be\label{2020may2eqn21}
 \times(  2^{ H( |\mathcal{L}|+ |\rho|)\delta_1 m}+ 2^{m+k})   \big( 2^{H(Id, \iota_2)\delta_0 m } + 2^{m+H(Id, \iota_1)\delta_0 m} \big) \epsilon_1^3\lesssim 2^{2 H(\mathcal{L}, \rho)\delta_0 m }\epsilon_1^3 .
\ee

Now, we estimate $ I_{\rho;4}^{\mathcal{L}}(k,t_1, t_2)$. Recall (\ref{2020april15eqn15}) and (\ref{2022may6eqn2}). Note that,  we have $| \rho |+|\mathcal{L}| \geq 2 $. By using the same strategy used in the above estimate, we have 
\[
\sum_{k\leq -m + c(\rho, \mathcal{L})\delta_0 m, k\in \Z} | I_{\rho;4}^{\mathcal{L}}(k,t_1, t_2)| \lesssim  \sum_{\begin{subarray}{c} 
k\leq -m + c(\rho, \mathcal{L}) \delta_0 m \\
\rho_1 \circ \rho_2 \preceq \rho, \mathcal{L}_1 \circ \mathcal{L}_2\preceq \mathcal{L}
\end{subarray} } \sum_{\begin{subarray}{c}
\Gamma  \in P_1, \iota\in \mathcal{S},  |\iota|=1,\tilde{c}(\iota)=0\\ 
   | \rho_1 |  + |\mathcal{L}_1|< |\mathcal{L}| +  | \rho | -1, \\ 
 |  \rho_2 | + |\mathcal{L}_2| < |\mathcal{L}|+| \rho |
 \end{subarray} } 2^{m+k} 2^{ H( \mathcal{L}, \rho)\delta_0 m }
\]
 \be\label{2020may2eqn22}
\times   (2^{m+k} +  2^{H(1+|\rho_2|+|\mathcal{L}_2|)\delta_1 m }) \big(2^{H( \mathcal{L}_1,\iota\circ \rho_1)\delta_0 m     }+2^{k+ H( \Gamma \circ  \mathcal{L}_1, \rho_1)\delta_0 m     }\big)   \epsilon_1^3 \lesssim   2^{2 H(\mathcal{L}, \rho)\delta_0 m }\epsilon_1^3 .
\ee
Hence   our desired estimate (\ref{2020april15eqn31}) holds from the estimates  (\ref{2020may2eqn21}--\ref{2020may2eqn22}). 
\end{proof}
 \begin{lemma}\label{notbulkcase}
Under the bootstrap assumptions \textup{(\ref{BAmetricold})} and \textup{(\ref{BAvlasovori})}, the following estimate holds for any   $t_1, t_2\in [2^{m-1}, 2^m]\subset [0, T]$,   $\mathcal{L}\in \cup_{l\leq N_0}\mathcal{P}_l$, $\rho\in \mathcal{S},$ s.t.,  s.t., $  | \mathcal{L}| +   | \rho |  \leq N_0,$   
\be\label{2020april21eqn71}
\sum_{k\in \Z, k\geq -m + c(\rho, \mathcal{L}) \delta_0 m }\big|  I_{\rho;3 }^{\mathcal{L}}(k,t_1, t_2)\big| + \big|  I_{\rho;4}^{\mathcal{L}}(k,t_1, t_2)\big|\lesssim   2^{2 H(\mathcal{L}, \rho)\delta_0 m }\epsilon_1^3,
\ee
where $c(\rho, \mathcal{L}):=\min\{(H(\mathcal{L}, \rho)-H(1))/3, \big(H(|\mathcal{L}+|\tilde{c}(\rho)| )-H(|\mathcal{L}+|\tilde{c}(\rho)|-1))/2 \}    $.
 \end{lemma}
 \begin{proof}
 Recall the detailed formulas of $ I_{\rho;3}^{\mathcal{L}}(k,t_1, t_2)$  and $ I_{\rho;4}^{\mathcal{L}}(k,t_1, t_2)$ in (\ref{2020april15eqn13}) and (\ref{2020april15eqn15}), the definition of index set $S_{\rho, \mathcal{L}}^1$ in (\ref{2022may6eqn1}), and  the profile of perturbed metric in (\ref{2020april15eqn35}). We have $ | \mathcal{L}| +   | \rho |\geq 2.$

Due to the fact that $| \rho_1 |+|\tilde{\mathcal{L}}_1|< |  {\mathcal{L}}|+| \rho |-1$, by taking the advantage of the different growth rates of the Vlasov part and the metric part, see  the estimate \eqref{2022march19eqn34} in Lemma \ref{basicestimates}, we can first rule out the case $| \rho |+|\mathcal{L}|\leq N_0-5$ and $k\leq 0$ by using the $L^\infty_{x,v}-L^2_{x,v}$ type bilinear estimate, i.e., using the decay estimate of the metric component. From now on, we restrict ourselves to the case $| \rho |+|\mathcal{L}|\geq N_0-5$ and the case $| \rho |+|\mathcal{L}|\leq  N_0-5$ and  $k\geq 0$.

 After  using the decomposition in (\ref{2022march19eqn16}) and doing integration by parts in time on the Fourier side for the modified perturbed metric part, we have
\be\label{2020april21eqn72}
I_{\rho;i}^{\mathcal{L}}(k,t_1, t_2)= \sum_{\kappa\in \{+,-\}} \widetilde{I}_{\rho;i}^{\mathcal{L};\kappa}(k,t_1, t_2)+ \widetilde{II}_{\rho;i}^{\mathcal{L};\kappa}(k,t_1, t_2)+  \widetilde{III}_{\rho;i}^{\mathcal{L};\kappa }(k,t_1, t_2), \quad i=3,4, 
\ee
where
\be\label{2020aug18eqn1}
 \widetilde{I}_{\rho;3}^{\mathcal{L};\kappa }(k,t_1, t_2)=   \sum_{\begin{subarray}{c}
  \tilde{L}\in\cup_{l+|\alpha|\leq N_0 }   \nabla_x^\alpha P_l \\
   |\tilde{L}|\leq \tilde{c}(\rho), |\alpha|\leq |\rho|-\tilde{c}(\rho)
   \end{subarray}}  \int_{t_1}^{t_2} \int_{\R^3} \int_{\R^3}    \mathcal{N}^{\tilde{L}\mathcal{L}}_{\rho;\mu\nu}(t,x,v) \cdot \nabla \d^{-1} \big(P_{k} ( ( \tilde{\rho}^{h_{\mu\nu  }}_{\tilde{L}\mathcal{L}}  )^{\kappa} )\big) (t, x+t\hat{v})    d x d  v d t, 
\ee
\[
\widetilde{II}_{\rho;3}^{\mathcal{L};\kappa}(k,t_1, t_2)=  \sum_{\begin{subarray}{c}
  \tilde{L}\in\cup_{l+|\alpha|\leq N_0 }   \nabla_x^\alpha P_l\\ 
   |\tilde{L}|\leq \tilde{c}(\rho), |\alpha|\leq |\rho|-\tilde{c}(\rho)
   \end{subarray}}      \int_{\R^3} \int_{\R^3}    e^{i t_i \kappa |\xi | + it_i\hat{v}\cdot \xi }       \frac{  (-1)^i \xi  \psi_k(\xi) \widehat{\mathcal{N}^{\tilde{L}\mathcal{L}}_{\rho;\mu\nu}}(t_i, -\xi, v) }{2 |\xi|(|\xi|+  \kappa  \hat{v}\cdot \xi )}  \widehat{ (\widetilde{V^{ \tilde{L} \mathcal{L} h_{\mu\nu}} } )^{\kappa}  }(t_i, \xi )     d \xi  d v  
\]
\be\label{2020april15eqn41}
+ \int_{t_1}^{t_2 } \int_{\R^3} \int_{\R^3}   \frac{ \xi  \psi_k(\xi)}{2 |\xi|(|\xi|+  \kappa  \hat{v}\cdot \xi )}   e^{i t  \kappa |\xi | + it \hat{v}\cdot \xi }   \widehat{ \mathcal{N}^{\tilde{L}\mathcal{L}}_{\rho;\mu\nu}}(t , -\xi, v)   \p_t \widehat{ (\widetilde{V^{ \tilde{L} \mathcal{L} h_{\mu\nu}}})^{\kappa}  }(t  , \xi )     d \xi  d v   d t, 
\ee
\[
\widetilde{III}_{\rho;3}^{\mathcal{L};\kappa}(k,t_1, t_2)=   \sum_{\begin{subarray}{c}
  \tilde{L}\in\cup_{l+|\alpha|\leq N_0 }   \nabla_x^\alpha P_l\\ 
   |\tilde{L}|\leq \tilde{c}(\rho), |\alpha|\leq |\rho|-\tilde{c}(\rho)
   \end{subarray}}  \int_{t_1}^{t_2 } \int_{\R^3} \int_{\R^3}   \frac{  e^{i t  \kappa |\xi | + it \hat{v}\cdot \xi } \xi  \psi_k(\xi)}{2 |\xi|(|\xi|+  \kappa  \hat{v}\cdot \xi )}
\]
\be\label{2020april15eqn71}
   \times \p_t \widehat{ \mathcal{N}^{\tilde{L}\mathcal{L}}_{\rho;\mu\nu}}(t ,- \xi, v)    \widehat{ (\widetilde{V^{ \tilde{L} \mathcal{L} h_{\mu\nu}}})^{\kappa}  }(t  , \xi )     d \xi  d v   d t, 
\ee
\[
\widetilde{I}_{\rho;4}^{\mathcal{L} }(k,t_1, t_2)= \sum_{  \begin{subarray}{c}  ( \Gamma_1, \Gamma_2)\in P\times\{\p_{\alpha}\} \cup \{\p_{\alpha}\}\times \mathcal{P}\\ 
  (\tilde{L}, \rho_1, \rho_2, \mathcal{L}_1, \mathcal{L}_2)\in S_{\rho, \mathcal{L}}^1
  \end{subarray}}  \int_{t_1}^{t_2} \int_{\R^3} \int_{\R^3}  \mathcal{N}_{ \tilde{\mathcal{L}}_1 \tilde{\mathcal{L}}_2 ;\tilde{L}, \mu\nu}^{  {\mathcal{L}};\Gamma_1,\Gamma_2;\rho, \rho_1, \rho_2} (t,x,v)  P_{k}\big( ( \tilde{\rho}^{h_{\mu\nu  }}_{  \tilde{L} \mathcal{L}_2 }   )^{\kappa} \big)(t, x+t\hat{v})   \]
\be\label{2020aug18eqn2}
 +  \mathcal{N}^{\tilde{L}\mathcal{L}\mathcal{L}_1\mathcal{L}_2}_{\rho\rho_1\rho_2;\mu\nu}(t,x,v)  \d^{-1} P_{k}\big( (\tilde{\rho}^{h_{\mu\nu  }}_{ \tilde{L}\Gamma_2 \tilde{\mathcal{L}} }   )^{\kappa} \big)(t, x+t\hat{v}) d x d  v dt , 
\ee 
\[
\widetilde{II}_{\rho;4}^{\mathcal{L};\kappa}(k,t_1, t_2)= \sum_{  \begin{subarray}{c}  ( \Gamma_1, \Gamma_2)\in P\times\{\p_{\alpha}\} \cup \{\p_{\alpha}\}\times \mathcal{P}\\ 
  (\tilde{L}, \rho_1, \rho_2, \mathcal{L}_1, \mathcal{L}_2)\in S_{\rho, \mathcal{L}}^1
  \end{subarray}}      (-1)^i \int_{\R^3} \int_{\R^3}      e^{i t_i \kappa |\xi | + it_i\hat{v}\cdot \xi }  \frac{   \psi_k(\xi)}{2 |\xi|(|\xi|+  \kappa  \hat{v}\cdot \xi )} \big[   \widehat{\mathcal{N}^{\tilde{L}\mathcal{L}\mathcal{L}_1\mathcal{L}_2}_{\rho\rho_1\rho_2;\mu\nu}}(t_i, -\xi, v)\]
 \[
\times  \widehat{ (\widetilde{V^{ \tilde{L} \mathcal{L}_2 h_{\mu\nu}}})^{\kappa}  }(t_i, \xi )   +   \widehat{ \mathcal{N}_{ \tilde{\mathcal{L}}_1 \tilde{\mathcal{L}}_2 ;\tilde{L}, \mu\nu}^{  {\mathcal{L}};\Gamma_1,\Gamma_2;\rho, \rho_1, \rho_2}}(t_i, -\xi, v) \widehat{ (\widetilde{V^{ \tilde{L}\Gamma_2 \tilde{\mathcal{L}}_2 h_{\mu\nu}}})^{\kappa}  }(t_i, \xi ) \big]   d \xi  d v 
\]
\[
+  \int_{t_1}^{t_2 } \int_{\R^3} \int_{\R^3}      e^{i t  \kappa |\xi | + it \hat{v}\cdot \xi }\frac{   \psi_k(\xi)}{2 |\xi|(|\xi|+  \kappa  \hat{v}\cdot \xi )}     \big(\widehat{\mathcal{N}^{\tilde{L}\mathcal{L}\mathcal{L}_1\mathcal{L}_2}_{\rho\rho_1\rho_2;\mu\nu}}(t , -\xi, v)  \p_t \widehat{ (\widetilde{V^{ \tilde{L} \mathcal{L}_2 h_{\mu\nu}}})^{\kappa}  }(t , \xi )
\]
 \be\label{2020june1eqn2}
 +   \widehat{ \mathcal{N}_{ \tilde{\mathcal{L}}_1 \tilde{\mathcal{L}}_2 ;\tilde{L}, \mu\nu}^{  {\mathcal{L}};\Gamma_1,\Gamma_2;\rho, \rho_1, \rho_2}}(t , -\xi, v) \p_t  \widehat{ (\widetilde{ V^{ \tilde{L}\Gamma_2 \tilde{\mathcal{L}}_2 h_{\mu\nu}}})^{\kappa}  }(t , \xi ) \big)  d \xi  d v d t,
\ee
\[
\widetilde{III}_{\rho;4}^{\mathcal{L};\kappa}(k,t_1, t_2)= \sum_{  \begin{subarray}{c}  ( \Gamma_1, \Gamma_2)\in P\times\{\p_{\alpha}\} \cup \{\p_{\alpha}\}\times \mathcal{P}\\ 
  (\tilde{L}, \rho_1, \rho_2, \mathcal{L}_1, \mathcal{L}_2)\in S_{\rho, \mathcal{L}}^1
  \end{subarray}} \int_{t_1}^{t_2 } \int_{\R^3} \int_{\R^3}      e^{i t  \kappa |\xi | + it \hat{v}\cdot \xi }  \frac{   \psi_k(\xi)}{2 |\xi|(|\xi|+  \kappa  \hat{v}\cdot \xi )} \big( \widehat{ (\widetilde{V^{ \tilde{L} \mathcal{L}_2 h_{\mu\nu}} } )^{\kappa}  }(t , \xi )   
  \]
 \be\label{2020april15eqn75}
  \times   \p_t  \widehat{\mathcal{N}^{\tilde{L}\mathcal{L}\mathcal{L}_1\mathcal{L}_2}_{\rho\rho_1\rho_2;\mu\nu}}(t , -\xi, v) +   \p_t\widehat{ \mathcal{N}_{ \tilde{\mathcal{L}}_1 \tilde{\mathcal{L}}_2 ;\tilde{L}, \mu\nu}^{  {\mathcal{L}};\Gamma_1,\Gamma_2;\rho, \rho_1, \rho_2}}(t , -\xi, v) \widehat{ (\widetilde{V^{ \tilde{L}\Gamma_2 \tilde{\mathcal{L}}_2 h_{\mu\nu}}})^{\kappa}  }(t , \xi ) \big)  d \xi  d v d t,
\ee
where
\be\label{2020april15eqn61}
 \mathcal{N}^{\tilde{L}\mathcal{L}}_{\rho;\mu\nu}(t,x,v):=    \big(\omega^{\mathcal{L}}_{\rho} ( x, v)   \big)^2   \Lambda^{\rho}{u}^{\mathcal{L}}(t,x,v) C^{\tilde{L}}_{\mu\nu;\rho}(x,v)D_v  u^{ }(t,x,v),
\ee
\be\label{2020april15eqn62}
 \mathcal{N}^{\tilde{L}\mathcal{L}\mathcal{L}_1\mathcal{L}_2}_{\rho\rho_1\rho_2;\mu\nu}(t,x,v):=     \big(\omega^{\mathcal{L}}_{\rho} ( x, v)   \big)^2   \Lambda^{\rho}{u}^{\mathcal{L}}(t,x,v) {C}_{\mathcal{L}_1\mathcal{L}_2;\tilde{L}, \mu, \nu}^{\mathcal{L};\rho,\rho_1,\rho_2}(x,v)  \cdot \nabla_x \Lambda^{\rho_1}  u^{\mathcal{L}_1}(t,x,v),
\ee
\be\label{2020april15eqn63}
\mathcal{N}_{ \tilde{\mathcal{L}}_1 \tilde{\mathcal{L}}_2 ;\tilde{L}, \mu\nu}^{  {\mathcal{L}};\Gamma_1,\Gamma_2;\rho, \rho_1, \rho_2}(t,x,v):=      \big(\omega^{\mathcal{L}}_{\rho} ( x, v)   \big)^2   \Lambda^{\rho}{u}^{\mathcal{L}}(t,x,v)  C_{ \tilde{\mathcal{L}}_1 \tilde{\mathcal{L}}_2 ;\tilde{L}, \mu\nu}^{  {\mathcal{L}};\Gamma_1,\Gamma_2;\rho, \rho_1, \rho_2}(x,v)  \Lambda^{\rho_1} u^{\Gamma_1 \tilde{\mathcal{L}}_1}(t,x,v). 
\ee

$\noindent$ $\bullet$\quad The estimate of $\widetilde{I}_{\rho;i}^{\mathcal{L}}(k,t_1, t_2), i\in\{3,4\}.$
\quad

Recall (\ref{2020aug18eqn1}) and (\ref{2020aug18eqn2}).  Note that, these two terms are Vlasov-density type interaction. We first introduce two bilinear estimates which have been studied in \cite{wang3} and will be used later. 

 For any fixed sign $\mu\in\{+,-\}$, any two distribution functions $f_1(t,x,v)$ and $f_2(t,x,v)$, any fixed $k\in\mathbb{Z}$, any symbol $m(\xi,v)\in L^\infty_v \mathcal{S}^\infty_k$, and  any differentiable coefficient $c(v)$,  we define a bilinear operator as follows, 
\be\label{noveq260}
B_k  (f_1,f_2)(t,x,v):= f_1(t,x,v) E (P_{k}[f_2(t)])(x+ \hat{v} t), 
\ee
where
\[
E(P_k[f])(t,x):=\int_{\R^3}\int_{\R^3} e^{i x \cdot \xi} e^{-i \mu t \hat{u}\cdot \xi} {c(u)  m(\xi,u) \psi_k(\xi)}  \widehat{f}(t, \xi, u) d\xi d u. 
\]
For the above defined bilinear operator, as stated in  \cite{wang3}[Lemma 2.4],  $\forall k\in \Z$, we have 
  \[
\| B_k (f_1,f_2)(t,x  ,v)\|_{L^2_x L^2_v}  \lesssim \sum_{|\alpha|\leq 5}\min\{ |t|^{-3}  2^{ k_{+}/2} ,2^{3k}\}   \|m(\xi,v)\|_{L^\infty_v \mathcal{S}^\infty_k}    
 \]
\be\label{bilineardensitylargek}
 \times  \|(1+ |v|+|x|)^{20}c(v) f_2(t,x,v) \|_{  L^2_x L^2_v}   \|(1+ |v|+|x|)^{20} \nabla_v^\alpha f_1(t,x,v)\|_{L^2_xL^2_v}.
\ee
Moreover, if $ k\in \mathbb{Z}, |t|^{-1}\lesssim  2^{k}\leq 1,$ we have 
\[
\| B_k (f_1,f_2)(t,x  ,v)\|_{L^2_x L^2_v}  \lesssim \sum_{|\alpha|\leq 5} \big( \|m(\xi,v)\|_{L^\infty_v\mathcal{S}^\infty_k} + \|m(\xi,v)\|_{L^\infty_v\mathcal{S}^\infty_k}\big) 
\]
\[
\times \big[ |t|^{-2}2^{k} \|\big(|c(v)|+|\nabla_v c(v)|\big) f_3(t,v) \|_{L^2_v} + |t|^{-3}  2^{ k }    \|(1+ |v|+|x|)^{20}c(v) f_2(t,x,v) \|_{  L^2_x L^2_v}
\]
 \be\label{bilineardensity}
 +  |t|^{-3} \| c(v)\big( \widehat{f_2}(t,0,v ) -\nabla_v\cdot f_3(t, v)\big)\|_{L^2_v}   \big] \|(1+ |v|+|x|)^{20}  \nabla_v^\alpha f_1(t,x,v)\|_{L^2_xL^2_v}, 
 \ee
 where  $f_3(t,v):\R_t\times \R_v^3\longrightarrow \mathbb{C}$, is 
 any localized differentiable function.

 Based on the possible distribution of vector fields and spatial derivatives, we separate into three subcases as follows.

\textbf{Subcase} $1$. \quad  If  $  |\tilde{L}| +|\mathcal{L}_2|\leq  | \rho |+|\mathcal{L}|  -5$. 

For this case,  we use  the $L^\infty_{x,v}-L^2_{x,v}$ type bilinear estimate. As a result,  from the decay estimate (\ref{densitydecay}) of the density type function in Lemma \ref{decayestimateofdensity}, we have
\[
\sum_{k\in \Z, k\geq -m + c(\rho, \mathcal{L}) \delta_0 m   }\sum_{i=3,4} | \widetilde{I }_{\rho;i}^{\mathcal{L} }(k,t_1, t_2)| 
\lesssim  \sum_{ k\geq -m + c(\rho, \mathcal{L}) \delta_0 m} 2^{-2m-2k}(1+2^{k+H(5)\delta_0 m })   2^{2 H(\mathcal{L}, \rho)\delta_0 m }\epsilon_1^3 
\]
\[
\lesssim   2^{2 H(\mathcal{L}, \rho)\delta_0 m }\epsilon_1^3.
\]
 
 \textbf{Subcase} $2$. \quad  If  $  |\tilde{L}| +|\mathcal{L}_2|>  | \rho |+|\mathcal{L}| -5$ and $ | \rho |+|\mathcal{L}|\leq   N_0-5$.  

 For this subcase, as the case $k\leq 0$ has been ruled out, it remains to consider the case $k\geq 0$. From the Cauchy-Schwarz inequality, the bilinear estimate (\ref{bilineardensitylargek}),   we have
\be
\sum_{k\in \Z, k\geq 0   }\sum_{i=3,4} | \widetilde{I }_{\rho;i}^{\mathcal{L} }(k,t_1, t_2)|  \lesssim \sum_{k\geq 0} 2^{   2 H(\mathcal{L}, \rho)\delta_0 m} 2^{-m-k/2+H(10)\delta_0 m }\epsilon_1^3 \lesssim \epsilon_1^3.
\ee

\textbf{Subcase} $3$. \quad      $  |\tilde{L}| +|\mathcal{L}_2|>   | \rho |+|\mathcal{L}|  -5$ and $  | \rho |+|\mathcal{L}|>  N_0-5$.

For this case, we have $ |\tilde{c}(\rho_1)|+|\mathcal{L}_1|\leq  5$.  From the Cauchy-Schwarz inequality, the bilinear estimate (\ref{bilineardensitylargek}) and (\ref{bilineardensity}),   we have
\[
 \sum_{k\in \Z, k\geq -m + c(\rho, \mathcal{L}) \delta_0 m   } \sum_{i=3,4} | \widetilde{I }_{\rho;i}^{\mathcal{L} }(k,t_1, t_2)|  \lesssim \sum_{k\geq 0} 2^{   2 H(\mathcal{L}, \rho)\delta_0 m} 2^{-m-k/2+H(10)\delta_0 m }
\]
\[
+ \sum_{k \in [-m + c(\rho, \mathcal{L}) \delta_0 m, 0 ]\cap \Z}  2^{    H(\mathcal{L}, \rho)\delta_0 m+H(10)\delta_0 m } \big( 1+ 2^{-m-k } + 2^{-m+   H(\mathcal{L}, \rho)\delta_0 m}\big)\epsilon_1^3 \lesssim 2^{ 2  H(\mathcal{L}, \rho)\delta_0 m- \delta_0 m } \epsilon_1^3. 
\]

$\noindent$ $\bullet$\quad The estimate of $\widetilde{II}_{\rho;i}^{\mathcal{L};\kappa}(k,t_1, t_2), i\in\{3,4\}.$\qquad Recall (\ref{2020april15eqn41}) and (\ref{2020june1eqn2}).

 \textbf{Subcase} $1$. \quad    If  $  |\tilde{L}| +|\mathcal{L}_2|\leq   | \rho |+|\mathcal{L}|  -5$  . 

From the $L^\infty_{x,v}-L^2_{x,v}$ type bilinear estimate, the estimate (\ref{oct4eqn41}) in Proposition \ref{fixedtimenonlinarityestimate}  and  the estimate (\ref{july27eqn4}) in Lemma \ref{superlocalizedaug},   we conclude that 
\[
\sum_{ k\in \Z, k\in [ -m +  c(\rho, \mathcal{L}) \delta_0 m, -d(N_0,0)\delta_0 m ], i=3,4} | \widetilde{II}_{\rho;i}^{\mathcal{L};\kappa}(k,t_1, t_2)| 
\]
\[
 \lesssim \sum_{\begin{subarray}{c}
 k\in \Z,  k\in [ -m +  c(\rho, \mathcal{L}) \delta_0 m, -d(N_0,0)\delta_0 m ]\\ 
 0\leq |\rho_1|+|\mathcal{L}_1|\leq |\mathcal{L}|+|\rho|-1
 \end{subarray} }  2^{-k}  2^{ H(\mathcal{L}, \rho)\delta_0m + H(\mathcal{L}_1, \rho_1)\delta_0 m }\big[   2^{ k/2+k_{-}/2} 2^{  H( |\mathcal{L}|+ |\rho |+1-|\rho_1|-|\mathcal{L}_1|)\delta_1 m+\delta_1 m }
\]
 \be\label{2020april21eqn41}
  +    2^{ k/2+k_{-}/2}  2^{  H(  |\mathcal{L}|+ |\rho |+1-|\rho_1|-|\mathcal{L}_1|)\delta_0 m +k_{-}}\big]\lesssim 2^{2 H(\mathcal{L}, \rho)\delta_0m - \delta_0 m } \epsilon_1^3. 
\ee

For $k$ relatively large, we first use the equality (\ref{2020june20eqn24}) and  the pointwise estimate (\ref{2020aug31eqn31}) and then rerun the above argument. As a result, we have
\be\label{2020sep22eqn1}
\sum_{ k\in \Z, k\geq -d(N_0,0)\delta_0 m , i=3,4} | \widetilde{II}_{\rho;i}^{\mathcal{L};\kappa}(k,t_1, t_2)|  \lesssim \sum_{ k\in \Z, k\geq -H(N_0,0)\delta_0 m }2^{-m-2k+3H(\mathcal{L}, \rho)\delta_0 m } 2^{k/2+k_{-}/2}\epsilon_1^3\lesssim \epsilon_1^3. 
\ee

\textbf{Subcase} $2$. \quad   If $  |\tilde{L}| +|\mathcal{L}_2|>  | \rho |+|\mathcal{L}| -5$,  $   | \rho |+|\mathcal{L}|\leq N_0-5$.

 For this case, we have  $k\geq 0$.   From the  the $L^2_x-L^\infty_x L^2_v$ type bilinear estimate,  the decay estimate (\ref{densitydecay}) of the density type function in Lemma \ref{decayestimateofdensity},   and the estimate (\ref{oct4eqn41}) in Proposition \ref{fixedtimenonlinarityestimate},   we have 
\[
\sum_{ k\in \Z,k\geq 0, i=3,4} | \widetilde{II}_{\rho;i}^{\mathcal{L};\kappa}(k,t_1, t_2)|   
\] 
\be\label{2020april21eqn42}
\lesssim \sum_{ k\in \Z,k\geq 0 }   2^{ H(\mathcal{L}, \rho)  \delta_0 m  +  H(|\mathcal{L}|+|\rho|+1,0)\delta_0m+ H(10)\delta_0 m }   ( 2^{-k/2} 2^{-m/2} + 2^{-3k/2-3m/2})     \epsilon_1^3 \lesssim   \epsilon_1^3. 
\ee

\textbf{Subcase} $3$. \quad     If $  |\tilde{L}| +|\mathcal{L}_2|>   | \rho |+|\mathcal{L}|-5$,  $  | \rho |+|\mathcal{L}|>  N_0-5$.

For this case, we use  the $L^2_x-L^\infty_x L^2_v$ type bilinear estimate.  As a result,   from    the decay estimate (\ref{densitydecay}) of the density type function in Lemma \ref{decayestimateofdensity},   and the estimate (\ref{oct4eqn41}) in Proposition \ref{fixedtimenonlinarityestimate},  we have
\[
\sum_{ k\in \Z, k\geq -m +  c(\rho, \mathcal{L}) \delta_0 m, i=3,4} | \widetilde{II}_{\rho;i}^{\mathcal{L};\kappa}(k,t_1, t_2)|  \lesssim  \sum_{ k\in \Z,k \geq -m +  c(\rho, \mathcal{L}) \delta_0 m  }   2^{ H(\mathcal{L}, \rho)\delta_0 m } 2^{H(10)\delta_0 m }  
\]
\[
   \times  ( 2^{-k/2} 2^{-m/2} + 2^{-3k/2-3m/2}) \big( 2^{  H(  |\mathcal{L}|+ |\rho|+1)\delta_1 m+\delta_1 m } + 2^{  H(|\mathcal{L}|+ |\rho|+1,0)\delta_0 m +k_{-}}\big)\epsilon_1^3
\]
\be\label{2020aug18eqn11}
\lesssim 2^{2  H(\mathcal{L}, \rho)  \delta_0 m - \delta_0 m } \epsilon_1^3.
\ee
 
 $\noindent$ $\bullet$\quad The estimate of $\widetilde{III}_{\rho;i}^{\mathcal{L};\kappa}(k,t_1, t_2), i\in\{3,4\}.$\quad 

 Recall the detailed formulas of $\widetilde{III}_{\rho;i}^{\mathcal{L};\kappa}(k,t_1, t_2), i\in\{3,4\},$ in (\ref{2020april15eqn71}) and (\ref{2020april15eqn75}) and the nonlinearities in 
 (\ref{2020april15eqn61}--\ref{2020april15eqn63}). 

We first show how to handle $\widetilde{III}_{\rho;3}^{\mathcal{L};\kappa}(k,t_1, t_2)$  in details as follows. Recall  the decomposition of ``$D_v$'' in (\ref{2020april21eqn11}), and the equation satisfied by $\Lambda^{\rho} u^{\mathcal{L}}$ in (\ref{2020april8eqn31}), we classify $\p_t \mathcal{N}^{\tilde{L}\mathcal{L}}_{\rho;\mu\nu}(t,x,v)$ as follows, 
\be\label{2020aug31eqn11}
 \p_t \mathcal{N}^{\tilde{L}\mathcal{L}}_{\rho;\mu\nu}(t,x,v)=  \nabla_x\cdot  \mathcal{N}^{\tilde{L}\mathcal{L};1}_{\rho;\mu\nu}(t,x,v)+  D_v\cdot  \mathcal{N}^{\tilde{L}\mathcal{L};2}_{\rho;\mu\nu}(t,x,v)+   \mathcal{N}^{\tilde{L}\mathcal{L};3}_{\rho;\mu\nu}(t,x,v) +  \mathcal{N}^{\tilde{L}\mathcal{L};4}_{\rho;\mu\nu}(t,x,v),
\ee
where
\[
 \mathcal{N}^{\tilde{L}\mathcal{L};1}_{\rho;\mu\nu}(t,x,v):=- v c_{\iota}(t,x,v)\big(\omega^{\mathcal{L}}_{\rho} ( t,x, v)   \big)^2 C^{\tilde{L}}_{\mu\nu;\rho}(x,v)  \Lambda_{\geq 1}[(v^0)^{-1}](t, x+t\hat{v})  \Lambda^{\rho} u^{\mathcal{L}}(t,x,v)  \Lambda^{\iota}   u^{ }(t,x,v),
\]
\[
\mathcal{N}^{\tilde{L}\mathcal{L};2}_{\rho;\mu\nu}(t,x,v):= -  \h \big( (v^0)^{-1}   v_{\alpha} v_{\beta} \nabla_x g^{\alpha\beta}\big)(t, x+t\hat{v})   \big(\omega^{\mathcal{L}}_{\rho} ( t,x, v)   \big)^2 
\]
\be
\times c_{\iota}(t,x,v) C^{\tilde{L}}_{\mu\nu;\rho}(x,v)   \Lambda^{\rho} u^{\mathcal{L}}(t,x,v)  \Lambda^{\iota}   u^{ }(t,x,v),
\ee
\[
\mathcal{N}^{\tilde{L}\mathcal{L};3}_{\rho;\mu\nu}(t,x,v)=\sum_{\iota\in \mathcal{S}, |\iota|=1} 2   \omega^{\mathcal{L}}_{\rho} ( t,x, v)    \p_t\omega^{\mathcal{L}}_{\rho} ( t,x, v)    \Lambda^{\rho}{u}^{\mathcal{L}}(t,x,v) C^{\tilde{L}}_{\mu\nu;\rho}(x,v) c_{\iota}(t,x,v) \Lambda^{\iota}u^{ }(t,x,v)
\]
\[
- \big(\omega^{\mathcal{L}}_{\rho} ( t,x, v)\big)^2  C^{\tilde{L}}_{\mu\nu;\rho}(x,v) \big(  \p_t c_{\iota}(t,x,v) \Lambda^{\iota}u^{ }(t,x,v)+  c_{\iota}(t,x,v)  (    {\mathfrak{N}}_{\iota,3}^{ }(t,x,v)+ {\mathfrak{N}}_{\iota,4}^{ }(t,x,v) ) \big)  \Lambda^{\rho}{u}^{\mathcal{L}}(t,x,v)
\]
\[
  +  \big(\omega^{\mathcal{L}}_{\rho} ( t,x, v)   \big)^2C^{\tilde{L}}_{\mu\nu;\rho}(x,v)     \mathfrak{N}_{\rho,4 }^{\mathcal{L}}(t,x,v)    c_{\iota}(t,x,v) \Lambda^{\iota}u^{ }(t,x,v) + \nabla_x\cdot\big(v  c_{\iota}(t,x,v)\big(\omega^{\mathcal{L}}_{\rho} ( t,x, v)   \big)^2 C^{\tilde{L}}_{\mu\nu;\rho}(x,v) \]
  \[
  \times  \Lambda_{\geq 1}[(v^0)^{-1}](t, x+t\hat{v})\big)\Lambda^{\rho} u^{\mathcal{L}}(t,x,v)  \Lambda^{\iota}   u^{ }(t,x,v)+  \Lambda^{\rho} u^{\mathcal{L}}(t,x,v)  \Lambda^{\iota}   u^{ }(t,x,v)
\]
\be
 \times D_v\cdot \big[  \h \big( (v^0)^{-1}   v_{\alpha} v_{\beta} \nabla_x g^{\alpha\beta}\big)(t, x+t\hat{v})   c_{\iota}(t,x,v)\big(\omega^{\mathcal{L}}_{\rho} ( t,x, v)   \big)^2 C^{\tilde{L}}_{\mu\nu;\rho}(x,v) \big],
\ee
\be
\mathcal{N}^{\tilde{L}\mathcal{L};4}_{\rho;\mu\nu}(t,x,v)=\sum_{\tilde{L}\in \cup_{\alpha\in \Z_+^3, l\in \Z_+}\nabla_x^\alpha P_l,  |\tilde{L}|\leq  \tilde{c}(\rho),  |
 \alpha|\leq |\rho|- \tilde{c}(\rho)} \mathfrak{N}_{\tilde{L}}^{\alpha\beta}(t,x,v) \nabla_x \tilde{L}\mathcal{L} g^{\alpha\beta} (t, x+t\hat{v}),  
 \ee
 where
\be\label{2020aug31eqn12}
  \mathfrak{N}_{\tilde{L}}^{\alpha\beta}(t,x,v):=c_{\tilde{L}}^{\rho}(t,x,v) (v^0)^{-1}   v_{\alpha} v_{\beta}   u(t,x,v)\cdot D_v \big( \big(\omega^{\mathcal{L}}_{\rho} ( t,x, v)   \big)^2C^{\tilde{L}}_{\mu\nu;\rho}(x,v)      D_v u(t,x,v)\big). 
\ee

From (\ref{2020april15eqn71}) and (\ref{2020aug31eqn11}), after doing integration by parts in ``$v$'' for $D_v\cdot  \mathcal{N}^{\tilde{L}\mathcal{L};2}_{\rho;\mu\nu}(t,x,v)$ and doing dyadic decomposition for the metric component in $\mathcal{N}^{\tilde{L}\mathcal{L};4}_{\rho;\mu\nu}(t,x,v)$, we have 
\[
\widetilde{III}_{\rho;3}^{\mathcal{L};\kappa}(k,t_1, t_2)= H_{\rho;1}^{\mathcal{L};\kappa}(k,t_1, t_2)+ \sum_{k_1\in \Z} H_{\rho;2}^{\mathcal{L};\kappa}(k,k_1;t_1, t_2),
\]
\[
H_{\rho;1}^{\mathcal{L};\kappa}(k,t_1, t_2)=   \sum_{\begin{subarray}{c}
  \tilde{L}\in\cup_{l+|\alpha|\leq N_0 }   \nabla_x^\alpha P_l, 
   |\tilde{L}|\leq \tilde{c}(\rho), |\alpha|\leq |\rho|-\tilde{c}(\rho)
   \end{subarray}}   \int_{t_1}^{t_2 } \int_{\R^3} \int_{\R^3}  e^{i t  \kappa |\xi | + it \hat{v}\cdot \xi }    \frac{ \xi  \psi_k(\xi)   \widehat{ (\widetilde{V^{ \tilde{L} \mathcal{L} h_{\mu\nu}}})^{\kappa}  }(t  , \xi )    }{2 |\xi|(|\xi|+  \kappa  \hat{v}\cdot \xi )} 
\]
\[
\times \big(   i \xi \cdot \widehat{ \widetilde{\mathcal{N}}^{\tilde{L}\mathcal{L};1}_{\rho;\mu\nu}}(t , \xi, v) +  \widehat{ \widetilde{\mathcal{N}}^{\tilde{L}\mathcal{L};3}_{\rho;\mu\nu}}(t , \xi, v)    \big)      -\nabla_v\big(   \frac{ \xi  \psi_k(\xi)}{2 |\xi|(|\xi|+  \kappa  \hat{v}\cdot \xi )} \big)\cdot  \widehat{ \widetilde{\mathcal{N}}^{\tilde{L}\mathcal{L};2}_{\rho;\mu\nu}}(t , \xi, v)     d \xi  d v   d t,  
\]
\[
H_{\rho;2}^{\mathcal{L};\kappa}(k,k_1;t_1, t_2)= \sum_{\begin{subarray}{c}
\tilde{L}_1,\tilde{L}_2 \in \cup_{\alpha\in \Z_+^3, l\in \Z_+}\nabla_x^\alpha P_l\\ 
 |\tilde{L}_1|,|\tilde{L}_2|\leq  \tilde{c}(\rho),  |
 \alpha|\leq |\rho|- \tilde{c}(\rho)
\end{subarray}
  } \int_{t_1}^{t_2 } \int_{\R^3} \int_{\R^3} \int_{\R^3}  e^{  -  it \hat{v}\cdot \eta }    \frac{ \xi  \psi_k(\xi)   \widehat{ (\widetilde{U^{ \tilde{L}_1 \mathcal{L} h_{\mu\nu}}})^{\kappa}  }(t  , \xi )    }{2 |\xi|(|\xi|+  \kappa  \hat{v}\cdot \xi )} 
\]
\be
\times i(-\xi-\eta)  \widehat{\mathfrak{N}_{\tilde{L}_2}^{\alpha\beta}}(t, \eta, v) \widehat{\tilde{L}_2\mathcal{L}g^{\alpha\beta}}(t, -\xi-\eta ) \psi_{k_1}(\xi-\eta) d \eta d \xi d v d t. 
\ee

Based on the possible size of $| \rho|+|\mathcal{L}|$, we separate into two subcases as follows.

\textbf{Subcase} $1$. If $| \rho |+|\mathcal{L}|\geq 10$.

From the estimate (\ref{aug9eqn87}) in Lemma \ref{estimateofremaindervlasov}, the estimate (\ref{2020april9eqn21}) in Lemma  \ref{estimateofremaindervlasov3},  the estimate (\ref{2020julybasicestiamte}) in Lemma \ref{basicestimates} and the decay estimate (\ref{densitydecay}) of the density type function in Lemma \ref{decayestimateofdensity},  the following estimate holds ,
 \[
 \sum_{i=1,2,3}\| \int_{\R^3}(1+|v|)^{4} | {\mathcal{N}}^{\tilde{L}\mathcal{L};i}_{\rho;\mu\nu} (t,x-t\hat{v},v)|  d v \|_{L^2_{x }} \lesssim 2^{-3m/2+ H(5)\delta_0m  }\epsilon_1^2. 
 \]
 From the above estimate, we have
 \[
 \sum_{k\geq -m + c(\rho, \mathcal{L})\delta_0 m}| H_{\rho;1}^{\mathcal{L};\kappa}(k,t_1, t_2) \lesssim \big[\big(\sum_{k\in  [-m + c(\rho, \mathcal{L})\delta_0 m, 0]\cap \Z} 2^{- k/2-\gamma k } \big) + 1 \big] 2^{- m/2 +H(5)\delta_0 m +    H(\mathcal{L}, \rho) \delta_0m}
 \]
 \be
 \times  2^{  H(|\mathcal{L}|+ |\tilde{c}(\rho)|)\delta_1m } \epsilon_1^3  \lesssim 2^{ 2   H(\mathcal{L}, \rho) \delta_0m- \delta_0 m}\epsilon_1^3.
 \ee
 
Recall (\ref{2020aug31eqn12}). From the decay estimate (\ref{densitydecay}) of the density type function in Lemma \ref{decayestimateofdensity} and the $L^2_x-L^2_x-L^\infty_x $ type multilinear estimate, we have
\[
\sum_{k, k_1\in \Z,k\geq -m + c(\rho, \mathcal{L}) \delta_0 m } |H_{\rho;2}^{\mathcal{L};\kappa}(k,k_1;t_1, t_2)|\lesssim  \sum_{k, k_1\in \Z,k\geq -m + c(\rho, \mathcal{L}) m } 2^{-k +(1/2-\gamma)k_{-}}  2^{ H(|\mathcal{L}|+|\tilde{c}(\rho)|)\delta_1 m + H(5)\delta_0 m }
\]
\be
\times 2^{ H(\mathcal{L}, \rho) \delta_0 m } 2^{(1/2-\gamma)k_{1,-} } \big(\mathbf{1}_{|k-k_1|\leq 10} + 2^{-m-\max\{k,k_1\}}\mathbf{1}_{|k-k_1|\geq 10}\big)\epsilon_1^4\lesssim  2^{ 2   H(\mathcal{L}, \rho) \delta_0m- \delta_0 m}\epsilon_1^4. 
\ee

\textbf{Subcase} $2$. If $| \rho |+|\mathcal{L}|< 10$. 

 Note that for this case  we have   $k\geq 0 $. Since the total number of vector fields is far less than $N_0$, we first use the equality (\ref{2020june20eqn24}) for the metric component to trade one spatial derivative for one vector field.

  Thanks to  the pointwise estimate (\ref{2020aug31eqn31}), which provides the smallness of $2^{-m}$, after rerunning the above argument, we have 
\be
\sum_{k\geq 0, k\in \Z}|\widetilde{III}_{\rho;3}^{\mathcal{L};\kappa}(k,t_1, t_2)|\lesssim  \sum_{k\geq 0, k\in \Z}  2^{-k } 2^{- m/2 +H(5)\delta_0 m +  2  H(\mathcal{L}, \rho) \delta_0m}  \epsilon_1^3\lesssim  \epsilon_1^3.
\ee
 
 Due to an obvious loss of ``$t$'' comes from the coefficient of $D_v$ in $\widetilde{III}_{\rho;3}^{\mathcal{L};\kappa}(k,t_1, t_2)$, which is not an issue for $\widetilde{III}_{\rho;4}^{\mathcal{L};\kappa}(k,t_1, t_2)$, the estimate of $\widetilde{III}_{\rho;4}^{\mathcal{L};\kappa}(k,t_1, t_2)$ is similar and also easier than $\widetilde{III}_{\rho;3}^{\mathcal{L};\kappa}(k,t_1, t_2)$. With minor modifications in the above argument, we have
 \be
 \sum_{k\geq 0, k\in \Z}|\widetilde{III}_{\rho;4}^{\mathcal{L};\kappa}(k,t_1, t_2)| \lesssim 2^{ 2   H(\mathcal{L}, \rho) \delta_0m- \delta_0 m}\epsilon_1^4. 
\ee
Hence finishing the proof of our desired estimate (\ref{2020april21eqn71}). 
 
 \end{proof}

 \begin{lemma}\label{vlasovenergy4}
 Under the bootstrap assumptions \textup{(\ref{BAmetricold})} and \textup{(\ref{BAvlasovori})}, the following estimate holds for any   $t_1, t_2\in [2^{m-1}, 2^m]\subset [0, T]$, $\mathcal{L}\in \cup_{l\leq N_0}\mathcal{P}_l, \rho\in \mathcal{S},  $  s.t., $   | \mathcal{L}| +   |  \rho |  \leq N_0,$ 
\be\label{2020april22eqn1}
\sum_{   k\in \Z}\big|  I_{\rho;5}^{\mathcal{L}}(k,t_1, t_2)\big|  \lesssim 2^{2  H(\mathcal{L}, \rho)\delta_0 m}\epsilon_1^3 
\ee
 \end{lemma}
 \begin{proof}
Recall (\ref{2020april15eqn51}) and the index set $ S_{\rho, \mathcal{L}}^2$ defined in (\ref{2022may6eqn2}). We have  $| \mathcal{L}| +   | \rho |\geq 1$, $| \rho_1 |+|\mathcal{L}_1|= | \rho  |+|\mathcal{L}|-1$ and $|\tilde{L}|+|\mathcal{L}_2|\leq 1$. From the $L^2_{x,v}-L^2_{x,v}-L^\infty_{x,v}$ type multilinear estimate and the estimate (\ref{2020julybasicestiamte}) in Lemma \ref{basicestimates}, the following estimate holds if $(\Gamma_1,\Gamma_2)\in   \{\p_\alpha\}\times \mathcal{P} $ and $\tilde{c}(\rho_2)+|\mathcal{L}_2|=1$ 
\[
\sum_{k\in \Z}\big| \int_{t_1}^{t_2} \int_{\R^3} \int_{\R^3}  \big(\omega^{\mathcal{L}}_{\rho} ( x, v)   \big)^2   \Lambda^{\rho}{u}^{\mathcal{L}}(t,x,v)P_k(\tilde{L}\mathcal{L}_2  \Lambda_{1}[(v^0)^{-1}] ) (t, x+t\hat{v} ) \]
\[
\times    {C}_{\mathcal{L}_1\mathcal{L}_2;\tilde{L} }^{\mathcal{L};\rho,\rho_1,\rho_2}(x,v)  \cdot \nabla_x \Lambda^{\rho_1}  u^{\mathcal{L}_1}(t,x,v) d x d v d t \big|+\big| \int_{t_1}^{t_2} \int_{\R^3} \int_{\R^3}  \big(\omega^{\mathcal{L}}_{\rho} ( x, v)   \big)^2   \Lambda^{\rho}{u}^{\mathcal{L}}(t,x,v)  
\]
\[
\times C_{  {\mathcal{L}}_1 {\mathcal{L}}_2 ;\tilde{L}, \mu\nu}^{  {\mathcal{L}};\Gamma_1,\Gamma_2;\rho, \rho_1, \rho_2}(x,v)  \Lambda^{\rho_1} u^{\Gamma_1  {\mathcal{L}}_1}(t,x,v)  P_k( \tilde{L}  \Gamma_2 {\mathcal{L}}_2  h_{\mu\nu})(t,x+t\hat{v})d x d v d t \big|
\]
\[
\lesssim \sum_{k\in \Z}\min\{2^{m+k+H(1)\delta_1 m }, 2^{H(3)\delta_1 m + \delta_1 m  -10k_{+} }\} 2^{       H(\mathcal{L}, \rho) \delta_0 m +   H(\Gamma_1\circ \mathcal{L}_1, \rho_1)\delta_0 m}\epsilon_1^3
\]
\[
\lesssim 2^{ 2  H(\mathcal{L}, \rho) \delta_0 m  -\delta_0 m  }\epsilon_1^3 .
\]

Now we consider the case $(\Gamma_1,\Gamma_2)\in P\times \{\p_\alpha\}  $ or $\tilde{c}(\rho_2)+|\mathcal{L}_2|=0$. For this case we have $|\rho_2|=1, i.e., \tilde{L}\in \{\nabla_x\}$.  In whichever case, the metric component is improved at the low frequency. More precisely, from the $L^2_{x,v}-L^2_{x,v}-L^\infty_{x,v}$ type multilinear estimate and  the decay estimate (\ref{aug16eqn1}), we have
\[
\sum_{  k\in \mathbb{Z}, k\notin [-2H( 4)\delta_0 m, H(  4)\delta_0 m ]} 
\big| \int_{t_1}^{t_2} \int_{\R^3} \int_{\R^3}  \big(\omega^{\mathcal{L}}_{\rho} ( x, v)   \big)^2   \Lambda^{\rho}{u}^{\mathcal{L}}(t,x,v) P_k(\tilde{L}\mathcal{L}_2  \Lambda_{1}[(v^0)^{-1}] ) (t, x+t\hat{v} )  
\]
\[
\times    {C}_{\mathcal{L}_1\mathcal{L}_2;\tilde{L} }^{\mathcal{L};\rho,\rho_1,\rho_2}(x,v)   \cdot \nabla_x \Lambda^{\rho_1}  u^{\mathcal{L}_1}(t,x,v) d x d v d t \big|+
\big| \int_{t_1}^{t_2} \int_{\R^3} \int_{\R^3}  \big(\omega^{\mathcal{L}}_{\rho} ( x, v)   \big)^2   \Lambda^{\rho}{u}^{\mathcal{L}}(t,x,v) 
\]
 \be\label{april26eqn9}
\times  C_{ \tilde{\mathcal{L}}_1 \tilde{\mathcal{L}}_2 ;\tilde{L}, \mu\nu}^{  {\mathcal{L}};\Gamma_1,\Gamma_2;\rho, \rho_1, \rho_2}(x,v)   \Lambda^{\rho_1} u^{\Gamma_1 \tilde{\mathcal{L}}_1}(t,x,v)  P_k( \tilde{L}  \Gamma_2 \tilde{\mathcal{L}}_2  h_{\mu\nu})(t,x+t\hat{v})d x d v d t \big| \lesssim 2^{  2   H(\mathcal{L}, \rho)\delta_0 m }\epsilon_1^3. 
\ee

From now on, we restrict ourself further to the case $k\in [-2H( 4)\delta_0 m, H( 4)\delta_0 m ]$. Moreover, from the pointwise  estimate (\ref{2020aug31eqn31}), we can rule out further the small $(x,v)$ case and the far away from the light cone case. More precisely the following estimate holds from the estimates (\ref{definitionofcoefficient1}) and (\ref{2020aug31eqn31}),
\[
\big| \int_{t_1}^{t_2} \int_{\R^3} \int_{\R^3}  \Lambda^{\rho}{u}^{\mathcal{L}}(t,x,v) P_k(\tilde{L}\mathcal{L}_2  \Lambda_{1}[(v^0)^{-1}] ) (t, x+t\hat{v} ){C}_{\mathcal{L}_1\mathcal{L}_2;\tilde{L} }^{\mathcal{L};\rho,\rho_1,\rho_2}(x,v)   \cdot \nabla_x \Lambda^{\rho_1}  u^{\mathcal{L}_1}(t,x,v) 
\]
\[
\times  \tilde{\varphi}(t,x,v)  \big(\omega^{\mathcal{L}}_{\rho} ( x, v)   \big)^2      d x d v d t \big|+ \big| \int_{t_1}^{t_2} \int_{\R^3} \int_{\R^3}  \big(\omega^{\mathcal{L}}_{\rho} ( x, v)   \big)^2   \Lambda^{\rho}{u}^{\mathcal{L}}(t,x,v)  C_{ \tilde{\mathcal{L}}_1 \tilde{\mathcal{L}}_2 ;\tilde{L}, \mu\nu}^{  {\mathcal{L}};\Gamma_1,\Gamma_2;\rho, \rho_1, \rho_2}(x,v)
\]
\be \label{2020sep1eqn61}
 \times   \Lambda^{\rho_1} u^{\Gamma_1 \tilde{\mathcal{L}}_1}(t,x,v)  P_k( \tilde{L}  \Gamma_2 \tilde{\mathcal{L}}_2  h_{\mu\nu})(t,x+t\hat{v}) \tilde{\varphi}(t,x,v) d x d v d t \big| \lesssim \epsilon_1^3 . 
\ee
where
\be
\tilde{\varphi}(t,x,v):=\psi_{\leq m/10}(|x|+|v|)+\psi_{ > m/10}(|x|+|v|) \psi_{\geq m-10}(|x+t\hat{v}|-|t|).
\ee

Now, we focus on the case when  $(x,v)$ is large    and close to  the light cone case. 
  From the estimate (\ref{2020april22eqn21}) and  the estimate of the coefficient (\ref{sepeqn88}) in Lemma \ref{decompositionofderivatives}, we identify the leading coefficients as follows, 
\be\label{2020april22eqn31}
|  \tilde{C}(t,x,v) |\lesssim (1+|v|)^{-1-c(\rho_2)},   \quad \tilde{C}(t,x,v):=  C_{ \tilde{\mathcal{L}}_1 \tilde{\mathcal{L}}_2 ;\tilde{L}, \mu\nu}^{  {\mathcal{L}};\Gamma_1,\Gamma_2;\rho, \rho_1, \rho_2}(x,v)- c_{\rho_2}^{\tilde{L}}(t,x,v) \hat{v}_{\mu} \hat{v}_{\nu},
\ee 
where $\hat{v}_0:=1.$

 Motivated from the above decompositions of coefficients and the decompositions of $ \hat{v}_{\mu}\hat{v}_{\nu} P_k( L h_{\mu\nu})$ in \eqref{april26eqn1} and $ \hat{v}_{\mu}\hat{v}_{\nu} P_k( \p_t L h_{\mu\nu})$  in \eqref{april26eqn11}, we have
\[
 \big| \int_{t_1}^{t_2} \int_{\R^3} \int_{\R^3}  \big(\omega^{\mathcal{L}}_{\rho} ( x, v)   \big)^2   \Lambda^{\rho}{u}^{\mathcal{L}}(t,x,v)  C_{ \tilde{\mathcal{L}}_1 \tilde{\mathcal{L}}_2 ;\tilde{L}, \mu\nu}^{  {\mathcal{L}};\Gamma_1,\Gamma_2;\rho, \rho_1, \rho_2}(x,v)  \Lambda^{\rho_1} u^{\Gamma_1 \tilde{\mathcal{L}}_1}(t,x,v)  P_k( \tilde{L}  \Gamma_2 \tilde{\mathcal{L}}_2h_{\mu\nu}) (t,x+t\hat{v}) 
  \]
  \[
   \times \psi_{> m/10}(|x|+|v|)   \psi_{< m-10}(|x+t\hat{v}|-|t|)    d x d v d t\big| 
  \]
\be\label{2020april22eqn32}
 \lesssim \sum_{\tilde{h}\in\{F, \underline{F}, \omega_k,  \vartheta_{mn}\}}
  \sum_{\mu\in\{+,-\}} \big| I_{\tilde{h};k}^\mu \big|  + \mathfrak{R}_k + \big|II_k \big| +  \big|III_k \big|, 
\ee
where
\[
  I_{\tilde{h};k}^\mu=  \int_{t_1}^{t_2} \int_{\R^3} \int_{\R^3}  \big(\omega^{\mathcal{L}}_{\rho} ( x, v)   \big)^2  c_{\rho_2}^{\tilde{L}}(t,x,v)  \Lambda^{\rho}{u}^{\mathcal{L}}(t,x,v)\Lambda^{\rho_1} u^{\Gamma_1 \tilde{\mathcal{L}}_1}(t,x,v)    \big(T_{k;\mu}^{\tilde{h};
 \alpha }(v) U^{\tilde{h}^{\tilde{L}\tilde{\mathcal{L}}_2}  }\big)^{\mu} (t,x+t\hat{v})
\]
\be\label{april26eqn14}
\times     \psi_{> m/10}(|x|+|v|)   \psi_{< m-10}(|x+t\hat{v}|-|t|) d x d v d t
\ee
\[
 \mathfrak{R}_k = \sum_{\mu\in \{+,-\}}
 \int_{t_1}^{t_2} \int_{\R^3} \int_{\R^3}  \big(\omega^{\mathcal{L}}_{\rho} ( x, v)   \big)^2  c_{\rho_2}^{\tilde{L}}(t,x,v)  \Lambda^{\rho}{u}^{\mathcal{L}}(t,x,v)    \Lambda^{\rho_1} u^{\Gamma_1 \tilde{\mathcal{L}}_1}(t,x,v) \big[  \p_\alpha \mathcal{R}_k 
\]
\be
+ a_\alpha \big(T_{k;\mu}^{\tilde{h}}(v)  \square    \tilde{L}  \tilde{\mathcal{L}}_2  h_{\mu\nu}\big)^{\mu}\big] \psi_{> m/10}(|x|+|v|)  \psi_{< m-10}(|x+t\hat{v}|-|t|)     d x d v d t,
\ee
\[
II_k =  \int_{t_1}^{t_2} \int_{\R^3} \int_{\R^3}  \big(\omega^{\mathcal{L}}_{\rho} ( x, v)   \big)^2   \Lambda^{\rho}{u}^{\mathcal{L}}(t,x,v)  \tilde{C}(t,x,v) \Lambda^{\rho_1} u^{\Gamma_1 \tilde{\mathcal{L}}_1}(t,x,v) \p_\alpha  P_k( \tilde{L} \tilde{\mathcal{L}}_2  h_{\mu\nu})(t,x+t\hat{v})
\]
\be\label{2020april22eqn34}
\times \psi_{\leq H(5)\delta_1 m}(v)\psi_{> m/10}(|x|+|v|)  \psi_{< m-10}(|x+t\hat{v}|-|t|)    d x d v d t.
\ee
\[
III_k =  \int_{t_1}^{t_2} \int_{\R^3} \int_{\R^3}  \big(\omega^{\mathcal{L}}_{\rho} ( x, v)   \big)^2   \Lambda^{\rho}{u}^{\mathcal{L}}(t,x,v)  \tilde{C}(t,x,v) \Lambda^{\rho_1} u^{\Gamma_1 \tilde{\mathcal{L}}_1}(t,x,v) \p_\alpha  P_k( \tilde{L} \tilde{\mathcal{L}}_2  h_{\mu\nu})(t,x+t\hat{v})
\]
\be\label{2020april27eqn34}
\times \psi_{> H(5)\delta_1 m }(v) \psi_{> m/10}(|x|+|v|)  \psi_{< m-10}(|x+t\hat{v}|-|t|)   d x d v d t,
\ee
 where $a_0=1$ and $a_i=0$ and the Fourier multiplier operators $T_{k;\mu}^{\tilde{h};
 \alpha }(v)$ are defined as follow, 
 \be\label{2020aug31eqn1}
 T_{k;\mu}^{\tilde{h};
0 }(v)= -i T_{k;\mu}^{\tilde{h} }(v)  \d , \quad T_{k;\mu}^{\tilde{h};
j}(v):= \p_j T_{k;\mu}^{\tilde{h} }(v) .
 \ee

Recall  (\ref{june16eqn51}). For the main term of  $\Lambda_{1}[(v^0)^{-1}]=-(1+|v|^2)^{-1}\Lambda_{1}[ v^0 ]$, we have 
 \be\label{2020sep1eqn41}
P_k h_{00} -  \hat{v}_i \hat{v}_j   P_k h_{ij}= G_k + B_k,\quad B_k:= P_k \underline{F} -\hat{v}_i \hat{v}_j  R_i R_j P_k \underline{F}= \sum_{\mu\in\{+,-\}} \big(\tilde{T}_{k;\mu}^{ \underline{F}}(v) U^{ \underline{F}}   \big)^{\mu},
  \ee
  where  $G_k$ denotes the good part of $P_k h_{00} -  \hat{v}_i \hat{v}_j   P_k h_{ij}$, which doesn't depend on $\underline{F}$, and the symbol of the Fourier multiplier operator is given as follows,
  \be\label{2022may30eqn1}
  \tilde{T}_{k;\mu}^{\underline{F}}(v)(\xi)= c_{\mu}|\xi|^{-1}( 1-(\hat{v}\cdot \xi/|\xi|)^2)\psi_k(\xi). 
  \ee

Similar to the decomposition \eqref{2020april22eqn32}, from \eqref{2020sep1eqn41}, we have 
\[
\big| \int_{t_1}^{t_2} \int_{\R^3} \int_{\R^3}  \big(\omega^{\mathcal{L}}_{\rho} ( x, v)   \big)^2   \Lambda^{\rho}{u}^{\mathcal{L}}(t,x,v) P_k(\tilde{L}\mathcal{L}_2  \Lambda_{1}[(v^0)^{-1}] ) (t, x+t\hat{v} )  {C}_{\mathcal{L}_1\mathcal{L}_2;\tilde{L} }^{\mathcal{L};\rho,\rho_1,\rho_2}(x,v)   \cdot \nabla_x \Lambda^{\rho_1}  u^{\mathcal{L}_1}(t,x,v)  
\]
\[
\times  \psi_{> m/10}(|x|+|v|)  \psi_{< m-10}(|x+t\hat{v}|-|t|)    d x d v d t \big| \lesssim  \big| J_k^1 \big|+ \big|J_k^2 \big|,
\]
where
\[
J_k^1=  \int_{t_1}^{t_2} \int_{\R^3} \int_{\R^3}  \big(\omega^{\mathcal{L}}_{\rho} ( x, v)   \big)^2   \Lambda^{\rho}{u}^{\mathcal{L}}(t,x,v)\langle v \rangle^{-1}  (\tilde{L}    G_k ) (t, x+t\hat{v} )  {C}_{\mathcal{L}_1\mathcal{L}_2;\tilde{L} }^{\mathcal{L};\rho,\rho_1,\rho_2}(x,v)   \cdot \nabla_x \Lambda^{\rho_1}  u^{\mathcal{L}_1}(t,x,v) 
\]
\be
\times\psi_{ > m/10}(|x|+|v|)    \psi_{< m-10}(|x+t\hat{v}|-|t|)     d x d v d t,
\ee
\[
 J_k^2=\sum_{\mu\in \{+,-\}}  \int_{t_1}^{t_2} \int_{\R^3} \int_{\R^3}  \big(\omega^{\mathcal{L}}_{\rho} ( x, v)   \big)^2   \Lambda^{\rho}{u}^{\mathcal{L}}(t,x,v)\langle v \rangle^{-1}  \tilde{L}\big(\tilde{T}_{k;\mu}^{ \underline{F}}(v) U^{ \underline{F}}   \big)^{\mu} (t, x+t\hat{v} )  
\]
\be
\times {C}_{\mathcal{L}_1\mathcal{L}_2;\tilde{L} }^{\mathcal{L};\rho,\rho_1,\rho_2}(x,v)   \cdot \nabla_x \Lambda^{\rho_1}  u^{\mathcal{L}_1}(t,x,v)  \psi_{> m/10}(|x|+|v|)  \psi_{< m-10}(|x+t\hat{v}|-|t|)      d x d v d t,
\ee

Recall (\ref{april1eqn11}) and  (\ref{aug25eqn2}). Note that, if $\tilde{c}(\rho_2)=0$, then the commutator $E_{\tilde{L},k}^{comm}$vanishes in $\mathcal{R}_k $.  From the estimate (\ref{april26eqn35}) in Lemma \ref{firstordercommutator}, the decay estimate of the density type function in Lemma \ref{decayestimateofdensity}, the decay estimate (\ref{july27eqn4}) in Lemma \ref{superlocalizedaug}, the estimate (\ref{oct7eqn73}) in Lemma \ref{harmonicgaugehigher}, and the estimate (\ref{oct4eqn41}) in Proposition \ref{fixedtimenonlinarityestimate}, we have
\be\label{2020sep1eqn42}
\begin{split}
\|\p_\alpha \mathcal{R}_k   \mathbf{1}_{\tilde{c}(\rho_2)\geq 1 }\|_{L^\infty_x} +\| \p_\alpha     G_k\|_{L^\infty} & \lesssim 2^{-m+0.999k-5k_{+}} \epsilon_1,\\ 
 \| \p_\alpha \mathcal{R}_k(t,x)\psi_{\geq m-100}(|x|)  \mathbf{1}_{\tilde{c}(\rho_2)=0}\|_{L^\infty_x} &\lesssim 2^{-10k_{+}}\min\{ 2^{-m + H(4)\delta_0m +k}, 2^{-2m + H(4 )\delta_0m   } \} \epsilon_1^2. 
\end{split}
\ee
 From the  $L^2_{x,v}-L^2_{x,v}-L^\infty_{x,v}$ type multilinear estimate, the estimate of coefficients in (\ref{sepeqn88}) and (\ref{2020april22eqn31}) and  the above estimate,  we have
 \be\label{april26eqn15}
\sum_{k\in [-2H( 4)\delta_0 m, H( 4)\delta_0 m ]\cap \mathbb{Z}}|\mathfrak{R}_k|+ |III_k|  +|J_k^1|\lesssim 2^{2H( \mathcal{L}, \rho)\delta_0 m } \epsilon_1^3.
 \ee
In the above estimate, we have  used the extra decay comes from  the smallness of coefficient $\tilde{C}(t,x,v)$ as $v$ is relatively large (see (\ref{2020april27eqn34})) in the above estimate of $III_k$. 

 It remains to estimate $ I_{\tilde{h};k}^\mu$, $II_k$,   and $J_k^2$ for the case $ k\in [-2H( 4)\delta_0 m, H(  4)\delta_0 m ]$.  We use the normal form transformation by  doing integration by parts in time  to utilize the null structure in $ I_{\tilde{h};k}^\mu$ and  $J_k^2$, see \eqref{april26eqn2} and \eqref{2022may30eqn1}.  As a result, on the Fourier side,  we have
 \[
 I_{\tilde{h};k}^\mu= \sum_{j=1,2} (-1)^{j-1}   \int_{\R^3} \int_{\R^3} \overline{\widehat{\mathcal{N}_{\tilde{h}}}( t_j, \xi, v) } e^{i  t_j\hat{v}\cdot \xi  - i t_j \mu|\xi|}  T_{k;\mu}^{\tilde{h};\alpha;1}(v,\xi)  i (\hat{v}\cdot \xi-  \mu |\xi|)^{-1}    \widehat{(V^{\tilde{h}^{\tilde{L}\tilde{\mathcal{L}}_2}})^{\mu}}( t_j, \xi)    d\xi d v  
 \]
  \be\label{2020may3eqn31}
 + \int_{t_1}^{t_2} \int_{\R^3} \int_{\R^3}  e^{i t\hat{v}\cdot \xi  - i t \mu|\xi|}     T_{k;\mu}^{\tilde{h};\alpha;1}(v,\xi) i ( \hat{v}\cdot \xi-  \mu |\xi|)^{-1}    \p_t\big(\overline{\widehat{\mathcal{N}_{\tilde{h}}}(t, \xi, v) } \widehat{(V^{\tilde{h}^{\tilde{L}\tilde{\mathcal{L}}_2}})^{\mu}}(t, \xi) \big)   d\xi d v d t
 \ee
 where
 \[
 \mathcal{N}_{\tilde{h}}(t,x,v):=\big(\omega^{\mathcal{L}}_{\rho} ( x, v)   \big)^2  c_{\rho_2}^{\tilde{L}}(x,v)  \Lambda^{\rho}{u}^{\mathcal{L}}(t,x,v)    \Lambda^{\rho_1} u^{\Gamma_1 \tilde{\mathcal{L}}_1}(t,x,v). 
 \]
Similar formula holds for $J_k^2$, we omit details here.

From the detailed formulas of symbols $ T_{k;\mu}^{\tilde{h}}(v)(\xi) $ in (\ref{april26eqn2}), after redoing the proof of lineaer decay estimate in Lemma \ref{superlocalizedaug}, we have 
\[
\sup_{v\in \R^3}\big| \int_{\R^3} e^{-i \mu t|\xi|} T_{k;\mu}^{\tilde{h};\alpha;1}(v,\xi)  i (\hat{v}\cdot \xi-  \mu |\xi|)^{-1}    \widehat{(V^{\tilde{h}^{\tilde{L}\tilde{\mathcal{L}}_2}})^{\mu}}( t_j, \xi) \psi_k(\xi) d \xi\big| \lesssim 2^{-m + H(2)\delta_1 m+ \delta_1 m   }  \epsilon_1. 
\]
From the above estimate, we know that we lose little but gain at least $2^{-m+H(N_0+3)\delta_0 m }$ from doing integration by parts in time. 

The estimate of $II_k$ is similar. Recall \eqref{2020april22eqn34}, by exploiting the fact   that $v$ is small in $II_k$, we also doing normal form transformation for $II_k$, we  lose at most $(1+|v|)^22^{|k|}$, which is less than $2^{10H(5)\delta_0 m}$ as $|v|\leq 2^{H(5)\delta_1  m}$ and $k\in  [-2H(  4)\delta_0 m, H( 4)\delta_0 m ] $ but gain at least $2^{-m+H(N_0+3)\delta_0 m }$.  

To sum up, similar to the procedure we did in the estimate of $I_{\rho;3}^{\mathcal{L}}(k,t_1, t_2)$ in Lemma \ref{notbulkcase}, we have 
\be\label{2020april27eqn41}
 \big|I_{\tilde{h};k}^\mu\big| + |II_k| +|J_k^2|\lesssim 2^{-m/2}\epsilon_1^3. 
\ee
Hence finishing the proof of our desired estimate (\ref{2020april22eqn1}).
 \end{proof}

\section{The $Z$-norm estimate for the Vlasov part}\label{ZnormVlas}
  
Recall the equation satisfied by $u^{\mathcal{L}}$  in (\ref{2020april7eqn6}) and the definition of the correction term in (\ref{correctionterm}). As a result of direct computations, after doing integration by parts in $x$ once for $\nabla_x \nabla_v^{\alpha_2}u^{\mathcal{L}_2}$,   based on the order of derivatives in $v$ and the number of vector fields in $\cup_{n}\mathcal{P}_n$ distributed,   we classify the nonlinearity of the equation satisfied by $\nabla_v^{\alpha} \widehat{u^{\mathcal{L}}}(t, 0,v)-    \nabla_v\cdot {u^{\mathcal{L};\alpha}_{corr}}(t,v) $ as follows, 
\be\label{2020may4eqn69}
\p_t \big( \nabla_v^{\alpha} \widehat{u^{\mathcal{L}}}(t, 0,v)-    \nabla_v\cdot {u^{\mathcal{L};\alpha}_{corr}}(t,v)  \big)=   \sum_{|\alpha_1|+|\alpha_2|\leq |\alpha|,\mathcal{L}_1\circ \mathcal{L}_2\preceq \mathcal{L}}  \int_{\R^3}\mathfrak{R}_{\mathcal{L};\mathcal{L}_1, \mathcal{L}_2}^{\alpha;\alpha_1, \alpha_2}(t,x,v) d x ,
\ee
where, if $ \alpha_2 = \alpha,\mathcal{L}_2=\mathcal{L}  $, we have  
\[\mathfrak{R}_{\mathcal{L};\mathcal{L}_1, \mathcal{L}_2}^{\alpha;\alpha_1, \alpha_2}(t,x,v)= \h   (\nabla_v -t \nabla_v \hat{v}\cdot \nabla_x)\cdot \big( (v^0)^{-1}   v_{\alpha} v_{\beta} \nabla_x g^{\alpha\beta}\big)(t, x+t\hat{v}, v)\nabla_v^{\alpha }u^{\mathcal{L}}(t,x,v)
\]
\[
+ v\cdot\nabla_x  \big( \Lambda_{\geq 1}[(v^0)^{-1}]\big)(t, x+t\hat{v})  \nabla_v^{\alpha }u^{\mathcal{L}}(t,x,v),
\]
and if $ |\alpha_2| +|\mathcal{L}_2|< |\alpha|+ |\mathcal{L}| $, then $\mathfrak{R}_{\mathcal{L};\mathcal{L}_1, \mathcal{L}_2}^{\alpha;\alpha_1, \alpha_2}(t,x,v)$ is defined as follows, 
\[
\mathfrak{R}_{\mathcal{L};\mathcal{L}_1, \mathcal{L}_2}^{\alpha;\alpha_1, \alpha_2}(t,x,v)=c_{1}(v) \nabla_v^{\alpha_1}( v\cdot\nabla_x  \big( \Lambda_{\geq 1}[\mathcal{L}_1(v^0)^{-1}]\big))(t, x+t\hat{v})  \nabla_v^{\alpha_2} u^{\mathcal{L}_2}(t,x,v) 
\]
\[
 + c_{2}(v) \nabla_v^{\alpha_1}\big( \mathcal{L}_1((v^0)^{-1}v_\alpha v_{\beta} \nabla_x g^{\alpha\beta} ) \big)(t,x+t\hat{v},v)\nabla_v \nabla_v^{\alpha_2} u^{\mathcal{L}_2}(t,x,v)
\]
\be\label{2020may4eqn11}
 + c_3(v)t(\nabla_v\hat{v}\cdot\nabla_x) \nabla_v^{\alpha_1}\big( \mathcal{L}_1((v^0)^{-1}v_\alpha v_{\beta} \nabla_x g^{\alpha\beta} ) \big)(t, x+t\hat{v}, v)  \nabla_v^{\alpha_2} u^{\mathcal{L}_2}(t,x,v).
 \ee
For simplicity of notation, we suppress the dependence of coefficients $c_i(v),i\in\{1,2,3\},$ with respect to $\alpha, \alpha_1, \alpha_2, \mathcal{L},\mathcal{L}_1,$ and $\mathcal{L}_2$.

\begin{proposition}\label{estimateofzerofrequency}
Under the bootstrap assumptions \textup{(\ref{BAmetricold})} and \textup{(\ref{BAvlasovori})}, the following estimate holds for any $\mathcal{L}\in \mathcal{P}_n, $ and any  time $\tau \in[0,T)$,  
\be\label{oct11eqn6}
\|  \widetilde{ \omega^{\mathcal{L}}_{\alpha} }(v) {u^{\mathcal{L};\alpha}_{corr}}(\tau ,v)  \|_{L^2_v}+\|   \widetilde{ \omega^{\mathcal{L}}_{\alpha} }(v)\big(\nabla_v^\alpha \widehat{u^{\mathcal{L} }}(\tau , 0, v) - \nabla_v\cdot {u^{\mathcal{L};\alpha}_{corr}}(\tau ,v)\big) \|_{L^2_v}   \lesssim \epsilon_0,
\ee
where $ {u^{\mathcal{L};\alpha}_{corr}}(t,v)$ was defined   in  \textup{(\ref{correctionterm})}.  
\end{proposition}
\begin{proof}

\noindent $\oplus$ The estimate of correction term ${u^{\mathcal{L};\alpha}_{corr}}(\tau ,v)$.

Recall (\ref{correctionterm}).  
From   the pointwise  estimate   (\ref{definitionofcoefficient1}) and (\ref{2020aug31eqn31}) and the decay estimate (\ref{2020julybasicestiamte}) in Lemma \ref{basicestimates}, the following estimate holds after  using the equality (\ref{definitionofcoefficient1}) once and gaining the decay rate of distance with respect to the light cone, 
\be\label{2020may4eqn2}
\| \widetilde{ \omega^{\mathcal{L}}_{\alpha} }(v) {u^{\mathcal{L};\alpha}_{corr}}(\tau ,v)  \|_{L^2_v}\lesssim \epsilon_0 + \sum_{|\mathcal{L}|+| \rho|\leq N_0}  \int_{0}^{\tau } (1+t)^{-2+2H(\mathcal{L}, \rho)\delta_0 } \epsilon_1^2 d t \lesssim  \epsilon_0. 
\ee

\noindent $\oplus$ The estimate of $ \nabla_v^\alpha \widehat{u^{\mathcal{L} }}(\tau, 0, v) - \nabla_v\cdot {u^{\mathcal{L};\alpha}_{corr}}(\tau,v) $.

Recall (\ref{2020may4eqn11}). Note that, from  (\ref{eqq10}),  we can represent $\nabla_v$ as a linear combination of second set of vector fields as follows, 
\be\label{2020may4eqn21}
\nabla_v = \sum_{\iota\in \mathcal{S}, |\iota|=1} c_{\iota}(x,v) \Lambda^{\iota}, \quad |c_{\iota}(x,v)|\lesssim   \langle x \rangle \langle  v\rangle.   
\ee
From the equality (\ref{sepeqn610}) in Lemma \ref{decompositionofderivatives}, if  $|\alpha_2|+|\mathcal{L}_2|< |\alpha| +|\mathcal{L}|$, we reduce $\mathfrak{R}_{\mathcal{L};\mathcal{L}_1, \mathcal{L}_2}^{\alpha;\alpha_1, \alpha_2}(t,x,v)$ further as follows, 
\[
\mathfrak{R}_{\mathcal{L};\mathcal{L}_1, \mathcal{L}_2}^{\alpha;\alpha_1, \alpha_2}(t,x,v)=\sum_{|\tilde{\mathcal{L}}_1|\leq |\mathcal{L}_1|+|\alpha_1|}\sum_{\rho, \iota\in \mathcal{S}, |\rho|\leq |\alpha_2|, |\iota|\leq 1} \tilde{c}^1_{\rho,\iota}(t,x,v) v\cdot \nabla_x  \tilde{\mathcal{L}}_1(\Lambda_{\geq 1}[(v^0)^{-1}])(t, x+t\hat{v},v) 
\]
\[
\times \Lambda^{\rho} u^{\mathcal{L}_2}(t,x,v) +  \tilde{c}^2_{\rho,\iota}(t,x,v) \tilde{\mathcal{L}}_1( (v^0)^{-1}v_\alpha v_{\beta} \nabla_x g^{\alpha\beta} )(t, x+t\hat{v},v) \Lambda^{\iota}  \Lambda^{\rho} u^{\mathcal{L}_2}(t,x,v)
 \]
 \be\label{2020may4eqn22}
  +  \tilde{c}^3_{\rho,\iota}(t,x,v) t\nabla_x  \tilde{\mathcal{L}}_1( (v^0)^{-1}v_\alpha v_{\beta} \nabla_x g^{\alpha\beta} )(t, x+t\hat{v},v) \Lambda^{\rho} u^{\mathcal{L}_2}(t,x,v),
 \ee
where in the above equality we suppress the dependence of coefficients with respect to the order of derivatives. From the estimate (\ref{sepeqn88}) in Lemma \ref{decompositionofderivatives} and the estimate (\ref{2020may4eqn21}), the rough estimate holds 
\be\label{2020may4eqn23} 
\sum_{i=1,2,3}| \tilde{c}^i_{\rho,\iota}(t,x,v)|\lesssim (\langle x \rangle +\langle v \rangle)^{2|\alpha |+10} .
\ee

Let $t\in[2^{m-1}, 2^m] \subset[0,T), m\in\mathbb{Z}_+$. Based on the total number of vector fields act on the Vlasov part, we separate into two cases as follows. 

\textbf{Subcase $1$:\qquad} If $|\alpha_2|+|\mathcal{L}_2|\geq N_0/2 + 5 $.

 For this case, we know that the total   number of vector fields act on the perturbed metric is less than $N_0/2-4$. From the equality (\ref{definitionofcoefficient1}),   the estimate  (\ref{2020aug31eqn31}), the estimate of coefficients in (\ref{2020may4eqn23}),   the decay estimate of perturbed metric (\ref{2020julybasicestiamte}) in Lemma \ref{basicestimates}, and the decay estimate of high order terms in  Lemma \ref{highorderterm2}
and  Lemma \ref{highorderterm5}, we have
\[
\| \int_{\R^3} \widetilde{ \omega^{\mathcal{L}}_{\alpha} }(v) \mathfrak{R}_{\mathcal{L};\mathcal{L}_1, \mathcal{L}_2}^{\alpha;\alpha_1, \alpha_2}(t,x,v) d x\|_{L^2_v}\lesssim \|\langle x \rangle^2  \widetilde{ \omega^{\mathcal{L}}_{\alpha} }(v)\mathfrak{R}_{\mathcal{L};\mathcal{L}_1, \mathcal{L}_2}^{\alpha;\alpha_1, \alpha_2}(t,x,v)\|_{L^2_{x,v}}
\]
 \be\label{2020may4eqn32}
\lesssim\sum_{\rho\in \mathcal{S}, |\rho|\leq |\alpha|} 2^{-2m+3d(|\mathcal{L}|+3, |\rho|)\delta_0 m }\epsilon_1^3. 
\ee

\textbf{Subcase $2$:\qquad} If $|\alpha_2|+|\mathcal{L}_2|\leq N_0/2 + 4 $.

From the $L^2$-estimate of the high order terms in Lemma \ref{highorderterm2} and Lemma \ref{highorderterm3}, and the decay estimate of density type function (\ref{densitydecay})  in Lemma \ref{decayestimateofdensity}, we have
\be\label{2020may4eqn65}
\|\langle x \rangle^2  \widetilde{ \omega^{\mathcal{L}}_{\alpha} }(v)\Lambda_{\geq 3}[\mathfrak{R}_{\mathcal{L};\mathcal{L}_1, \mathcal{L}_2}^{\alpha;\alpha_1, \alpha_2}](t,x,v)\|_{L^2_{x,v}} \lesssim 2^{-2m/3-m/2+3d(|\mathcal{L}|+3, |\rho|)\delta_0 m } \epsilon_1^3.
\ee

For the quadratic terms of $\mathfrak{R}_{\mathcal{L};\mathcal{L}_1, \mathcal{L}_2}^{\alpha;\alpha_1, \alpha_2}(t,x,v)$, after doing  dyadic decomposition for the metric component and using (\ref{2020may4eqn21}) and the equality (\ref{sepeqn610}) in Lemma \ref{decompositionofderivatives}, from the estimate \eqref{sepeqn88} in Lemma \ref{decompositionofderivatives}, we have 
\[
\Lambda_{2}[\mathfrak{R}_{\mathcal{L};\mathcal{L}_1, \mathcal{L}_2}^{\alpha;\alpha_1, \alpha_2}](t,x,v) = \sum_{k\in \mathbb{Z}} \sum_{i=1,2,3}\sum_{|\tilde{\mathcal{L}}_1|\leq |\mathcal{L}_1|+|\alpha_1|}   \sum_{\rho, \iota\in \mathcal{S}, |\rho|\leq |\alpha_2|, |\iota|\leq 1}   H_{k;i}^{\tilde{\mathcal{L}}_1, \rho, \iota}(t,x,v),
\]
where
\be\label{2020may4eqn48}
\begin{split}
 H_{k;1}^{\tilde{\mathcal{L}}_1, \rho, \iota}(t,x,v)&=  {c}_{\alpha\beta}^{\rho,\iota,1}(t,x,v) P_k(\nabla_x \tilde{\mathcal{L}}_1 h_{\alpha\beta})(t, x+t\hat{v}) \Lambda^{\rho} u^{\mathcal{L}_2}(t,x,v),\\ 
  H_{k;2}^{\tilde{\mathcal{L}}_1, \rho, \iota}(t,x,v)&={c}_{\alpha\beta}^{\rho,\iota,2}(t,x,v) P_k(\nabla_x \tilde{\mathcal{L}}_1 h_{\alpha\beta})(t, x+t\hat{v}) \Lambda^{\iota\circ \rho} u^{\mathcal{L}_2}(t,x,v),\\ 
   H_{k;3}^{\tilde{\mathcal{L}}_1, \rho, \iota}(t,x,v)&=t {c}_{\alpha\beta}^{\rho,\iota,3}(t,x,v) P_k(\nabla_x^2 \tilde{\mathcal{L}}_1 h_{\alpha\beta})(t, x+t\hat{v})    \Lambda^\rho  u^{\mathcal{L}_2}(t,x,v),\\ 
   \sum_{i=1,2,3}|{c}_{\alpha\beta}^{\rho,\iota,i}(t,x,v) | + |t\p_t{c}_{\alpha\beta}^{\rho,\iota,i}(t,x,v) |&\lesssim (\langle x \rangle +\langle v \rangle)^{2|\alpha |+10}
\end{split}
\ee

From the $L^2_{x,v}-L^2_xL^\infty_v-L^\infty_x L^2_v$ type multilinear estimate and the decay estimate of density type function (\ref{densitydecay})  in Lemma \ref{decayestimateofdensity}, we can rule out the case   $k\notin[-m/2, {m/10}]$ as follows, 
\be\label{2020may4eqn66}
\begin{split}
&\sum_{i=1,2,3}\sum_{k\in \mathbb{Z},  k\notin[-m/2, {m/10}]} \|\int_{\R^3} \widetilde{ \omega^{\mathcal{L}}_{\alpha} }(v)   H_{k;i}^{\tilde{\mathcal{L}}_1, \rho, \iota}(t,x,v) d x \|_{L^2_{v}}   \\
&  \lesssim \sum_{k\in \mathbb{Z},  k\notin[-m/2, {m/10}]} \big( 2^{-3m/2+k/2} + 2^{-m/2+3k/2}\big) 2^{-10k_{+}+ 3 d(|\mathcal{L}|+3, |\rho|)\delta_0 m }\epsilon_1^3 \\
& \lesssim 2^{-11m/10}\epsilon_1^3. 
\end{split}
\ee
 
For the case $k\in[-m/2, {m/10}]$, same as what we did in   the proof of Lemma \ref{notbulkcase} , we do normal form transformation.  Note that  the price of using normal form transformation is minor because the weight function we used in the low order energy is much weaker than the weight function  we propagated in the energy estimate, see \eqref{weightfunctions2} and \eqref{weightfunctions3}. 

Recall (\ref{2020may4eqn48}). As a typical example, which is also the most difficult case because of the extra loss of ``$t$'', we estimate $ H_{k;3}^{\tilde{\mathcal{L}}_1, \rho, \iota}(t,x,v)$ in details. With minor modifications, the estimate of $ H_{k;i}^{\tilde{\mathcal{L}}_1, \rho, \iota}(t,x,v), i\in\{1,2\}$ follows in the same way. 

For any $t_1, t_2\in [2^{m-1}, 2^m]\subset[0, T], m\in\mathbb{Z}_{+},k\in[-m/2, {m/10}]\cap \Z$, after doing normal form transformation,  we have
 \[
   \big|\int_{t_1}^{t_2}\int_{\R^3}  H_{k;3}^{\tilde{\mathcal{L}}_1, \rho, \iota}(t,x,v) d x d t \big|\lesssim 
   \sum_{  \mu\in\{+,-\}}  \big| End^{\tilde{\mathcal{L}}_1, \rho, \iota}_{k, \mu}(t_1, t_2 , v) \big|
    + \big| \mathcal{H}^{\tilde{\mathcal{L}}_1, \rho, \iota}_{k;  \mu;1}(t_1, t_2 , v) \big| + \big| \mathcal{H}^{\tilde{\mathcal{L}}_1, \rho, \iota}_{k;  \mu; 2}(t_1, t_2 , v) \big|,  
 \]
 where 
 \be\label{2022may23eqn31}
 \begin{split}
 End^{\tilde{\mathcal{L}}_1, \rho, \iota}_{k; \mu }(t_1, t_2 , v) &: =\sum_{l=1,2}(-1)^{l} \int_{\R^3}   t_l {c}_{\alpha\beta}^{\rho,\iota,3}(t_l,x,v) \widetilde{M}^\mu_k( (U^{\tilde{\mathcal{L}}_1 h_{\alpha\beta}})^{\mu} )(t_l, x+t_l\hat{v}, v)   \Lambda^\rho  u^{\mathcal{L}_2}(t_l,x,v)   d x,\\ 
 \mathcal{H}^{\tilde{\mathcal{L}}_1, \rho, \iota}_{k; \mu;1}(t_1, t_2 , v) &: =  \int_{t_1}^{t_2} \int_{\R^3} \Lambda^\rho u^{\mathcal{L}_2}(t ,x,v) \big[ t   {c}_{\alpha\beta}^{\rho,\iota,3}(t ,x,v) \widetilde{M}^\mu_k(v)( (e^{-it \d}\p_t V^{\tilde{\mathcal{L}}_1 h_{\alpha\beta}} )^{\mu} )(t , x+t \hat{v})\\
 &+ \p_t\big( t  {c}_{\alpha\beta}^{\rho,\iota,3}(t ,x,v)\big) \widetilde{M}^\mu_k(v)(   (U^{\tilde{\mathcal{L}}_1 h_{\alpha\beta}})^{\mu}  )(t , x+t \hat{v} )\big]d x dt, \\
 \mathcal{H}^{\tilde{\mathcal{L}}_1, \rho, \iota}_{k;2}(t_1, t_2 , v)&: = \int_{t_1}^{t_2}  \int_{\R^3}    t   {c}_{\alpha\beta}^{\rho,\iota,3} (t ,x,v)  \widetilde{M}^\mu_k(v)(   (U^{\tilde{\mathcal{L}}_1 h_{\alpha\beta}})^{\mu}  )(t , x+t \hat{v} )   \p_t     \Lambda^\rho  u^{\mathcal{L}_2}(t ,x,v)  d x d t,
 \end{split}
 \ee
 and the Fourier multiplier operator $\widetilde{M}^\mu_k(v)$ is given as follows, 
\[
\widetilde{M}^\mu_k(v)(f)(x):=\int_{\R^3} e^{i x\cdot \xi} \frac{ \xi \otimes \xi \psi_k(\xi) }{|\xi|\big(\mu |\xi| - \hat{v}\cdot \xi\big) } \widehat{f}(\xi) d \xi. 
\]

From the $L^2_{x,v}-L^2_xL^\infty_v-L^\infty_x L^2_v$ type multilinear estimate and the decay estimate of density type function (\ref{densitydecay})  in Lemma \ref{decayestimateofdensity}, we have
 \be\label{2022may23eqn32}
 \big\| \widetilde{ \omega^{\mathcal{L}}_{\alpha} }(v)   End^{\tilde{\mathcal{L}}_1, \rho, \iota}_{k, \mu}(t_1, t_2 , v) \big\|_{L^2_{v}} \lesssim 2^{-m/2+(1/2-\gamma)k_{-}-10k_{+} + 3 d(|\mathcal{L}|+3, |\rho|)\delta_0 m }\epsilon_1^3 . 
\ee

Similarly, by using the same strategy,  from the bilinear estimate \eqref{bilineardensitylargek} for $\Lambda_{1}[\p_t V^{\tilde{\mathcal{L}}_1 h_{\alpha\beta}}]$ and  the estimate \eqref{oct4eqn41} in Proposition   \ref{fixedtimenonlinarityestimate} for $\Lambda_{\geq 2}[\p_t V^{\tilde{\mathcal{L}}_1 h_{\alpha\beta}}]$, we have
 \be\label{2022may23eqn33}
 \big\| \widetilde{ \omega^{\mathcal{L}}_{\alpha} }(v)    \mathcal{H}^{\tilde{\mathcal{L}}_1, \rho, \iota}_{k; \mu;1}(t_1, t_2 , v) \big\|_{L^2_{v}} \lesssim 2^{-m/2+(1/2-\gamma)k_{-}-10k_{+} + 3 d(|\mathcal{L}|+3, |\rho|) \delta_0 m }\epsilon_1^3 + 2^{-m+k_{+}/2}\epsilon^3_1.
\ee

Lastly, we estimate $\mathcal{H}^{\tilde{\mathcal{L}}_1, \rho, \iota}_{k;2}(t_1, t_2 , v)$. Due to the fact that we have $|\rho|+|\mathcal{L}_2|\leq |\alpha_2|+|\mathcal{L}_2|\leq N_0/2+4$ for the case are considering. Recall  (\ref{july5eqn41}) and the equality (\ref{march31eqn1}) in Lemma \ref{innerproduct}. From the decay estimate \eqref{densitydecay} in Lemma \ref{decayestimateofdensity}, the following estimate holds after distributing derivative and put the metric part in $L^\infty$,
\[
\|   \langle v \rangle^{20}\widetilde{ \omega^{\mathcal{L}}_{\alpha} }(v)  {c}_{\alpha\beta}^{\rho,\iota,3} (t ,x-t\hat{v},v) \p_t     \Lambda^\rho  u^{\mathcal{L}_2}(t ,x-t\hat{v},v)  \|_{L^\infty_x L^2_v} \lesssim 2^{-5m/2+2 d(|\mathcal{L}|+3, |\rho|) \delta_0 m }\epsilon_1^2.
\]
From the above estimate and the Cauchy-Schwarz inequality, we have 
 \be\label{2022may23eqn35}
 \big\| \widetilde{ \omega^{\mathcal{L}}_{\alpha} }(v)    \mathcal{H}^{\tilde{\mathcal{L}}_1, \rho, \iota}_{k; \mu;2}(t_1, t_2 , v) \big\|_{L^2_{v}} \lesssim 2^{-m/2+(1/2-\gamma)k_{-}-10k_{+} + 3  d(|\mathcal{L}|+3, |\rho|)\delta_0 m }\epsilon_1^3  .
\ee
After combining the above obtained estimates, we have 
\be\label{2020may4eqn61}
\sum_{i=1,2,3} \sum_{  k\in[-m/2, {m/10}]\cap \mathbb{Z}}  \big\| \int_{t_1}^{t_2}\int_{\R^3}  \widetilde{ \omega^{\mathcal{L}}_{\alpha} }(v)  H_{k;i}^{\tilde{\mathcal{L}}_1, \rho, \iota}(t,x,v) d x d t \big\|_{L^2_{v}} \lesssim 2^{-m/3}\epsilon_1^3. 
\ee
To sum up, from the estimates (\ref{2020may4eqn32}), (\ref{2020may4eqn65}), (\ref{2020may4eqn66}), and (\ref{2020may4eqn61}), we have
\[
\sum_{|\alpha_1|+|\alpha_2|\leq |\alpha|, \mathcal{L}_1\circ\mathcal{L}_2\preceq \mathcal{L}}\| \int_{0}^{\tau}\int_{\R^3} \widetilde{ \omega^{\mathcal{L}}_{\alpha} }(v) \mathfrak{R}_{\mathcal{L};\mathcal{L}_1, \mathcal{L}_2}^{\alpha;\alpha_1, \alpha_2}(t,x,v) d x d \tau \|_{L^2_v} \lesssim \epsilon_1^3. 
\]
Recall (\ref{2020may4eqn69}). Our desired estimate (\ref{oct11eqn6}) holds from the above estimate and (\ref{2020may4eqn2}). 
\end{proof}

\end{document}